\documentclass[11pt]{amsart}
\usepackage{amstext}
\usepackage{amsthm,euscript}
\usepackage{amscd}
\usepackage{amssymb}
\usepackage{amsmath}
\usepackage{mathtools}
\usepackage[all]{xy}
\usepackage{paralist}
\usepackage{enumitem}
\usepackage{moreenum}
\usepackage{scalerel,stackengine}
\usepackage[colorlinks=true, pdfstartview=FitV, linkcolor=blue,
citecolor=blue,urlcolor=blue,breaklinks=true]{hyperref}
\usepackage{subfiles} 

\usepackage{lmodern} 
\usepackage{xcolor}
\usepackage{tikz-cd}
\usepackage{arydshln}

\allowdisplaybreaks

\usepackage{forloop}
\newcounter{fonts}
\let\eeee\edef
\forloop{fonts}{1}{\value{fonts} < 27}{%
\expandafter\eeee\csname \Alph{fonts}\Alph{fonts}\endcsname{\noexpand\mathbb{\Alph{fonts}}} 
\expandafter\eeee\csname bb\Alph{fonts}\endcsname{\noexpand\mathbb{\Alph{fonts}}} 
\expandafter\eeee\csname b\Alph{fonts}\endcsname{\noexpand\mathbf{\Alph{fonts}}} 
\expandafter\eeee\csname u\Alph{fonts}\endcsname{\noexpand\mathbf{\Alph{fonts}}} 
\expandafter\eeee\csname bf\Alph{fonts}\endcsname{\noexpand\mathbf{\Alph{fonts}}} 
\expandafter\eeee\csname bf\alph{fonts}\endcsname{\noexpand\mathbf{\alph{fonts}}} 
\expandafter\eeee\csname u\Alph{fonts}\Alph{fonts}\endcsname{\noexpand\mathbf{\Alph{fonts}}_0} 
\expandafter\eeee\csname c\Alph{fonts}\endcsname{\noexpand\mathcal{\Alph{fonts}}} 
\expandafter\eeee\csname cal\Alph{fonts}\endcsname{\noexpand\mathcal{\Alph{fonts}}} 
\expandafter\eeee\csname fr\Alph{fonts}\endcsname{\noexpand\mathfrak{\Alph{fonts}}} 
\expandafter\eeee\csname f\Alph{fonts}\endcsname{\noexpand\mathfrak{\Alph{fonts}}} 
\expandafter\eeee\csname fr\alph{fonts}\endcsname{\noexpand\mathfrak{\alph{fonts}}} 
\expandafter\eeee\csname s\Alph{fonts}\endcsname{\noexpand\mathscr{\Alph{fonts}}} 
\expandafter\eeee\csname eu\Alph{fonts}\endcsname{\noexpand\EuScript{\Alph{fonts}}} 
\expandafter\eeee\csname sf\Alph{fonts}\endcsname{{\noexpand\sf \Alph{fonts}}} 
\expandafter\eeee\csname sf\alph{fonts}\endcsname{{\noexpand\sf \alph{fonts}}} 
\expandafter\eeee\csname ul\Alph{fonts}\endcsname{\noexpand\underline{\noexpand\rm \Alph{fonts}}} 
}

\swapnumbers 

\newcommand{\arxiv}[1]{\href{http://arxiv.org/abs/#1}{\tt arXiv:\nolinkurl{#1}}}
\newtheorem{thm}[subsection]{Theorem}
\newtheorem*{thm*}{Theorem}

\newtheorem{prop}[subsection]{Proposition}
\newtheorem{lem}[subsection]{Lemma}

\newtheorem{cor}[subsection]{Corollary}

\theoremstyle{definition}
\newtheorem{defn}[subsection]{Definition}
\newtheorem*{defn*}{Definition}

\newtheorem{remark}[subsection]{Remark}

\newtheorem{remarks}[subsection]{Remarks}

\newtheorem{example}[subsection]{Example}

\newtheorem{question}[subsection]{Question}

\numberwithin{equation}{subsection}

\newcommand\sm{\smallskip}

\newcommand{\lv}[1]{} 

\newcounter{question} 
\newcounter{suggestion}

\newcommand{\simlgr}{\buildrel \sim \over \longrightarrow}

\newcommand{\wdh}{\widehat}

\newcommand{\fppf}{_\mathrm{fppf}}

\newcommand{\longto}{\longrightarrow}
\newcommand{\we}{\wedge}

\def\co{\colon}
\def\ot{\otimes} 
\def\op{^{\rm op}}

\newcommand{\me}{^{-1}}
\def\dar[#1]{\ar@<2pt>[#1]\ar@<-2pt>[#1]}

\newcommand{\alg}{_\mathrm{alg}}

\newcommand{\sep}{_\mathrm{sep}}
\stackMath
\newcommand\reallywidehat[1]{%
\savestack{\tmpbox}{\stretchto{%
  \scaleto{%
    \scalerel*[\widthof{\ensuremath{#1}}]{\kern.1pt\mathchar"0362\kern.1pt}%
    {\rule{0ex}{\textheight}}
  }{\textheight}%
}{2.4ex}}%
\stackon[-6.9pt]{#1}{\tmpbox}%
}

\newcommand{\Aut}{\operatorname{Aut}}

\newcommand{\Br}{\operatorname{Br}} 

\newcommand{\End}{\operatorname{End}}

\newcommand{\Gal}{\operatorname{Gal}}

\newcommand{\GL}{\mathbf{GL}}

\newcommand{\Hom}{\operatorname{Hom}}

\newcommand{\cHom}{\mathcal{H}\hspace{-0.4ex}\textit{o\hspace{-0.2ex}m}} 
\newcommand{\cEnd}{\mathcal{E}\hspace{-0.4ex}\textit{n\hspace{-0.2ex}d}} 
\newcommand{\cSym}{\mathcal{S}\hspace{-0.5ex}\textit{y\hspace{-0.3ex}m}}
\newcommand{\cSymd}{\mathcal{S}\hspace{-0.5ex}\textit{y\hspace{-0.3ex}m\hspace{-0.2ex}d}}
\newcommand{\cSkew}{\mathcal{S}\hspace{-0.4ex}\textit{k\hspace{-0.2ex}e\hspace{-0.3ex}w}}

\newcommand{\Id}{\operatorname{Id}}

\newcommand{\inc}{{\operatorname{inc}}}
\newcommand{\Isom}{\operatorname{Isom}}

\newcommand{\Ker}{\operatorname{Ker}}

\newcommand{\Mat}{{\operatorname{M}}}

\newcommand{\Nrd}{\operatorname{Nrd}}

\newcommand{\PGO}{{\mathbf{PGO}}}

\newcommand{\PGL}{\mathbf{PGL}}

\newcommand{\res}{\operatorname{res}}

\newcommand{\Spec}{\operatorname{Spec}}

\newcommand{\Span}{\operatorname{Span}}

\newcommand{\Symm}{{\operatorname{Sym}}}

\newcommand{\SL}{\mathbf{SL}}

\newcommand{\Trd}{\operatorname{Trd}}

\newcommand{\tr}{\operatorname{tr}}
\newcommand{\Tr}{\operatorname{Tr}}

\newcommand{\uSL}{\mathbf {SL}}

\newcommand{\Ralg}{R\mathchar45\mathbf{alg}}

\newcommand{\Sch}{\mathfrak{Sch}} 
\newcommand{\Aff}{\mathfrak{Aff}} 
\newcommand{\Sh}{\mathfrak{Sh}} 

\newcommand{\Sets}{\mathfrak{Sets}} 
\newcommand{\Rings}{\mathfrak{Rings}} 
\newcommand{\Ab}{\mathfrak{Ab}} 
\newcommand{\Grp}{\mathfrak{Grp}} 
\newcommand{\fMod}{\mathfrak{Mod}} 
\newcommand{\QCoh}{\mathfrak{QCoh}}

\newcommand\al{\alpha}
\newcommand\be{\beta}
\newcommand\ga{\gamma} 
\newcommand\Ga{\Gamma}

\newcommand\veps{\varepsilon}

 \newcommand\vphi{\varphi}

\newcommand\si{\sigma}

\newcommand\ze{\zeta}

\newcommand{\bmu}{\boldsymbol{\mu}}
\newcommand{\bnu}{\boldsymbol{\nu}}
\newcommand{\bgamma}{\boldsymbol{\gamma}}
\newcommand{\boeta}{\boldsymbol{\eta}}

\newcommand{\und}{\underline{\hspace{2ex}}} 

\newcommand{\Inn}{\mathrm{Inn}} 

\newcommand{\inj}{\hookrightarrow}
\newcommand{\surj}{\twoheadrightarrow}
\newcommand{\iso}{\overset{\sim}{\longrightarrow}}
\newcommand{\norm}{\mathrm{norm}}
\newcommand{\bnorm}{\mathbf{norm}}

\newcommand{\Seg}{\mathrm{Seg}}
\newcommand{\PSeg}{\mathrm{PSeg}}

\newcommand{\SP}{\mathbf{Sp}}
\newcommand{\PSP}{\mathbf{PSp}}
\newcommand{\bPGO}{\mathbf{PGO}}

\newcommand{\cIsom}{\mathcal{I}\hspace{-0.35ex}\textit{s\hspace{-0.15ex}o\hspace{-0.15ex}m}}

\newcommand{\NPL}{\mathfrak{NPL}} 
\newcommand{\fAzu}{\mathfrak{Azu}}
\newcommand{\fVec}{\mathfrak{Vec}}

\newcommand{\fSh}{\mathfrak{Sh}}
\newcommand{\Sd}{S^{\sqcup d}}

\newcommand{\fTors}{\mathfrak{Tors}}
\newcommand{\fForms}{\mathfrak{Forms}}
\newcommand{\flf}{_\mathrm{flf}}
\newcommand{\QAlg}{\mathfrak{QAlg}}
\newcommand{\etale}{\mathrm{-\acute{e}t}}

\mathchardef\mdash="2D
\newcommand{\fTMod}{T\mdash\fMod}
\newcommand{\fTAzu}{T\mdash\fAzu}
\newcommand{\cTEnd}{T\mdash\cEnd}

\newcommand{\fAff}{\mathfrak{AffMor}}
\newcommand{\fAffMor}{\mathfrak{AffMor}}

\DeclareMathOperator{\TS}{TS}
\DeclareMathOperator{\T}{T}
\DeclareMathOperator{\RN}{\mathcal{RN}}
\DeclareMathOperator{\FN}{\mathcal{FN}}

\newcommand{\bAut}{\mathcal{A}\hspace{-0.4ex}\textit{u\hspace{-0.2ex}t}}
\newcommand{\cAut}{\mathcal{A}\hspace{-0.4ex}\textit{u\hspace{-0.2ex}t}}

\newcommand{\Cl}{\mathcal{C}\hspace{-0.3ex}\ell}
\DeclareMathOperator{\sw}{sw}

\DeclareMathOperator{\Img}{Img}

\newcommand{\tens}{\mathrm{tens}}

\DeclareMathOperator{\ST}{TS}
\newcommand{\FFN}{\FN}

\DeclareMathOperator{\Arr}{Arr}

\begin{document}
\title[The Norm Functor]{The Norm Functor over Schemes}
\author[P. Gille]{Philippe Gille}
\address{UMR 5208 du CNRS - Institut Camille Jordan - Universit\'e Claude Bernard Lyon 1, 43 boulevard du
11 novembre 1918, 69622 Villeurbanne cedex - France \\
and Institute of Mathematics ``Simion Stoilow" of the Romanian Academy,
21 Calea Grivitei Street, 010702 Bucharest, Romania. }
\email{gille@math.univ-lyon1.fr}

\author[E. Neher]{Erhard Neher}
\address{Department of Mathematics and Statistics, University of Ottawa, 150 Louis-Pasteur Private,
Ottawa, Ontario, Canada, K1N 9A7}
\email{Erhard.Neher@uottawa.ca}

\author[C. Ruether]{Cameron Ruether}
\address{The ``Simion Stoilow" Institute of Mathematics of the Romanian Academy, 21 Calea Grivitei Street, 010702 Bucharest, Romania.}
\email{cameronruether@gmail.com}

\thanks{The first and third authors were supported by the project ``Group schemes, root systems, and related representations" founded by the European Union - NextGenerationEU through Romania's National Recovery and Resilience Plan (PNRR) call no. PNRR-III-C9-2023-I8, Project CF159/31.07.2023, and coordinated by the Ministry of Research, Innovation and Digitalization (MCID) of Romania. The research of the second author was partially supported by an NSERC grant. The research of the third author was also partially supported by the NSERC grants of the second author, Kirill Zainoulline,  Mikhail Kotchetov, and Yorck Sommerh\"auser. All three authors thank the referee for their thorough and insightful comments.}

\date{December 10, 2024.}

\maketitle

\noindent{\bf Abstract:} We construct a globalization of Ferrand's norm functor over rings which generalizes it to the setting of a finite locally free morphism of schemes $T\to S$ of constant rank. It sends quasi-coherent modules over $T$ to quasi-coherent modules over $S$. These functors restrict to the category of quasi-coherent algebras. We also assemble these functors into a norm morphism from the stack of quasi-coherent modules over a finite locally free of constant rank extension of the base scheme into the stack of quasi-coherent modules. This morphism also restricts to the analogous stacks of algebras. Restricting our attention to finite \'etale covers, we give a cohomological description of the norm morphism in terms of the Segre embedding. Using this cohomological description, we show that the norm gives an equivalence of stacks of algebras $A_1^2 \equiv D_2$, akin to the result shown in The Book of Involutions.
\medskip

\bigskip

\tableofcontents

\section*{Introduction}
{%
\renewcommand{\thesubsection}{\Alph{subsection}}
One of the coincidences in the theory of algebraic groups is the exceptional isomorphism between the Dynkin diagrams of type $A_1+A_1$ and of type $D_2$. One way this manifests is as an isomorphism between split simply connected groups $\SL_2\times \SL_2 \cong \mathbf{Spin}_4$, or between split adjoint groups $\PGL_2\times \PGL_2 \cong \mathbf{PSO}_4$. However, due to the relationship between algebraic groups and algebras with involution, this also manifests as the following equivalence of groupoids shown in \cite[15.B]{KMRT}. Let $\FF$ be an arbitrary field and
\begin{enumerate}[label={\rm(\roman*)}]
\item let $A_1^2$ be the groupoid of Azumaya algebras of degree $2$ over a quadratic \'etale extension of $\FF$ with $\FF$--algebra isomorphisms as arrows, and
\item let $D_2$ be the groupoid of central simple $\FF$--algebras of degree $4$ equipped with quadratic pairs (see \cite[\S 5]{KMRT}) with $\FF$--algebra isomorphisms respecting the quadratic pair as arrows.
\end{enumerate}
Then, there is an equivalence of categories $A_1^2 \equiv D_2$. In particular, they show in \cite[15.7]{KMRT} that a \emph{norm functor} $\bN \colon A_1^2 \to D_2$ provides this equivalence. 

The norm functor used in \cite{KMRT} is with respect to finite \'etale extensions of the base field. It is a generalization of the \emph{corestriction} with respect to a finite separable field extension $\KK/\FF$ introduced by Riehm in \cite{Riehm}. Riehm's corestriction sends a central simple $\KK$--algebra $A$ of degree $r$ to a central simple $\FF$--algebra $\mathrm{cor}_{\KK/\FF}(A)$ of degree $r^{[\KK:\FF]}$ in such a way that the induced map on Brauer groups
\begin{align*}
\Br(\KK)=H^2(\Gal(\FF\sep,\KK),\FF\sep^\times) &\to H^2(\Gal(\FF\sep,\FF),\FF\sep^\times)=\Br(\FF) \\
[A] &\mapsto [\mathrm{cor}_{\KK/\FF}(A)]
\end{align*} 
agrees with the usual corestriction in Galois cohomology. This was generalized by Knus and Ojanguren, who defined the norm functor of a finite \'etale extensions of rings in \cite{KO75}, and theirs is the version used in \cite{KMRT}. The norm functor was then extended further to the case of a finite locally free extension of rings in two slightly different ways, one by Ferrand in \cite{F} and one by Rost in a preprint \cite{Rost}. Ferrand's norm agrees with the previous norms in the \'etale extension case, see \cite[\S 5]{F}. The Ferrand norm and the Rost norm agree with each other in many common cases, for example on flat modules, see Proposition \ref{noho}. However, these two norms are not isomorphic to one another in general. We review Ferrand's construction in Section \ref{sec_Ferrands_norm}, pointing out its important properties. In particular we review how it is compatible with arbitrary base change. In Section \ref{sec_Rosts_norm}, we explain Rost's construction, show how it differs from Ferrand's, and that it is not compatible with base change in general.

In this paper we continue the ``tradition" of extending the norm functor to new settings. We fix a base scheme $S$ and work on the big fppf ringed site $(\Sch_S,\cO)$ of schemes over $S$ with the global sections functor $\cO$. In fact, we extend the norm functor to a morphism of stacks over $\Sch_S$ between certain stacks of quasi-coherent sheaves. We are able to do so because Ferrand's norm is compatible with arbitrary base change. Precisely, for a finite locally free ring extension $R \to R'$, let $N_{R'/R}\colon \fMod_{R'} \to \fMod_R$ denote Ferrand's norm functor. If $R \to Q$ is any other ring homomorphism, thus making $Q\to R'\otimes_R Q$ a finite locally free extension as well, then for any $R'$--module $M'$ there is an isomorphism
\[
N_{R'/R}(M')\otimes_R Q \cong N_{(R'\otimes_R Q)/Q}(M'\otimes_R Q)
\]
and this is functorial in $M'$. This compatibility with tensor products is exactly what allows the norm to be generalized to quasi-coherent sheaves as they are characterized by a similar condition, see Lemma \ref{lem_quasi_coh_characterization}. Our construction is a general one which takes
\begin{enumerate}[label={\rm(\roman*)}]
\item a family $\fI$ of affine morphisms in $\Sch_S$ which is closed under arbitrary pullbacks and which allows descent (precisely, $\fI$ should be a substack of the stack of affine morphisms $\fAffMor$ as in Appendix \ref{app_stack_morphism}),
\item for each $h\colon U' \to U$ in $\fI$ with $U$ and $U'$ affine, a functor
\[
\cF_h \colon \fMod_{\cO(U')} \to \fMod_{\cO(U)},
\]
\item and for every fiber product diagram in $\Sch_S$
\[
\begin{tikzcd}
V' \ar[r] \ar[d,"h'"] & U' \ar[d,"h"] \\
V \ar[r] & U
\end{tikzcd}
\] 
where $U,U',V,V'$ are all affine, a natural isomorphism of functors
\[
\cF_h(\und) \otimes_{\cO(U)} \cO(V) \iso \cF_{h'}(\und \otimes_{\cO(U')} \cO(V'))
\]
\end{enumerate}
and assembles them into a morphism of stacks
\[
\cF\colon \QCoh_{\fI} \to \QCoh.
\]
Here, $\QCoh_{\fI}$ is the stack with objects $(T' \to T,\cM')$ consisting of a morphism $T' \to T$ in $\fI$ and a quasi-coherent $\cO|_{T'}$--module $\cM'$, while $\QCoh$ is the stack with objects $(X,\cM)$ consisting of a scheme $X\in \Sch_S$ and a quasi-coherent $\cO|_X$--module $\cM$. For a review of quasi-coherent modules on $(\Sch_S,\cO)$ and the details of this construction, see Appendix \ref{app_quasi_coh}.

Unsurprisingly, Ferrand's norm functors, ranging over the family of finite locally free morphisms of a fixed degree $d$, satisfy the necessary conditions of the constructions in Appendix \ref{app_quasi_coh}. When $\fI$ is the family of finite locally free morphisms of degree $d$, we write $\QCoh_\fI = \QCoh\flf^d$. Therefore, we obtain a stack morphism $N \colon \QCoh\flf^d \to \QCoh$, as well as a functor $N_{T/S}\colon \QCoh(T) \to \QCoh(S)$ between categories of quasi-coherent $\cO|_T$--modules and quasi-coherent $\cO$--modules for any finite locally free morphism $T\to S$ of degree $d$. We verify this in Section \ref{sec_norm_of_modules}. The functor $N_{T/S}$ of course has additional specific properties analogous to Ferrand's functor, most of which are in regards to polynomial laws. Since $f\colon T\to S$ is finite locally free, by \cite[Tag 0BD2]{Stacks} there is a functor
\[
\norm \colon f_*(\cO|_T) \to \cO
\]
which arises from the determinant of left multiplication by elements of $f_*(\cO|_T)$ on itself. Given a quasi-coherent $\cO|_T$--module $\cM$, we define a \emph{normic polynomial law} to be a natural transformation $\phi\colon f_*(\cM) \to \cN$, where $\cN$ is a quasi-coherent $\cO$--module, such that
\[
\phi(tm) = \norm(t)\phi(m)
\]
holds for all sections $t\in f_*(\cO|_T)(X)$ and $m\in f_*(\cM)(X)$ and for all $X\in \Sch_S$. We show that the norm functor $N_{T/S}$ has the following properties.

\begin{thm} \label{intro_norm_thm}
Let $f\colon T\to S$ be a finite locally free morphism of schemes and let $N_{T/S} \colon \QCoh(T) \to \QCoh(S)$ be the norm functor.
\begin{enumerate}[label={\rm(\roman*)}]
\item \label{intro_norm_thm_i} For every quasi-coherent $\cM$ over $T$ there exists a normic polynomial law $\nu_{\cM} \colon f_*(\cM)\to N_{T/S}(\cM)$ such that the pair $(N_{T/S}(\cM), \nu_{\cM})$ is universal in the following sense; if $\nu' \colon f_*(\cM) \to \cN'$ is any other normic polynomial law into a quasi-coherent $\cO$--module, then there is a unique $\cO$--module morphism $\varphi \colon N_{T/S}(\cM) \to \cN'$ such that $\nu' = \varphi \circ \nu_{\cM}$.

\item \label{intro_norm_thm_ii} The universal property determines the image of $N_{T/S}$ on morphisms. If $\varphi \colon \cM_1 \to \cM_2$ is a morphism of quasi-coherent modules over $T$, then $\nu_{\cM_2}\circ f_*(\varphi)$ is a normic polynomial law and $N_{T/S}(\varphi)$ is the unique $\cO$--module morphism making the diagram
\[
\begin{tikzcd}
f_*(\cM_1) \ar[r,"\nu_{\cM_1}"] \ar[d,"f_*(\varphi)"] & N_{T/S}(\cM_1) \ar[d,"N_{T/S}(\varphi)"] \\
f_*(\cM_2) \ar[r,"\nu_{\cM_2}"] & N_{T/S}(\cM_2)
\end{tikzcd}
\]
commute.

\item \label{intro_norm_thm_v} The norm functor and universal normic polynomial law respects base change. If we have a fiber product diagram in $\Sch_S$
\[
\begin{tikzcd}
T' \ar[r,"g'"] \ar[d,"f'"] & T \ar[d,"f"] \\
S' \ar[r,"g"] & S
\end{tikzcd}
\]
then there is an isomorphism of functors $N_{T'/S'}\circ g'^* \iso g^*\circ N_{T/S}$ and for any $\cM\in \QCoh(T)$ the diagram
\[
\begin{tikzcd}
f'_*(\cM|_{T'}) \ar[r,"\nu_{(\cM|_{T'})}"] \ar[d,"\mathrm{can}"] & N_{T'/S'}(\cM|_{T'}) \ar[d,"\rotatebox{90}{$\sim$}"] \\
f_*(\cM)|_{S'} \ar[r,"(\nu_{\cM})|_{S'}"]  & N_{T/S}(\cM)|_{S'}
\end{tikzcd}
\]
commutes.

\item \label{intro_norm_thm_iii} $N_{T/S}(\cO|_T) = \cO$ and $\nu_{\cO|_T} = \norm \colon f_*(\cO|_T) \to \cO$.

\item \label{intro_norm_thm_iv} If $\cB$ is a quasi-coherent $\cO|_T$--algebra, then $N_{T/S}(\cB)$ is naturally a quasi-coherent $\cO$--algebra and $\nu_{\cB}$ is multiplicative. The norm also preserves algebra homomorphisms and thus restricts to the categories of quasi-coherent algebras.
\end{enumerate}

Further, if $T\to S$ is finite \'etale of degree $d$ we have the following.
\begin{enumerate}[resume,label={\rm(\roman*)}]
\item \label{intro_norm_thm_vi} $N_{T/S}$ sends (finite) locally free modules to (finite) locally free modules.
\item \label{intro_norm_thm_vii}  $N_{T/S}$ sends Azumaya algebras to Azumaya algebras.
\item \label{intro_norm_thm_viii} $N_{T/S}$ sends locally free $\cO|_T$--modules of constant rank $r$ to locally free modules of constant rank $r^d$ and it sends Azumaya algebras of constant degree $r$ to Azumaya algebras of constant degree $r^d$.
\end{enumerate}
\end{thm}

Properties \ref{intro_norm_thm_i} and \ref{intro_norm_thm_ii} are shown in Proposition~\ref{prop_universal_property} and Corollary \ref{cor_morphism_from_universal} respectively. The isomorphism of functors in Property \ref{intro_norm_thm_v} comes from Lemma \ref{lem_phi_isomorphism} and the statement about the universal normic laws is Corollary \ref{cor_unviersal_base_change}. Property \ref{intro_norm_thm_iii} follows from Example \ref{ex_split_norm} after considering a sufficient localization. Property \ref{intro_norm_thm_iv} is shown in Lemma \ref{lem_norm_az}. Finally, when $T\to S$ is \'etale, property \ref{intro_norm_thm_vi} follows from Example \ref{ex_split_norm}, property \ref{intro_norm_thm_vii} follows from property \ref{intro_norm_thm_viii}, and property \ref{intro_norm_thm_viii} is Lemma \ref{lem_norm_modules_etale} and Lemma \ref{lem_ferrand}.

In Section \ref{cohomological_description} we give a description of maps on cohomology induced by the norm morphism. In particular, we restrict the norm morphism of stacks to various substacks, in fact subgerbes, which are equivalent to the gerbes of torsors for some semi-direct products of groups. As is discussed in Appendix \ref{app_semi_direct}, the cohomology set $H^1(S,(\GL_r)^d \rtimes \SS_d)$, where $\SS_d$ is the permutation group, classifies isomorphism classes of objects in the fiber over $S$ of the gerbe whose objects are pairs $(T' \to T,\cM')$ where $T'\to T$ is a degree $d$ \'etale cover and $\cM'$ is a locally free $\cO|_{T'}$--module of constant rank $r$. Similarly, isomorphism classes in the fiber over $S$ of the gerbe whose objects are pairs $(T'\to T,\cA')$, where now $\cA'$ is an $\cO|_{T'}$--Azumaya algebra of degree $r$, are classified by $H^1(S,(\PGL_r)^d\rtimes \SS_d)$. Details on the definitions of these gerbes are given at the beginning of Section \ref{cohomological_description}. We know by Theorem \ref{intro_norm_thm}\ref{intro_norm_thm_viii} that the norm will send such objects to $\cO$--modules of rank $r^d$ or Azumaya $\cO$--algebras of degree $r^d$ respectively. Isomorphism classes of these modules are classified by $H^1(S,\GL_{r^d})$ and isomorphisms of those Azumaya algebras are classified by $H^1(S,\PGL_{r^d})$. A general fact about stack morphisms, \cite[III.2.5.3]{Gir}, states that the resulting maps on isomorphism classes will be induced from group homomorphisms $(\GL_r)^d\rtimes \SS_d \to \GL_{r^d}$ and  $(\PGL_r)^d\rtimes \SS_d \to \PGL_{r^d}$. We show that these are the Segre embeddings. On the $\GL$ level, it sends elements $(A_1,\ldots,A_d)\in (\GL_r)^d$ to $A_1\otimes \ldots \otimes A_d$ and elements of $\SS_d$ to the corresponding permutation of the tensor factors of $\cO^{(r^d)} \cong (\cO^r)^{\otimes d}$. The map on the $\PGL$ level is defined similarly.
\begin{thm}
The map on cohomology sets induced by the norm functor, namely
\begin{align*}
\widetilde{N} \colon H^1(S,(\GL_r)^d\rtimes \SS_d) &\to H^1(S,\GL_{r^d})\\
[(T\to S,\cM)] &\mapsto [N_{T/S}(\cM)],
\end{align*}
agrees with the map induced by the Segre embedding $(\GL_r)^d \rtimes \SS_d \to \GL_{r^d}$. Furthermore, the behaviour of the norm functor on Azumaya algebras induces a map of cohomology sets
\begin{align*}
\widetilde{N\alg} \colon H^1(S,(\PGL_r)^d \rtimes \SS_d) &\to H^1(S,\PGL_{r^d}) \\
[(T\to S,\cA)] &\mapsto [N_{T/S}(\cA)].
\end{align*}
which likewise agrees with the map induced by the Segre embedding $(\PGL_r)^d\rtimes \SS_d \to \PGL_{r^d}$.
\end{thm}

Our cohomological description of the norm functor also extends beyond first cohomology. Over a scheme, Brauer classes of Azumaya algebras lie in the \emph{Brauer-Grothendieck group} $\Br(S) = H^2(S,\GG_m)$, where $\GG_m$ is the multiplicative group. Since we are assuming $f\colon T\to S$ is finite \'etale, we know that the norm functor preserves Azumaya algebras and thus maps classes of Azumaya algebras in $\Br(T)$ to classes in $\Br(S)$. We show that this induced action is compatible with the trace morphism $\tr \colon \Br(T) \to \Br(S)$ induced by the trace morphism $f_*(\GG_m|_T) \to \GG_m$ of \cite[IX.5.1.2]{SGA4}. The following is Proposition \ref{prop_brauer_group}.
\begin{prop}
Let $T \to S$ be a degree $d$ \'etale cover. Let $\cB$ be an Azumaya $\cO|_T$--algebra of constant degree and let $\cA$ be an Azumaya $\cO$--algebra of constant degree. Denoting Brauer classes in square brackets, we have
\begin{enumerate}[label={\rm(\roman*)}]
\item $[N_{T/S}(\cB)] = \tr([\cB]) \in \Br(S)$, and
\item $[N_{T/S}(\cA|_T)] = d[\cA] \in \Br(S)$.
\end{enumerate}
\end{prop}

Our goal in Section \ref{Equivalence} is to show that the exceptional isomorphism discussed at the beginning of this introduction still occurs at the level of Azumaya algebras over schemes. In fact, we show that it holds as an equivalence of stacks. We consider the gerbes
\begin{enumerate}[label={\rm (\roman*)}]
\item $\fA_1^{2\etale}$ of quaternion algebras over a degree $2$ \'etale extension, and
\item $\fD_2$ of degree $4$ Azumaya algebras with a quadratic pair.
\end{enumerate}
Quadratic pairs are a characteristic agnostic analogue of orthogonal involutions and thus are appropriate for discussing groups of type $D$ and related objects over a general scheme. We include background on quadratic pairs in Section \ref{sec_quad_pairs}. We begin Section \ref{Equivalence} by constructing a quadratic pair over $\ZZ$ which, after restriction, acts as the split object in $\fD_n$, the gerbe of Azumaya $\cO$--algebras of degree $2n$ with quadratic pair. By restricting the Segre homomorphism to symplectic and orthogonal groups, we then show how the norm gives a morphism from the stack of Azumaya algebras of degree $2r$ with symplectic involution over an \'etale extension of degree $d$ into the stack $\fD_{2^{d-1}r^d}$. By noting that when $r=1$ and $d=2$ we get an equivalence of categories, we obtain the following generalization of \cite[15.7]{KMRT} to our setting.
\begin{thm}\label{thm_D}
There is an equivalence of stacks
\begin{align*}
N\colon \fA_1^{2\etale} &\to \fD_2 \\
(T'\to T,\cB) &\mapsto \big(T,(N_{T'/T}(\cB),\sigma_{N_{T'/T}(\cB)},f_{N_{T'/T}(\cB)})\big)
\end{align*}
given by the norm functor. Over each fiber, a quasi-inverse is given by the Clifford algebra functor.
\end{thm}
This is Theorem \ref{thm_norm_eq} and Theorem \ref{thm_quasi_inverse}. If we focus on the fiber over $S$ of this morphism, we get an equivalence of categories between the groupoid of quaternion algebras over a degree $2$ \'etale extension of $S$ into the groupoid of degree $4$ Azumaya $\cO$--algebras with quadratic pair, which is a more direct analogue of \cite[15.7]{KMRT}. 

The organization of the paper is essentially outlined above. Section \ref{Preliminaries} recalls the common objects and some basic results we use throughout the paper. The important details of Ferrand's construction over rings are in Section \ref{sec_review_norms}, which also contains a review of the Rost norm over rings. The construction of our new norm morphism/functor and the proofs of the many parts of Theorem \ref{intro_norm_thm} are in Section \ref{sec_glue_Ferrand}. Section \ref{cohomological_description} develops the cohomological interpretation of the newly constructed norm functor. Section \ref{Equivalence} uses these tools to show Theorem \ref{thm_D}.

We also include three Appendices containing general technical lemmas. Appendix \ref{app_Weil} considers a degree $d$ \'etale cover $f\colon T\to S$ and then for any sheaf $\cF$ on $\Sch_S$ describes $f_*(\cF|_T)$ in terms of twisting sheaves with torsors or in terms of Weil restrictions in the case $\cF$ is representable. Appendix \ref{app_semi_direct} gives a cohomological description of stacks related to semi-direct products of groups which appear frequently in Sections \ref{cohomological_description} and \ref{Equivalence}. Finally, as mentioned above, Appendix \ref{app_quasi_coh} contains a review of the properties of quasi-coherent module on $(\Sch_S,\cO)$ as well as the details of the general construction we use to define the norm morphism.
}

\section{Preliminaries}\label{Preliminaries}

\subsection{Flat Sites}\label{sec_ringed_site}
Following the style of \cite{CF} and \cite{GNR}, most of our objects will be sheaves on a category of schemes equipped with the fppf topology. We now review this setting, highlighting definitions of objects that are most important to us.

For a scheme $X$ we denote by $\Sch_{X}$ the big fppf site of $X$ as in \cite[Expos\'e IV]{SGA3}. The objects of $\Sch_X$ are schemes with a fixed structure morphism $Y\to X$, morphisms are scheme morphisms which respect the structure morphisms, and coverings in the site consist of families of the form $\{Y_i \to Y\}_{i\in I}$ which are jointly surjective and where each $Y_i \to Y$ is flat and locally of finite presentation. When given a cover, we denote $Y_{ij} = Y_i \times_Y Y_j$, as well as $Y_{ijk} = Y_i \times_Y Y_j \times_Y Y_k$, etc. Affine schemes will commonly be denoted with $U$ or $V$ and affine covers by $\{U_i \to Y\}_{i\in I}$ or likewise with $V$. Since $Y$ need not be a separated scheme, an affine cover need not have its $U_{ij}$'s be affine. 
\begin{remark}
In \cite[Tag 021S]{Stacks}, the Stacks project defines ``a'' big fppf site of $X$, instead of ``the''. This distinction, due to set theoretic considerations, is avoided in \cite{SGA3} through the use of universes. We also simply use ``the'' big fppf site and similarly use ``the" big affine fppf site introduced below.
\end{remark}
We denote by $\Aff_X$ the big affine fppf site of $X$ as in \cite[Tag 021S (2)]{Stacks}. It is the full subcategory of $\Sch_X$ consisting of schemes which are affine and have an arbitrary structure morphism to $X$ and where the covers are fppf covers of the form $\{U_i \to U\}_{i=1}^m$ where each $U_i$ and $U$ are affine schemes. We will frequently use the following lemma to discuss sheaves on all of $\Sch_X$ by instead working with sheaves on $\Aff_X$.
\begin{lem}\label{lem_equiv_affine_sheaves}
There is an equivalence of categories between sheaves on $\Sch_X$ and sheaves on $\Aff_X$ given by restricting a sheaf $\cF \colon \Sch_X \to \Sets$ to the objects of $\Aff_X$. Under this equivalence, intrinsic properties, such as being finite locally free or being quasi-coherent as recalled in Appendix \ref{app_quasi_coh}, are preserved.
\end{lem}
\begin{proof}
The equivalence of categories (more precisely, equivalence of topoi) is \cite[Tag 021V]{Stacks}. The properties of Appendix \ref{app_quasi_coh} are intrinsic properties by \cite[Tag 03DM]{Stacks}. By the definition of intrinsic property at the beginning of \cite[Tag 03DG]{Stacks}, they are preserved under equivalences of topoi.
\end{proof}

If $Y \in \Sch_X$, then $\Sch_Y$ is naturally a subcategory of $\Sch_X$ by composing the structure morphisms $Y'\to Y$ with the structure morphism $Y \to X$. For a sheaf $\calF$ on $\Sch_X$, we denote by $\calF|_Y$ the restriction of the sheaf to $\Sch_Y$. If $Y'\to Y$ is a morphism of $X$--schemes, then borrowing notation from \cite{Stacks}, we use $t|_{Y'}$ to denote the image under $\cF(Y) \to \cF(Y')$ of a section $t\in \cF(Y)$. This will also be referred to as the restriction of $t$ to $Y'$. It will be clear from context which notion of restriction, for sheaves or for sections, is intended. By a slightly further abuse of notation, we may talk about a section $t\in \cF$, by which we mean a section $t\in \cF(Y)$ for some $Y\in \Sch_S$, i.e., $t$ may be any section over any $Y$.

\subsection{Ringed Sites}\label{sec_O_modules}
We now fix a base scheme $S$. Unless otherwise stated, we assume that a ring is unital, commutative, and associative. The global sections functor 
\begin{align*}
\cO \co \Sch_S &\to \Rings \\
X &\mapsto \cO_X(X)
\end{align*}
where $\Rings$ is the category of commutative rings, is a sheaf with respect to the fppf topology by \cite[Tag 03DU]{Stacks}. It makes $(\Sch_S,\cO)$ into a ringed site as in \cite[Tag 03AD]{Stacks} and we call $\cO$ the \emph{structure sheaf}. If $X \in \Sch_S$ is another scheme, then the structure sheaf of $\Sch_X$ is $\cO|_X$.

From \cite[Tag 03CW]{Stacks}, an \emph{$\cO$--module} is a sheaf $\cM \co \Sch_S \to \Ab$ of abelian groups with a map of sheaves
\[
\cO \times \cM \to \cM
\]
such that for each $X\in \Sch_S$, the map $\cO(X)\times \cM(X) \to \cM(X)$ gives $\cM(X)$ the structure of an $\cO(X)$--module. A morphism of $\cO$--modules is a morphism of sheaves such that the map on $X$ points is $\cO(X)$--linear for all $X\in \Sch_S$. The notion of $\cO|_X$--module on $\Sch_X$ is analogous and likewise for the properties discussed below.

The \emph{internal homomorphism} functor of two $\cO$--modules $\cM$ and $\cN$ is
\begin{align*}
\cHom_{\cO}(\cM,\cN)\co \Sch_S &\to \Ab \\
T &\mapsto \Hom_{\cO|_T}(\cM|_T,\cN|_T).
\end{align*}
It is another $\cO$--module by \cite[03EM]{Stacks}. The internal endomorphisms of an $\cO$--module $\cM$ are denoted by $\cEnd_{\cO}(\cM)=\cHom_{\cO}(\cM,\cM)$, and the subsheaf of automorphisms is denoted $\cAut_{\cO}(\cM)$.

If $g\colon X \to S$ is a morphism of schemes, then we have a direct image (or pushforward) functor $g_*$ and a pull-back functor $g^*$ as in \cite[Tag 03D6]{Stacks}. These form an adjoint pair where $g^*$ is left adjoint to $g_*$, \cite[Tag 03D7]{Stacks}. That is, for each  $\cO|_X$--module $\cE$ and each $\cO$--module $\cF$, we have a natural isomorphism 
\[
\Hom_{\cO|_X-Mod}( g^*\cF, \cE) \simlgr \Hom_{\cO-Mod}( \cF, g_*\cE).
\]
In particular for $\cF= \cO$ we have 
\[
\Hom_{\cO|_X-Mod}( \cO|_X, \cE) \simlgr \Hom_{\cO-Mod}( \cO, g_*\cE)
\]
or in other words $\cE(X) \simlgr (g_*\cE)(S)$.

\subsubsection{Local Types of $\cO$--modules}
We refer to \cite[Tags 03DE, 03DL]{Stacks} for definitions of various properties of $\cO$--modules. Since $S \in \Sch_S$ is a final object, \cite[Tag 03DN]{Stacks} applies and it suffices for us to define local conditions for an fppf-covering of $S$. Quasi-coherent modules are reviewed in Appendix \ref{app_quasi_coh}. Here we briefly review finite locally free modules.

We call an $\cO$--module $\cM$ \emph{finite locally free} or \emph{locally free of finite type} if for all $X\in \Sch_S$, there is a covering $\{X_i \to X\}_{i\in I}$ such that for each $i\in I$, the restriction $\cM|_{X_i}$ is a free $\cO|_{X_i}$--module of finite rank. Explicitly, $\cM|_{X_i} \cong \cO|_{X_i}^{n_i}$ for some non-negative $n_i \in \ZZ$. If $n_i > 0$ for all $i\in I$ we say that $\cM$ is of \emph{finite positive rank}. If all $n_i = n$ for some integer $n$ then we say $\cM$ has \emph{constant rank $n$}. Neither of these notions depend on the cover. If $\cM$ is finite locally free, then so is $\cEnd_{\cO}(\cM)$.

\subsubsection{$\cO$--Algebras}
An \emph{$\cO$--algebra} is an abelian sheaf $\cB \colon \Sch_S \to \Ab$ together with sheaf morphisms
\[
\cO \to \cB \text{ and } \cB\otimes_{\cO} \cB \to \cB
\]
which makes $\cB(X)$ into a $\cO(X)$--algebra for all $X\in \Sch_S$. It is unital, associative, commutative, etc., if each $\calB(X)$ has that property. For an $\cO$--module $\cM$, the sheaf $\cEnd_{\cO}(\cM)$ is naturally an $\cO$--algebra with multiplication coming from composition as usual. 

An $\cO$-algebra $\cA$ is an \emph{Azumaya $\cO$--algebra} if it is finite locally free and it satisfies the following equivalent conditions.
\begin{enumerate}[label={\rm(\roman*)}]
\item \label{defn_Azumaya_i} The enveloping morphism
\begin{align*}
\cA\otimes_{\cO} \cA\op &\to \cEnd_{\cO}(\cA)\\
a\otimes b &\mapsto (x\mapsto axb)
\end{align*}
is an isomorphism.
\item \label{defn_Azumaya_ii} For any $U\in \Aff_S$, we have that $\cA(U)$ is an Azumaya $\cO(U)$--algebra in the sense over rings such as in \cite{Ford} or \cite[III \S 5]{K}. In particular, on $\Aff_S$ an Azumaya $\cO$--algebra is a sheaf of Azumaya algebras.
\item \label{defn_Azumaya_iii} There exists a cover $\{X_i \to S\}_{i\in I}$ such that for each $i\in I$, $\cA|_{X_i} \cong \cEnd_{\cO|_{X_i}}(\cM_i)$ for a locally free $\cO|_{X_i}$--module $\cM$ of finite positive rank.
\item \label{defn_Azumaya_iv} There exists a cover $\{X_i \to S\}_{i\in I}$ such that for each $i\in I$, $\cA|_{X_i} \cong \Mat_{n_i}(\cO|_{X_i})$ for some $0< n_i \in \ZZ$.
\end{enumerate}
Definition \ref{defn_Azumaya_i} above is from \cite[5.1]{Gro-Brau}, and definition \ref{defn_Azumaya_ii} is from \cite[2.5.3.4]{CF}. 

Since an Azumaya $\cO$--algebra is locally a matrix algebra, it locally has a notion of the trace $\Tr \colon \Mat_{n_i}(\cO|_{X_i}) \to \cO|_{X_i}$. These functions are compatible on $X_{ij}$ and therefore glue into a global $\cO$--linear map $\Trd_{\cA}\colon \cA \to \cO$ called the \emph{reduced trace} of $\cA$. The local determinant maps are also compatible and glue into the \emph{reduced norm} $\Nrd_{\cA}\colon \cA \to \cO$, which is multiplicative.

By \cite[2.5.3.6]{CF}, whenever $\cA|_X \cong \cO|_X^{n}$ for some $X\in \Sch_S$ and $n\in \ZZ$, the integer $n$ will be a square. If $n=m^2$, then $m$ is called the \emph{degree} of $\cA|_X$. If $\cA$ is of constant rank $m^2$, we say it is of \emph{constant degree} $m$.

\subsection{Quadratic Pairs}\label{sec_quad_pairs}
Quadratic pairs are the characteristic independent analogue of orthogonal involutions and are used to study quadratic forms and groups of type $D$ when $2$ need not be invertible. We refer to \cite[\S 2.7]{CF} and \cite{GNR}. Background on involutions and quadratic pairs over fields can be found in \cite{KMRT} and background on involutions over rings in \cite{K}.

An \emph{involution} (of the first kind) on an Azumaya $\cO$--algebra is an $\cO$--linear order $2$ anti-automorphism, i.e., $\sigma \colon \cA \to \cA$ such that $\sigma^2=\Id$ and $\sigma(ab)=\sigma(b)\sigma(a)$. We refer to such a pair $(\cA,\sigma)$ as an Azumaya $\cO$--algebra with involution. By \cite[2.7.0.25]{CF}, any Azumaya $\cO$--algebra with involution will have a cover $\{X_i \to S\}_{i\in I}$ (which in fact can be an \'etale cover if one desires) over which $(\cA|_{X_i},\sigma|_{X_i}) \cong (\Mat_{n_i}(\cO|_{X_i}),\eta_{b_i})$ where $\eta_{b_i}$ is the adjoint involution of a regular bilinear form $b_i \colon \cO|_{X_i}^{n_i} \times \cO|_{X_i}^{n_i} \to \cO|_{X_i}$. If the resulting $b_i$ are all symmetric bilinear forms, i.e., $b_i(x,y)=b_i(y,x)$, then we call $\sigma$ an \emph{orthogonal involution}; if they are skew-symmetric, $b_i(x,y)=-b_i(y,x)$, then we call $\sigma$ a \emph{weakly-symplectic involution}; and if they are alternating, $b_i(x,x)=0$, then we call $\sigma$ a \emph{symplectic involution}. These properties do not depend on the cover chosen and are not the only possibilities; for example, a bilinear form can be $\varepsilon$--symmetric for some $\varepsilon\in \cO(X_i)$ with $\varepsilon^2=1$. Our terminology follows \cite{CF} and \cite{GNR} but differ slightly from \cite{KMRT}. In our conventions, if $2=0\in \cO$, then orthogonal and weakly-symplectic coincide. In all cases, a symplectic involution is also weakly-symplectic and so we allow symplectic involutions to also be orthogonal.

Associated to an Azumaya $\cO$--algebra with involution $(\cA,\sigma)$ are a few important subsheaves. We have two $\cO$--module endomorphisms $\Id \pm \sigma \colon \cA \to \cA$ and we define
\begin{align*}
\cSym_{\cA,\sigma} &= \Ker(\Id -\sigma), \\
\cSkew_{\cA,\sigma} &= \Ker(\Id+\sigma),\text{ and} \\
\cSymd_{\cA,\sigma} &= \Img(\Id+\sigma),
\end{align*}
called the modules of \emph{symmetric}, \emph{skew-symmetric}, and \emph{symmetrized} elements, respectively. The symmetrized elements are defined by an image sheaf which involves sheafification. If $\sigma$ is orthogonal, then $\cSym_{\cA,\sigma}$ is finite locally free by \cite[3.3 (ii)]{GNR}. A \emph{quadratic pair} on $\cA$ is a pair $(\sigma,f)$ where
\begin{enumerate}[label={\rm(\roman*)}]
\item $\cA$ is an Azumaya $\cO$--algebra,
\item $\sigma$ is an orthogonal involution on $\cA$, and
\item $f\colon \cSym_{\cA,\sigma} \to \cO$ is an $\cO$--linear map such that $f(x+\sigma(x))=\Trd_{\cA}(x)$ for all $x\in \cA$.
\end{enumerate}
We call $f$ a \emph{semitrace} and also refer to $(\cA,\sigma,f)$ as a \emph{quadratic triple}. If $\frac{1}{2} \in \cO$, then the third condition implies that $f=\frac{1}{2}\Trd_{\cA}$ is the unique $f$ making $(\cA,\sigma,f)$ a quadratic triple, see \cite[4.3(a)]{GNR}. The fact that a unique $f$ exists when $2$ is invertible is a reflection of the bijective correspondence between symmetric bilinear forms and quadratic forms in characteristic not $2$. The connection between quadratic triples and quadratic forms is detailed in \cite[4.4]{GNR} over schemes or \cite[5.11]{KMRT} over fields.

\subsection{Algebraic Groups}\label{sec_algebraic_groups}
We review the groups we need from \cite{CF}.

If $\calB$ is a unital associative $\cO$--algebra which is finite locally free, then the functor of invertible elements
\begin{align*}
\GL_{1,\calB} \co \Sch_S &\to \Grp \\
X &\mapsto \calB(X)^\times,
\end{align*}
is called the \emph{general linear group} of $\cB$. If $\cB = \cEnd_{\cO}(\cO^n)$ we write $\GL_{1,\cB} = \GL_n$. The \emph{multiplicative group} is $\GG_m = \GL_1$. The \emph{$n^{\mathrm{th}}$--roots of unity} are denoted $\bmu_n$ and are the kernel of the map $\GG_m \to \GG_m$ given by $x\mapsto x^n$.

If $\cA$ is an Azumaya $\cO$--algebra, then the reduced norm is a group homomorphism $\Nrd_{\cA} \colon \GL_{1,\cA} \to \GG_m$ and the kernel of this map, i.e., the group of norm $1$ elements, is the \emph{special linear group} $\SL_{\cA}$, \cite[3.5.0.91]{CF}. If $\cA = \cEnd_{\cO}(\cO^n)$, we similarly write $\SL_n$.

The \emph{projective general linear group} of an Azumaya $\cO$--algebra $\cA$ is the subsheaf of $\cAut_{\cO}(\cA)$ consisting of algebra automorphisms. So we may write $\PGL_{\cA} = \cAut_{\cO\mathrm{\mdash alg}}(\cA)$. When $\cA=\cEnd_{\cO}(\cO^n)$, we write $\PGL_n$. The canonical projection $\GL_{1,\cA} \to \PGL_{\cA}$ sends an element to its inner automorphism and this projection has kernel $\GG_m$.

Orthogonal groups are defined in \cite[\S 4]{CF} for quadratic forms and quadratic triples. In particular, let $q\colon \cO^{2n} \to \cO$ be a regular quadratic form. By \cite[4.4 (i)]{GNR} or \cite[2.7.0.31]{CF}, it has a corresponding quadratic triple $(\Mat_{2n}(\cO),\sigma_q,f_q)$ where $\sigma_q$ is the adjoint involution. The \emph{orthogonal groups} are then
\begin{align*}
\mathbf{O}_q(X) &= \{ x\in \GL_{2n}(X) \mid q\circ x = q \}, \text{ and}\\
\mathbf{O}_{\Mat_{2n}(\cO),\sigma_q,f_q}(X) &= \{x\in \GL_{2n}(X) \mid \sigma_q(x)=x^{-1}, f_q(x \und x^{-1}) = f_q\}
\end{align*}
for all $X\in \Sch_{S}$. We have that $\mathbf{O}_q \cong \mathbf{O}_{\Mat_{2n}(\cO), \sigma_q, f_{q}}$ by \cite[4.4.0.44]{CF}. The group $\mathbf{O}_{q}^+$ is a smooth affine $S$--group scheme which is the kernel of the map
\begin{equation}\label{eq_Dickson}
\mathbf{O}_q \to \ZZ/2\ZZ
\end{equation}
called the Dickson map in \cite[IV, 5.1]{K} and called the Arf map in \cite[4.3.0.27]{CF}. An automorphism of a quadratic form induces an automorphism of the associated even Clifford algebra and the Dickson map sends the quadratic form automorphism to the induced map's restriction to the center of the even Clifford algebra.  We use the terminology of \cite{K}.

The \emph{projective orthogonal group} of a regular quadratic form q is defined to be the automorphism group scheme $\bPGO_{\Mat_{2n}(\cO),\sigma_q,f_q} = \cAut(\Mat_{2n}(\cO),\sigma_q,f_q)$. This is also simply denoted by $\bPGO_q$. An automorphism of $(\Mat_{2n}(\cO),\sigma_q,f_q)$ similarly induces an automorphism of the even Clifford algebra, and so there is also a Dickson map $\bPGO_q \to \ZZ/2\ZZ$ whose kernel is denoted $\bPGO_q^+$. The kernel of the canonical projection $\bO_q^+ \to \bPGO_q^+$ is the center of $\bO_q^+$, which is isomorphic to $\bmu_2$.

We will also work with the \emph{symplectic group} of an Azumaya $\cO$--algebra with symplectic involution $(\cA,\sigma)$. As in \cite[\S 7]{CF}, this is
\[
\SP_{\cA,\sigma}(X) = \{ x \in \cA(X) \mid \sigma(x)=x^{-1} \}
\]
for all $X\in \Sch_S$. The \emph{projective symplectic group} associated to $(\cA,\sigma)$ is $\PSP_{\cA,\sigma} = \cAut(\cA,\sigma)$: the automorphism group as an algebra with involution. The kernel of the canonical projection $\SP_{\cA,\sigma} \to \PSP_{\cA,\sigma}$ is the center of $\SP_{\cA,\sigma}$, which is isomorphic to $\bmu_2$.

\subsection{Torsors and Contracted Products}\label{sec_contracted_products}
We work with torsors for various groups using the definition of \cite[2.2.2.2]{CF} or \cite[Tag 03AH]{Stacks}. If
\[
\bG\colon \Sch_S \to \Grp
\]
is a group sheaf, a \emph{$\bG$--torsor} is a sheaf $\cP \colon \Sch_S \to \Sets$ with a map of sheaves $\cP\times \bG \to \cP$ such that
\begin{enumerate}[label={\rm(\roman*)}]
\item $\cP(X)\times \bG(X) \to \cP(X)$ gives a simply transitive right $\bG(X)$--action on $\cP(X)$ for all $X\in \Sch_S$ such that $\cP(X)\neq \O$, and
\item there exists a cover $\{X_i \to S\}_{i\in I}$ such that $\cP(X_i)\neq \O$ for all $i\in I$. 
\end{enumerate}
The sheaf $\bG$ itself, viewed simply as a sheaf of sets with right action given by right multiplication, is called the \emph{trivial torsor}. We will work with non-abelian cohomology as in \cite[III, 2.4.2]{Gir}, defining $H^1(S,\bG)$ to be the set of isomorphism classes of $\bG$--torsors over $S$ for any group sheaf $\bG$. If $\cP$ is a $\bG$--torsor and $X\in \Sch_S$, then it is clear that $\cP|_X$ is a $\bG|_X$--torsor. When we wish to emphasize the morphism $g\colon X \to S$ we will also write $g^*(\cP)=\cP|_X$, mirroring the pullback notation for $\cO$--modules and algebras.

Given a group sheaf $\bG$, a sheaf of sets $\cX$ with a right action of $\bG$, and a sheaf of sets $\cY$ with a left action of $\bG$, the contracted product as defined in \cite[2.2.2.9]{CF} is denoted $\cX \wedge^{\bG} \cY$. It is the sheaf associated to the presheaf $(\cX\times \cY)/\sim$ where the equivalence relation $\sim$ identifies elements $(xg,y)$ and $(x,gy)$ for all appropriate sections $x\in \cX, y\in \cY$, and $g\in \cG$. If $\cX$ also has a left action of $\cG$, then $\cX\wedge^{\bG} \cY$ has an inherited left action, and similarly if $\cY$ also has a right action then so does $\cX\wedge^{\bG} \cY$.

In particular, if $\varphi \colon \bG \to \bH$ is a group sheaf homomorphism and $\cP$ is a $\bG$--torsor, then giving $\bH$ a left action of $\bG$ through $\varphi$ and a right action on itself by multiplication, the contracted product $\cP \wedge^{\bG} \bH$ will have a right $\bH$--action and by \cite[2.2.2.12]{CF} will be an $\bH$--torsor. This is denoted $\varphi_*(\cP)$ and is called \emph{the pushforward of $\cP$ along $\varphi$} or the \emph{twist of $\bH$ by $\cP$}. If we have a morphism $g\colon \cP_1 \to \cP_2$ of $\bG$--torsors, it induces a morphism 
\[
g\times\Id \colon (\cP_1\times\bH)/\sim \to (\cP_2\times\bH)/\sim
\]
between the presheaves defining the contracted products, and thus also induces a map between the contracted products themselves, denoted
\[
\varphi_*(g) \colon \varphi_*(\cP_1) \to \varphi_*(\cP_2).
\]

Furthermore, if $\cP$ is a $\bG$--torsor and if $\cY$ has additional structure, say it is a sheaf of groups, rings, modules, etc., and the left action of $\bG$ is by automorphisms which respect that structure, then the contracted product $\cP \wedge^{\bG} \cY$ will also be a sheaf of groups, rings, modules, etc. This has been formalized in \cite[2.1]{CF}, which we invite the interested reader to consult. The contracted product $\cP \wedge^{\bG} \cY$ will also be locally isomorphic to $\cY$. Indeed, over any cover $\{Y_i \to S\}_{i\in I}$ which trivializes $\cP$, we will have $(\cP\wedge^{\bG} \cY)|_{Y_i} \cong \cY|_{Y_i}$.

\begin{remark}
We have limited ourselves in the above discussion to considering torsors on the site $\Sch_S$. Of course, there is a general notion of a torsor for a sheaf of groups $\bG \colon \cC \to \Grp$ on any site $\cC$ defined in \cite[Tag 03AH]{Stacks}. We use this more general notion only briefly in Appendix \ref{app_semi_direct}.
\end{remark}

\subsection{Stacks}
We make use of stacks over the fppf site $\Sch_S$ defined in Section \ref{sec_ringed_site} though readers familiar with stacks will recognize that the following discussion also holds over a general site. We recall this notion from \cite{Gir} (where the French word for ``stack" is ``\emph{champ}"), \cite{Olsson}, and \cite{Stacks}.

First is the notion of a fibered category. We take the following definition from \cite[3.1]{Olsson}, specifying all notions to be over the site $\Sch_S$. Alternatively, this is \cite[Tag 02XJ]{Stacks}, where they use the terminology ``strongly cartesian" where we use ``cartesian".
\begin{defn}\label{defn_fibered_category}
Let $\fF$ be a category with a functor $p\colon \fF \to \Sch_S$. We call a morphism $\phi \colon u \to v$ in $\fF$ \emph{cartesian} if whenever we have another morphism $\psi \colon w \to v$ such that $p(w)$ factors through $p(\phi)$, diagrammatically
\[
\begin{tikzcd}
w \arrow[dashed]{d} \arrow{dr}{\psi}& \\
u \arrow{r}{\phi} & v 
\end{tikzcd} \overset{p}{\mapsto}
\begin{tikzcd}
p(w) \arrow{d}{h} \arrow{dr}{p(\psi)} & \\
p(u) \arrow{r}{p(\phi)} & p(v),
\end{tikzcd}
\]
then there exists a unique morphism $\lambda \colon w \to u$ in $\fF$, taking the place of the dashed arrow above, such that the diagram commutes and $p(\lambda)=h$.

The category $\fF$, or more precisely the data $p\colon \fF \to \Sch_S$, is called a \emph{fibered category over $\Sch_S$} if for every $x \in \fF$ and morphism $f\colon Y \to p(x)$ in $\Sch_S$, there exists a cartesian morphism $\phi \colon y \to x$ such that $p(\phi)= f$, and therefore also $p(y)=Y$. The object $y$ may be denoted $y=f^*(x)$ and called a \emph{pullback} of $x$ along $f$. We will call the functor $p$ the \emph{structure functor}.
\end{defn}

In the above definition we say ``a" pullback, instead of ``the" pullback, since pullbacks need not be unique. However, they are unique up to unique isomorphism by \cite[3.1]{Olsson}. If the morphism $f\colon Y \to p(x)$ is clear from context we may write $f^*(x) = x|_Y$. By \cite[3.4]{Vis}, compositions of cartesian morphisms are cartesian and isomorphisms in $\fF$ are also cartesian.

A morphism from a fibered category $(\fF,p_\fF)$ over $\Sch_S$ to another fibered category $(\fG,p_\fG)$ over $\Sch_S$ is a functor $\varphi \colon \fF \to \fG$ such that $p_{\fG}\circ \varphi = p_{\fF}$ is an equality of functors into $\Sch_S$ and $\varphi$ sends cartesian morphisms to cartesian morphisms.

Given a fibered category $p \colon \fF \to \Sch_S$, for any scheme $X \in \Sch_S$ we denote by $\fF(X)$ the subcategory of objects $x\in \fF$ such that $p(x)=X$, and morphisms $\phi \colon x \to x'$ such that $p(\phi)=\Id_X$. This is called the \emph{fiber over $X$}. A morphism of fibered categories $\varphi \colon \fF \to \fG$ induces functors $\varphi_X \colon \fF(X) \to \fG(X)$ between the corresponding fibers. We take \cite[3.1.10]{Olsson} as our definition and say that $\varphi$ is an \emph{equivalence of fibered categories} if every $\varphi_X$ is an equivalence of categories.

A fibered category is called \emph{fibered in groupoids} if $\fF(X)$ is a groupoid for all $X\in \Sch_S$. A convenient fact about categories fibered in groupoids is the following.
\begin{lem}\label{lem_cartesian_groupoids}{\cite[3.22]{Vis}}
If $p\colon \fF \to \Sch_S$ is a fibered category, then $\fF$ is fibered in groupoids if and only if every morphism in $\fF$ is a cartesian morphism.
\end{lem}
In particular, this means that if $\fF$ is a fibered category and $\fG$ is fibered in groupoids, then any functor $\fF \to \fG$ which is compatible with the structure functors into $\Sch_S$ will be a morphism of fibered categories.

Following \cite[Tag 02ZB]{Stacks} or \cite[\S 3.7]{Vis} (or \cite[3.4.7]{Olsson} when $\fF$ is fibered in groupoids), given two objects $x,x' \in \fF(X)$, we define their \emph{internal homomorphism functor}
\begin{align*}
\cHom(x,x')=\cHom_{\fF}(x,x') \colon \Sch_X &\to \Sets \\
Y &\mapsto \Hom_{\fF(Y)}(x|_Y,x'|_Y).
\end{align*}
Independence from the choice of pullbacks in shown in \cite[\S 3.7]{Vis}. The restriction maps for this functor are defined as follows. If $f\colon Y' \to Y$ is a morphism of $X$--schemes and $\varphi \colon x|_Y \to x'|_Y$ is a morphism in $\fF(Y)$, we have a diagram in $\fF$
\[
\begin{tikzcd}
x|_{Y'} \arrow{d}{\phi_x} \arrow[dashed]{r} & x'|_{Y'} \arrow{d}{\phi_{x'}}\\
x|_Y \arrow{r}{\varphi} & x'|_Y
\end{tikzcd}
\]
where $\phi_x$ and $\phi_{x'}$ are cartesian morphisms. Since $\varphi$ is a morphism in $\fF(Y)$, we have $p(\varphi)=\Id_Y$ and so tautologically $p(\varphi\circ \phi_x)$ factors through $p(\phi_{x'}) = p(\phi_x)$. Therefore by Definition \ref{defn_fibered_category} there is a unique morphism in $\fF(Y')$, which we denote $\varphi|_{Y'}$ or $f^*(\varphi)$ when emphasizing $f$, that takes the place of the dashed arrow. These assignments $\varphi \mapsto \varphi|_{Y'}$ give the restriction maps of the functor $\cHom(x,x')$. When $x=x'$ we write $\cEnd(x) = \cHom(x,x)$. The subfunctors of invertible elements will be denoted $\cIsom(x,x')$ and $\cAut(x)$ respectively. If $\varphi \colon \fF \to \fG$ is a morphism of fibered categories, it induces natural transformations of functors $\varphi_{x,x'} \colon \cHom_{\fF}(x,x') \to \cHom_{\fG}(\varphi(x),\varphi(x'))$ in the canonical way, and this restricts to the $\cIsom$ subfunctors.

\begin{example}\label{hom_sheaf_of_schemes}
Here we present \cite[Example 3.15]{Vis} applied to the category $\Sch_S$. The \emph{arrow fibered category} $p\colon \Arr(\Sch_S) \to \Sch_S$ has
\begin{enumerate}[label={(\roman*)}]
\item objects which are morphisms $T \to W$ in $\Sch_S$,
\item morphisms which are pairs $(f,g)\colon (T_1\to W_1) \to (T_2 \to W_2)$ where $f\colon T_1 \to T_2$ and $g\colon W_1 \to W_2$ are in $\Sch_S$ such that
\[
\begin{tikzcd}
T_1 \ar[r,"f"] \ar[d] & T_2 \ar[d] \\
W_1 \ar[r,"g"] & W_2
\end{tikzcd}
\]
commutes, and
\item structure functor $p\colon \Arr(\Sch_S) \to \Sch_S$ defined by $p(T\to W) = W$ on objects and $p(f,g) = g$ on morphisms.
\end{enumerate} 
The cartesian morphisms are those $(f,g)$ such that the associated commutative square is a fiber product diagram. In particular, this means that a pullback of an object $T \to W_2$ along a morphism $g\colon W_1 \to W_2$ is the object $T\times_{W_2} W_1 \to W_1$ given by the fiber product. This means that for two objects $T_1 \to W$ and $T_2 \to W$ in the same fiber, the homomorphism functor takes the form
\begin{align*}
\cHom_{\Arr(\Sch_S)}(T_1\to W,T_2 \to W) \colon \Sch_W &\to \Sets \\
W' &\mapsto \Hom_{W'}(T_1\times_W W', T_2\times_W W').
\end{align*}
This is a sheaf for the fppf topology (in fact, even for the fpqc topology) by \cite[Tag 040L]{Stacks} or by \cite[2.1.10(2)]{Alper}. For later use, we will use the notation $\cHom_W(T_1,T_2)$ for this sheaf. Of course, this means that the functor $\cIsom_W(T_1,T_2)$, and therefore also $\cAut_W(T_1)$, are sheaves as well.
\end{example}

We take our definition of stack from \cite[Tag 026F]{Stacks}. This is slightly more general than the definition in \cite[4.6.1]{Olsson}, which requires that stacks be fibered in groupoids.
\begin{defn}\label{defn_stacks}
An \emph{$S$--stack}, or simply a \emph{stack} when $S$ is clear from context, is a fibered category $p\colon \fF \to \Sch_S$ such that the following hold.
\begin{enumerate}[label={\rm(\roman*)}]
\item \label{defn_stacks_i} For any $X\in \Sch_S$ and objects $x,x'\in \fF(X)$, the functor
\[
\cHom(x,x')\colon \Sch_X \to \Sets
\]
is an fppf sheaf on $\Sch_X$.
\item \label{defn_stacks_ii} For any fppf covering $\{X_i \to X\}_{i\in I}$ of $\Sch_S$, objects $x_i \in \fF(X_i)$, and isomorphisms $\psi_{ij} \colon x_i|_{X_{ij}} \iso x_j|_{X_{ij}}$ in $\fF$ satisfying the cocycle condition
\[
\psi_{jk}|_{X_{ijk}} \circ \psi_{ij}|_{X_{ijk}} = \psi_{ik}|_{X_{ijk}},
\]
there exists an object $x\in \fF(X)$ and isomorphisms $\alpha_i \colon x_i \iso x|_{X_i}$ in $\fF(X_i)$ such that
\[
\psi_{ij}=(\alpha_j)|_{X_{ij}}^{-1}\circ (\alpha_i)|_{X_{ij}}.
\]
\end{enumerate}
\end{defn}
Intuitively, this means that a stack is a fibered category which allows gluing (or descent) of local objects along glueing data (or descent data). Condition \ref{defn_stacks_i} in Definition \ref{defn_stacks} implies that that the functors $\cEnd(x)$, $\cIsom(x,x')$, and $\cAut(x)$ are sheaves as well.

Most of the stacks we will consider consist of categories of sheaves with certain properties, and are therefore all variations (or more precisely, substacks as defined in \cite[3.29, 4.18]{Vis}) of the following example.
\begin{example}\label{ex_stack_sheaves}
Let $\fSh_S$, usually abbreviated to just $\fSh$ when $S$ is clear, be the category whose
\begin{enumerate}[label={\rm(\roman*)}]
\item objects are pairs $(X,\cF \colon \Sch_X \to \Sets)$ consisting of a scheme $X \in \Sch_S$ and a sheaf $\cF$ on the fppf site of schemes over $X$, and whose
\item morphisms are pairs $(g,\varphi)\colon (X',\cF') \to (X,\cF)$ where $g\colon X' \to X$ is a morphism of $S$--schemes and $\varphi \colon \cF' \to \cF|_{X'}$ is a morphism of sheaves on $\Sch_{X'}$. If $(h,\psi)\colon (X'',\cF'')\to (X',\cF')$ is another morphism, then composition is given by
\[
(g,\varphi)\circ(h,\psi) = (g\circ h,\varphi|_{X''}\circ \psi).
\]
\end{enumerate}
Equipping $\fSh$ with the structure functor $p \colon \fSh \to \Sch_S$ sending $(X,\cF) \mapsto X$ and $(g,\varphi) \mapsto g$ makes $\fSh$ the fibered category of \cite[3.20]{Vis} where we take $\cC = \Sch_S$ and $\cT$ to be the fppf topology. For an object $(X,\cF)$ and a morphism $g\colon X' \to X \in \Sch_S$, the pullback is given by the cartesian morphism $(g,\Id_{\cF|_{X'}}) \colon (X',\cF|_{X'}) \to (X,\cF)$ as outlined in \cite[3.1.3]{Vis}. Since any cartesian morphism is uniquely isomorphic to one of the form $(g,\Id_{\cF|_{X'}})$, we see that a morphism $(g,\varphi)$ is cartesian if and only if $\varphi$ is an isomorphism. The fiber over $X\in \Sch_S$ is the category of sheaves on $\Sch_X$. In fact, \cite[4.11]{Vis} shows that $\fSh$ is a stack. Briefly, this is because for any two $\cF,\cF' \in \fSh(X)$, the homomorphism functor $\cHom(\cF,\cF')$ is simply the sheaf of internal homomorphisms between $\cF$ and $\cF'$ and because $\fSh$ allows descent since sheaves can be glued along glueing data as in \cite[Tag 04TR]{Stacks}. We call $\fSh$ the stack of sheaves over $\Sch_S$.
\end{example}

There is a notion of twisting objects, similar to the contracted product discussed above, for objects in a stack. In particular, let $p\colon \fF \to \Sch_S$ be a stack and assume we have and object $x\in \fF(X)$ in the fiber over $X\in \Sch_S$, a group sheaf $\bG \colon \Sch_X  \to \Grp$, and a homomorphism $\varphi \colon \bG \to \cAut(x)$. Then, by \cite[III.2.3.1]{Gir}, for each $\bG$--torsor $\cP$ there is an object in $\fF(X)$, denoted by $\cP\wedge^\bG x$ or $\cP\wedge^\varphi x$ when we wish to emphasize $\varphi$, such that
\begin{enumerate}[label={\rm(\roman*)}]
\item $\cP\wedge^\bG x$ is locally isomorphic to $x$, and
\item by \cite[III.2.3.2.1]{Gir} there is an isomorphism of $\cAut(x)$--torsors
\[
\cP\wedge^\bG \cAut(x) \iso \cIsom(x,\cP\wedge^\bG x)
\]
where the sheaf on the left is given by the contracted product as in Section \ref{sec_contracted_products}.
\end{enumerate}
We also refer to this construction as the \emph{contracted product}.

Most of the stacks we will be interested in are fibered in groupoids and therefore match the definition of stack from \cite{Olsson}. In fact, they are gerbes in the sense of \cite{Gir} or \cite{Stacks}. We note that the notion of gerbe discussed in \cite[Ch. 12]{Olsson}, namely $\bmu$--gerbe, is more restrictive.
\begin{defn}\label{defn_gerbe}[{\cite[III.2.1.1]{Gir} or \cite[Tag 06NZ]{Stacks} (or \cite[2.2.5.1]{CF} in the split case)}]
A stack $p\colon \fF \to \Sch_S$ is a \emph{gerbe} if
\begin{enumerate}[label={\rm(\roman*)}]
\item \label{defn_gerbe_i} $\fF$ is fibered in groupoids,
\item \label{defn_gerbe_ii} for every $X\in \Sch_S$ there exists a covering $\{X_i \to X\}_{i\in I}$ such that each $\fF(X_i)$ is non-empty, and
\item \label{defn_gerbe_iii} for every $X\in \Sch_S$ and objects $x,x'\in \fF(X)$ there exists a cover $\{X_i \to X\}_{i\in I}$ such that there are isomorphisms $x|_{X_i} \cong x'|_{X_i}$ for all $i \in I$.
\end{enumerate}
\end{defn}

\begin{example}
Consider a group sheaf $\bG\colon \Sch_S \to \Grp$. Let $\fTors(\bG)$ be the fibered category over $\Sch_S$ whose
\begin{enumerate}[label={\rm(\roman*)}]
\item objects are pairs $(X,\cP)$ where $X\in \Sch_S$ and $\cP\colon \Sch_X \to \Sets$ is a $\bG|_X$--torsor,
\item morphisms are pairs $(g,\varphi)\colon (X',\cP') \to (X,\cP)$ where $g\colon X' \to X$ is a morphism of $S$--schemes and $\varphi \colon \cP' \iso \cP|_{X'}$ is a morphism (and hence isomorphism) of $\bG|_{X'}$--torsors, and whose
\item structure functor $p\colon \fTors(\bG) \to \Sch_S$ sends $(X,\cP) \to X$.
\end{enumerate}
This is a gerbe by \cite[III.1.4.5]{Gir}, however we include a justification in our own notation. The category $\fTors(\bG)$ is indeed a fibered category where the pullback of an object $(X,\cP)$ along a morphism $g \colon X' \to X$ in $\Sch_S$ is given by the morphism $(g,\Id_{\cP|_{X'}}) \colon (X',\cP|_{X'}) \to (X,\cP)$. Because all morphisms of torsors are isomorphisms, $\fTors(\bG)$ is fibered in groupoids as any morphism over $\Id_X$ is of the form $(\Id_X,\varphi)$ for an isomorphism $\varphi$. Therefore by Lemma \ref{lem_cartesian_groupoids}, all morphisms in $\fTors(\bG)$ are cartesian.

For $X\in \Sch_S$ and two objects $(X,\cP_1),(X,\cP_2) \in \fTors(\bG)(X)$, the functor of homomorphisms $\cHom_{\fTors(\bG)}((X,\cP_1),(X,\cP_2))$ is a subfunctor of the homomorphisms $\cHom_{\fSh}((X,\cP_1),(X,\cP_2))$ in $\Sh$. Since $\bG$--equivariance of morphisms can be checked locally, compatible local torsor isomorphisms glue into a torsor isomorphism and so $\cHom_{\fTors(\bG)}((X,\cP_1),(X,\cP_2))$ is itself a sheaf. Similarly, gluing torsors along descent data consisting of torsor isomorphism will produce another torsor. Therefore, $\fTors(\bG)$ is a stack over $\Sch_S$. Since torsors are by definition locally isomorphic to $\bG$, they are locally isomorphic to each other. Therefore, $\fTors(\bG)$ is a gerbe.

There is a clear inclusion functor $\fTors(\bG) \to \Sh$ which we claim is a morphism of stacks. It is straightforward to see that it respects the structure functors and so we only need to check that cartesian morphism are preserved, i.e., that the image of any morphism is again cartesian in $\Sh$. A morphism $(g,\varphi)\colon (X',\cP') \to (X,\cP)$ in $\fTors(\bG)$ can be decomposed as
\[
\begin{tikzcd}[column sep = 1in]
(X',\cP') \arrow{dr}{(g,\varphi)} \arrow[swap]{d}{(\Id_{X'},\varphi)} & \\
(X',\cP|_{X'}) \arrow[near start]{r}{(g,\Id_{\cP|_{X'}})} & (X,\cP)
\end{tikzcd}
\]
where $(g,\Id_{\cP|_{X'}})$ is cartesian in $\fSh$ and $(\Id_{X'},\varphi)$ is also cartesian in $\fSh$ by virtue of being an isomorphism. Therefore, their composition $(g,\varphi)$ is also cartesian in $\fSh$ as desired. Hence, the inclusion functor $\fTors(\bG) \to \Sh$ makes the gerbe $\fTors(\bG)$ a substack of $\Sh$.
\end{example}

The following results are analogous to those found in \cite[2.2.4]{CF}, which we borrow some notation from. However, the authors there only consider split fibered categories and stacks. See \cite[2.1.2.1]{CF} for their notion of fibered category and \cite[2.1.3.4]{CF} for stack. Because of this, and also because the result corresponding to Example \ref{ex_forms} is left as an exercise, we include short proofs of these facts in our context.

\begin{example}\label{ex_forms}
Let $p\colon \fF \to \Sch_S$ be a stack and let $X\in \Sch_S$. We call objects $x,x'\in \fF(X)$ \emph{twisted forms} of one another if they are locally isomorphic, that is, if there exists a cover $\{X_i \to X\}_{i \in I}$ and isomorphisms $x|_{X_i} \cong x'|_{X_i}$ in $\fF(X_i)$ for each $i\in I$. Note that this definition is independent of the choice of pullbacks (e.g. \cite[Remark 3.3]{Vis}). For an object $s\in \fF(S)$ we denote by $\fForms(s)$ the subcategory of $\fF$ whose
\begin{enumerate}[label={\rm(\roman*)}]
\item objects are those $x \in \fF$ (not necessarily in $\fF(S)$) such that $x$ is a twisted form of $s|_{p(x)}$, and whose
\item morphisms $y \to x$ are the cartesian morphisms between $y$ and $x$ in $\fF$.
\end{enumerate}
By \cite[4.20]{Vis}, the subcategory of $\fF$ consisting of the same objects but only cartesian morphisms, denoted $\fF_{\mathrm{cart}}$, is itself a stack. It is clear that $\fForms(s)$ is a full subcategory of $\fF_{\mathrm{cart}}$. To see that $\fForms(s)$ is a substack, we will apply \cite[4.19]{Vis} after checking that the following two conditions hold.
\begin{enumerate}[label={\rm(\roman*)}]
\item Any arrow in $\fF_{\mathrm{cart}}$ whose target is in $\fForms(s)$ is also in $\fForms(s)$.
\item If $\{X_i \to X\}_{i\in I}$ is a covering in $\Sch_S$ and $x\in \fF_{\mathrm{cart}}(X)$ is an object such that the pullbacks $x|_{X_i} \in \fForms(s)(X_i)$ for all $i\in I$, then $x\in \fForms(s)(X)$.
\end{enumerate}
To check the first condition, consider an object $x\in \fForms(s)(X)$ and a morphism $\varphi \colon y \to x$ in $\fF_{\mathrm{cart}}$. Let $g=p(\varphi) \colon Y \to X$ be the underlying map of schemes in $\Sch_S$. Since all maps in $\fF_{\mathrm{cart}}$ are cartesian, $y \cong x|_Y$ is a pullback of $x$ along $g$. Let $\{X_i \to X\}_{i\in I}$ be a cover over which $x$ is locally isomorphic to $s|_X$. The local isomorphisms $x|_{X_i} \cong s|_{X_i}$ restrict to isomorphisms $x|_{X_i\times_X Y} \cong s|_{X_i \times_X Y}$, and so $y \cong x|_Y$ is locally isomorphic to $s|_Y$ with respect to the cover $\{X_i\times_X Y \to Y\}_{i\in I}$. Therefore $y \in \fForms(s)(Y)$ and so the morphism $y \to x$ is in $\fForms(s)$ as desired.

The second condition is clear since if each $x|_{X_i}$ is locally isomorphic to $s|_{X_i}$, then there exists a refined cover over which $x$ is locally isomorphic to $s$.

Finally, we show that $\fForms(s)$ is a gerbe. First, since any arrow in $\fForms(s)(X)$ is cartesian, it is an isomorphism by \cite[3.3]{Vis}, proving condition \ref{defn_gerbe_i} of Definition \ref{defn_gerbe}. Next, Definition \ref{defn_gerbe}\ref{defn_gerbe_ii} is obvious since $s|_X \in \fForms(s)(X)$. Finally, by definition, all objects of $\fForms(s)$ in the same fiber are locally isomorphic, thus also Definition \ref{defn_gerbe}\ref{defn_gerbe_iii} holds. This shows that $\fForms(s)$ is a gerbe.
\end{example}

The following proposition connects isomorphism classes in a stack to cohomology sets of automorphism sheaves. In the context of \cite{CF}, part \ref{lem_gerbes_and_torsors_iv} is \cite[2.2.4.5]{CF} and part \ref{lem_gerbes_and_torsors_ii} follows from \cite[2.2.3.6]{CF}. Part \ref{lem_gerbes_and_torsors_iv} is \cite[III.2.5.1]{Gir}.
\begin{prop}\label{lem_gerbes_and_torsors}
Let $p\colon \fF \to \Sch_S$ be a stack and let $s\in \fF(S)$. Consider the gerbe $\fForms(s)$.
\begin{enumerate}[label={\rm(\roman*)}]
\item \label{lem_gerbes_and_torsors_i} For any $X\in \Sch_S$ and $x\in \fForms(s)(X)$, the sheaf $\cIsom(s|_X,x)$ is an $\cAut(s)|_X$--torsor.
\item \label{lem_gerbes_and_torsors_iv} There is an equivalence of stacks $\varphi \colon \fForms(s) \to \fTors(\cAut(s))$ which acts on objects by
\[
x \mapsto (p(x),\cIsom(s|_{p(x)},x))
\]
and on morphisms as follows. Let $f:y \to x$ be a morphism in $\fForms(s)$ and $p(f)\colon Y \to X$ be the underlying morphism in $\Sch_S$. For a chosen cartesian morphism $f''\colon x|_Y \to x$, there exists a unique isomorphism $f'\colon y \iso x|_Y$ in the fiber $\fForms(s)(Y)$ such that the composition $y \xrightarrow{f'} x|_Y \xrightarrow{f''} x$ is $f$. Furthermore, there is a canonical isomorphism
\[
g\colon \cIsom(s|_Y,x|_Y) \iso \cIsom(s|_X,x)|_Y.
\]
Therefore, we define $\varphi(f)$ to be the morphism $(p(f),\phi)$ where $\phi$ is the isomorphism of $\cAut(s)|_Y$--torsors
\begin{align*}
\phi \colon \cIsom(s|_Y,y) &\iso \cIsom(s|_X,x)|_Y \\
h &\mapsto g(f'\circ h).
\end{align*}
A quasi-inverse is given by the contracted product, namely the functor
\begin{align*}
\fTors(\cAut(s)) &\to \fForms(s) \\
(X,\cP) &\mapsto \cP\wedge^{\cAut(s)|_X} s|_X.
\end{align*}

\item \label{lem_gerbes_and_torsors_ii} For any $s'\in \fForms(s)(S)$, the sheaves $\cAut(s)$ and $\cAut(s')$ are twisted forms of one another as group sheaves. In particular, giving $\cAut(s)$ the left action on itself by inner automorphisms, we have
\[
\cIsom(s,s') \wedge^{\cAut(s)} \cAut(s) \cong \cAut(s').
\]
\item \label{lem_gerbes_and_torsors_iii} Denote by $H^1(S,\fForms(s))$ the isomorphism classes of the groupoid $\fForms(s)(S)$ and recall that $H^1(S,\cAut(s))$ is the set of isomorphism classes of $\cAut(s)$--torsors. This is a pointed set and there is a bijective map of pointed sets
\begin{align*}
H^1(S,\fForms(s)) &\to H^1(S,\cAut(s)) \\
[s'] &\mapsto [\cIsom(s,s')]
\end{align*}
where we choose $[s]$ as the basepoint of $H^1(S,\fForms(s))$.
\end{enumerate}
\end{prop}  
\begin{proof}
\noindent\ref{lem_gerbes_and_torsors_i}: The sheaf $\cIsom(s|_X,x)$ has a natural right action of $\cAut(s)|_X$ given by composition. For any $Y \in \Sch_X$, the action $\cIsom(s|_X,x)(Y) \times \cAut(s)(Y) \to \cIsom(s|_X,x)(Y)$ is simply transitive since any two isomorphisms $\varphi_1,\varphi_2 \in \cIsom(s|_X,x)(Y)$ differ by $\varphi_1^{-1}\circ\varphi_2 \in \cAut(s)(Y)$. There also exists a cover over which $\cIsom(s|_X,x)$ is non-empty since $x$ is a twisted form of $s|_X$ by definition. Thus, $\cIsom(s|_X,x)$ is a $\cAut(s)|_X$--torsor.

\noindent\ref{lem_gerbes_and_torsors_iv}: As mentioned above, this is \cite[III.2.5.1]{Gir}.

\noindent\ref{lem_gerbes_and_torsors_ii}: The map of presheaves
\begin{align*}
\big(\cIsom(s,s')\times \cAut(s)\big)/\sim &\to \cAut(s') \\
(g,\varphi) &\mapsto g\circ \varphi \circ g^{-1},
\end{align*}
where $\sim$ is the equivalence relation defined in Section \ref{sec_contracted_products}, is easily seen to be well defined and it is bijective whenever $\cIsom(s,s')$ has a section. Therefore, since $\cIsom(s,s')$ is an $\cAut(s)$--torsor, after sheafification it produces a sheaf isomorphism $\cIsom(s,s') \wedge^{\cAut(s)} \cAut(s) \iso \cAut(s')$ as desired.

\noindent\ref{lem_gerbes_and_torsors_iii}: This follows from \ref{lem_gerbes_and_torsors_iv}. The equivalence of stacks includes an equivalence of categories $\fForms(s)(S) \to \fTors(\cAut(s))(S)$ which in turn produces a bijection of pointed sets
\begin{align*}
H^1(S,\fForms(s)) &\iso H^1(S,\fTors(\cAut(s))) \\
[s'] &\mapsto [\cIsom(s,s')]
\end{align*}
and the latter cohomology set is by definition $H^1(S,\cAut(s))$.
\end{proof}
\begin{remark}\label{rem_gerbe_forms}
If $p\colon \fF \to \Sch_S$ is itself a gerbe, then $\fForms(s) = \fF$ by condition \ref{defn_gerbe}\ref{defn_gerbe_iii} and we may apply Proposition \ref{lem_gerbes_and_torsors} to all of $\fF$. In this case, we will often view the isomorphism $H^1(S,\fF) \iso H^1(S,\cAut(s))$ as an identification and write expressions such as $[s']\in H^1(S,\cAut(s))$ where $s'\in \fF(S)$. This is a natural shorthand and, for example, is equivalent to considering $H^1(S,\GL_r)$ as the set of isomorphism classes of locally free $\cO$--modules of constant rank $r$.
\end{remark}

Our main technique throughout Section \ref{cohomological_description} will be an application of the following lemma. This lemma is \cite[III.2.5.3]{Gir} and it also follows from \cite[2.2.3.9]{CF}.
\begin{lem}\label{lem_loos_cohom}
Let $\varphi \colon \fF \to \fG$ be a morphism of gerbes. For $s\in \fF(S)$ there is an associated morphism of group sheaves $\varphi_s \colon \cAut_{\fF}(s) \to \cAut_{\fG}(\varphi(s))$. Then, the map on first cohomology induced by $\varphi_s$ is the map
\begin{align*}
H^1(S,\cAut_{\fF}(s)) &\to H^1(S,\cAut_{\fG}(\varphi(s))) \\
[s'] &\mapsto [\varphi(s')]
\end{align*}
where we view $H^1(S,\cAut_{\fF}(s))$ as the set of isomorphism classes in $\fF(S)$ and $H^1(S,\cAut_{\fG}(\varphi(s)))$ as the isomorphism classes in $\fG(S)$.
\end{lem}
\lv{%
\begin{proof}
An element $[s'] \in H^1(S,\cAut_{\fF}(s))$ corresponds to the $\cAut_{\fF}(s)$--torsor $\cIsom_{\fF}(s,s')$. In turn, $\cIsom_{\fG}(\varphi(s),\varphi(s'))$ is an $\cAut_{\fG}(\varphi(s))$--torsor whose isomorphism class is $[\varphi(s')]$. Since $\varphi$ is a morphism of stacks, if gives a natural transformation of sheaves
\[
\varphi_{s,s'} \colon \cIsom_{\fF}(s,s') \to \cIsom_{\fG}(\varphi(s),\varphi(s'))
\]
which is compatible with the torsor structures via the group homomorphism $\varphi_s \colon \cAut_{\fF}(s) \to \cAut_{\fG}(\varphi(s))$. In particular, for sections $\psi \in \cAut_{\fF}(s)$ and $g \in \cIsom_{\fF}(s,s')$ we have
\[
\varphi_{s,s'}(g\cdot \psi) = \varphi_{s,s'}(g) \cdot \varphi_{s}(\psi).
\]
Thus, $\varphi_{s,s'}$ is what Loos calls a $\varphi_s$--morphism between torsors and so we may apply \cite[6.4]{Loos} to conclude that $[s']$ maps to $[\varphi(s')]$ under the morphism of cohomology induced by $\varphi_s$.
\end{proof}
}

We also identify equivalences of gerbes using the following Theorem. This is folklore, but we provide a proof since it is not stated in this exact fashion in \cite{Gir}.
\begin{thm}[{\cite{Gir}}]\label{lem_gerbe_equivalence}
Let $\varphi \colon \fF \to \fG$ be a morphism of gerbes. Assume there exists $s\in \fF(S)$ such that the induced group sheaf homomorphism $\varphi_s \colon \cAut_\fF(s) \to \cAut_\fG(\varphi(s))$ is an isomorphism. Then, $\varphi$ is an equivalence of gerbes.
\end{thm}
\begin{proof}
By \cite[III.2.5.1]{Gir} there is an equivalence of stacks
\begin{align*}
\fF &\to \fTors(\cAut_\fF(s)) \\
x &\mapsto (p(x),\cIsom(s|_{p(x)},x))
\end{align*}
and likewise for $\fG \to \fTors\big(\cAut_\fG(\varphi(s))\big)$. In the case of this second equivalence, we use the quasi-inverse
\begin{align*}
\fTors\big(\cAut_\fG(\varphi(s))\big) &\to \fG \\
(X,\cP) &\mapsto \cP \wedge^{\cAut_\fG(\varphi(s))|_X} \varphi(s)|_X.
\end{align*}
where $X\in \Sch_S$. Additionally, since $\varphi_s$ is an isomorphism, there is an obvious equivalence of categories
\begin{align*}
\fTors\big(\cAut_\fF(s)\big) &\to \fTors\big(\cAut_\fG(\varphi(s))\big) \\
(X,\cP) &\mapsto (X,\cP \wedge^{\varphi_s|_X} \cAut_\fG(\varphi(s))|_X)
\end{align*}
which simply interprets an $\cAut_\fF(s)$--torsor $\cP$ as an $\cAut_\fG(\varphi(s))$--torsor by giving it the right action coming from $\varphi_s^{-1}$. Therefore, we have a chain of equivalences
\[
\begin{tikzcd}[column sep=2ex]
\fF \ar[r] & \fTors(\cAut_\fF(s)) \ar[r] & \fTors\big(\cAut_\fG(\varphi(s))\big) \ar[r] & \fG \\[-5ex]
x \ar[rrr,mapsto] & & & \cIsom(s|_{p(x)},x)\wedge^{\varphi_s|_{p(x)}} \varphi(s)|_{p(x)}.
\end{tikzcd}
\]
Now, by composing the equivalence $\fF \to \fTors(\cAut_\fF(s))$ with its quasi-inverse, we see that there is an isomorphism
\[
x \iso \cIsom(s|_{p(x)},x)\wedge^{\cAut_\fF(s)|_{p(x)}} s|_{p(x)}.
\]
Furthermore, by \cite[III.2.3.11]{Gir} we have that
\[
\varphi\big(\cIsom(s|_{p(x)},x)\wedge^{\cAut_\fF(s)|_{p(x)}} s|_{p(x)}\big) \iso \cIsom(s|_{p(x)},x)\wedge^{\varphi_s|_{p(x)}} \varphi(s)|_{p(x)}.
\]
Therefore, $\varphi$ is canonically isomorphic to the composition of three equivalences above and is hence an equivalence itself, as claimed.
\end{proof}

We will use the following results, in particular Corollary \ref{cor_identity_equiv}, to identify a fiberwise quasi-inverse of an equivalence of gerbes.
\begin{prop}\label{identity_equiv}
Let $\fF \to \Sch_S$ be a gerbe with a global object $x\in \fF(S)$. If $\varphi \colon \fF \to \fF$ is a morphism of gerbes such that there is an isomorphism $g \in \cIsom(\varphi(x),x)(S)$ and the induced group sheaf homomorphism
\[
\bAut_\fF(x) \xrightarrow{\varphi_x} \bAut_\fF(\varphi(x)) \xrightarrow{\Inn(g)} \bAut_\fF(x)
\]
is the identity, then for every $T\in \Sch_S$ there is a natural isomorphism of functors $\rho_T \colon \varphi_T \iso \Id_{\fF(T)}$, on the fiber $\fF(T)$ over $T$.
\end{prop}
\begin{proof}
As a preliminary step, we replace $\varphi$ with an isomorphic functor $\varphi'$ which has the property that $\varphi'(x|_T)=x|_T$ for all $T\in \Sch_S$ and the induced group sheaf homomorphism $\varphi'_x \colon \bAut_\fF(x) \to \bAut_\fF(x)$ is the identity. This is of course possible by only modifying the image of $x$ under $\varphi$, and using $g$ to modify homomorphisms involving $x$. In detail, we define $\varphi' \colon \fF \to \fF$ to behave on objects by
\[
y \mapsto {\begin{cases} \varphi(y) & y\neq x|_{p(y)} \\ x|_{p(y)} & y = x|_{p(y)} \end{cases}}
\]
and on morphisms by
\[
(f\colon y \to z) \mapsto {\begin{cases} \varphi(f), & y\neq x|_{p(y)} \text{ and } z\neq x|_{p(z)} \\ g|_{p(z)}\circ \varphi(f), & y\neq x|_{p(y)} \text{ and } z = x|_{p(z)} \\ \varphi(f)\circ g|_{p(y)}^{-1}, & y = x|_{p(y)} \text{ and } z \neq x|_{p(z)} \\ g|_{p(z)}\circ \varphi(f) \circ g|_{p(y)}^{-1}, & y = x|_{p(y)} \text{ and } z = x|_{p(z)}.  \end{cases}}
\]
Then, the isomorphisms
\begin{align*}
\Id_\varphi(y) \colon \varphi(y) &\iso \varphi(y)=\varphi'(y) \text{ for } y\neq x|_{p(y)} \text{, and} \\ 
g|_T \colon \varphi(x|_T) &\iso x|_T = \varphi'(x|_T) \text{ for all } T\in \Sch_S
\end{align*}
provide a natural isomorphism $\varphi \iso \varphi'$. The fact that $\varphi'_x$ is the identity follows immediately from our assumption that $\Inn(g)\circ \varphi_x = \Id$. Hence, without loss of generality we may assume that $\varphi$ itself has these properties. By Theorem \ref{lem_gerbe_equivalence}, $\varphi$ is an equivalence of categories and so for each $T\in \Sch_S$, $\varphi_T \colon \fF(T) \to \fF(T)$ is an equivalence of categories between the fibers.

The core of our approach is the following point of view. Given $T\in \Sch_S$ and $y,z\in \fF(T)$, we may choose a cover $\{T_i\to T\}_{i\in I}$ which ``splits" both $y$ and $z$, i.e., such that there are isomorphisms $\alpha_i \colon y|_{T_i} \iso x|_{T_i}$ and $\beta_i \colon z|_{T_i} \iso x|_{T_i}$ for each $i\in I$. Setting $\alpha_{ij} = \alpha_j|_{T_{ij}} \circ (\alpha_i)^{-1}|_{T_{ij}}$ and likewise for $\beta_{ij}$, we obtain 1-cocycles in $\bAut_\fF(x|_T)$ defining $y$ and $z$ respectively. With this setup, a morphism $f \colon y \to z$ is equivalent to automorphisms $g_i \in \bAut_\fF(x)(T_i)$ such that all the diagrams
\[
\begin{tikzcd}
x|_{T_{ij}} \ar[r,"g_i|_{T_{ij}}"] \ar[d,"\alpha_{ij}"] & x|_{T_{ij}} \ar[d,"\beta_{ij}"] \\
x|_{T_{ij}} \ar[r,"g_j|_{T_{ij}}"] & x|_{T_{ij}}
\end{tikzcd}
\]
commute. Indeed, we think of this as part of the following diagram.
\[
\begin{tikzcd}
(y|_{T_i})|_{T_{ij}} \ar[d,equals] \ar[r,"\alpha_i|_{T_{ij}}"] \ar[rrr,bend left,"(f|_{T_i})|_{T_{ij}}"] & x|_{T_{ij}} \ar[r,"g_i|_{T_{ij}}"] \ar[d,"\alpha_{ij}"] & x|_{T_{ij}} \ar[d,"\beta_{ij}"] & (z|_{T_i})|_{T_{ij}} \ar[d,equals] \ar[l,swap,"\beta_i|_{T_{ij}}"] \\
(y|_{T_j})|_{T_{ij}} \ar[r,"\alpha_j|_{T_{ij}}"] \ar[rrr,bend right,"(f|_{T_j})|_{T_{ij}}"] & x|_{T_{ij}} \ar[r,"g_j|_{T_{ij}}"] & x|_{T_{ij}} & (z|_{T_j})|_{T_{ij}} \ar[l,swap,"\beta_j|_{T_{ij}}"]
\end{tikzcd}
\]
To provide some foreshadowing, the result essentially follows from the fact that, because of our assumption on $\varphi_x \colon \bAut_\fF(x) \to \bAut_\fF(x)$, the center square of such diagrams is fixed by $\varphi$.

For a $T\in \Sch_S$, we construct a natural isomorphism $\rho_T \colon \varphi_T \to \Id_{\fF(T)}$ using the approach above. Take an object $y\in \fF(T)$ and consider a cover $\{T_i \to T\}_{i\in I}$ and isomorphisms $\alpha_i \colon y|_{T_i} \iso x|_{T_i}$ in $\fF(T_i)$. Because $x|_{T_i}$ is fixed by $\varphi$, we also obtain a splitting of $\varphi(y)$ over the same cover, namely $\varphi(\alpha_i) \colon \varphi(y)|_{T_i} \iso x|_{T_i}$. Also, because $\alpha_{ij} \in \bAut_\fF(x)(T_{ij})$ is fixed, we have
\[
\alpha_j|_{T_{ij}} \circ (\alpha_i)^{-1}|_{T_{ij}} = \alpha_{ij} = \varphi(\alpha_j)|_{T_{ij}} \circ \varphi(\alpha_i)^{-1}|_{T_{ij}}.
\]
Therefore, if we choose the identity $x|_{T_i} \to x|_{T_i}$ for each $i\in I$, the diagrams
\[
\begin{tikzcd}
(\varphi(y)|_{T_i})|_{T_{ij}} \ar[d,equals] \ar[r,"\varphi(\alpha_i)|_{T_{ij}}"] & x|_{T_{ij}} \ar[r,equals] \ar[d,"\alpha_{ij}"] & x|_{T_{ij}} \ar[d,"\alpha_{ij}"] & (y|_{T_i})|_{T_{ij}} \ar[d,equals] \ar[l,swap,"\alpha_i|_{T_{ij}}"] \\
(\varphi(y)|_{T_j})|_{T_{ij}} \ar[r,"\varphi(\alpha_j)|_{T_{ij}}"] & x|_{T_{ij}} \ar[r,equals] & x|_{T_{ij}} & (y|_{T_j})|_{T_{ij}} \ar[l,swap,"\alpha_j|_{T_{ij}}"]
\end{tikzcd}
\]
commute and therefore define an isomorphism $\rho_y \colon \varphi(y) \to y$ in $\fF(T)$. Note that $\rho_y$ does not depend on the choice of cover or isomorphisms. Indeed, if we have a second cover we can compare the construction over some refinement of the cover, therefore it is sufficient to assume that we have a cover $\{T_i \to T\}_{i\in I}$ and that we make two choices of isomorphisms
\[
\alpha_i \colon y|_{T_i} \iso x|_{T_i} \text{ and } \beta_i \colon y|_{T_i} \iso x|_{T_i}.
\]
Using these, we define $\rho_y \colon \varphi(y) \to y$ and $\rho_y' \colon \varphi(y) \to y$ such that
\[
\rho_y|_{T_i} = \alpha_i^{-1} \circ \varphi(\alpha_i) \text{ and } \rho_y'|_{T_i} = \beta_i^{-1} \circ \varphi(\beta_i).
\]
However, the morphism $\beta_i\circ \alpha_i^{-1}$ is an automorphism of $x|_{T_i}$ and is thus fixed by $\varphi$ by assumption. Therefore
\[
\beta_i\circ \alpha_i^{-1} = \varphi(\beta_i)\circ \varphi(\alpha_i)^{-1}
\]
which can be rearranged to
\[
\rho_y|_{T_i} = \alpha_i^{-1}\circ \varphi(\alpha_i) = \beta_i^{-1}\circ \varphi(\beta_i) = \rho_y'|_{T_i},
\]
so $\rho_y=\rho_y'$ globally. Thus, we may describe $\rho_y$ as the unique map such that, for any $W\in \Sch_T$ and isomorphism $\alpha \colon y|_W \to x|_W$, we have $\rho_y|_W = \alpha^{-1}\circ \varphi(\alpha)$. This shows that the $\rho_y$ are canonical morphisms.

We claim that the isomorphisms $\rho_y$ for $y\in \fF(T)$ define a natural isomorphism of functors $\rho_T \colon \varphi_T \to \Id_{\fF(T)}$. We must show that, given a morphism $f\colon y \to z$ in $\fF(T)$, the diagram
\[
\begin{tikzcd}
\varphi(y) \ar[d,"\varphi(f)"] \ar[r,"\rho_y"] & y \ar[d,"f"] \\
\varphi(z) \ar[r,"\rho_z"] & z
\end{tikzcd}
\]
commutes. Choose a cover $\{T_i \to T\}_{i\in I}$ which mutually splits $y$ and $z$, denote their splitting isomorphisms with $\alpha$'s and $\beta$'s as above. Let $g_i \in \bAut_\fF(x)(T_i)$ be the isomorphisms defining $f$, i.e.,
\[
f|_{T_i} = \beta_i^{-1} \circ g_i \circ \alpha_i
\]
for each $i\in I$. This means that
\[
\varphi(f)|_{T_i} = \varphi(\beta_i)^{-1} \circ g_i \circ \varphi(\alpha_i).
\]
We can then compute that
\begin{align*}
(\rho_z \circ \varphi(f))|_{T_i} &= \big(\beta_i^{-1}\circ \varphi(\beta_i)\big)\circ \varphi(\beta_i)^{-1} \circ g_i \circ \varphi(\alpha_i) \\
&= \beta_i^{-1} \circ g_i \circ \varphi(\alpha_i) \\
&= \beta_i^{-1} \circ g_i \circ \alpha_i \circ \big(\alpha_i^{-1}\circ \varphi(\alpha_i)\big) \\
&= (f\circ \rho_y)|_{T_i}
\end{align*}
for each $i\in I$, and therefore $\rho_z\circ \varphi(f) = f \circ \rho_y$ globally, as required. This finishes the proof.
\end{proof}

\begin{cor}\label{cor_identity_equiv}
Let $\varphi \colon \fF \to \fG$ and $\psi\colon \fG \to \fF$ be morphisms of gerbes over $\Sch_S$. Assume there exists a global object $x\in \fF(S)$ and an isomorphism $g \in \cIsom((\psi\circ\varphi)(x),x)(S)$ such that the induced group sheaf homomorphism
\[
\bAut_\fF(x) \xrightarrow{(\psi\circ\varphi)_x} \bAut_\fF((\psi\circ\varphi)(x)) \xrightarrow{\Inn(g)} \bAut_\fF(x)
\]
is the identity. Then, $\varphi$ and $\psi$ are fiberwise mutual quasi-inverse equivalences of categories. I.e., for every $T\in \Sch_S$, the fiber functors $\varphi_T$ and $\psi_T$ are quasi-inverses.
\end{cor}
\begin{proof}
It is immediate from Proposition \ref{identity_equiv} that for all $T\in \Sch_S$ there is a natural isomorphism of functors $\psi_T \circ \varphi_T \iso \Id_{\fF(T)}$. For the other composition, $\varphi\circ \psi$, we take $\varphi(x) \in \fG(S)$ as the global object and
\[
\varphi(g) \colon (\varphi\circ \psi)(\varphi(x)) \iso \varphi(x)
\]
as the isomorphism. Now let $h\in \bAut_\fG(\varphi(x))$ be any isomorphism. Since $\varphi$ is an equivalence of gerbes it is full, and so there exists $f\in \bAut_\fF(x)$ such that $\varphi(f)=h$. Then we can compute,
\begin{align*}
&\varphi(g) \circ (\varphi\circ\psi)(h) \circ \varphi(g)^{-1} \\
=& \varphi(g) \circ (\varphi\circ\psi)(\varphi(f)) \circ \varphi(g)^{-1} \\
=& \varphi\big(g\circ (\psi\circ\varphi)(f) \circ g^{-1}) \\
=& \varphi(f) \\
=& h
\end{align*}
and so Proposition \ref{identity_equiv} applies again. This means that for any $T\in \Sch_S$, we also have a natural isomorphism of functors $\varphi_T \circ \psi_T \iso \Id_{\fG(T)}$, and so $\varphi_T$ and $\psi_T$ are quasi-inverses.
\end{proof}

Throughout the paper we use various stacks and gerbes, some of which we have created our own notation for. Here we provide a small index noting where they are each defined.
\begin{enumerate}[label={(\roman*)}]
\item $\QCoh$ and $\QCoh\flf$ -- Section \ref{the_construction}.
\item $\QAlg$ and $\QAlg\flf$ -- Remark \ref{rem_QAlg}.
\item $\fMod_r$, $\fMod_r^{d\etale}$, $\fAzu_r$, and $\fAzu_r^{d\etale}$ -- start of Section \ref{cohomological_description}.
\item $\fTMod_r$ -- after \eqref{eq_comp_d}.
\item $\fTAzu_r$ -- before Lemma \ref{lem_TAzu_aut}.
\item $\fA_m^{d\etale}$, $\fC_m^{d\etale}$, and $\fD_m$ -- start of Section \ref{Equivalence}.
\item $\fC_{(n_1,\ldots,n_d)}$ -- Section \ref{subsec_twisting_triples}.
\item $\fF(\bG)^{d\etale}$, which is a generalization of the previous stacks featuring ``$d\etale$" -- above \eqref{AppenB_gerbe_tors_equiv}.
\item $\QCoh_\fJ$ and $\fAffMor$ -- Appendix \ref{app_stack_morphism}.
\end{enumerate}

\section{The Norm Functor over Rings}\label{sec_review_norms}
We begin by reviewing Ferrand's construction over rings from \cite{F}, which generalizes the construction of Knus and Ojanguren \cite{KO75} (see \cite[5.3]{F}) and also Tignol's construction over a field from \cite{Ti}. The norm functor has also been generalized in various other ways, for example to quasi-coherent sheaves on algebraic spaces, by Rydh in his thesis \cite{Rydh}. The interested reader can find the details in \cite{Rydh2}.

We also discuss the norm functor defined by Rost in \cite{Rost}, which has arguably a simpler definition and agrees with Ferrand's in many cases, such as when the module is flat, see Proposition \ref{noho}. We do so in the interest of being comprehensive and to justify why we generalize Ferrand's norm instead of Rost's. Critically, we provide an example showing that Rost's norm is not compatible with arbitrary base change and therefore does not generalize to our setting of stacks.

\subsection{Ferrand's Norm Functor over Rings}\label{sec_Ferrands_norm}
Ferrand's construction begins with the \emph{$\Gamma$--algebra}, also called the \emph{divided power algebra}, of a module. We summarize this construction and refer to \cite[IV,\S 5.4]{B:A2} and the related exercises as well as \cite{Roby} for more details. 

\subsubsection{The Construction}\label{sec_The_Construction}
Let $R$ be a unital, commutative, associative ring and let $M$ be an $R$--module. We let $\Gamma_R(M)$ be the unital, commutative, associative $R$--algebra generated by symbols $\gamma^d(m)$ for $d\in \NN = \{0,1,\ldots\}$ and $m\in M$, subject to the relations
\begin{align*}
\gamma^0(m) &= 1_{\Gamma(M)} \\
\gamma^d(rm) &= r^d \gamma^d(m) \\
\gamma^d(m+n) &= \sum_{r=0}^d \gamma^r(m)\gamma^{d-r}(n) \\
\gamma^{d_1}(m)\gamma^{d_2}(m) &= \binom{d_1 + d_2}{d_1} \gamma^{d_1+d_2}(m)
\end{align*}
for all $d,d_1,d_2 \in \NN$, $m,n\in M$, and $r\in R$. Here $\binom{d_1 + d_2}{d_1}=\frac{(d_1+d_2)!}{d_1!d_2!}$ is the binomial coefficient. We note that \cite{B:A2} uses the notation $\gamma_d(m)$ where we follow \cite{F} and use $\gamma^d(m)$. The algebra $\Gamma_R(M)$ is $\NN$--graded by total degree of the ``exponents", i.e., setting
\[
\Gamma_R^d(M) = \Span_R\left(\left\{\gamma^{d_1}(m_1)\gamma^{d_2}(m_2)\ldots \gamma^{d_k}(m_k) \mid \sum_{i=1}^k d_i = d, m_i\in M \right\}\right)
\]
gives an $\NN$--grading $\Gamma_R(M) = \oplus_{d\in \NN} \Gamma_R^d(M)$. The assignment $M\mapsto \Gamma_R(M)$ is functorial in $M$. Given a morphism $f\colon M_1 \to M_2$ of $R$--modules, we set $\Gamma_R(f)$ to be the map defined on generators as
\begin{align*}
\Gamma_R(f) \colon \Gamma_R(M_1) &\to \Gamma_R(M_2) \\
\gamma^d(m) &\mapsto \gamma^d(f(m)).
\end{align*}
This is an algebra homomorphism and thus we have a functor $\Gamma_R \colon \fMod_R \to \Rings_R$. The morphisms $\Gamma_R(f)$ are also graded morphisms and therefore via restriction to homogeneous components we also obtain endofunctors $\Gamma_R^d \colon \fMod_R \to \fMod_R$.

Given two $R$--modules $M$ and $N$, \cite[2.4.1]{F} says that there exists a unique $R$--linear map
\begin{equation}\label{eq_Ferrand_mu}
\mu \colon \Gamma_R^d(M)\otimes_R \Gamma_R^d(N) \to \Gamma_R^d(M\otimes_R N)
\end{equation}
with the property that $\gamma^d(m)\otimes \gamma^d(n) \mapsto \gamma^d(m\otimes n)$ for all $m\in M$ and $n\in N$.

Now consider a ring extension $R\to R'$. We may consider $R'$ as an $R$--module and for $d\in \NN$ obtain the $R$--module $\Gamma_R^d(R')$. Using the map of \eqref{eq_Ferrand_mu} and the multiplication of $R'$, $m \colon R'\otimes_R R' \to R'$, we define a multiplication on $\Gamma_R^d(R')$ by
\[
\Gamma_R^d(R') \otimes_R \Gamma_R^d(R') \xrightarrow{\mu} \Gamma_R^d(R'\otimes_R R') \xrightarrow{\Gamma_R^d(m)} \Gamma_R^d(R').
\]
This makes $\Gamma_R^d(R')$ a unital, commutative, associative $R$--algebra as it inherits these properties from $R'$. If we are now also given an $R'$--module $M'$, we may similarly view it as an $R$--module and construct $\Gamma_R^d(M')$. This can be equipped with the structure of a $\Gamma_R^d(R')$--module, again using the map \eqref{eq_Ferrand_mu} and the map $R'\otimes_R M' \to M'$ defined by $r'\otimes m' \mapsto r'm'$ coming from the $R'$--module structure of $M'$. Therefore, the composition
\[
\Gamma_R^d(R')\otimes_R \Gamma_R^d(M') \xrightarrow{\mu} \Gamma_R^d(R'\otimes_R M') \to \Gamma_R^d(M')
\]
makes $\Gamma_R^d(M')$ a $\Gamma_R^d(R')$--module, as in \cite[2.4.6]{F}. This structure has the property that $\gamma^d(r')\cdot \gamma^d(m') = \gamma^d(r'm')$.

\begin{example}\label{aldm-c2}
Let $d=2$. The $\Ga^2_R(A)$--module structure of $\Ga^2_R(M)$ is determined by the formulas below where $a,a_1, a_2\in A$ and $m, m_1, m_2\in M$: 
\begin{align*}
  \ga^2(a) \cdot \ga^2(m) &= \ga^2(a m), \\
  \ga^2(a) \cdot (\ga^1(m_1) \times \ga^1(m_2)) &= \ga^1(am_1) \times \ga^1(am_2), \\
   (\ga^1(a_1) \times \ga^1(a_2))\cdot \ga^2(m) &= \ga^1(a_1m) \times \ga^1(a_2m),
  \\ (\ga^1(a_1) \times \ga^1(a_2))\cdot(\ga^1(m_1) \times \ga^1(m_2)) & = 
  (\ga^1(a_1 m_1) \times \ga^1(a_2m_2) )
  \\ & \quad + (\ga^1(a_1 m_2) \times \ga^1(a_2 m_1)).
\end{align*}
Specializing $M=A$, these formulas become the ones in \cite[2.4.5]{F}, where the $\Ga^2_R(A)$--action on itself (= multiplication in $\Ga^2_R(A)$) is denoted by $\star$ and where $\ga^1(\cdot)$ is taken as identification.
\end{example}

Now, assume that the ring extension $R \to R'$ is locally free of finite rank $d$. We therefore have the determinant map, $\det \colon \End_R(R') \to R$. For $r' \in R'$, the determinant of the left multiplication by $r'$ yields the norm map
\begin{equation}\label{eq_ring_norm}
\norm_{R'/R} \colon R' \to R.
\end{equation}
By \cite[3.1.2]{F}, there exists an $R$--algebra homomorphism
\begin{equation}\label{eq_Gamma_pi}
\pi \colon \Gamma_R^d(R') \to R
\end{equation}
with the property that $\pi(\gamma^d(r'))=\norm_{R'/R}(r')$ for all $r'\in R'$. This is used to define the \emph{norm} of an $R'$--module $M'$. Namely, the $R$--module
\begin{equation}\label{eq_Ferrands_norm}
N_{R'/R}(M') = \Gamma_R^d(M')\otimes_{\Gamma_R^d(R')} R
\end{equation}
where $\Gamma_R^d(R')$ acts on $R$ via $\pi$. Since $\Gamma_R^d$ is a functor, so is $N_{R'/R} \colon \fMod_{R'} \to \fMod_R$. The norm of each $M'$ comes equipped with a canonical (non-linear) function
\begin{align*}
\nu_{M'} \colon M' &\to N_{R'/R}(M') \\
m' &\mapsto \gamma^d(m')\otimes 1_R.
\end{align*}
This function has the property that for all $m'\in M'$ and $r'\in R'$, we have $\nu_{M'}(r'm') = \norm_{R'/R}(r')\cdot \nu_{M'}(m')$. This can be seen by calculation since
\begin{align*}
r'm' \mapsto \gamma^d(r'm')\otimes 1_R &= (\gamma^d(r')\cdot \gamma^d(m'))\otimes 1_R \\ 
&= \gamma^d(m')\otimes \pi(\gamma^d(r'))\cdot 1_R \\
&= \gamma^d(m')\otimes \norm_{R'/R}(r').
\end{align*}

\subsubsection{Polynomial Laws}
While the above description is the essentials of Ferrand's construction, he instead primarily works with \emph{polynomial laws}. Let $\Rings_R$ denote the category of $R$--algebras which are themselves associative, commutative, and unital. For an $R$--module $N$, denote by $\bW_R(N) \colon \Rings_R \to \Ab$ the functor $Q \mapsto N\otimes_R Q$. Note that $\bW_R$ is functorial itself. A morphism $\varphi \colon N_1 \to N_2$ of $R$--modules gives rise to a natural transformation $\bW_R(\varphi) \colon \bW_R(N_1) \to \bW_R(N_2)$ defined over $Q\in \Rings_R$ by $\bW_R(\varphi)(Q)=\varphi\otimes 1 \colon N_1\otimes_R Q \to N_2\otimes_R Q$. For two $R$--modules $N_1$ and $N_2$, a natural transformation of functors $\bW_R(N_1)\to \bW_R(N_2)$ is called a \emph{polynomial law}. Of course, $\bW_R(\varphi)$ defined above is an example. Such examples are linear, however a general polynomial law need not be. For example, a polynomial law $\bnu \colon \bW_R(N_1) \to \bW_R(N_2)$ is called \emph{homogeneous of degree $d$} if we have $\bnu(qn)=q^d\bnu(n)$ for all $Q\in \Rings_R$, $q\in Q$, and $n\in N_1\otimes_R Q$. We will generally denote polynomial laws in bold. The canonical, and indeed universal as explained below, example of such a homogeneous of degree $d$ polynomial law is
\begin{equation}\label{eq_gamma_poly_law}
\bgamma^d \colon \bW_R(N_1) \to \bW_R(\Gamma_R^d(N_1))
\end{equation}
which behaves over $Q \in \Rings_R$ by
\[
\sum_{i=1}^k n_i\otimes q_i \mapsto \sum_{(a_1,\ldots,a_k)\in \NN^k \atop a_1+\ldots + a_k = d} \gamma^{a_1}(n_1)\dots \gamma^{a_k}(n_k) \otimes q_1^{a_1}\dots q_k^{a_k}.
\]

By \cite[2.2.4]{F}, which itself quotes \cite[IV.1]{Roby}, if $\bnu \colon \bW_R(N_1) \to \bW_R(N_2)$ is a homogeneous polynomial law of degree $d$ between two $R$--modules, then there exists a unique $R$--module homomorphism $\varphi_{\bnu} \colon \Gamma_R^d(N_1) \to N_2$ such that $\bnu = \bW_R(\varphi_{\bnu})\circ \bgamma^d$.

If $R \to R'$ is a finite locally free extension, for any $Q\in \Rings_R$ the extension $Q\to R'\otimes_R Q$ is also finite locally free and hence has a its own norm map. Given a morphism $f\colon Q_1 \to Q_2$ in $\Rings_R$, the associated norm maps are related by the commutative diagram
\begin{equation}\label{eq_norm_poly_law}
\begin{tikzcd}[column sep=1in]
R'\otimes_R Q_2 \ar[r,"\norm_{(R'\otimes_R Q_2)/Q_2}"] & Q_2 \ar[r,"1\otimes \Id"] & R\otimes_R Q_2 \\
R'\otimes_R Q_1 \ar[u,swap,"\Id\otimes f"] \ar[r,"\norm_{(R'\otimes_R Q_1)/Q_1}"] & Q_1 \ar[u,swap,"f"] \ar[r,"1\otimes \Id"] & R\otimes_R Q_2. \ar[u,swap,"\Id\otimes f"]
\end{tikzcd}
\end{equation}
Therefore, we have a polynomial law, $\bnorm \colon \bW_R(R') \to \bW_R(R)$, given by the norms. If $R'$ is of rank $d$, then the norm is a homogeneous polynomial law of degree $d$. The map $\pi \colon \Gamma_R^d(R')\to R$ used above is simply $\varphi_{\bnorm}$ coming from the universal property of $\Gamma_R^d(R')$.

If $M'$ is an $R'$--module and $N$ is an $R$--module, we may consider polynomial laws $\bnu \colon \bW_R(M') \to \bW_R(N)$. If such a $\bnu$ has the property that
\[
\bnu(r'm')=\bnorm(r')\bnu(m')
\]
for all $Q\in \Rings_R$, $r'\in R'\otimes_R Q$, and $m'\in M'\otimes_R Q$, then using the terminology of Ferrand we say that $\bnu$ is a \emph{normic polynomial law}. Of course, $\bnorm \colon \bW_R(R') \to \bW_R(R)$ is a normic polynomial law since the norm is multiplicative, i.e., we have
\[
\norm_{(R'\otimes_R Q)/Q}(ab)=\norm_{(R'\otimes_R Q)/Q}(a)\norm_{(R'\otimes_R Q)/Q}(b)
\]
for all $Q\in \Rings_R$.

\subsubsection{Base Change}
The construction of the $\Gamma$--algebra is compatible with ring extensions as follows, demonstrated in \cite[III.3]{Roby}. Given an arbitrary ring extension $R\to Q$ and an $R$--module $M$, there is a canonical graded isomorphism of $Q$--algebras
\begin{align}
\varphi_Q \colon \Gamma_R(M)\otimes_R Q &\iso \Gamma_Q(M\otimes_R Q) \label{eq_f_tilde} \\
\gamma^d(m)\otimes q &\mapsto q\cdot \gamma^d(m\otimes 1). \nonumber
\end{align}
Since this isomorphism is graded, for each $d$ it restricts to an isomorphism of $Q$--modules $\varphi_Q^d \colon \Gamma_R^d(M)\otimes_R Q \iso \Gamma_R^d(M\otimes_R Q)$. In the case $R \to R'$ is another ring extension, then $\varphi_Q^d \colon \Gamma_R^d(R')\otimes_R Q \iso \Gamma_R^d(R'\otimes_R Q)$ is an isomorphism of $Q$--algebras.

Since $\varphi_Q^d$ is the restriction of a ring homomorphism, for $\sum a_i = d$ and $m_i \in M$ we have
\begin{align*}
&\varphi_Q^d(\gamma^{a_1}(m_1)\dots\gamma^{a_k}(m_k)\otimes q) \\
=& \varphi_Q^d\big((\gamma^{a_1}(m_1)\otimes q)\cdot (\gamma^{a_2}(m_2)\otimes 1)\dots (\gamma^{a_k}(m_k)\otimes 1)\big) \\
=& q\cdot \gamma^{a_1}(m_1\otimes 1)\dots \gamma^{a_k}(m_k\otimes 1).
\end{align*}

For our purposes, we will be interested in the following notion of base change. Consider a pushout diagram of commutative $R$--algebras, or equivalently simply a pushout diagram of rings,
\[
D = \begin{tikzcd}
R' \ar[r,"f'"] & Q' \\
R \ar[u] \ar[r,"f"] & Q \ar[u]
\end{tikzcd}
\]
where the left vertical arrow is a finite locally free extensions of rank $d$, which then also holds for the right vertical arrow. Moreover, since this is a pushout diagram, there is a unique isomorphism of $R$--algebras
\begin{equation}\label{eq_pushout_psi}
\psi \colon R'\otimes_R Q \iso Q'.
\end{equation}
However, because we have applications to stacks in mind, we avoid identifying $Q'$ and $R'\otimes_R Q$.
\begin{lem}\label{lem_pushout_Gamma_base_change}
Assume we have a pushout diagram of rings as above. Then, there is a commutative diagram
\[
\begin{tikzcd}
\Gamma_R^d(R') \ar[r,"\Id\otimes 1"] \ar[d,"\pi_{R'}"] & \Gamma_R^d(R')\otimes_R Q \ar[r,"\varphi_Q^d"] \ar[d,"\pi_{R'}\otimes \Id"] & \Gamma_Q^d(R'\otimes_R Q) \ar[r,"\Gamma_Q^d(\psi)"] \ar[d,"\pi_{R'\otimes_R Q}"] & \Gamma_Q^d(Q') \ar[d,"\pi_{Q'}"] \\
R \ar[rr,bend right=20,swap,"f"] \ar[r,"\Id\otimes 1"] & R\otimes_R Q \ar[r,"\mathrm{can}"] & Q \ar[r,equals] & Q
\end{tikzcd}
\]
where $\varphi_Q^d$ is as defined in \eqref{eq_f_tilde}, $\pi_{R'}$ is the unique morphism such that $\bW_R(\pi_{R'})\circ \bgamma^d =\bnorm \colon \bW_R(R')\to \bW_R(R)$, and likewise for $\pi_{R'\otimes_R Q}$ and $\pi_{Q'}$.
\end{lem}
\begin{proof}
The commutativity of the left square of the diagram is obvious. The commutativity of the middle square of the diagram is more involved, following from the universal properties defining $\pi_{R'}$ and $\pi_{R'\otimes_R Q}$.

We consider the canonical isomorphism $\bW_Q(f) \colon \bW_R(R)|_Q \iso \bW_Q(Q)$ arising from the isomorphisms $f\otimes\Id \colon R\otimes_R P \iso Q\otimes_Q P$ for all $P\in \Rings_Q$ as well as the isomorphism $\bW_Q(\Id\otimes 1_Q)\colon \bW_R(R')|_Q \iso \bW_Q(R'\otimes_R Q)$ arising from the isomorphisms $R'\otimes_R P \iso (R'\otimes_R Q)\otimes_Q P$. Similarly, there is an isomorphism
\[
\bW_Q(\varphi_Q^d\circ (\Id\otimes 1_Q)) \colon \bW_R(\Gamma_R^d(R'))|_Q \iso \bW_Q(\Gamma_Q^d(R'\otimes_R Q)).
\]
Let $\bgamma_R^d \colon \bW_R(R')\to \bW_R(\Gamma_R^d(R'))$ be the polynomial law of \eqref{eq_gamma_poly_law} and likewise let $\bgamma_Q^d \colon \bW_Q(R'\otimes_R Q) \to \bW_Q(\Gamma_Q^d(R'\otimes_R Q))$ be the analogous polynomial law for $R'\otimes_R Q$. We claim that the diagram
\[
\begin{tikzcd}
\bW_R(R')|_Q \ar[r,"\bgamma_R^d|_Q"] \ar[d,"\bW_Q(\Id\otimes 1_Q)"] & \bW_R(\Gamma_R^d(R'))|_Q \ar[d,"\bW_Q(\varphi_Q^d\circ(\Id\otimes 1_Q))"] \\
\bW_Q(R'\otimes_R Q) \ar[r,"\gamma_Q^d"] & \bW_Q(\Gamma_Q^d(R'\otimes_R Q))
\end{tikzcd}
\]
commutes. For $P\in \Rings_Q$, one can trace the image of an element $\sum_{i=1}^k r'_i\otimes p_i \in R'\otimes_R P$ through the diagram, obtaining
\[
\begin{tikzcd}
\displaystyle\sum_{i=1}^k r'_i\otimes p_i \ar[r,mapsto] \ar[d,mapsto] & \displaystyle\sum_{(a_1,\ldots,a_k)\in \NN^k \atop a_1+\ldots +a_k = d} \gamma^{a_1}(r'_1)\dots \gamma^{a_k}(r'_k)\otimes p_1^{a_1}\dots p_k^{a_k} \ar[d,mapsto] \\
\displaystyle\sum_{i=1}^k (r'_i\otimes 1_Q)\otimes p_i \ar[r,mapsto] & \displaystyle\sum_{(a_1,\ldots,a_k)\in \NN^k \atop a_1+\ldots+a_k = d} \gamma^{a_1}(r'_1\otimes 1_Q)\dots \gamma^{a_k}(r'_k\otimes 1_Q)\otimes p_1^{a_1}\dots p_k^{a_k}
\end{tikzcd}
\]
which justifies our claim. Next, we let $\bnorm_{R'} \colon \bW_R(R')\to \bW_R(R)$ be the normic polynomial law associated to the norm of $R'$ and likewise we let $\bnorm_{R'\otimes_R Q} \colon \bW_Q(R'\otimes_R Q) \to \bW_Q(Q)$ be the one associated to $R'\otimes_R Q$. We claim that the diagram
\[
\begin{tikzcd}[column sep=10ex]
\bW_R(R')|_Q \ar[r,"\bnorm_{R'}|_Q"] \ar[d,"\bW_Q(\Id\otimes 1_Q)"] & \bW_R(R)|_Q \ar[d,"\bW_Q(f)"] \\
\bW_Q(R'\otimes_R Q) \ar[r,"\bnorm_{R'\otimes_R Q}"] & \bW_Q(Q)
\end{tikzcd}
\]
commutes. For a ring $P\in \Rings_Q$ this diagram becomes
\[
\begin{tikzcd}[column sep=20ex]
R'\otimes_R P \ar[r,"\norm_{(R'\otimes_R P)/P}"] \ar[d,"(\Id\otimes 1_Q)\otimes \Id"] & P \ar[r,"1_R \otimes \Id"] \ar[d,equals] & R\otimes_R P \ar[d,"f\otimes \Id"] \\
(R'\otimes_R Q)\otimes Q P \ar[r,"\norm_{((R'\otimes_R Q)\otimes_Q P)/P}"] & P \ar[r,"1_Q\otimes \Id"] & Q\otimes_Q P.
\end{tikzcd}
\]
The right square clearly commutes. The leftmost arrow in the diagram is an isomorphism of $P$--modules and the determinant respects such isomorphisms. In particular, the determinant of left multiplication by $x\in R'\otimes_R P$ will be the same as the determinant of left multiplication by $(\Id\otimes 1_Q)(x) \in R'\otimes_R Q)\otimes_Q P$. Thus, the left square commutes and hence the original diagram commutes as claimed.

Therefore, we have a commutative diagram
\[
\begin{tikzcd}
\bW_R(R')|_Q \ar[r,"\bgamma_R^d|_Q"] \ar[d,"\bW_Q(\Id\otimes 1_Q)"] \ar[rr,bend left,"\bnorm_{R'}|_Q"] & \bW_R(\Gamma_R^d(R'))|_Q \ar[d,"\bW_Q(\varphi_Q^d\circ(\Id\otimes 1_Q))"] \ar[r,"\bW_R(\pi_{R'})|_Q"] & \bW_R(R)|_Q \ar[d,"\bW_Q(f)"] \\
\bW_Q(R'\otimes_R Q) \ar[r,"\gamma_Q^d"] \ar[rr,bend right,"\bnorm_{R'\otimes_R Q}"] & \bW_Q(\Gamma_Q^d(R'\otimes_R Q)) & \bW_Q(Q).
\end{tikzcd}
\]
The composition $\bW_Q(f) \circ \bW_R(\pi_{R'})|_Q \circ \bW_Q(\varphi_Q^d\circ(\Id\otimes 1_Q))^{-1}$ appears over $P \in \Rings_Q$ as the outer edges in the following commutative diagram.
\[
\begin{tikzcd}[column sep=15ex]
\Gamma_R^d(R')\otimes_R P \ar[r,"\pi_{R'}\otimes \Id"] & R\otimes_R P \ar[d,swap,"(\Id\otimes 1_Q)\otimes \Id"] \ar[dd,bend left=60,start anchor=east,end anchor=east,"f\otimes \Id"] \\
(\Gamma_R^d(R')\otimes_R Q)\otimes_Q P \ar[u,swap,"\mathrm{can}\otimes \Id"] \ar[r,"(\pi_{R'}\otimes\Id)\otimes \Id"] & (R\otimes_R Q)\otimes_Q P \ar[d,swap,"\mathrm{can}\otimes \Id"] \\
\Gamma_Q^d(R'\otimes_R Q)\otimes_Q P \ar[u,swap,"(\varphi_Q^d)^{-1}\otimes \Id"] & Q\otimes_Q P
\end{tikzcd}
\]
By instead using the bottom rectangle, we obtain the equality
\[
\bW_Q(f) \circ \bW_R(\pi_{R'})|_Q \circ \bW_Q(\varphi_Q^d\circ(\Id\otimes 1_Q))^{-1} = \bW_Q(\mathrm{can}\circ (\pi_{R'}\otimes\Id) \circ (\varphi_Q^d)^{-1}).
\]
Finally, the universal property which defines $\pi_{R'\otimes_R Q}$ then enforces that
\[
\pi_{R'\otimes_R Q} = \mathrm{can}\circ (\pi_{R'}\otimes\Id) \circ (\varphi_Q^d)^{-1}
\]
as desired. Thus the middle square commutes.

For the commutativity of the right square we have a similar argument. We consider the commutative diagram
\[
\begin{tikzcd}
\bW_Q(Q') \ar[r,"\bgamma_Q^d"] \ar[rr,bend left,"\bnorm_{Q'}"] & \bW_Q(\Gamma_Q^d(Q')) \ar[r,"\bW_Q(\pi_{Q'})"] & \bW_Q(Q) \ar[d,equals] \\
\bW_Q(R'\otimes_R Q) \ar[u,swap,"\bW_Q(\psi)"] \ar[r,"\bgamma_{R'\otimes_R Q}^d"] \ar[rr,bend right,"\bnorm_{R'\otimes_R Q}"] & \bW_Q(\Gamma_Q^d(R'\otimes_R Q)) \ar[u,swap,"\bW_Q(\Gamma_Q^d(\psi))"] & \bW_Q(Q)
\end{tikzcd}
\]
where it is clear the left square commutes and the outer square involving the norms commutes because the determinant is invariant under module isomorphisms as above. Therefore, since $\bW_Q(\pi_{Q'})\circ \bW_Q(\Gamma_Q^d(\psi)) = \bW_Q(\pi_{Q'}\circ \Gamma_Q^d(\psi))$, the universal property of $\pi_{R'\otimes_R Q}$ enforces that
\[
\pi_{R'\otimes_R Q} = \pi_{Q'}\circ \Gamma_Q^d(\psi)
\]
as claimed. This finishes the proof.
\end{proof}

\begin{lem}\label{lem_Ferrand_base_change_isos}
Consider a pushout diagram of rings
\[
D = \begin{tikzcd}
R' \ar[r,"f'"] & Q' \\
R \ar[u] \ar[r,"f"] & Q \ar[u]
\end{tikzcd}
\]
where the vertical arrows are finite locally free extensions of rank $d$.
\begin{enumerate}[label={\rm(\roman*)}]
\item \label{lem_Ferrand_base_change_isos_i} For an $R'$--module $M'$ there is a canonical isomorphism of $Q$--modules
\[
\begin{tikzcd}
\theta_{D,M'} \colon N_{R'/R}(M')\otimes_R Q \ar[r,"\sim"] \ar[d,equals] & N_{Q'/Q}(M'\otimes_{R'} Q') \ar[d,equals] \\[-3ex]
(\Gamma_R^d(M')\otimes_{\Gamma_R^d(R')} R)\otimes_R Q & \Gamma_Q^d(M'\otimes_{R'} Q')\otimes_{\Gamma_Q^d(Q')} Q \\[-4ex]
\big(\gamma^{a_1}(m'_1)\dots\gamma^{a_k}(m'_k)\otimes r \big)\otimes q \ar[r,mapsto] & \gamma^{a_1}(m'_1\otimes 1)\dots\gamma^{a_k}(m'_k\otimes 1)\otimes f(r)q
\end{tikzcd}
\]
for $a_i\in \NN$ with $\sum a_i = d$.
\item \label{lem_Ferrand_base_change_isos_ii} There is an isomorphism of functors
\[
\theta_D \colon N_{R'/R}(\und)\otimes_R Q \iso N_{Q'/Q}(\und \otimes_{R'} Q') 
\]
induced by the isomorphisms of \ref{lem_Ferrand_base_change_isos_i}.
\end{enumerate} 
\end{lem}
\begin{proof}
\noindent\ref{lem_Ferrand_base_change_isos_i}: Because $D$ is a pushout diagram, we have a unique isomorphism $\psi \colon R'\otimes_R Q \iso Q'$ as in \eqref{eq_pushout_psi}. This isomorphism also gives us an isomorphism
\[
\psi_{M'} \colon M'\otimes_R Q \xrightarrow{(\Id\otimes 1)\otimes \Id} M'\otimes_{R'} R'\otimes_R Q \xrightarrow{\Id\otimes \psi} M'\otimes_{R'} Q'
\]
which is $\psi$--equivariant. Now, we consider the following composition of isomorphisms
\begin{align*}
(\Gamma_R^d(M')\otimes_{\Gamma_R^d(R')} R)\otimes_R Q \xrightarrow{\mathrm{can}} &(\Gamma_R^d(M')\otimes_R Q)\otimes_{\Gamma_R^d(R')\otimes_R Q} (R\otimes_R Q) \\
\xrightarrow{\varphi_Q^d  \otimes \mathrm{can}} &\Gamma_Q^d(M'\otimes_R Q)\otimes_{\Gamma_Q^d(R'\otimes_R Q)} Q \\
\xrightarrow{\Gamma_Q^d(\psi_{M'})\otimes\Id} &\Gamma_Q^d(M'\otimes_{R'} Q')\otimes_{\Gamma_Q^d(Q')} Q
\end{align*}
where $\varphi_Q^d$ is the isomorphism of \eqref{eq_f_tilde} and the final two isomorphisms are well-defined due to the results of Lemma \ref{lem_pushout_Gamma_base_change}. Tracing an element through this composition yields
\begin{align*}
\big(\gamma^{a_1}(m'_1)\dots \gamma^{a_k}(m'_k)\otimes r\big)\otimes q \mapsto &(\gamma^{a_1}(m'_1)\dots \gamma^{a_k}(m'_k)\otimes 1_Q)\otimes (r\otimes q) \\
\mapsto &\gamma^{a_1}(m'_1\otimes 1_Q)\dots \gamma^{a_k}(m'_k\otimes 1_Q) \otimes f(r)q \\
\mapsto &\gamma^{a_1}(m'_1\otimes 1_{Q'})\dots \gamma^{a_k}(m'_k\otimes 1_{Q'}) \otimes f(r)q
\end{align*}
which is the claimed formula.

\noindent\ref{lem_Ferrand_base_change_isos_ii}: The fact that $\theta_D$ is a natural transformation follows from the functoriality of the $\Gamma$--algebras. In particular, for a morphism $\phi \colon M'_1 \to M'_2$ of $R'$--modules, tracing an element of $N_{R'/R}(M'_1)\otimes_R Q$ through the diagram
\[
\begin{tikzcd}[column sep=10ex]
N_{R'/R}(M'_2)\otimes_R Q \ar[r,"\theta_D(M'_2)"] & N_{Q'/Q}(M'_2\otimes_{R'}Q') \\
N_{R'/R}(M'_1)\otimes_R Q \ar[r,"\theta_D(M'_1)"] \ar[u,swap,"N_{R'/R}(\phi)\otimes\Id"] & N_{Q'/Q}(M'_1\otimes_{R'}Q') \ar[u,swap,"N_{Q'/Q}(\phi\otimes\Id)"]
\end{tikzcd}
\]
yields, using the formula of \ref{lem_Ferrand_base_change_isos_i} above,
\[\small
\begin{tikzcd}[column sep=2ex]
\gamma^{a_1}(\phi(m'_1))\dots\gamma^{a_k}(\phi(m'_k))\otimes r\otimes q \ar[r,mapsto] & \gamma^{a_1}(\phi(m'_1)\otimes 1)\dots\gamma^{a_k}(\phi(m'_k)\otimes 1)\otimes f(r)q \\
\gamma^{a_1}(m'_1)\dots\gamma^{a_k}(m'_k)\otimes r\otimes q \ar[r,mapsto] \ar[u,mapsto] & \gamma^{a_1}(m'_1\otimes 1)\dots\gamma^{a_k}(m'_k\otimes 1)\otimes f(r)q \ar[u,mapsto]
\end{tikzcd}
\]
and so we see the diagram commutes.
\end{proof}

\begin{remark}\label{rem_Ferrand_iso}
In the case when $Q'=R'\otimes_R Q$, Ferrand states that there is an isomorphism of functors $N_{R'/R}(\und)\otimes_R Q \cong N_{(R'\otimes_R Q)/Q}(\und\otimes_R Q)$ as his property (N2) in \cite[\S 1]{F}.  Composing this with the isomorphism $N_{(R'\otimes_R Q)/Q}(\rho)$, where $\rho \colon M'\otimes_R Q \iso M'\otimes_{R'} (R'\otimes_R Q)$ is the canonical isomorphism, yields the isomorphism of Lemma \ref{lem_Ferrand_base_change_isos}\ref{lem_Ferrand_base_change_isos_ii} above.
\end{remark}

\subsubsection{Universal Property}\label{sec_Ferrand_universal_property}
For an $R'$--module $M'$, Ferrand assembles the canonical functions $\nu_{M'} \colon M' \to N_{R'/R}(M')$ from Section \ref{sec_The_Construction} into a normic polynomial law as follows. For $Q\in \Rings_R$, we define the function
\[
\bnu_{M'}(Q) \colon M'\otimes_R Q \xrightarrow{\nu_{M'\otimes_R Q}} N_{(R'\otimes_R Q)/Q}(M'\otimes_R Q) \xrightarrow{\phi_Q} N_{R'/R}(M')\otimes_R Q
\]
where the final map $\phi_Q$ is the isomorphism $\big((\varphi_Q^d)\otimes \mathrm{can})\circ \mathrm{can}\big)^{-1}$, using the notation of Lemma \ref{lem_Ferrand_base_change_isos}. For a morphism $f\colon Q_1 \to Q_2$ in $\Rings_R$, we claim the diagram
\[
\begin{tikzcd}
M'\otimes_R Q_2 \ar[r,"\nu_{M'\otimes_R Q_2}"] & N_{(R'\otimes_R Q_2)/Q_2}(M'\otimes_R Q_2) \ar[r,"\phi_{Q_2}"] & N_{R'/R}(M')\otimes_R Q_2 \\
M'\otimes_R Q_1 \ar[r,"\nu_{M'\otimes_R Q_1}"] \ar[u,swap,"\Id\otimes f"] & N_{(R'\otimes_R Q_1)/Q_1}(M'\otimes_R Q_1) \ar[r,"\phi_{Q_1}"] & N_{R'/R}(M')\otimes_R Q_1 \ar[u,swap,"\Id\otimes f"]
\end{tikzcd}
\]
commutes. Indeed, tracing an element along the bottom and up we get
\begin{align*}
\sum_{i=1}^k m'_i\otimes q_i \mapsto &\gamma^d\big(\sum_{i=1}^k m'_i\otimes q_i\big)\otimes 1_{Q_1} \\
= &\sum_{(a_1,\ldots,a_k)\in \NN^k \atop a_1+\ldots +a_k = d} \gamma^{a_1}(m'_1\otimes q_1)\dots\gamma^{a_k}(m'_k\otimes q_k) \otimes 1_{Q_1} \\
= &\sum_{(a_1,\ldots,a_k)\in \NN^k \atop a_1+\ldots +a_k = d} q_1^{a_1}\dots q_k^{a_k}\gamma^{a_1}(m'_1\otimes 1_{Q_1})\dots\gamma^{a_k}(m'_k\otimes 1_{Q_1}) \otimes 1_{Q_1} \\
= &\sum_{(a_1,\ldots,a_k)\in \NN^k \atop a_1+\ldots +a_k = d} \gamma^{a_1}(m'_1\otimes 1_{Q_1})\dots\gamma^{a_k}(m'_k\otimes 1_{Q_1}) \otimes q_1^{a_1}\dots q_k^{a_k} \\
\mapsto &\left(\sum_{(a_1,\ldots,a_k)\in \NN^k \atop a_1+\ldots +a_k = d} \gamma^{a_1}(m'_1)\dots \gamma^{a_k}(m'_k) \otimes 1_R \right)\otimes q_1^{a_1}\dots q_k^{a_k} \\
\mapsto &\left(\sum_{(a_1,\ldots,a_k)\in \NN^k \atop a_1+\ldots +a_k = d} \gamma^{a_1}(m'_1)\dots \gamma^{a_k}(m'_k) \otimes 1_R \right)\otimes f(q_1^{a_1}\dots q_k^{a_k}).
\end{align*}
Instead, going up and then across the top begins by sending
\[
\sum_{i=1}^k m'_i\otimes q_i \mapsto \sum_{i=1}^k m'_i\otimes f(q_i),
\]
which then follows similar computations to arrive at
\[
\left(\sum_{(a_1,\ldots,a_k)\in \NN^k \atop a_1+\ldots +a_k = d} \gamma^{a_1}(m'_1)\dots \gamma^{a_k}(m'_k) \otimes 1_R \right)\otimes f(q_1)^{a_1}\dots f(q_k)^{a_k}.
\]
The two paths are therefore equal.

This shows that $\bnu_{M'} \colon \bW_R(M') \to \bW_R(N_{R'/R}(M'))$ is a well-defined natural transformation given over $Q \in \Rings_R$ by $\bnu_{M'}(Q)$, justifying our previous notation. It is clear that this is a normic polynomial law since each $\nu_{M'}$ is normic. Ferrand proves that this polynomial law has the following properties.
\begin{thm}[Ferrand]\label{Ferrand_property}
The functor $N_{R'/R} \colon \fMod_{R'} \to \fMod_R$ together with the normic polynomial laws $\bnu_{M'} \colon \bW_R(M') \to \bW_R(N_{R'/R}(M'))$ for every $R'$--module $M'$ have the following properties.
\begin{enumerate}[label={\rm (\roman*)}]
\item\label{Ferrand_property_i}  $N_{R'/R}(R')=R$ and $\bnu_{R'} = \bnorm$.
\item\label{Ferrand_property_iii} The pair $(N_{R'/R}(M'),\bnu_{M'})$ is universal among such pairs. If $(E,\bnu')$ is an $R$--module and normic polynomial law $\bnu' \colon \bW_R(M') \to \bW_R(E)$ pair, then there is a unique morphism of $R$--modules $\varphi \colon N_{R'/R}(M') \to E$ such that $\bnu' = \bW_R(\varphi) \circ \bnu_{M'}$.
\item\label{Ferrand_property_iv} The universal property of $\bnu_{M'}$ above induces the image of the norm functor on morphisms. If $\varphi \colon M'_1 \to M'_2$ is an $R'$--module morphism, then $\bnu_{M'_2}\circ \bW_R(\varphi)$ is a normic polynomial law and
\[
N_{R'/R}(\varphi)\colon N_{R'/R}(M'_1) \to N_{R'/R}(M'_2)
\]
is the unique map making the diagram below commute.
\[
\begin{tikzcd}
\bW_R(M'_1) \arrow{d}{\bW_R(\varphi)} \arrow{r}{\bnu_{M'_1}} & \bW_R\big(N_{R'/R}(M'_1)\big) \arrow{d}{\bW_R(N_{R'/R}(\varphi))} \\
\bW_R(M'_2) \arrow{r}{\bnu_{M'_2}} & \bW_R\big(N_{R'/R}(M'_2)\big)
\end{tikzcd}
\]
\end{enumerate}
\end{thm}

To give an example of the form these norms take, and for later reference, we quote another result of Ferrand.
\begin{prop}[{\cite[3.2.4]{F}}]\label{Ferrand_property_v}
Let $S_1, \ldots, S_m$ be finite projective $R$--algebras of ranks $d_1, \ldots, \allowbreak d_m$ respectively. Put $S=S_1 \times \cdots \times S_m$ and let $F$ be an $S$--module. Thus, $F=F_1 \times \cdots \times F_m$ for $S_i$--modules $F_i$, $i=1, \ldots , m$. Then, there exists an isomorphism
\[ \phi \co N_{S/R}(F_1 \times \cdots \times F_m) \simlgr 
      N_{S_1/R}(F_1) \ot_R \cdots \ot_R N_{S_m/R}(F_m)
\] 
of $R$--modules such that the normic polynomial $\bnu_F$ satisfies 
\[ 
   (\bW_R(\phi) \circ \bnu_F) (x_1, \ldots, x_m) = \bnu_{F_1}(x_1) \ot \cdots 
      \ot \bnu_{F_m}(x_m)
\]
for $x_i \in F_i \ot_R Q$, $Q\in \Rings_R$. \sm 

In particular, if $S = R\times \cdots \times R$ and $E_1, \ldots, E_m$ is a family of $R$--modules, then we have an isomorphism
\[ \phi \co N_{S/R}(E_1 \times \cdots \times E_m)
    \simlgr E_1 \ot_R \cdots \ot_R E_m\]
such that the normic polynomial law of $E_1 \times \cdots \times E_m$ is given by 
\[ (\bW_R(\phi) \circ \bnu_{E_1 \times \cdots \times E_m}) (y_1 , \ldots y_m) 
   = y_1 \ot \cdots \ot y_m\]
for $y_i \in E_i \ot_R Q$.
\end{prop}
The existence of the isomorphism $\phi$ in Proposition \ref{Ferrand_property_v} is proven in \cite[3.2.4]{F} and the formulas for the normic polynomials can be inferred from the proof of loc. cit.

The universal normic polynomial laws are compatible with the isomorphisms $\theta_D$ of Lemma \ref{lem_Ferrand_base_change_isos}\ref{lem_Ferrand_base_change_isos_ii} in the following way.
\begin{lem}\label{lem_Ferrand_natural_iso}
Consider a pushout diagram of rings
\[
D = \begin{tikzcd}
R' \ar[r,"f'"] & Q' \\
R \ar[u] \ar[r,"f"] & Q \ar[u]
\end{tikzcd}
\]
where the vertical arrows are finite locally free extensions of rank $d$. Consider the natural isomorphism $\theta_D$ of Lemma \ref{lem_Ferrand_base_change_isos}\ref{lem_Ferrand_base_change_isos_ii}. Let $\psi \colon R'\otimes_R Q \iso Q'$ be the unique isomorphism of \eqref{eq_pushout_psi} and $\psi_{M'} \colon M'\otimes_R Q \iso M'\otimes_{R'} Q'$ the associated isomorphism of $Q$--modules. For any $R'$--module $M'$, the universal normic polynomial laws are related via the diagram
\[
\begin{tikzcd}[column sep=10ex]
\bW_R(M')|_Q \ar[r,"\bnu_{M'}|_Q"] & \bW_R(N_{R'/R}(M'))|_Q \\
 & \bW_Q(N_{R'/R}(M')\otimes_R Q) \ar[d,"\bW_Q(\theta_D(M'))"] \ar[u,swap,"\mathrm{can}"] \\
\bW_Q(M'\otimes_{R'} Q') \ar[r,"\bnu_{(M'\otimes_{R'} Q')}"] \ar[uu,swap,"\psi'"] & \bW_Q(N_{Q'/Q}(M'\otimes_{R'} Q')).
\end{tikzcd}
\]
Here, the canonical isomorphism is given over $P \in \Rings_Q$ by the canonical map $(N_{R'/R}(M')\otimes_R Q)\otimes_Q P \iso N_{R'/R}(M')\otimes_R P$ and the isomorphism $\psi'$ is given by
\[
(M'\otimes_{R'} Q')\otimes_Q P \xrightarrow{\psi_{M'}^{-1}\otimes \Id} (M'\otimes_R Q)\otimes_Q P \xrightarrow{\mathrm{can}} M'\otimes_R P.
\]
\end{lem}
\begin{proof}
Working over a ring $P\in \Rings_Q$, the diagram becomes
\[
\begin{tikzcd}[column sep=10ex]
M'\otimes_R P \ar[r,"\bnu_{M'}(P)"] & N_{R'/R}(M')\otimes_R P \\
 & N_{R'/R}(M')\otimes_R Q \otimes_Q P \ar[d,"\theta_D(M')\otimes\Id"] \ar[u,swap,"\mathrm{can}"] \\
(M'\otimes_{R'} Q')\otimes_Q P \ar[r,"\bnu_{M'\otimes_{R'} Q'}(P)"] \ar[uu,"\psi'(P)"] & N_{Q'/Q}(M'\otimes_{R'}Q')\otimes_Q P.
\end{tikzcd}
\]
Starting in $M'\otimes_R P$ and tracing an element, we obtain
\[\small
\begin{tikzcd}[column sep=2ex]
\displaystyle\sum_{i=1}^k m'_i\otimes p_i \ar[r,mapsto] & \left(\displaystyle\sum_{(a_1,\ldots,a_k)\in \NN^k \atop a_1+\ldots +a_k = d} \gamma^{a_1}(m'_1)\dots \gamma^{a_k}(m'_k)\otimes 1_R \right)\otimes p_1^{a_1}\dots p_k^{a_k} \\
 & \left(\displaystyle\sum_{(a_1,\ldots,a_k)\in \NN^k \atop a_1+\ldots +a_k = d} \gamma^{a_1}(m'_1)\dots \gamma^{a_k}(m'_k) \otimes 1_R \right)\otimes 1_Q\otimes p_1^{a_1}\dots p_k^{a_k} \ar[u,mapsto] \ar[d,mapsto] \\
\displaystyle\sum_{i=1}^k (m'_i\otimes 1)\otimes p_i \ar[uu,mapsto] \ar[r,mapsto] &  \left(\displaystyle\sum_{(a_1,\ldots,a_k)\in \NN^k \atop a_1+\ldots +a_k = d} \gamma^{a_1}(m'_1\otimes 1)\dots \gamma^{a_k}(m'_k\otimes 1)\otimes 1_Q\right)\otimes p_1^{a_1}\dots p_k^{a_k}
\end{tikzcd}
\]
which verifies that the diagram commutes as claimed.
\end{proof}

\subsection{Rost's Norm Functor}\label{sec_Rosts_norm}
Here, we outline the construction of Rost's norm which was used in his preprint \cite{Rost} and a special case of which was addressed in \cite{Kr}. The reader will immediately notice similarities between Rost's and Ferrand's definitions, the most obvious being that where Ferrand uses the $\Gamma$--module $\Gamma_R^d(M')$, Rost instead uses the symmetric tensors $\TS_R^d(M')$. These similarities are what lead to the two norms being isomorphic in many usual cases. For example, it is show in \cite[2.2]{Lu} that $\Gamma_R^d(M') \cong \TS_R^d(M')$ if $M'$ is flat as an $R$--module or if $d! \in R^\times$. However, we provide an example at the end of this section which shows that the Rost norm does not respect arbitrary (i.e., non-flat) base change and thus is not always isomorphic to the Ferrand norm. This also implies that the Rost norm is not suitable for constructing a morphism of stacks using the construction of Appendix \ref{app_quasi_coh}, as we do with Ferrand's norm in Section \ref{sec_glue_Ferrand} below. In this section, for a finite locally free ring extension of constant rank $R \to R'$, we will use the notation $\RN_{R'/R}$ for Rost's norm and $\FN_{R'/R}$ for Ferrand's norm as in \eqref{eq_Ferrands_norm}.

To begin, we review some basic notions while setting notation. Let $R$ be our base ring and let $M$ be an $R$--module. We denote the tensor algebra by $\T_R(M)$ and the $d$--graded component by $\T_R^d(M) = \Span(\{m_1\otimes \ldots \otimes m_d \mid m_i \in M\})$. There is a natural action of the symmetric group $\SS_d$ on $\T_R^d(M)$ by permuting the tensor factors and the fixed points of this action are called the \emph{symmetric tensors}, denoted
\[
\TS_R^d(M) = (\T_R^d(M))^{\SS_d}.
\]
It is clear that both $\T_R^d \colon \fMod_R \to \fMod_R$ and $\TS_R^d \colon \fMod_R \to \fMod_R$ are functors by sending a morphism $\varphi \colon M \to N$ to the morphism which acts as
\[
m_1\otimes\ldots\otimes m_d \mapsto \varphi(m_1)\otimes\ldots\otimes \varphi(m_d).
\]
If $A$ is an $R$--algebra, then $\T_R^d(A)$ and $\TS_R^d(A)$ are naturally $R$--algebras as well by defining the multiplication tensor component-wise, i.e.,
\[
(a_1\otimes\ldots\otimes a_d)(b_1\otimes\ldots\otimes b_d)=(a_1b_1)\otimes\ldots\otimes(a_db_d)
\]
and extended linearly.
\begin{example}\label{bant-sym2}
Let $n=2$. For $a_1, a_2\in A$ put $a_1 \star a_2 = a_1\ot a_2 + a_2 \ot a_1$, the shuffle product of $a_1$ and $a_2$. We then have the following multiplication rules in $\TS^2_R(A)$ where $a, a_1, a_2, b, b_1, b_2\in A$:
\begin{align*}
  (a^{\ot 2})(b^{\ot 2}) & = (ab)^{\ot 2}, 
 \\  (a_1 \star a_2) b^{\ot 2} &= a_1 b \star a_2 b, 
 \\ (a_1 \star a_2) (b_1 \star b_2) &= (a_1b_1 \star a_2 b_2) + (a_1 b_2 \star a_2b_1).
\end{align*}
\end{example}

If in addition $M$ is an $A$--module, then $\T_R^d(M)$ is naturally a $\T_R^d(A)$--module in the same tensor component-wise way.

Both $\T_R^d(A)$ and $\T_R^d(M)$ have natural actions of $\SS_d$, which interact as follows. For $\sigma \in \SS_d$ we have
\[
\sigma\big((a_1\otimes\ldots\otimes a_d)\cdot (m_1\otimes\ldots\otimes m_d)\big) = \sigma(a_1\otimes\ldots\otimes a_d) \cdot \sigma(m_1\otimes\ldots\otimes m_d).
\]
Additionally, when we consider the symmetric tensors, $\TS_R^d(M)$ is naturally a $\TS_R^d(A)$--module. Of course, we may also simply consider $\T_R^d(M)$ as a $\TS_R^d(A)$--module.

For $S\in \Ralg$ and an $R$--module $M$ we have a canonical isomorphism
\begin{equation}\label{sybi-bc1} \begin{split}
 \sfT^n_S(M\ot_R S)\qquad &\simlgr \qquad \sfT^n_R(M) \ot_R S, \\
    (v_1\ot s_1)\ot_S \cdots \ot_S (v_n \ot s_n) &\mapsto
      \;  (v_1 \ot_R \cdots \ot_R v_n)\ot_R (s_1 \cdots s_n)
\end{split}\end{equation}
of $S$--modules. The group $\SS_n$ acts on $\sfT_S(M\ot_R S)$ naturally, and on $\sfT^n_R(M) \ot_R S$ by acting on the first factor. The isomorphism \eqref{sybi-bc1} is equivariant with respect to these two actions. Hence, by restriction we get an isomorphism of $R$--modules,
\begin{equation}
  \label{sybi-bc2} \TS_S^n(M\ot_R S) \simlgr \big( \sfT^n_R(M)\ot_R S \big){}^{\SS_n}.
\end{equation}
The reader should however not confuse $\big( \sfT^n_R(M)\ot_R S \big){}^{\SS_n}$ and the base change $\sfT^N_R(M)^{\SS_n} \ot_R S = \TS^n_R(M) \ot_R S$. We always have a canonical $S$--linear map
\begin{equation} \label{sybi-bc3}
 \frb_{S/R} \co \TS^n_R(M) \ot_R S \longto \TS^n_S(M\ot_R S),
\end{equation}
sending $v^{\ot n}\ot_R 1_S \in \TS^n_R(M) \ot S$ to $(v \ot 1_S)^{\ot n} \in \TS^n_S(M\ot_R S)$,
which is in general neither injective nor surjective, \cite[5.3, 5.5]{Lu}. Nevertheless, by \cite[5.2]{Lu}, the base change map $\frb_{S/R}$ is an isomorphism in the following cases:
\begin{enumerate}[label={\rm (\roman*)}]
  \item \label{sybi-bc-i} $n!$ is invertible in $R$,
  \item\label{sybi-bc-ii} $M$ is a flat $R$--module,
  \item \label{sybi-bc-iii} $S$ is a flat $R$--module.
\end{enumerate}
We will come back to this in Section \ref{bcrn}.

Now, let $P$ be a finite locally free, i.e., finitely generated projective, $R$--module. We consider the exterior algebra $\wedge_R^d(P)$, which is the quotient of $\T_R^d(P)$ by the submodule
\[
I_d = \Span(\{p_1\otimes\ldots\otimes p_d \mid p_i=p_j \text{ for some }i,j\}).
\]
We define $I_d^{(ij)} = \Span(\{p_1\otimes\ldots\otimes p_d \mid p_i=p_j\})$ for a particular pair $i,j$. Thus it is clear that $I_d$ is the non-direct sum of the various $I_d^{(ij)}$ with $1\leq i < j \leq d$. Then, we have that
\[
I_d^{(ij)} = (\T_R^d(P))^{(ij)}
\]
is the fixed point set of the transposition $(ij)\in \SS_d$. This is easily verified if $P$ is finite free by calculations with a basis, and therefore also holds for all finite locally free module by faithfully flat descent. We use this to argue the following.
\begin{lem}\label{lem_TSA_on_wedge}
Let $A$ be an $R$--algebra and let $P$ be an $A$--module which is finite locally free as an $R$--module. Then the $\TS_R^d(A)$--module action on $\T_R^d(P)$ stabilizes the submodule $I_d$ and thus induces a natural $\TS_R^d(A)$--module structure on $\wedge_R^d(P)$.
\end{lem}
\begin{proof}
We will show that the $\TS_R^d(A)$--module structure stabilizes each submodule $I_d^{(ij)}$, and thus also stabilizes $I_d$. So, consider an $a\in \TS_R^d(A)$ and $x\in I_d^{(ij)}=(\T_R^d(P))^{(ij)}$. Then, it is simple to see that $a\cdot x$ is fixed by $(ij)$ since
\[
(ij)(a \cdot x) = (ij)(a) \cdot (ij)(x) = a\cdot x
\]
because both $(ij)(a)=a$ and $(ij)(x) = x$ by assumption. Therefore $a\cdot x \in (\T_R^d(P))^{(ij)}=I_d^{(ij)}$ also and we are done.
\end{proof}

The Rost norm is then constructed from these objects in the following scenario. Let $R \to R'$ be a finite locally free ring extension of constant rank $d$. Viewing $R'$ simply as a finite locally free $R$--module, its $d$--fold exterior product $\wedge_R^d(R')$ is a line bundle since $R'$ is rank $d$. Thus, $\End_R(\wedge_R^d(R'))\cong R$. Now, because $R'$ is also an $R'$--module, Lemma \ref{lem_TSA_on_wedge} says that $\wedge_R^d(R')$ has an action of $\TS_R^d(R')$, which is equivalent to the data of a ring homomorphism
\begin{equation}\label{eq_TS_rho}
\rho \colon \TS_R^d(R') \to R.
\end{equation}
This is a homomorphism of commutative rings since $R'$ is a commutative $R$--algebra. Now, let $M'$ be an arbitrary $R'$--module, which we also consider as an $R$--module. Then $\TS_R^d(M')$ is naturally a $\TS_R(R')$--module, and so we may define the Rost norm by
\[
\RN_{R'/R}(M') = \TS_R^d(M')\otimes_{\TS_R^d(R')} R.
\]
The functoriality of $\TS_R^d$ is inherited by the Rost norm and so we get a functor
\[
\RN_{R'/R} \colon \fMod_{R'} \to \fMod_R.
\]

\subsection{Comparison of the Norms}
Let $d\in \NN$ and let $M$ be an $R$--module. We have the $R$--algebra $(\Ga_R(M), \times)$ as defined in Section \ref{sec_The_Construction}. The shuffle product $\star$ equips
\[ 
\TS_R(M) = \textstyle \bigoplus_{n\in \NN} \, \TS_R^n(M)
\]
with the structure of a commutative, $\NN$-graded $R$--algebra, denoted  $(\TS_R(M), \star)$, (\cite[IV; \S5.3]{B:A2}, \cite[III, \S5]{Roby}). By \cite[Prop.~III.1]{Roby}, there exists a unique $R$--algebra homomorphism
\begin{equation*}\label{chm0}
(\Ga_R(M), \times) \to (\TS_R(M), \star), \quad \ga^d(m) \mapsto m^{\ot d}.
\end{equation*}
It is $\NN$--graded. Therefore, by restriction, it gives rise to an $R$--linear map
\begin{equation}\label{chm1}
\sfc_M \co \Ga^d_R(M) \to \ST^d_R(M),  
\end{equation} 
referred to as the {\em canonical homomorphism}. By \cite[Prop 2.2]{Lu}, $\sfc_M$ is an isomorphism in the following cases:
\begin{enumerate}[label={\rm (\roman*)}]
  \item\label{chm-i}  $n!$ is invertible in $R$, or
  \item \label{chm-ii} $M$ is a flat $R$--module. 
\end{enumerate}
However, it is in general neither injective nor surjective, see \cite[3.1, 3.2, 4.4]{Lu}.  
Regarding \ref{chm-ii}, the special case of $M$ being free has been established in \cite[Prop.~IV.5, p.~272]{Roby}. Since both functors $\Ga^d_R$ and $\TS_R^d$ commute with filtered direct limits, one gets \ref{chm-ii} as an application of Lazard's Theorem, see \cite[5.5.2.5]{Deligne}.

\begin{lem}\label{rco}
Let $d\in \NN_+$, and let $A$ be an $R$--algebra.
\begin{enumerate}[label={\rm (\roman*)}]
\item\label{rco-a} Let $M$ be an $A$--module. Then the $\Ga^d_R(A)$--module $\Ga^d_R(M)$ and the $\TS^d_R(A)$--module $\TS^d_R(M)$ are related by the canonical homomorphisms $\sfc_{M}$ and $\sfc_{A}$ of \eqref{chm1}:
\begin{equation}\label{rco1}
  \sfc_{M} (\wdh a \cdot \wdh m) = \sfc_{A}(\wdh a) \cdot \sfc_{M}(\wdh m)
\end{equation}
where $\wdh a\in \Ga^d_R(A)$ and $\wdh m\in \Ga^d_R(M)$.
\item\label{rco-b} In particular, 
\begin{equation}  \label{rco11}
\sfc_{A} \co \Ga^d_R(A) \to \TS_R^d(A)
\end{equation}
is a homomorphism of $R$--algebras.
\end{enumerate}
\end{lem}

\begin{proof}
\noindent\ref{rco-a}: It is sufficient to establish the equation \eqref{rco1} after a faithfully flat extension $S\in \Ralg$, i.e., we claim
\begin{equation}\label{rco2} 
(\sfc_{M}\ot 1_S)(\wdh a\cdot \wdh m) = (\sfc_{A}\ot 1_S)(a) \, \cdot \, 
    (\sfc_{M}\ot 1_S)(\wdh m)
\end{equation} 
for $\wdh a\in \Ga^d_R(A)\ot_R S$ and $\wdh m\in \Ga^d_R(M)\ot_R S$. By \cite[2.3.1]{F}, we know that there is a finite free $R$--algebra $S$ such that
\[ 
   \Span_S\{ \ga^d(m)\ot 1_S \mid m\in M\} = \Ga^d_R(M)\ot_R S,   
\] 
and that the analogous equation holds for $\Ga^d_R(A)$. Since \eqref{rco2} is bilinear in $\wdh a$ and $\wdh m$, it is enough to verify it for $\wdh a$ and $\wdh m$ of the form $\wdh a=\ga^d(a)\ot 1_S$ and $\wdh m=\ga^d(m)\ot 1_S$ for $a\in A$ and $m\in M$. In this case \eqref{rco2} follows from
\begin{align*}
  &(\sfc_{M}\ot 1_S)\big( (\ga^d(a)\ot 1_S) \cdot (\ga^d(m)\ot 1_S)\big)
  \\=& (\sfc_{M}\ot 1_S)\big( (\ga^d(a)\cdot \ga^d(m))\ot 1_S\big)
  \\=&(\sfc_{M}\ot 1_S)(\ga^d(a\cdot m) \ot 1_S)
      = \sfc_{M}(\ga^d(a\cdot m)) \ot 1_S
  \\ =& (a\cdot m)^{\ot d} \ot 1_S = (a{}^{\,\ot d} \cdot m{}^{\,\ot d}) \ot 1_S
  \\ =& [\sfc_{A}(\ga^d(a))\cdot \sfc_{M}( \ga^d(m))] \ot 1_S
  \\ =& [ \sfc_{A}(\ga^d(a))\ot 1_S] \cdot
                  [\sfc_{M}(\ga^d(m))\ot 1_S]
  \\=& \big( (\sfc_{A}\ot 1_S)(\ga^d(a)\ot 1_S) \big) \cdot
           \big( (\sfc_{M}\ot 1_S)(\ga^d(m)\ot 1_S)\big).
\end{align*}
\noindent\ref{rco-b}: The map $\sfc_{A}$ is a homomorphism of $R$--algebras by \eqref{rco1} with $M=A$.
\end{proof}

\begin{example}\label{extwo}
It is easy to show that \eqref{rco11} for $d=2$ is a homomorphism of $R$--algebras.
Indeed, $\Ga^2_R(A)$ is spanned by $\ga^2(a)$, $a\in A$, and $\ga^1(a_1) \times \ga^1(a_2)$ for $a_1, a_2 \in A$. By definition, $\sfc_A\big(\ga^2(a)\big) = a^{\ot 2}$. The image of $\ga^1(a_1) \times \ga^1(a_2)$ can be calculated from this using the relations given in Section \ref{sec_The_Construction}: 
\begin{align*}
\sfc_A\big(\ga^1(a_1) \times \ga^1(a_2)\big) 
  & = \sfc_A\big(\ga^2(a_1+ a_2) -\ga^2(a_1) - \ga^2(a_2)\big) 
\\ & = (a_1 + a_2)^{\ot 2} - a_1^{\ot 2} - a_2^{\ot 2} \\
 &= a_1 \ot a_2 + a_2 \ot a_1 \\
 &= a_1 \star a_2. 
\end{align*}
We then have, using the relations of Example \ref{aldm-c2} for $M=A$ and Example \ref{bant-sym2}:
\begin{align*}
\sfc_A(\ga^2(a) \cdot \ga^2(b)\big) &= \sfc_A\big(\ga^2(ab)\big) = (ab)^{\ot 2} \\
&= (a^{\ot 2})(b^{\ot 2}) = \sfc_A\big(\ga^2(a)\big)\sfc_A\big(\ga^2(b)\big). \\
\sfc_A\big( (\ga^1(a_1) \times \ga^1(a_2))\times \ga^2(b)\big) &= \sfc_A\big( \ga^1(a_1 b) \times \ga^1(a_2b)\big) \\
& = (a_1b)  \star (a_2b) = (a_1 \star a_2)b^{\ot 2} \\
& = \sfc_A\big( \ga^1(a_1) \times \ga^1(a_2)\big) \sfc_A\big(\ga^2(b)\big).
\end{align*}
We leave the verification for the other products to the reader.
\end{example} 

\begin{lem}\label{rcoco} Let $R'\in \Ralg$ be finite locally free of constant rank $d\in \NN_+$. Then the canonical homomorphism 
\[ 
\sfc_{R'} \co \Ga^d_R(R') \to \TS_R^d(R')
\] is an isomorphism of $R$--algebras. Moreover, both triangles of the diagram below are commutative:
\begin{equation} \label{rco-b1} \vcenter{
\xymatrix@C=60pt{
   R' \ar[r]^{\ga^d} \ar[d]_{\psi_R} & \Ga_R^d(R') \ar[d]^\pi \ar[dl]_{\sfc_{R'}}^\cong \\ \TS_R^d(R') \ar[r]_\rho & R
}}\end{equation}
where $\Psi_R(r')=r'\otimes\ldots\otimes r'$ with $d$ factors, $\rho$ is the $R$--algebra homomorphism of \eqref{eq_TS_rho}, and $\pi$ is the $R$--algebra homomorphism \eqref{eq_Gamma_pi}.
\end{lem}

\begin{proof} The map $\sfc_{R'}$ is a homomorphism of $R$--algebras by Lemma \ref{rco}. It is then an isomorphism by condition \ref{chm-ii} below \eqref{chm1}.  That $\psi_R = \sfc_{R'} \circ \ga^d$ is clear from the definitions of $\psi_R$ and $\sfc_{R'}$. Regarding $\rho \circ \sfc_{R'} = \pi$, we have
\begin{align*}
  (\rho \circ \sfc_{R'})\, \big( \ga^d(r')\big) = \rho(r'{}^{\, \ot d}) = \norm_{R'/R}(r') = \pi\big(\ga^r(r')\big)
\end{align*}
for any $r'\in R'$. Hence $\pi = \rho \circ \sfc_{R'}$ as soon as $\{\ga^d(r'): r'\in R'\}$ spans $\Ga^d_R(R')$. As in the proof of Lemma \ref{rco}\ref{rco-a}, this can be achieved by an application of \cite[2.3.1]{F}.
\end{proof}

The proof of the following lemma is straightforward and left to the reader.
\begin{lem}\label{tpl} Let $A$ and $A'$ be $R$--algebras, let $M$ and $N$ be be a right respectively left $A$--module, and let $M'$ and $N'$ be a right respectively left $A'$--module. Assume that
$\be_A \co A \to A'$ is a homomorphism of $R$--algebras, and that 
$\be_M \co M \to M'$ and $\be_N \co N \to N'$ are $R$--linear maps which are equivariant with respect to $\be_A$, i.e., 
\begin{equation}\label{tpl1}
\be_M(a \cdot m) = \be_A(a) \cdot \be_M(m) \quad \text{and}\quad
\be_N(a \cdot n) = \be_A(a) \cdot \be_N(n)
\end{equation}
holds for $a\in A$, $m\in M$ and $n\in N$. Then
    \[ \be \co M \ot_A N \simlgr M'\ot_{A'} N', \quad
         m \ot_A n \mapsto \be_M(m) \ot_{A'} \be_N(n)\]
    is a well-defined $R$--module homomorphism. It is an isomorphism if $\be_A$, $\be_M$ and $\be_N$ are isomorphisms of $R$--modules.
\end{lem}

We now show that the canonical morphism $\sfc_{M'}$ induces a canonical homomorphism from the Ferand norm to the Rost norm.
\begin{prop}\label{noho} Let $R'\in \Ralg$ be finite locally free of constant rank $d\in \NN_+$, and let $M'$ be an $R'$--module. Then the canonical map $\sfc_{M'}$ of \eqref{chm1} together with $\Id_R$ give rise to an $R$--linear map
\begin{align}
\sfn_{M'} \co \FFN_{R'/R}(M') &\longto \RN_{R'/R}(M') \label{noho1} \\
\wdh m \ot_{\Ga^d_R(R')} r &\; \mapsto \; \sfc_{M'}(\wdh m) \ot_{\TS^d_R(R')} r \nonumber
\end{align} 
 between the Ferrand and Rost norm modules of $M'$, which is an isomorphism if $\sfc_{M'}$ is so, e.g., if $M'$ is a flat $R$--module.
\end{prop}
\begin{proof} Recall 
 \begin{align*}
   \FFN_{R'/R}(M') &= \Ga^d_R(M') \ot_{\Ga^d_R(R')} R \quad \text{and}
   \\ \RN_{R'/R}(M') &= \TS^d_R(M') \ot_{\TS^d_R(R')} R. 
 \end{align*}
By \eqref{rco1}, the map $\sfc_{M'}$ satisfies the first condition in \eqref{tpl1} with respect to the $R$--algebra homomorphism $\sfc_{R'}$. The actions of $\Ga^d_R(R')$ and of $\TS^d_R(R')$ on $R$ are given by the homomorphisms $\pi$ and $\rho$ respectively, which satisfy $\pi = \rho \circ \sfc_{R'}$. This says that $\be_R = \Id_R$ satisfies $\be_R(\wdh a\cdot r) = \sfc_{R'}(\wdh a) \cdot \be_R(r)$ for $\wdh a \in \Ga^d_R(R')$. Thus also the second condition in \eqref{tpl1} is satisfied. The existence of the $R$--linear map $\sfn_{M'}$ then follows from Lemma~\ref{tpl}. Moreover, since both $\sfc_{R'}$ and $\be_R$ are isomorphisms, $\sfn_{M'}$ is an isomorphism as soon as $\sfc_{M'}$ is an isomorphism. By condition \ref{chm-ii} below \eqref{chm1}, this is the case if $M'$ is a flat $R$--module.   
\end{proof}
 
\begin{remarks}
\begin{enumerate}[label={\rm (\roman*)}]
\item[]
\item If $R'/R$ is finite \'etale of constant rank and $M'$ is a finite locally free $R$--module, the isomorphism $\FFN_{R'/R}(M') \cong \RN_{R'/R}(M')$ is proven in \cite[Prop.~5]{Bachmann} with a different approach.
\item Proposition~\ref{noho} allows the possibility that $\sfn_{M'}$ is an isomorphism, even if $M'$ is not a flat $R$-module, or that $\FFN_{R'/R}(M')$ and $\RN_{R'/R}(M')$ are isomorphic using a different map. We will address these questions in Corollary \ref{gleico} and Example \ref{noeq}.
\end{enumerate}
\end{remarks}

\begin{lem}\label{glei} Let $R'=R[\veps]$, $\veps^2 = 0$, let $M$ be an $R$--module, and let $M'$ be the $R'$--module obtained from $M$ by putting $\veps \cdot M' = 0$. Then there exists a commutative diagram of $R$--modules,
\begin{equation}\label{glei0}\vcenter{ \xymatrix@C=50pt{ 
 \FFN_{R'/R}(M') \ar[r]^\cong \ar[d]_{\sfn_{M'}} & \Ga^2_R(M) \ar[d]^{\sfc_{M'}} \\
  \RN_{R'/R}(M') \ar[r]^\cong & \TS^2_R(M) 
}}\end{equation}
in which the horizontal maps are isomorphisms, given by 
\begin{equation}\label{glei00}
 x \ot_{\Ga^2_R(R')} 1_R \mapsto x\quad \text{and}\quad 
 y \ot_{\TS^2_R(R')} 1_R \mapsto y 
\end{equation}
for $x\in \Ga^2_R(M')$ and $y\in \TS^2_R(M')$. \end{lem}

\begin{proof} It is immediate that 
\begin{equation} \label{glei1}
\TS^2_R(R') = R (1_R^{\ot 2}) \oplus R(1_R \star \veps) \oplus R \veps^{\ot 2},
\end{equation}
where $1_R \star \veps = 1_R \ot \veps + \veps \ot 1_R$. Let $\sfc_{R'}\co \Ga^2_R(R') \to \TS_R^2(R')$ be the canonical map which we know to be an isomorphism of $R$--algebras by Lemma \ref{rcoco}. We have seen in the Example~\ref{extwo} that $\sfc_{R'}\big( \ga^1(r'_1) \times \ga^1(r'_2) \big) = r'_1 \star r'_2$ for $r'_1, r'_2 \in R'$. Since $\sfc_{R'}$ is an isomorphism, this implies
\begin{equation} \label{glei2}
\Ga^2_R(R') = \Span_R \{ \ga^2(1_R), \ga^2(\veps), \ga^1(1_R) \times \ga^1(\veps)\}
\end{equation}
with $\sfc_{R'}$ acting as 
\begin{align*}
  \ga^2(1_R)= 1_{\Ga^2_R(R')} \quad &\mapsto\quad 1_R^{\ot 2} = 1_{\ST^2_R(R')} \\
  \ga^1(1_R) \times \ga^1(\veps) \quad &\mapsto \quad 1_R \star \veps, \\
  \ga^2(\veps) \quad &\mapsto \quad \veps^{\ot 2}. 
  \end{align*}
Next we claim that 
\begin{equation}\label{glei3} \begin{split}
  I_{\TS} &= R ( 1_R \star \veps) \oplus R \veps^{\ot 2}, \quad \text{and}\\ 
  I_{\Ga} &= R\big( \ga^1(1_R) \times \ga^1(\veps)\big) \oplus R \ga^2(\veps)
\end{split}\end{equation}
are nilpotent ideals of the $R$--algebras $\TS^2(R')$ and $\Ga^2_R(R')$ respectively. Indeed, for $I_{\TS}$ this follows from the multiplication rules, cf.~Example \ref{bant-sym2},
\begin{align*}
  &1_R^{\ot 2}  = 1_{\TS^2_R(R')} \, , \\
  & (1_R \star \veps)(1\star \veps) = 1_R \star \veps^2 + \veps \star \veps = 2 \veps^{\ot 2}\, , \\
  & (1_R \star \veps)\veps^{\ot 2} = \veps \star \veps^2 = 0\, , \\
  & (\veps^{\ot 2})(\veps^{\ot 2}) = (\veps^2)^{\ot 2} = 0\,. 
  \end{align*}  
Since $\sfc_{R'}(I_\Ga) = I_{TS}$, we also obtain that $I_{\Ga}$ is a nilpotent ideal of $\Ga^2_R(R')$. 

Our next claims are that 
\begin{equation}\label{glei4} \begin{split}
&\text{\em $I_\Ga$ annihilates the $\Ga^2_R(R')$--modules $\Ga^2_R(M')$ and $R$, and} \\
&\text{\em $I_{TS}$ annihilates the $\TS^2_R(R')$--modules $\TS^2_R(M')$ and $R$.}
\end{split}\end{equation}
Since $\Ga^2_R(M') = \Span_R \{\ga^2(m), \ga^1(m_1) \times \ga^1(m_2): m, m_1, m_2\in M\}$, the claim \eqref{glei4} for $\Ga^2_R(M')$ follows from the formulas below, cf.~Example \ref{aldm-c2}:  
\begin{align*}
  &\ga^2(\veps) \cdot \ga^2(m) = \ga^2(\veps m) = \ga^2(0) = 0, \\
& \ga^2(\veps) \cdot \big( \ga^1(m_1) \times \ga^1(m_2)\big) 
    = \ga^1(\veps m_1) \times \ga^1(\veps m_2) = 0 , \\
& \big( \ga^1(1_R) \times \ga^1(\veps)\big) \cdot \ga^2(m) 
  = \ga^1(m) \times \ga^1(\veps m) =  0, \\
& \big( \ga^1(1_R) \times \ga^1(\veps)\big) \cdot \big( \ga^1(m_1) \times \ga^1(m_2)\big)\\ &\quad =
  \ga^1(m_1) \star \ga^1(\veps m_2) + \ga^1(m_2 \star \ga^1(\veps m_1) = 0  .
  \end{align*}
We also have $I_{\TS} \cdot \TS^2_R(M')= 0$ because $I_{TS}$ even annihilates $M'\ot M'$. Thus, it remains to consider $R$. The $\TS^2_R(R')$--module $R$ is given by the algebra homomorphism $\rho$ of \eqref{eq_TS_rho}. Untangling the definition, it suffices to show that 
\[ \TS^2_R(R') \to \End_R\big(\we^2_R(R')\big), \quad \we^2_R(R') = R(1 \we \veps),\]
is zero on $I_{\TS}$, which follows from 
\begin{align*}
  &(1_R \star \veps) \cdot (1_R \we \veps)  = 1_R \we \veps^2 + \veps \we \veps = 0  , \\
 & \veps^{\ot 2} \cdot (1 \we \veps) = \veps \we \veps^2 = 0. 
\end{align*}
Finally, $\pi (I_{\Ga}) = 0$ because $\pi (I_\Ga) = (\rho \circ \sfc_{R'})(I_\Ga) = \rho(I_{\TS}) = 0$. This finishes the proof of \eqref{glei4}. 

Because $\Ga^2_R(R')/ I_\Ga \cong R$ and $\TS^2_R(R')/I_{\TS} \cong R$, we get 
\begin{align*}
  \FFN_{R'/R}(M') & = \Ga^2_R(M') \ot_{\Ga^2_R(R')} R = 
     \Ga^2_R(M)\ot_{\Ga^2_R(R')/I_\Ga} R \\ & \quad \cong \Ga^2_R(M) \ot_R R \cong \Ga^2_R(M)\text{, and} \\
 \RN_{R'/R}(M') &= \TS^2_R(M') \ot_{\TS^2_R(R')} R = \cdots = \TS^2_R(M). 
  \end{align*} 
We thus get the diagram \eqref{glei0} with horizontal isomorphisms. It is easily seen to be commutative. \end{proof} 

\begin{cor}  \label{gleico} The canonical map $\sfn_{M'}$ between the Ferrand norm and the Rost norm of an $R'$--module $M'$ is in general neither injective nor surjective.
\end{cor}

\begin{proof} By Lemma~\ref{glei}, it is enough to prove this for $\sfc_{M'}$. This has been done in \cite[Ex.~3.2 and Ex.~4.6]{Lu}. \end{proof}

\begin{example}\label{noeq}
We provide an example where $\FFN_{R'/R}(M') \not \cong \RN_{R'/R}(M')$. We specialize Lemma~\ref{glei} by taking $R=\ZZ$, $R'=\ZZ[\veps]$, and $M=\ZZ/2\ZZ = \{ \bar 0, \bar 1\}$. It then suffices to show that $\Ga^2_\ZZ(\ZZ/2\ZZ) \not \cong \TS^2_\ZZ(\ZZ/2\ZZ)$, which follows from 
\begin{equation}  \label{noeq1} \Ga^2_\ZZ(\ZZ/2\ZZ) \cong \ZZ/4\ZZ, \quad
\TS^2_\ZZ(\ZZ/2\ZZ) \cong \ZZ/2\ZZ.
\end{equation}
To prove \eqref{noeq1}, recall that the $\ZZ$--module $\Ga^2_\ZZ(\ZZ/2\ZZ)$ is spanned by $\ga^2(\ZZ/2\ZZ)$ and $\ga^1(\ZZ/2\ZZ) \times \ga^1(\ZZ/2\ZZ)$. Since $\ga^1(\bar 0) = 0 = \ga^2(\bar 0)$ and
\[ \ga^1(\bar 1) \times \ga^1(\bar 1) = 2 \ga^2(\bar 1)\] 
by the relations of Section \ref{sec_The_Construction}, it follows that $\Ga^2_\ZZ(\ZZ/2\ZZ) = \Span_\ZZ\{\ga^2(\bar 1)\}$. One easily verifies that the first two and the fourth relation of Section \ref{sec_The_Construction} are satisfied, and that from the third relation only
\[ 0 = \ga^1(\bar 1) \times \ga^1(\bar 1) = \ga^2(\bar 1) + 
    \ga^1(\bar 1) \times \ga^1(\bar 1) + \ga^2(\bar 1) = 4 \ga^2(\bar 1)
\]
yields some new information. Thus, $\Ga^2_\ZZ(\ZZ/2\ZZ) \cong \ZZ/4\ZZ$ holds. Finally, since $\sfT^2_\ZZ(\ZZ/2\ZZ) = \{ \bar 0 \ot \bar 0, \bar 1 \ot \bar 1\}=\TS^2_\ZZ(\ZZ/2\ZZ)$, we get $\TS^2_\ZZ(\ZZ/2\ZZ) \cong \ZZ/2\ZZ$.
\end{example}

\subsection{Base change for the Rost norm.} \label{bcrn} 
To construct a base change map for the Rost norm, we will use the following folklore lemma.
\begin{lem}\label{tpl-a} 
Let $A$ be an $R$--algebra, let $M$ be a right $A$--module, let $N$ be a left $A$--module and let $S\in \Ralg$. Then, with respect to the canonical actions of the $S$--algebra $A\ot_R S$ on $M\ot_R S$ and $N\ot_R S$ the map
\begin{equation} \label{tpla1}\begin{split}
  (M\ot_A N) \ot_R S &\simlgr (M\ot_R S) \ot_{A\ot_R S} (N\ot_R S)
\\  m\ot_A n \ot_R 1_S &\; \mapsto \; (m\ot_R 1_S) \ot_{A\ot_R S} (n \ot_R 1_S)
\end{split}\end{equation} 
is an isomorphism of $S$--modules with inverse
\[  (m\ot_R s_1) \ot_{A\ot S} (n \ot_R s_2) \mapsto m\ot_A n \ot_R (s_1s_2).
\]
\end{lem}

Now, let $R'\in \Ralg$ be finite locally free of constant rank $d$, let $M'$ be an $R'$--module, let $S\in \Ralg$, and put $S'=R'\ot_RS$. The {\em base change map for the Rost norm\/} is the $S$--linear map
\begin{equation} \label{bcrn1}\begin{split}
\frb_{M'}^{\RN} \co \RN_{R'/R}(M') \ot S &\longto \RN_{S'/S}(M'\ot_R S) \\
(x \ot_{\TS_R^d(R')} 1_R)\ot s  &\mapsto  \frb_{M'}(x\ot s)\ot_{\TS_S^d(S')} \ot 1_S,
\end{split}\end{equation}
obtained as the composition of the two $S$--linear maps 
\[\xymatrix{ 
  \RN_{R'/R}(M') \ot_R S =  ( \TS^d_R(M') \ot_{\TS^d_R(R')} R ) \ot_R S 
      \ar[d]^\cong_{\eqref{tpl-a}}
           \\  
   (\TS^d_R(M')\ot_R S) \ot_{\TS^d_R(R') \ot_R S} (R\ot_R S) 
        \ar[d]_{\eqref{tpl}} 
      \\
     \TS^d_S(M'\ot_R S)\ot_{\TS^d_S(S')}S  
}\]
where the second map uses the base change maps  
\begin{align*}
  \frb_{M'} \co \TS^d_R(M') \ot_R S &\longto \TS^d_S(M'\ot_R S) \quad \text{and} \\
  \frb_{R'} \co \TS^d_R(R') \ot_R S &\longto \TS^d_S(S')
\end{align*}
of \eqref{sybi-bc3} and the canonical isomorphism $R\ot_R S$, which satisfy the conditions \eqref{tpl1}. We know that $\frb_{R'}$ is an isomorphism since $R'$ is a flat $R$--module. Hence,  by \ref{sybi-bc-i} and \ref{sybi-bc-ii} after \eqref{sybi-bc3}, the base change map for the Rost norm is an isomorphism if
\begin{enumerate}[label={\rm (\roman*)}]
  \item \label{bcrn-i} $d!$ is invertible in $R$, or
  \item\label{bcrn-ii} $M$ is a flat $R$--module. 
\end{enumerate}
Of course, one suspects at this point that $\frb_{M'}^{\RN}$ is in general not an isomorphism. This is indeed the case, as we will show now.

\begin{lem}
Let $R'=R[\veps]$, $\veps^2 = 0$, let $S\in \Ralg$, and let $M$ be an $R$--module which we view as $R[\veps]$--module $M'$ by putting $\veps \cdot M = 0$. Then there exists a commutative diagram of $S$--linear maps, 
\[ \xymatrix@C=50pt{
   \RN_{R'/R}(M') \ot_R S \ar[r]^\cong \ar[d]_{\frb^{\RN}_{M'}} 
   & \TS^2_R(M) \ot_R S \ar[d]^{\frb_{M}}
 \\  
   \RN_{S'/S}(M'\ot_R S) \ar[r]^\cong & \TS^2_S(M\ot_R S)
}\]
in which the horizontal maps are isomorphisms.  
\end{lem}

\begin{proof} The top horizontal map is the $S$--extension of the isomorphism in the diagram \eqref{glei0}; the bottom horizontal map is the same isomorphism, but for the $S$--module $M\ot_R S$ and the extension $R'\ot_R S = S[\veps]$ of $S$. By \eqref{glei00} and \eqref{bcrn1}, an element $(x\ot_{\TS^2_R(R')} 1_R)\ot_R s \in \RN_{R'/R}(M') \ot_R S$ is mapped to $\frb_{M'} \ot s \in \TS^2_S(M\ot_R S)$ in both ways of the diagram. 
\end{proof}

\begin{example} By \cite[Ex.~5.3 and Ex.~5.5]{Lu}, there exist $R$, $S$ and $M$ such that $\frb_{M}$ is not injective and not surjective. Hence the same holds for $\frb_{M'}^{\RN}$.
\end{example}

\section{Globalizing Ferrand's Norm Functor}\label{sec_glue_Ferrand}
In this section we globalize Ferrand's norm over rings to our setting over a scheme by applying the constructions of Appendix \ref{app_quasi_coh} to produce a norm morphism of stacks. We first define the norm of quasi-coherent modules and then in the subsequent subsection we show that this norm also sends quasi-coherent algebras to quasi-coherent algebras in a natural way.

\subsection{The Norm of Quasi-coherent Modules}\label{sec_norm_of_modules} 
\subsubsection{The Construction}\label{the_construction}
We will apply Proposition \ref{prop_stack_morphism_assembled} in the following context. First, as in Appendix \ref{app_stack_morphism}, define the stack of quasi-coherent modules over $S$, denoted $p \colon \QCoh \to \Sch_S$, to have the following.
\begin{enumerate}[label={\rm(\roman*)}]
\item Its objects are pairs $(X,\cF)$ with $X\in \Sch_S$ and $\cF$ a quasi-coherent $\cO|_X$--module on $\Sch_X$.
\item Its morphisms are pairs $(g,\varphi)\colon (X',\cF') \to (X,\cF)$ where $g\colon X' \to X$ is a morphism of $S$--schemes and $\varphi \colon \cF' \to g^*(\cF)$ is a morphism of $\cO|_{X'}$--modules. Composition is given by $(g,\varphi)\circ (h,\psi) = (g\circ h,h^*(\varphi)\circ \psi)$.
\item Its structure functor is given by $(X,\cF)\mapsto X$ and $(g,\varphi)\mapsto g$.
\end{enumerate}
For a scheme $X \in \Sch_S$, the fiber $\QCoh(X)$ over $X$ is the category of quasi-coherent $\cO|_X$--modules. We point out that this is essentially the same definition as in \cite[Tag 03YL]{Stacks} but with some opposite conventions. For example, their fiber over $X$ is the opposite category of quasi-coherent $\cO|_X$--modules. Furthermore, the set theoretic issues mentioned in \cite[Tag 03YL]{Stacks} will not play a role in this paper. A definition of $\QCoh$ analogous to ours is used in \cite[4.3.11]{Olsson}, denoted $\mathrm{QCOH}$, but using the Zariski topology.

Next, define the stack of quasi-coherent modules over a finite locally free extension of rank $d$, denoted $p \colon \QCoh\flf^d \to \Sch_S$, to have
\begin{enumerate}[label={\rm(\roman*)}]
\item objects which are pairs $(h\colon T' \to T, \cM)$ where $h\colon T' \to T$ is a finite locally free morphism of constant rank $d$ in $\Sch_S$ and $\cM$ is a quasi-coherent $\cO|_{T'}$--module,
\item morphisms which are triples
\[
(f,g,\varphi)\colon (j\colon X' \to X,\cN) \to (h\colon T' \to T, \cM)
\]
where $f$ and $g$ are morphisms in $\Sch_S$ such that
\[
\begin{tikzcd}
X' \ar[r,"g"] \ar[d,"j"] & T' \ar[d,"h"] \\
X \ar[r,"f"] & T
\end{tikzcd}
\]
is a fiber product diagram and $\varphi \colon \cN \iso g^*(\cM)$ is an isomorphism of $\cO|_{X'}$--modules, and
\item structure functor given by $(h\colon T' \to T,\cM) \mapsto T$ and $(f,g,\varphi)\mapsto f$.
\end{enumerate}
In the notation of Appendix \ref{app_stack_morphism}, we write $\QCoh\flf^d = \QCoh_{\fI}$ where $\fI$ is the stack of finite locally free morphisms of constant rank $d$ in $\Sch_S$ (viewed as a full substack of the stack of affine morphisms $\fAff$, also defined in \ref{app_stack_morphism}). For each object $h\colon U' \to U$ in $\fI$, i.e., each finite locally free morphism of constant rank $d$, where $U$ and $U'$ are affine schemes, we set
\[
\cF_h = N_{\cO(U')/\cO(U)} \colon \fMod_{\cO(U')} \to \fMod_{\cO(U)}
\]
to be Ferrand's norm functor. For each fiber product diagram in $\Sch_S$ of the form on the left below with associated pushout diagram of rings on the right below
\[
D = \begin{tikzcd}
V' \ar[r,"f'"] \ar[d,"h'"] & U' \ar[d,"h"] \\
V \ar[r,"f"] & U,
\end{tikzcd}
\hspace{10ex}
\begin{tikzcd}
\cO(U') \ar[r] & \cO(V') \\
\cO(U) \ar[u] \ar[r] & \cO(V). \ar[u]
\end{tikzcd}
\]
where $h,h' \in \fI$ and where $U,U',V,V'$ are all affine schemes, we use the isomorphism of functors
\[
\theta_D \colon \cF_h(\und)\otimes_{\cO(U)}\cO(V) \iso \cF_{h'}(\und \otimes_{\cO(U')}\cO(V')).
\]
from Lemma \ref{lem_Ferrand_base_change_isos}\ref{lem_Ferrand_base_change_isos_ii} corresponding to the pushout diagram of rings. We now verify that these isomorphisms satisfy the required assumptions.
\begin{lem}\label{lem_verify_assumptions}
For the stack $\fI$ of finite locally free morphisms of rank $d$ in $\Sch_S$, the functors $\cF_h$ and natural isomorphisms $\theta_D$ chosen above satisfy assumptions \ref{assumption_1} and \ref{assumption_2}.
\end{lem}
\begin{proof}
To verify assumption \ref{assumption_1} holds, consider a fiber product diagram of schemes and associated pushout diagram of rings of the following form.
\[
D=
\begin{tikzcd}
U' \arrow{d}{h} \arrow[equals]{r} & U' \arrow{d}{h} \\
U \arrow[equals]{r} & U,
\end{tikzcd}
\hspace{10ex}
\begin{tikzcd}
\cO(U') \ar[r,equals] & \cO(U') \\
\cO(U) \ar[u] \ar[r,equals] & \cO(U). \ar[u]
\end{tikzcd}
\]
For any $\cO(U')$--module $M'$, we trace an element through the diagram
\[
\begin{tikzcd}[column sep=1.2in]
N_{\cO(U')/\cO(U)}(M')\otimes_{\cO(U)}\cO(U) \arrow{d}{\theta_D(M')} \arrow{dr}{\mathrm{can}} & \\
N_{\cO(U')/\cO(U)}\big(M'\otimes_{\cO(U')}\cO(U')\big) \arrow[swap]{r}{N_{\cO(U')/\cO(U)}(\mathrm{can})} & N_{\cO(U')/\cO(U)}(M')
\end{tikzcd}
\]
to obtain
\[
\begin{tikzcd}
(\gamma^{a_1}(m'_1)\dots\gamma^{a_k}(m'_k)\otimes u_1)\otimes u_2 \ar[dr,mapsto] \ar[d,mapsto] & \\
\gamma^{a_1}(m'_1\otimes 1)\dots\gamma^{a_k}(m'_k\otimes 1)\otimes u_1u_2 \ar[r,mapsto] & \gamma^{a_1}(m'_1)\dots\gamma^{a_k}(m'_k)\otimes u_1u_2
\end{tikzcd}
\]
where $\sum a_i = d$ and $m'_i\in M'$. This shows that the diagram
\[
\begin{tikzcd}[column sep=0.8in]
\cF_h(\und)\otimes_{\cO(U)}\cO(U) \arrow{d}{\theta_D} \arrow{dr}{\mathrm{can}} & \\
\cF_h\big(\und\otimes_{\cO(U')}\cO(U')\big) \arrow[near start]{r}{\cF_h(\mathrm{can})} & \cF_h(\und)
\end{tikzcd}
\]
commutes as required.

To verify assumption \ref{assumption_2} holds, consider fiber product diagrams of affine schemes
\[
D_f = \begin{tikzcd}
V' \arrow{d}{h'} \arrow{r}{f'} & U' \arrow{d}{h} \\
V \arrow{r}{f} & U
\end{tikzcd},\;
D_g = \begin{tikzcd}
W' \arrow{d}{h''} \arrow{r}{g'} & V' \arrow{d}{h'} \\
W \arrow{r}{g} & V
\end{tikzcd},
\text{ and }
D_{f\circ g} = \begin{tikzcd}
W' \arrow{d}{h''} \arrow{r}{f'\circ g'} & U' \arrow{d}{h} \\
W \arrow{r}{f \circ g} & U
\end{tikzcd}
\]
with $h,h',h'' \in \fI$ and the associated pushout diagrams of rings
\[
\begin{tikzcd}
\cO(U') \ar[r,"f'^*"] & \cO(V') \\
\cO(U) \ar[u,swap,"h^*"] \ar[r,"f^*"] & \cO(V), \ar[u,"h'^*"]
\end{tikzcd}\;
\begin{tikzcd}
\cO(V') \ar[r,"g'^*"] & \cO(W') \\
\cO(V) \ar[u,"h'^*"] \ar[r,"g^*"] & \cO(W), \ar[u,"h''^*"]
\end{tikzcd}
\text{ and }
\begin{tikzcd}
\cO(U') \ar[r,"(f'\circ g')^*"] & \cO(W') \\
\cO(U) \ar[u,"h^*"] \ar[r,"(f\circ g)^*"] & \cO(W). \ar[u,"h''^*"]
\end{tikzcd}
\]
For an $\cO(U')$--module $M'$, we claim that the diagram
\[
\begin{tikzcd}[column sep=-10ex]
 & N_{\cO(U')/\cO(U)}(M')\otimes_{\cO(U)}\cO(W) \arrow{dddd}{\theta_{D_{f\circ g}}(M')}\\
\big(N_{\cO(U')/\cO(U)}(M')\otimes_{\cO(U)}\cO(V)\big)\otimes_{\cO(V)}\cO(W) \arrow{ur}{\mathrm{can}} \arrow{d}{\theta_{D_f}(M')\otimes\Id} &  \\
N_{\cO(V')/\cO(V)}\big(M'\otimes_{\cO(U')}\cO(V')\big)\otimes_{\cO(V)}\cO(W) \arrow{d}{\theta_{D_g}(M'\otimes_{\cO(U')}\cO(V'))} & \\
N_{\cO(W')/\cO(W)}\big((M'\otimes_{\cO(U')}\cO(V'))\otimes_{\cO(V')}\cO(W')\big) \arrow[swap]{dr}{N_{\cO(W')/\cO(W)}(\mathrm{can})} & \\
& N_{\cO(W')/\cO(W)}(M' \otimes_{\cO(U')}\cO(W'))
\end{tikzcd}
\]
commutes. Indeed, this can be seen by tracing an element through the diagram. For $\overline{a}=(a_1,\ldots,a_k)\in \NN^k$ with $\sum a_i = d$ and $\overline{m'} = (m'_1,\ldots,m'_k)\in (M')^k$, we use the notation $\gamma^{\overline{a}}(\overline{m'}) = \gamma^{a_1}(m'_1)\dots \gamma^{a_k}(m'_k)$. We have
\[
\begin{tikzcd}[column sep=-5ex]
(\gamma^{\overline{a}}(\overline{m'})\otimes u)\otimes v \otimes w \ar[r,mapsto] \ar[d,mapsto] & (\gamma^{\overline{a}}(\overline{m'})\otimes u)\otimes g^*(v)w \ar[ddd,mapsto] \\
(\gamma^{\overline{a}}(\overline{m'\otimes 1_{\cO(V')}})\otimes f^*(u)v) \otimes w \ar[d,mapsto] & \\
\gamma^{\overline{a}}(\overline{m'\otimes 1_{\cO(V')}\otimes 1_{\cO(W')}})\otimes g^*(f^*(u)v)w \ar[dr,mapsto] & \\
 & \gamma^{\overline{a}}(\overline{m'\otimes 1_{\cO(W')}})\otimes (f\circ g)^*(u)g^*(v)w\\[-5ex]
 & =\gamma^{\overline{a}}(\overline{m'\otimes 1_{\cO(W')}})\otimes g^*(f^*(u)v)w
\end{tikzcd}
\]
for all $u\in \cO(U)$, $v\in \cO(V)$, and $w\in \cO(W)$, which shows that the diagram
\[\begin{tikzcd}[ampersand replacement=\&]
\big(\cF_h(\und)\otimes_{\cO(U)}\cO(V)\big)\otimes_{\cO(V)}\cO(W) \arrow{r}{\mathrm{can}} \arrow{d}{\theta_{D_f}\otimes\Id} \& \cF_h(\und)\otimes_{\cO(U)}\cO(W) \arrow{dd}{\theta_{D_{f\circ g}}} \\
\cF_{h'}\big(\und\otimes_{\cO(U')}\cO(V')\big)\otimes_{\cO(V)}\cO(W) \arrow{d}{\theta_{D_g}} \& \\
\cF_{h''}\big((\und\otimes_{\cO(U')}\cO(V'))\otimes_{\cO(V')}\cO(W')\big) \arrow{r}{\cF_{h''}(\mathrm{can})} \& \cF_{h''}(\und \otimes_{\cO(U')}\cO(W'))
\end{tikzcd}
\]
commutes as required. This finishes the proof.
\end{proof}

Lemma \ref{lem_verify_assumptions} allows us to apply Proposition \ref{prop_stack_morphism_assembled} to our choice of functors and natural isomorphisms, yielding a stack morphism
\begin{equation}\label{eq_norm_stack_morphism}
N \colon \QCoh\flf^d \to \QCoh
\end{equation}
which we call the \emph{norm morphism}. Since Proposition \ref{prop_stack_morphism_assembled} uses the construction of Lemma \ref{lem_quasi_coh_functor_construction}, for each $h\colon T' \to T$ in $\fI$ we have a norm functor $N_{T'/T}\colon \QCoh(T') \to \QCoh(T)$ between categories of quasi-coherent modules. In particular, for a finite locally free cover $T\to S$ of degree $d$ of our fixed base scheme $S$, we have a norm functor
\begin{equation}\label{eq_norm_functor}
N_{T/S} \colon \QCoh(T) \to \QCoh(S).
\end{equation}
This generalizes the norm appearing in \cite[Ch.2, 6.5.5]{EGA}, where they define the norm only for line bundles. By construction, $N_{T/S}$ has the property that for $\cM\in \QCoh(T)$ and $U\in \Aff_S$, we have
\[
N_{T/S}(\cM)(U) = N_{\cO(T\times_S U)/\cO(U)}(\cM(T\times_S U))
\]
where the right hand side is Ferrand's norm functor. Thus, $N_{T/S}$ is the unique extension of Ferrand's norm to the stack $\QCoh(T)$ using his compatibility isomorphisms as in \ref{rem_Ferrand_iso}. Furthermore, the restriction along a morphism $V \to U$ in $\Aff_S$ has the following nice expression,
\begin{align}
N_{\cO(T\times_S U)/\cO(U)}(\cM(T\times_S U)) &\to N_{\cO(T\times_S V)/\cO(V)}(\cM(T\times_S V)) \label{eq_norm_restriction} \\
\gamma^{\overline{a}}(\overline{m})\otimes u &\mapsto \gamma^{\overline{a}}(\overline{m|_{T\times_S V}})\otimes u|_V \nonumber
\end{align}
for $\overline{m}=(m_1,\ldots,m_k)\in \cM(T\times_S U)^k$ and $u \in \cO(U)$.

\subsubsection{The Universal Normic Polynomial Law}
Let $T\to S$ be a finite locally free cover of degree $d$ and consider the norm functor $N_{T/S}\colon \QCoh(T) \to \QCoh(S)$ of \eqref{eq_norm_functor}. Theorem \ref{Ferrand_property}\ref{Ferrand_property_iii} is preserved in a sense for this globalized norm functor. Since $f\colon T\to S$ is finite locally free, there is a canonical norm
\[
\bnorm \colon f_*(\cO|_T) \to \cO
\]
defined in \cite[Tag 0BD2]{Stacks} or \cite[II 6.5.1]{EGA} for sheaves on schemes, but which generalizes immediately to sheaves on $\Sch_S$. It is the globalized version of the norm of a finite locally free ring extension as in \eqref{eq_ring_norm}. For a quasi-coherent $\cO|_T$--module $\cM$, we define a \emph{normic polynomial law} to be a natural transformation $\bnu \colon f_*(\cM) \to \cN$, where $\cN$ is a quasi-coherent $\cO$--module, such that
\[
\bnu(tm) = \bnorm(t)\bnu(m)
\]
for all appropriate sections $t\in f_*(\cO|_T)$ and $m\in f_*(\cM)$. For a fixed $\cO|_T$--module $\cM$, we can form the category of normic polynomial laws, denoted $\NPL_{T/S}(\cM)$, whose
\begin{enumerate}[label={\rm(\roman*)}]
\item objects are pairs $(\cN,\bnu)$ where $\cN$ is a quasi-coherent $\cO$--module and $\bnu\colon f_*(\cM) \to \cN$ is a normic polynomial law, and whose
\item morphisms $(\cN,\bnu) \to (\cN',\bnu')$ are $\cO$--module maps $\varphi \colon \cN \to \cN'$ such that $\bnu' = \varphi \circ \bnu$.
\end{enumerate}
  
\begin{prop}\label{prop_universal_property}
Let $f\colon T\to S$ be a finite locally free morphism and let $\cM$ be a quasi-coherent $\cO|_T$--module. Then, there is a normic polynomial law $\bnu_{\cM} \colon f_*(\cM) \to N_{T/S}(\cM)$, given over $U \in \Aff_S$ by the function
\[
\nu_{\cM(T\times_S U)}\colon \cM(T\times_S U) \to N_{\cO(T\times_S U)/\cO(U)}(\cM(T\times_S U))
\]
of Section \ref{sec_Ferrand_universal_property}, such that $(N_{T/S}(\cM),\bnu_{\cM})$ is an initial object in the category $\NPL_{T/S}(\cM)$.
\end{prop}
\begin{proof}
We first check that the proposed functions over each $U\in \Aff_S$ assemble into a natural transformation of sheaves on $\Aff_S$. Let $g \colon V \to U$ be a morphism in $\Aff_S$. We then have a fiber product diagram
\[
D = \begin{tikzcd}
T\times_S V \ar[r,"g'"] \ar[d] & T\times_S U \ar[d] \\
V \ar[r,"g"] & U
\end{tikzcd}
\]
in $\Sch_S$. To shorten notation, we set
\begin{align*}
R &= \cO(U) & R' &=\cO(T\times_S U) \\
Q &=\cO(V) & Q' &=\cO(T\times_S V) \\
M'_U &= \cM(T\times_S U) & M'_V &= \cM(T\times_S V).
\end{align*}
Because this is a fiber product diagram, there is a canonical isomorphism $\varphi' \colon M'_U \otimes_{R'} Q' \iso M'_U\otimes_R Q$ and since $\cM$ is quasi-coherent, there is a canonical isomorphism $\rho\colon M'_U \otimes_{R'} Q' \iso M'_V$. We have a diagram
\[
\begin{tikzcd}[column sep=15ex]
M'_V \ar[r,"\nu_{M'_V}"] & N_{Q'/Q}(M'_V) \\
M'_U\otimes_{R'}Q' \ar[u,swap,"\rho"] \ar[r,"\nu_{M'_U\otimes_R' Q'}"] \ar[dd,"\varphi'"] & N_{Q'/Q}(M'_U\otimes_{R'} Q') \ar[u,swap,"N_{Q'/Q}(\rho)"] \\
 & N_{R'/R}(M'_U)\otimes_R Q \ar[u,swap,"\theta_D(M'_U)"] \\
M'_U\otimes_R Q \ar[r,"\nu_{M'_U}"] & N_{R'/R}(M'_U)\otimes_R Q \ar[u,equals] \\
M'_U \ar[u,swap,"\Id\otimes 1"] \ar[r,"\nu_{M'_U}"] \ar[uuuu,bend left=50,start anchor=west,end anchor=west,"f_*(\cM)(g)"] & N_{R'/R}(M'_U) \ar[u,swap,"\Id\otimes 1"] \ar[uuuu,swap,bend right=50,start anchor=east,end anchor=east,"N_{T/S}(\cM)(g)"]
\end{tikzcd}
\]
where the top square commutes by Theorem \ref{Ferrand_property}\ref{Ferrand_property_iv}, the middle square commutes by Lemma \ref{lem_Ferrand_natural_iso}, and the bottom square commutes because Ferrand's universal normic polynomial law $\bnu_{M'_U}\colon \bW_R(M'_U) \to \bW_R\big(N_{R'/R}(M'_U)\big)$ is a natural transformation. The two faces involving curved arrows commute by definition. Hence, we have a polynomial law
\[
\bnu_{\cM} \colon f_*(\cM) \to N_{T/S}(\cM)
\]
which is clearly normic since each $\nu_{M'_U}$ is so.

To justify that this polynomial law has the claimed universal property, let $\bnu \colon f_*(\cM) \to \cN$ be another normic polynomial law into a quasi-coherent $\cO$--module. Fix an affine scheme $U \in \Aff_S$ and a morphism $g \colon W \to V$ in $\Aff_U$. Using the same notation as above, as well as $\cO(W)=P$, $\cO(T\times_S W)=P'$, and $\cM(T\times_S W)=M'_W$, we have a diagram
\[
\begin{tikzcd}
M'_U\otimes_R P & M'_U\otimes_{R'} P' \ar[l,swap,"\sim"] \ar[r,"\sim"] & M'_W \ar[r,"\bnu(W)"] & \cN(W) & \cN(U)\otimes_R P \ar[l,swap,"\sim"] \\[2ex]
M'_U\otimes_R Q \ar[u,swap,"\Id\otimes g^*"] & M'_U\otimes_{R'} Q' \ar[l,swap,"\sim"] \ar[r,"\sim"] \ar[u,swap,"\Id\otimes g'^*"] & M'_V \ar[r,"\bnu(V)"] \ar[u,swap,"f_*(\cM)(g)"] & \cN(V) \ar[u,swap,"\cN(g)"] & \cN(U)\otimes_R Q \ar[l,swap,"\sim"] \ar[u,swap,"\Id\otimes g"]
\end{tikzcd}
\]
and hence we may take the long horizontal compositions as $\bnu_U(V)$ and $\bnu_U(W)$ respectively to define a normic polynomial law $\bnu_U \colon \bW_R(M'_U) \to \bW_R(\cN(U))$. Therefore, by Theorem \ref{Ferrand_property}\ref{Ferrand_property_iii}, there exists a unique $R$--linear map $\phi_U \colon N_{T/S}(\cM)(U)= N_{R'/R}(M'_U) \to \cN(U)$ such that $\bW_R(\phi_U) \circ \bnu_{M'_U} = \bnu_U$. For the affine scheme $V$, we similarly obtain a $Q$--linear homomorphism $\phi_V \colon N_{T/S}(\cM)(V) \to \cN(V)$. Due to the isomorphisms $M'_V \cong M'_U\otimes_R Q$ and $\cN(V)\cong \cN(U)\otimes_R Q$ as well as the uniqueness of $\phi_U$, we will have a diagram
\[
\begin{tikzcd}
N_{Q'/Q}(M'_V) \ar[r,"\phi_V"] & \cN(V) \\
N_{R'/R}(M'_U)\otimes_R Q \ar[u,swap,"\rotatebox{90}{$\sim$}"] \ar[r,"\phi_U\otimes 1"] & \cN(U)\otimes_R Q \ar[u,swap,"\rotatebox{90}{$\sim$}"] \\
N_{R'/R}(M'_U) \ar[u,swap,"\Id\otimes 1"] \ar[r,"\phi_U"] & \cN(U) \ar[u,swap,"\Id\otimes 1"]
\end{tikzcd}
\]
which shows that the various $\phi_U$ assemble into a map $\phi\colon N_{T/S}(\cM) \to \cN$ which is $\cO$--linear and satisfies $\phi \circ \bnu_{\cM} = \bnu$. If $\phi_2 \colon N_{T/S}(\cM) \to \cN$ is any other such map, then in a similar manner to above we can extract a natural transformation
\[
\phi_{2,U} \colon \bW_R(M'_U) \to \bW_R(N_{R'/R}(M'_U))
\]
satisfying $\bW_R(\phi_{2,U}) \circ \bnu_{M'_U} = \bnu_U$. Then, uniqueness of $\phi_U$ requires that $\phi_{2,U} = \phi_U$ and since this holds for all $U\in \Aff_S$, we must have $\phi_2 = \phi$ globally. Thus, $\phi$ is the unique such morphism as desired.
\end{proof}
For an $\cO|_T$--module $\cM$, we will refer to $\bnu_{\cM} \colon f_*(\cM) \to N_{T/S}(\cM)$ as the universal polynomial law associated to $\cM$. The analogue of Theorem \ref{Ferrand_property}\ref{Ferrand_property_iv} also holds.

\begin{cor}\label{cor_morphism_from_universal}
The universal property of Proposition \ref{prop_universal_property} above induces the image of the norm functor on morphisms. If $\varphi \colon \cM_1 \to \cM_2$ is a morphism of quasi-coherent $\cO|_T$--modules, then $\bnu_{\cM_2}\circ f_*(\varphi)$ is a normic polynomial law and $N_{T/S}(\varphi)$ is the unique $\cO$--module map making the diagram
\[
\begin{tikzcd}
f_*(\cM_1) \ar[r,"\bnu_{\cM_1}"] \ar[d,"f_*(\varphi)"] & N_{T/S}(\cM_1) \ar[d,"N_{T/S}(\varphi)"] \\
f_*(\cM_2) \ar[r,"\bnu_{\cM_2}"] & N_{T/S}(\cM_2)
\end{tikzcd}
\]
commute.
\end{cor}
\begin{proof}
It is clear from the explicit definitions of $\bnu_{\cM_i}$ as well as $N_{T/S}(\varphi)$ that the above diagram commutes. Therefore, the uniqueness claim follows from Proposition \ref{prop_universal_property}.
\end{proof}

The next corollary shows that the universal normic polynomial law is stable under base change.
\begin{cor}\label{cor_unviersal_base_change}
Consider a fiber product diagram in $\Sch_S$
\[
D = \begin{tikzcd}
T' \ar[r] \ar[d,"f'"] & T \ar[d,"f"] \\
S' \ar[r] & S
\end{tikzcd}
\]
where $f\colon T\to S$ is finite locally free of degree $d$ and hence so is $f'\colon T' \to S'$. Then, for any $\cM\in \QCoh(T)$, the diagram
\[
\begin{tikzcd}
f'_*(\cM|_{T'}) \ar[r,"\bnu_{(\cM|_{T'})}"] \ar[d,"\psi"] & N_{T'/S'}(\cM|_{T'}) \ar[d,"\phi_D"] \\
f_*(\cM)|_{S'} \ar[r,"(\bnu_{\cM})|_{S'}"]  & N_{T/S}(\cM)|_{S'}
\end{tikzcd}
\]
commutes. Here, $\psi$ is the canonical isomorphism coming from the isomorphisms $T\times_S X \iso T'\times_{S'} X$ for any $X\in \Sch_{S'}$ and $\phi_D$ is the isomorphism of Lemma \ref{lem_phi_isomorphism}.
\end{cor}
\begin{proof}
We may check over affine schemes $U\in \Aff_{S'}$, where the diagram becomes
\[
\begin{tikzcd}[column sep =10ex]
\cM(T'\times_{S'} U) \ar[r,"\nu_{\cM(T'\times_{S'} U)}"] \ar[d,"\psi(U)"] & N_{\cO(T'\times_{S'} U)/\cO(U)}(\cM(T'\times_{S'} U)) \ar[d,"\phi_D"] \\
\cM(T\times_S U) \ar[r,"\nu_{\cM(T\times_S U)}"] & N_{\cO(T\times_S U)/\cO(U)}(\cM(T\times_S U)).
\end{tikzcd}
\]
Tracing an element through yields
\[
\begin{tikzcd}
m' \ar[r,mapsto] \ar[d,mapsto] & \gamma^d(m')\otimes 1 \ar[d,mapsto] \\
m'|_{T\times_S U} \ar[r,mapsto] & \gamma^d(m'|_{T\times_S U})\otimes 1
\end{tikzcd}
\]
where $\phi_D$ is the composition
\[
\gamma^d(m')\otimes 1 \mapsto (\gamma^d(m')\otimes 1)\otimes 1 \mapsto \gamma^d(m'\otimes 1)\otimes 1 \mapsto \gamma^d(m'|_{T\times_S U})\otimes 1.
\]
This verifies that the diagram commutes.
\end{proof}

We will need the following result for the next section.
\begin{lem}\label{lem_norm_modules_etale} Assume that $T \to S$ is finite \'etale of constant degree $d$. If $\cM$ is a locally free $\cO|_T$--module of constant rank $r$, then $N_{T/S}(\cM)$ is a locally free $\cO$--module of constant rank $r^d$.
\end{lem}
\begin{proof}
By Lemma \ref{lem_equiv_affine_sheaves} we may examine the restriction of $N_{T/S}(\cM)$ to $\Aff_S$. There, for each $U \in \Aff_S$ the $\cO(T\times_S U)$--module $\cM(T\times_S U)$ is projective of rank $r$. Therefore, by \cite[4.1.3]{F}, $N_{T/S}(\cM)(U)$ is a projective $\cO(U)$--module of rank $r^d$, which implies the stated claim globally.
\end{proof}

\begin{example}\label{ex_split_norm}
For $i=1, \ldots, m$, let $f_i \co S_i \to S$ be finite locally free morphisms of constant degree $d_1, \ldots, d_m$ respectively. Hence the canonical map $f \co S'= S_1 \sqcup \cdots \sqcup S_m \to S$ is finite locally free of constant degree $d = d_1 + \cdots + d_m$. Any $S'$--scheme $T$ is of the form $T_1 \sqcup \cdots \sqcup T_m$ for $T_i \in \Sch_{S_i}$, so that any quasi-coherent $\cO|_{S'}$--module $\cE$ is given by the formula
\[
\cE(T_1\sqcup \ldots \sqcup T_m) = \cE_1(T_1)\times\ldots\times \cE_m(T_m)
\]
where $\cE_i$, $i=1, \ldots, m$,  are quasi-coherent $\cO|_{S_i}$--modules.
Then the following holds. 
\begin{enumerate}[label={\rm(\roman*)}]
\item \label{ex_split_norm_i} The norm of $\cE$ is 
\[ 
     N_{S'/S}(\cE) = N_{S_1/S}(\cE_1) \otimes_{\cO} \cdots \otimes_{\cO} N_{S_m/S}(\cE_m). 
\] 
In particular, for the split \'etale cover $f \colon \Sd \to S$ the norm of any $\cO|_{\Sd}$--module $\cE$ is 
  \[ N_{\Sd/S}(\cE) = \cE_1\otimes_{\cO} \ldots \otimes_{\cO} \cE_d. \]

\item \label{ex_split_norm_ii} We have $f_*(\cE)=f_{1 *}(\cE_1) \times \cdots \times f_{m*}(\cE_m)$, and the universal normic polynomial law is
\begin{align*}
\bnu \colon f_{1 *}(\cE_1) \times \cdots \times f_{m*}(\cE_m) &\mapsto N_{S_1/S}(\cE_1)\otimes_{\cO} \ldots \otimes_{\cO} N_{S_m/S}(\cE_m) \\
(e_1,\ldots,e_m) &\mapsto \bnu_1(e_1)\otimes \ldots \otimes \bnu_m(e_m)
\end{align*}
where $\bnu_i\colon f_{i*}(\cE_i) \to N_{S_i/S}(\cE_i)$ is the universal polynomial law for $N_{S_i/S}(\cE_i)$. In particular, for the split \'etale cover $\Sd \to S$ we get
\begin{align*}
\bnu \colon \cE_1\times\ldots\times \cE_d &\to \cE_1\otimes_\cO \ldots\otimes_\cO \cE_d \\
(e_1,\ldots,e_d)&\mapsto e_1\otimes\ldots \otimes e_d.
\end{align*}
\end{enumerate}
\end{example}
\begin{proof}
Both claims follow from Proposition \ref{Ferrand_property_v} after localizing with respect to an affine cover.
\end{proof}

\subsection{The Norm of Quasi-coherent Algebras}
The functor $N_{T/S}$ of \eqref{eq_norm_functor} restricts to the category of quasi-coherent $\cO|_T$--algebras. This is shown in Lemma \ref{lem_norm_az}, which is a globalization of parts of \cite[3.2.5]{F}. Alternatively, when $T\to S$ is \'etale, it is a globalization of \cite[4.5]{KO75}. In turn, this means that the morphism of stacks $N$ of \eqref{eq_norm_stack_morphism} restricts to the stack of quasi-coherent algebras. However, we begin with two technical lemmas about norms of modules which will be needed.
\begin{lem}\label{lem_nu_generates}
Let $f\colon T\to S$ be a finite locally free morphism and let $\cM$ be a quasi-coherent $\cO|_T$--module. Consider its universal normic polynomial law $(N_{T/S}(\cM),\bnu_{\cM})$ from Proposition \ref{prop_universal_property}. Then, $N_{T/S}(\cM)$ is generated as an $\cO$--module by the image of $\bnu_{\cM}$. Precisely, we mean that $N_{T/S}(\cM)$ is the sheaf associated to the presheaf
\[
U \mapsto \Span_{\cO(U)}(\{\bnu_{\cM}(m) \mid m\in f_*(\cM)(U)\})
\]
for $U\in \Sch_S$.
\end{lem}
\begin{proof}
This result follows from the analogous statement over rings that is used throughout \cite{F}, which in turn follows from \cite[2.3.1]{F}. In particular, for any affine scheme $U\in \Aff_S$ over which $T\times_S U \to U$ is of constant degree $d$, there exists a cover (depending on $d$) $\{V\to U\}$ such that $V$ is an affine scheme, $\cO(U) \to \cO(V)$ is a finite free ring extension, and $N_{T/S}(\cM)(V)$ is generated as an $\cO(V)$--module by the image under $\bnu_{\cM}$ of $f_*(\cM)(V)$. Therefore, any scheme in $\Sch_S$ has an affine cover on which the presheaf above and $N_{T/S}(\cM)$ agree, so the statement follows.
\end{proof}

Now we can argue that the norm functor preserves quasi-coherent algebras, thus globalizing \cite[3.2.5(a)]{F}.
\begin{lem} \label{lem_norm_az} Let $T\to S$ be a finite locally free morphism of schemes of degree $d$. Let $\cB$ be a quasi-coherent $\cO|_T$--algebra and let $\bnu_{\cB} \colon f_*(\cB) \to N_{T/S}(\cB)$ be the universal normic polynomial law of Proposition \ref{prop_universal_property}.
\begin{enumerate}[label={\rm (\roman*)}]
\item \label{lem_norm_az3} There is a unique morphism of $\cO$--modules
\[
\Phi \colon N_{T/S}(\cB)\otimes_{\cO} N_{T/S}(\cB) \to N_{T/S}(\cB\otimes_{\cO|_T}\cB)
\]
such that $\Phi(\bnu_{\cB}(b_1)\otimes \bnu_{\cB}(b_2)) = \bnu'(b_1\otimes b_2)$ where $\bnu'$ is the universal normic polynomial associated to $\cB\otimes_{\cO|_T}\cB$. In particular, since $\cB$ has an algebra structure morphism $\mu \colon \cB\otimes_{\cO|_T}\cB \to \cB$, the composition
\[
\begin{tikzcd}
N_{T/S}(\cB)\otimes_{\cO} N_{T/S}(\cB) \arrow{r}{\Phi} & N_{T/S}(\cB\otimes_{\cO|_T}\cB) \arrow{r}{N_{T/S}(\mu)} & N_{T/S}(\cB)
\end{tikzcd}
\]
gives $N_{T/S}(\cB)$ a natural algebra structure. It is associative or unital or commutative if $\cB$ is so. Further, the universal normic polynomial law $\bnu_{\cB}$ is multiplicative with respect to this natural structure. If $\cB$ is unital, then $\nu_\cB$ preserves the unit as well.

\item \label{lem_norm_az5} The norm preserves algebra homomorphisms. If $\varphi \colon \cB_1 \to \cB_2$ is an $\cO|_T$--algebra homomorphism between quasi-coherent $\cO|_T$--algebras, then $N_{T/S}(\varphi) \colon N_{T/S}(\cB_1) \to N_{T/S}(\cB_2)$ is an $\cO$--algebra homomorphism with respect to the natural algebra structures from \ref{lem_norm_az3}.

\item \label{lem_norm_az4} If $T\to S$ is finite \'etale, then $\Phi$ is an isomorphism.
\end{enumerate}
\end{lem}
\begin{proof}
\noindent\ref{lem_norm_az3}: The property $\Phi(\bnu_{\cB}(b_1)\otimes \bnu_{\cB}(b_2)) = \bnu'(b_1\otimes b_2)$ is sufficient to define a unique $\cO$--module morphism by Lemma \ref{lem_nu_generates}. We leave the verification that the resulting algebra structure preserves being associative or commutative to the reader. The multiplicativity of $\bnu_{\cB}$ follows from the calculation
\begin{align*}
\bnu_{\cB}(b_1b_2) &= \bnu_{\cB} \circ f_*(\mu)(b_1\otimes b_2) = N_{T/S}(\mu)\circ \bnu'(b_1\otimes b_2)\\
&= N_{T/S}(\mu)\circ \Phi(\bnu_{\cB}(b_1)\otimes\bnu_{\cB}(b_2)) = \bnu_{\cB}(b_1)\bnu_{\cB}(b_2)
\end{align*}
for $b_1,b_2$ sections in $f_*(\cB)$. 

Now, assume $\cB$ is unital. To see that $N_{T/S}(\cB)$ is also unital and that $\nu_\cB$ preserves this unit, we argue that $\nu_\cB(1_\cB)$ is the identity in $N_{T/S}(\cB)$. Let $x\in N_{T/S}(\cB)(U)$ be a section over some $U \in \Sch_S$. By Lemma \ref{lem_nu_generates}, there is a cover $\{U_i \to U\}_{i\in I}$ over which $x|_{U_i} = \sum a_j \nu_{\cB}(b_j)$ for sections $a_j \in \cO(U_i)$ and $b_j\in f_*(\cB)(U_i)$. Then, the product $\nu_\cB(1_\cB)|_U \cdot x$ is locally of the form
\begin{align*}
\nu_\cB(1_\cB)|_{U_i}\cdot x|_{U_i} &= \nu_\cB(1_\cB)|_{U_i}\cdot \sum a_j\nu_\cB(b_j) \\
&= \sum a_j\nu_\cB(1_\cB|_{U_i}\cdot b_j) \\
&= \sum a_j \nu_\cB(b_j) \\
&= x|_{U_i}
\end{align*}
where we use the multiplicativity of $\nu_\cB$ established above. This implies that $\nu_\cB(1_\cB)|_U \cdot x = x$ and therefore $\nu_\cB(1_\cB)$ is the identity in $N_{T/S}(\cB)$ as claimed.

\noindent\ref{lem_norm_az5}: This can be verified via direct computation on the generators given by Lemma \ref{lem_nu_generates}.

\noindent\ref{lem_norm_az4}: This follows since, under the new assumptions, the map is an isomorphism over affine schemes by \cite[3.2.5 (c)]{F}.
\end{proof}

\begin{example}\label{ex_split_algebra_norm}
As in Example \ref{ex_split_norm}, let $f_i \co S_i \to S$ for $i=1, \ldots, m$ be finite locally free morphisms of constant degree and consider $f \co S'= S_1 \sqcup \cdots \sqcup S_m \to S$. Let $\cB$ be a quasi-coherent $\cO|_{S'}$--algebra defined by
\[
\cB(T_1\sqcup \ldots \sqcup T_m) = \cB_1(T_1)\times\ldots\times \cB_m(T_m)
\]
where the $\cB_i$ are quasi-coherent $\cO|_{S_i}$--algebras. We know from Example \ref{ex_split_norm} that
\[
N_{S'/S}(\cB) = N_{S_1/S}(\cB_1)\otimes_\cO \ldots \otimes_\cO N_{S_m/S}(\cB_m)
\]
as a module and that is has universal polynomial law given by
\begin{align*}
\bnu \colon f_{1 *}(\cB_1) \times \cdots \times f_{m*}(\cB_m) &\mapsto N_{S_1/S}(\cB_1)\otimes_{\cO} \ldots \otimes_{\cO} N_{S_m/S}(\cB_m) \\
(b_1,\ldots,b_m) &\mapsto \bnu_1(b_1)\otimes \ldots \otimes \bnu_m(b_m).
\end{align*}
Here, we identify the natural product given to this module by Lemma \ref{lem_norm_az}\ref{lem_norm_az3}. First, we have that
\[
\cB\otimes_{\cO|_{S'}} \cB = (\cB_1\otimes_{\cO|_{S_1}} \cB_1,\ldots,\cB_m\otimes_{\cO|_{S_m}} \cB_m)
\]
where $(b_1,\ldots,b_m)\otimes(b_1',\ldots,b_m') = (b_1\otimes b_1',\ldots,b_m\otimes b_m')$. Then,
\begin{align*}
N_{S'/S}(\cB)\otimes_\cO N_{S'/S}(\cB) &= \big(N_{S_1/S}(\cB_1)\otimes_\cO \ldots \otimes_\cO N_{S_m/S}(\cB_m)\big)^{\otimes 2}, \\
N_{S'/S}(\cB\otimes_{\cO|_{S'}}\cB) &= N_{S_1/S}(\cB_1\otimes_{\cO|_{S_1}} \cB_1)\otimes_\cO \ldots \otimes_\cO N_{S_m/S}(\cB_m\otimes_{\cO|_{S_m}} \cB_m).
\end{align*}
The map
\begin{align*}
N_{S'/S}(\cB)\otimes_\cO N_{S'/S}(\cB) &\xrightarrow{\Phi} N_{S'/S}(\cB\otimes_{\cO|_{S'}}\cB) \\
(b_1\otimes\ldots\otimes b_m)\otimes(b_1'\otimes\ldots\otimes b_m') &\mapsto \Phi_i(b_1\otimes b_1')\otimes\ldots\otimes\Phi_m(b_m\otimes b_m'),
\end{align*}
where $\Phi_i$ is the unique morphism of Lemma \ref{lem_norm_az}\ref{lem_norm_az3} for $N_{S_i/S}(\cB_i)$, is the unique morphism of the same lemma but for $N_{S'/S}(\cB)$. Thus, since the product on $\cB\otimes_{\cO|_{S'}}\cB$ is the tensor product $\mu_1\otimes\ldots\otimes \mu_m$ of the products for the algebras $\cB_i$, the product on $N_{S'/S}(\cB)$ is then given by the composition $N_{S'/S}(\mu_1\otimes\ldots\otimes \mu_m)\circ \Phi$, which will just be the usual tensor component wise multiplication on $N_{S'/S}(\cB) = N_{S_1/S}(\cB_1)\otimes_\cO \ldots \otimes_\cO N_{S_m/S}(\cB_m)$.

In particular, when we consider the split \'etale cover $f\colon S^{\sqcup d} \to S$ and an $\cO|_{S^{\sqcup d}}$--algebra $\cB$ given by
\[
\cB(T_1\sqcup \ldots \sqcup T_d) = \cB_1(T_1)\times\ldots\times\cB_d(T_d),
\]
we see that the canonical multiplication on
\[
N_{\Sd/S}(\cB) = \cB_1\otimes_\cO \ldots \otimes_\cO \cB_d
\]
is the regular multiplication on a tensor product of algebras. A special case is when $\cB=\Mat_n(\cO|_{\Sd}) = (\Mat_n(\cO),\ldots,\Mat_n(\cO))$ where we get
\[
N_{\Sd/S}(\Mat_n(\cO|_{\Sd})) = \Mat_n(\cO)\otimes_{\cO}\ldots \otimes_\cO \Mat_n(\cO).
\]
\end{example}

We now turn our attention to Azumaya algebras and assume that $T\to S$ is finite \'etale. The following is inspired by \cite[3.2.5]{F}, where Ferrand claims that for arbitrary $R'$--modules $M'_1$ and $M'_2$ there exists an $R$--linear map between $N_{R'/R}\big(\Hom_{R'}(M'_1,M'_2))$ and $\Hom_{R}\big( N_{R'/R}(M'_1),N_{R'/R}(M_2')\big)$ arising via the universal property from a normic polynomial law 
\[
\bW_R(\Hom_{R'}(M'_1,M'_2)) \to \bW_R\big(\Hom_{R}\big( N_{R'/R}(M'_1),N_{R'/R}(M_2')\big)\big).
\]

which he defines on $R$--points by $f' \mapsto N_{R'/R}(f')$ for an $R'$--linear map $f' \colon M'_1 \to M'_2$. However, a normic polynomial law must also be defined on $S$--points for each $S\in \Rings_R$. Our interpretation of how Ferrand intended this construction to proceed is the following. We consider the diagram
\[
\begin{tikzcd}
\Hom_{R'}(M'_1,M'_2)\otimes_R S \ar[r] \ar[d,dashed] & \Hom_{R'\otimes_R S}(M'_1\otimes_R S, M'_2\otimes_R S) \ar[d] \\
\Hom_R(N(M'_1),N(M'_2))\otimes_R S \ar[r] & \Hom_S(N(M'_1)\otimes_R S,N(M'_2)\otimes_R S) 
\end{tikzcd}
\]
where for brevity we use $N=N_{R'/R}$. We must define a suitable dashed morphism. Since we have that $N_{(R'\otimes_R S)/S}(M'_i \otimes_R S) \cong N_{R'/R}(M'_i)\otimes_R S$ by Remark \ref{rem_Ferrand_iso}, the downward map on the right may be taken to be $f \mapsto N_{(R'\otimes_R S)/S}(f)$ as Ferrand proposes. However, the canonical horizontal maps are not isomorphisms in general. When $M'_1$ is finitely generated projective, and thus $N_{R'/R}(M'_1)$ is also finitely generated projective, then these horizontal maps are isomorphisms and we may define the dashed morphism to be the obvious composition. In the case when $M'_1$ is not finitely generated projective, we have not be able to verify that Ferrand's construction of this normic polynomial law works as suggested.

In any case, we make an equivalent assumption in Lemma \ref{lem_Hom_normic} below in order to ensure that our $\cHom$ modules are quasi-coherent. A similar assumption is made by Knus-Ojanguren in \cite[Prop. 4.4]{KO75}, where the authors prove Lemma \ref{lem_Hom_normic} over rings. We recall that by \cite[5.3]{F}, in the setting of \cite{KO75}, Ferrand's norm functor and the one constructed by Knus-Ojanguren are isomorphic. 

\begin{lem}\label{lem_Hom_normic}
Assume that $T\to S$ is finite \'etale. Let $\cM_1, \cM_2 \in \QCoh(T)$ and assume that $\cM_1$ is finite locally free. There is a normic polynomial law
\[
\boeta \colon f_*(\cHom_{\cO|_T}(\cM_1,\cM_2)) \to \cHom_{\cO}(N_{T/S}(\cM_1),N_{T/S}(\cM_2))
\]
defined over $\Aff_S$ as follows. For $U\in \Aff_S$ we consider the fiber product diagram
\[
D = \begin{tikzcd}
T\times_S U \ar[r] \ar[d,"f'"] & T \ar[d,"f"] \\
U \ar[r] & S
\end{tikzcd}
\]
and then for
\[
\varphi \in f_*(\cHom_{\cO|_T}(\cM_1,\cM_2))(U) = \Hom_{\cO|_{T\times_S U}}(\cM_1|_{T\times_S U},\cM_2|_{T\times_S U})
\]
we set $\boeta(\varphi)$ to be the composition
\[
\begin{tikzcd}[column sep=6ex]
N_{T/S}(\cM_1)|_U \ar[drr,bend right=10,end anchor=west,"\boeta(\varphi)"] & N_{(T\times_S U)/U}(\cM_1|_{T\times_S U}) \ar[l,swap,"\phi_D(\cM_1)"] \ar[r,"N_{(T\times_S U)/U}(\varphi)" {yshift=1ex}] & N_{(T\times_S U)/U}(\cM_2|_{T\times_S U}) \ar[d,"\phi_D(\cM_2)"] \\
 & & N_{T/S}(\cM_2)|_U
\end{tikzcd}
\]
where $\phi_D$ is the isomorphism of functors of Lemma \ref{lem_phi_isomorphism}.
\end{lem}
\begin{proof}
Assuming $\cM_1$ is finite locally free ensures that $\cHom_{\cO|_T}(\cM_1,\cM_2)$ is quasi-coherent by Lemma \ref{lem_hom_quasi_coh}. It also ensures that $N_{T/S}(\cM_1)$ is finite locally free by Lemma \ref{lem_norm_modules_etale} since we are assuming $T\to S$ is \'etale. Therefore, we know $\cHom_{\cO}(N_{T/S}(\cM_1),N_{T/S}(\cM_2))$ is quasi-coherent as well by another application of Lemma \ref{lem_hom_quasi_coh}.

We verify that $\boeta$ defines a natural transformation. Let $g\colon V \to U$ be a morphism in $\Aff_S$ and let $W\in \Aff_V$. We will show that $\boeta(\varphi)|_V(W) = \boeta(\varphi|_{T\times_S V})(W)$ for a morphism $\varphi \in f_*(\cHom_{\cO|_T}(\cM_1,\cM_2))(U)$. Since $W$ is arbitrary, we will conclude that $\boeta(\varphi)|_V = \boeta(\varphi|_{T\times_S V})$ as is required.

We have the following commutative diagram
\[
\begin{tikzcd}[column sep=1ex]
(T\times_S V)\times_V W \ar[rr,"g''"] \ar[dr,near end,"h'_V"] \ar[dd] &[-8ex] & (T\times_S U)\times_U W \ar[rr,"g_0''"] \ar[dr,near end,"h'_U"] \ar[dd] &[-8ex] & T\times_S W \ar[dd] \ar[dr,near end,"h'_S"] & \\
 & T\times_S V \ar[rr,crossing over,"g'"] & & T\times_S U \ar[rr,crossing over,near start,"g_0'"] & & T  \\
W \ar[rr,equals] \ar[dr,swap,near start,"h_V"] & & W \ar[rr,equals] \ar[dr,swap,near start,"h_U"] & & W \ar[dr,swap,near start,"h_S"] & \\
& V \ar[rr,"g"] \ar[from=uu,crossing over,near start,"f''"] & & U \ar[from=uu,crossing over,near start,"f'"] \ar[rr,"g_0"] & & S \ar[from=uu,crossing over,near start,"f"]
\end{tikzcd}
\]
where the vertical faces are fiber product diagrams. This means that $g''$ and $g_0''$ are isomorphisms. We set
\[
D_U = \begin{tikzcd}
T\times_S U \ar[r,"g_0'"] \ar[d,"f'"] & T \ar[d,"f"] \\
U \ar[r,"g_0"] & S
\end{tikzcd} \text{ and }
D_V = \begin{tikzcd}
T\times_S V \ar[r,"g_0'\circ g'"] \ar[d,"f''"] & T \ar[d,"f"] \\
V \ar[r,"g_0\circ g"] & S.
\end{tikzcd}
\]
The morphism $\boeta(\varphi)|_V(W)$ is then the composition
\[
\begin{tikzcd}
N_{\cO(T\times_S W)/\cO(W)}(\cM_1(T\times_S W))  \\
N_{\cO((T\times_S U)\times_U W)/\cO(W)}(\cM_1((T\times_S U)\times_U W)) \ar[u,swap,"\phi_{D_U}(\cM_1)(W)"] \ar[d,"\Gamma_{\cO(W)}^d(\varphi((T\times_S U)\times_U W))\otimes \Id"] \\
N_{\cO((T\times_S U)\times_U W)/\cO(W)}(\cM_2((T\times_S U)\times_U W)) \ar[d,"\phi_{D_U}(\cM_2)(W)"] \\
N_{\cO(T\times_S W)/\cO(W)}(\cM_2(T\times_S W)).
\end{tikzcd}
\]
Since $g_0'$ is an isomorphism, the restriction map $\cM_1(g_0')\colon \cM_1(T\times_S W)\to \cM_1((T\times_S U)\times_U W)$ is also an isomorphism. Therefore, the isomorphism $\phi_{D_U}(\cM_1)(W)$ takes the form
\[
\gamma^{\overline{a}}(\overline{m''})\otimes w \mapsto \gamma^{\overline{a}}(\overline{\cM_1(g_0')^{-1}(m'')})\otimes w.
\]
for $\overline{a}=(a_1,\ldots,a_k)$ with $\sum a_i = d$ and $\overline{m''}=(m''_1,\ldots,m''_k) \in \cM_1((T\times_S U)\times_U W)^k$. So, denoting $\varphi((T\times_S U)\times_U W)=\varphi''$, the morphism $\boeta(\varphi)|_V(W)$ behaves as
\begin{align*}
\gamma^{\overline{a}}(\overline{m'})\otimes w &\mapsto \gamma^{\overline{a}}(\overline{\cM_1(g_0')(m')})\otimes w \\
&\mapsto \gamma^{\overline{a}}(\overline{(\varphi''\circ \cM_1(g_0'))(m')})\otimes w \\
&\mapsto \gamma^{\overline{a}}(\overline{(\cM_2(g_0')^{-1}\circ \varphi''\circ \cM_1(g_0'))(m')})\otimes w
\end{align*}
for $\overline{m'} \in \cM_1(T\times_S W)^k$. However, since $\varphi$ is a natural transformation, the diagram
\[
\begin{tikzcd}
\cM_1(T\times_S W) \ar[r,"\varphi(T\times_S W)"] \ar[d,"\cM_1(g_0')"] & \cM_2(T\times_S W) \ar[d,"\cM_2(g_0')"] \\
\cM_1((T\times_S U)\times_U W) \ar[r,"\varphi''"] & \cM_2((T\times_S U)\times_U W)
\end{tikzcd}
\]
commutes and therefore we have that
\[
\boeta(\varphi)|_V(W)(\gamma^{\overline{a}}(\overline{m'})\otimes w) = \gamma^{\overline{a}}(\overline{\varphi(T\times_S W)(m')})\otimes w.
\]

Next, we consider $\boeta(\varphi|_{T\times_S V})(W)$. This is the composition
\[
\begin{tikzcd}
N_{\cO(T\times_S W)/\cO(W)}(\cM_1(T\times_S W))  \\
N_{\cO((T\times_S V)\times_V W)/\cO(W)}(\cM_1((T\times_S V)\times_V W)) \ar[u,swap,"\phi_{D_V}(\cM_1)(W)"] \ar[d,"\Gamma_{\cO(W)}^d(\varphi((T\times_S V)\times_V W))\otimes \Id"] \\
N_{\cO((T\times_S V)\times_V W)/\cO(W)}(\cM_2((T\times_S V)\times_V W)) \ar[d,"\phi_{D_V}(\cM_2)(W)"] \\
N_{\cO(T\times_S W)/\cO(W)}(\cM_2(T\times_S W)).
\end{tikzcd}
\]
A symmetric argument, simply swapping $U$ for $V$, then yields that
\[
\boeta(\varphi|_{T\times_S V})(W)(\gamma^{\overline{a}}(\overline{m'})\otimes w) = \gamma^{\overline{a}}(\overline{\varphi(T\times_S W)(m')})\otimes w
\]
and hence $\boeta(\varphi)|_V(W)=\boeta(\varphi|_{T\times_S V})(W)$ as desired. Thus $\boeta$ is a well-defined natural transformation. 

Finally, it follows from the explicit descriptions of $\boeta(\varphi)(W)$ above, which also hold when $W\in \Aff_U$, that $\boeta$ is normic. For $t\in \cO(T\times_S W)$, the morphism $\boeta(t\varphi)$ will have the following formula over $W$.
\begin{align*}
\boeta(t\varphi)(\gamma^{\overline{a}}(\overline{m'})\otimes w) =& \gamma^{\overline{a}}(\overline{t\varphi(T\times_S W)(m')})\otimes w \\
=&\gamma^d(t)\cdot \gamma^{\overline{a}}(\overline{\varphi(T\times_S W)(m')})\otimes w \\
=& \gamma^{\overline{a}}(\overline{\varphi(T\times_S W)(m')})\otimes \norm_{\cO(T\times_S W)/\cO(W)}(t)\cdot w \\
=& \bnorm(t)\cdot \boeta(\varphi)(\gamma^{\overline{a}}(\overline{m'})\otimes w)
\end{align*}
Thus, we conclude $\boeta(t\varphi)=\bnorm(t)\cdot \boeta(\varphi)$, which finishes the proof.
\end{proof}

\begin{lem} \label{lem_norm_etale_neutral_azu} Let $T\to S$ be a finite \'etale morphism of schemes of degree $d$. Consider a neutral Azumaya $\cO|_T$--algebra $\cB=\cEnd_{\cO|_T}(\mathcal{Q})$ for a finitely locally free $\cO|_T$--module $\mathcal{Q}$. Let $\bnu_{\cB} \colon f_*(\cB) \to N_{T/S}(\cB)$ be the universal normic polynomial law of Proposition \ref{prop_universal_property}. Then, there is a unique isomorphism of $\cO$--algebras $\Psi$ making the following diagram commute
\[
\begin{tikzcd}
f_*(\cB) \arrow{r}{\bnu_{\cB}} \arrow[swap]{dr}{\boeta} & N_{T/S}(\cB) \arrow{d}{\Psi} \\
 & \cEnd_{\cO} \bigl( N_{T/S}(\mathcal{Q}) \bigr)
\end{tikzcd}
\]
where $\boeta$ is the normic polynomial law of Lemma \ref{lem_Hom_normic}.
\end{lem}
\begin{proof}
Since $\boeta$ is a normic polynomial law, a unique such $\cO$--module map $\Psi$ exists by Proposition \ref{prop_universal_property}. Additionally, since $\boeta$ is induced by the norm functor, multiplication in endomorphism algebras is by composition, and functors respect composition, we see that $\boeta$ is multiplicative. Using this along with the fact that $\bnu_{\cB}$ is multiplicative by Lemma \ref{lem_norm_az}\ref{lem_norm_az3}, we compute
\begin{align*}
\Psi(\bnu_{\cB}(b_1)\bnu_{\cB}(b_2)) &= \Psi\circ N_{T/S}(\mu)\circ \Phi(\bnu_{\cB}(b_1)\otimes\bnu_{\cB}(b_2)) \\
&= \Psi\circ N_{T/S}(\mu)\circ \bnu'(b_1\otimes b_2) \\
&= \Psi\circ \bnu_{\cB} \circ f_*(\mu)(b_1\otimes b_2) = \Psi\circ \bnu_{\cB}(b_1b_2) \\
&= \boeta(b_1b_2) = \boeta(b_1)\boeta(b_2) \\
&= \Psi(\bnu_{\cB}(b_1))\Psi(\bnu_{\cB}(b_2))
\end{align*}
and so $\Psi$ is multiplicative on the image of $\bnu_{\cB}$. Therefore, by Lemma \ref{lem_nu_generates} and the linearity of $\Psi$, this means $\Psi$ is multiplicative in general and hence is an algebra homomorphism. The fact that $\Psi$ is an isomorphism follows since it is an isomorphism over affine schemes by \cite[3.2.5 (c)]{F}.
\end{proof}

The following is the Azumaya algebra analogue of Lemma \ref{lem_norm_modules_etale}.
\begin{lem}\label{lem_ferrand} Assume that $T \to S$ is finite \'etale of constant degree $d$. If $\cA$ is an Azumaya $\cO|_T$--algebra of constant degree $r$, then $N_{T/S}(\cA)$ is an Azumaya $\cO$--algebra of constant degree $r^d$.
\end{lem}
\begin{proof}
This follows from Lemma \ref{lem_norm_az}\ref{lem_norm_az3} and Lemma \ref{lem_norm_etale_neutral_azu}. There will be a cover $\{U_i \to S\}_{i\in I}$ over which we have $\cA|_{T\times_S U_i} \cong \cEnd_{\cO|_{T\times_S U_i}}(\cQ_i)$ for a locally free $\cO|_{T\times_S U_i}$--module $\cQ_i$ of constant rank $r$, and so we have the isomorphism of Lemma \ref{lem_norm_etale_neutral_azu}
\[
N_{T/S}(\cA)|_{U_i} = N_{T\times_S U_i / U_i}(\cA|_{T\times_S U_i}) \xrightarrow{\Psi} \cEnd_{\cO|_{U_i}}(N_{T\times_S U_i / U_i}(\cQ_i))
\]
where $N_{T\times_S U_i / U_i}(\cQ_i)$ is a locally free $\cO|_{U_i}$--module of constant rank $r^d$ by Lemma \ref{lem_norm_modules_etale}. Therefore, $N_{T/S}(\cA)$ is an Azumaya $\cO$--algebra of degree $r^d$ as claimed.
\end{proof}

\begin{remark}\label{rem_QAlg}
Since Lemma \ref{lem_norm_az}\ref{lem_norm_az3} and \ref{lem_norm_az5} show that the norm respects quasi-coherent algebras, it is immediate that the morphism $N\colon \QCoh\flf^d \to \QCoh$ of \eqref{eq_norm_stack_morphism} restricts to a morphism 
\begin{equation}\label{eq_normalg_stack_morphism}
N\alg \colon \QAlg\flf^d \to \QAlg
\end{equation}
between stacks of quasi-coherent algebras. In detail, we define $\QAlg\flf^d$ to be the substack of $\QCoh\flf^d$ which has
\begin{enumerate}[label={\rm(\roman*)}]
\item objects $(T'\to T,\cB) \in \QCoh\flf^d$ where $\cB$ is a quasi-coherent $\cO|_{T'}$--algebra, and
\item morphisms $(f,g,\varphi)\colon (X'\to X,\cB_1) \to (T'\to T,\cB_2)$ of $\QCoh\flf^d$ where $\varphi \colon \cB_1 \to g^*(\cB_2)$ is an $\cO|_{X'}$--algebra isomorphism. 
\end{enumerate}
Similarly, the unadorned $\QAlg$ is the substack of $\QCoh$ which has
\begin{enumerate}[label={\rm(\roman*)}]
\item objects $(X,\cB)\in \QCoh$ where $\cB$ is a quasi-coherent $\cO|_X$--algebra, and
\item morphisms $(g,\varphi)\colon (X',\cB') \to (X,\cB)$ of $\QCoh$ where $\varphi \colon \cB' \to g^*(\cB)$ is an $\cO|_{X'}$--algebra morphism.
\end{enumerate}
\end{remark}

\section{Cohomological Description}\label{cohomological_description}
In this section we give a cohomological description of the norm functor over finite \'etale covers of degree $d$ by analyzing restrictions of the morphisms $N$ and $N\alg$ of \eqref{eq_norm_stack_morphism} and \eqref{eq_normalg_stack_morphism} to various substacks of the stacks $\QCoh\flf^d$ of Section \ref{the_construction} and $\QAlg\flf^d$ of Remark \ref{rem_QAlg} respectively. In particular, we will consider the following stacks.
\begin{enumerate}[label={\rm(\roman*)}]
\item Let $\fMod_r$ be the substack of $\QCoh$ whose objects are those $(X,\cM) \in \QCoh$ where $\cM$ is a finite locally free $\cO|_X$--module of constant rank $r$, and whose morphisms are the cartesian morphisms from $\QCoh$. This is equivalent to the split stack $\fVec_r$ considered in \cite[2.4.1.8]{CF}.
\item Let $\fAzu_r$ be the substack of $\QAlg$ whose objects are those $(X,\cA) \in \QAlg$ where $\cA$ is an Azumaya $\cO|_X$--algebra of constant degree $r$, and whose morphisms are the cartesian morphisms of $\QAlg$. This is equivalent to the split stack of \cite[2.5.3.10]{CF}.
\item Let $\fMod_r^{d\etale}$ be the substack of $\QCoh\flf^d$ whose objects are those $(T'\to T,\cM)\in \QCoh\flf^d$ where $T'\to T$ is an \'etale cover (of degree $d$) and $\cM$ is a locally free $\cO|_{T'}$--module of constant rank $r$, and whose morphisms are the cartesian morphisms of $\QCoh\flf^d$.
\item \label{defn_Azu_r^d} Let $\fAzu_r^{d\etale}$ be the substack of $\QAlg\flf^d$ whose objects are those $(T'\to T,\cA)\in \QAlg\flf^d$ where $T'\to T$ is an \'etale cover (of degree $d$) and $\cA$ is an Azumaya $\cO|_{T'}$--algebra of constant degree $r$, and whose morphisms are the cartesian morphisms of $\QAlg\flf^d$.
\end{enumerate}
Since all four of the above stacks only contain cartesian morphisms, they are fibered in groupoids by Lemma \ref{lem_cartesian_groupoids}. In fact, all four stacks are gerbes, which we will justify for the first two in \ref{sec_aut_sheaves}, for $\fMod_r^{d\etale}$ before Lemma \ref{lem_Azu_torsor_gerbe}, and for $\fAzu_r^{d\etale}$ after Lemma \ref{lem_Azu_torsor_gerbe}.

The results of Lemma \ref{lem_ferrand} imply that the norm morphism $N\colon \QCoh\flf^d \to \QCoh$ restricts to two morphisms of stacks,
\begin{equation}\label{eq_N_Mod_Azu}
N_{\fMod} \colon \fMod_r^{d\etale} \to \fMod_{r^d} \text{ and } N_{\fAzu}\colon \fAzu_r^{d\etale} \to \fAzu_{r^d}.
\end{equation}
Since all four of these stacks are gerbes, we will obtain a cohomological description of $N_\fMod$ and $N_\fAzu$ by applying Lemma \ref{lem_loos_cohom}. To do so, we first identify some of the automorphism sheaves of objects in these stacks.

\subsection{Automorphism Sheaves}\label{sec_aut_sheaves}
To begin, if we consider the object $(S,\cO^r) \in \fMod_r$, then it is clear that its automorphism sheaf is $\cAut(S,\cO^r) = \GL_r$. Any locally free module of rank $r$ is by definition locally isomorphic to $\cO^r$, the stack $\fMod_r$ is fibered in groupoids, and the fibers $\fMod_r(U)$ for $U\in \Sch_S$ are nonempty since they contain the free module, so we know $\fMod_r$ is a gerbe. We call $(S,\cO^r)$ the split object in $\fMod_r(S)$.

If we consider the object $(S,\Mat_r(\cO)) \in \fAzu_r$, then we have $\cAut(S,\Mat_r(\cO))=\PGL_r$. Any Azumaya algebra of degree $r$ is locally isomorphic to $\Mat_r(\cO)$ and so $\fAzu_r$ is a gerbe as well. We call $(S,\Mat_r(\cO))$ the split object of $\fAzu_r(S)$.

Next, we define some semi-direct products of groups which will appear later as automorphism sheaves. Let $f\colon T\to S$ be a degree $d$ \'etale cover of our base scheme.

First, we define the group $f_*(\GL_r|_T)\rtimes \cAut_S(T)$. Let $X\in \Sch_S$ and let $g\colon T\times_S X \iso T \times_S X$ be an isomorphism of $X$--schemes, i.e., $g\in \cAut_S(T)(X)$. We then have a pullback functor $g^*$ from the category of $\cO|_{T\times_S X}$--modules to itself and $g^*(\cO|_{T\times_S X}) = \cO|_{T\times_S X}$. This also means that $g^*(\cO|_{T\times_S X}^r) = \cO_{T\times_S X}^r$. Therefore, for a section $\varphi \in f_*(\GL_r|_T)(X) = \Aut_{\cO_{T\times_S X}}(\cO|_{T\times_S X}^r)$ we also have that $g^*(\varphi)$ is an automorphism of $\cO|_{T\times_S X}^r$. We use this to define the semidirect product structure on $f_*(\GL_r|_T)\rtimes\cAut_S(T)$ on by
\[
\varphi \cdot g = g\cdot g^*(\varphi)
\]
for appropriate sections.

When $T=\Sd$, this group becomes $(\GL_r)^d \rtimes \SS_d$ where $\SS_d$ acts on $(\GL_r)^d$ be permuting the factors as follows. To keep track of position, write $S^{\sqcup d}=S_1\sqcup \ldots \sqcup S_d$ where each $S_j=S$. Let $X\in \Sch_S$ be any scheme. Since $X$ is possibly disconnected, let $X=\bigsqcup_{i\in I}X_i$ be its decomposition into connected components. Then $S^{\sqcup d}\times_S X = \bigsqcup_{i\in I} (X_{1,i}\sqcup \ldots\sqcup X_{d,i})$ where each $X_{j,i}=X_i$. For $\sigma = (\sigma_i)_{i\in I} \in \SS_d(X) = \prod_{i\in I}\SS_d(\ZZ)$ (here $\SS_d(\ZZ)$ is simply the abstract group of permutations since $\Spec(\ZZ)$ is connected), we view it as the scheme isomorphism which sends component $X_{j,i} \to X_{\sigma_i(j),i}$ via $\Id_{X_i}$. 

Now, let $\cM$ be an $\cO|_{X^{\sqcup d}}$--module. A $X^{\sqcup d}$ scheme is of the form $\sqcup_{i\in I}(Y_{1,i}\sqcup \ldots \sqcup Y_{d,i})$ where each $Y_{j,i}$ are arbitrary $X$--schemes and the structure morphism sends $Y_{j,i} \to X_{j,i}$. The module $\cM$ will then evaluate as
\[
\cM\big(\sqcup_{i\in I}(Y_{1,i}\sqcup \ldots \sqcup Y_{d,i})\big) = \prod_{i\in I}(\cM_{1,i}(Y_{1,i})\times\ldots\times \cM_{d,i}(Y_{d,i}))
\]
for $\cO|_X$--modules $\cM_{j,i}$. We express this as $\cM=(\cM_{1,i},\ldots,\cM_{d,i})_{i\in I}$. The pullback module $\sigma^*(\cM)$ will evaluate the $X^{\sqcup d}$-scheme $\sqcup_{i\in I}(Y_{1,i}\sqcup \ldots \sqcup Y_{d,i}$ as if the structure morphism sends $Y_{j,i} \to X_{\sigma_i(j),i}$. Therefore,
\[
\sigma^*(\cM)\big(\sqcup_{i\in I}(Y_{1,i}\sqcup \ldots \sqcup Y_{d,i})\big) = \prod_{i\in I}(\cM_{\sigma_i(1),i}(Y_{1,i})\times\ldots\times \cM_{\sigma_i(d),i}(Y_{d,i})),
\]
i.e., $\sigma^*(\cM) = (\cM_{\sigma_i(1),i},\ldots,\cM_{\sigma_i(d),i})_{i\in I}$. So, for an automorphism $\varphi = (\varphi_{1,i},\ldots,\varphi_{d,i})_{i\in I}$ of $\cM$, we have 
\[
\sigma^*(\varphi) = (\varphi_{\sigma_i(1),i},\ldots,\varphi_{\sigma_i(d),i})_{i\in I}.
\]
In particular, this applies when all $\cM_{j,i} = \cO|_{X_i}^r$ to describe $\sigma^*(\varphi)$ for $\varphi \in (\GL_r)^d(X)$. As an example, if $X$ is connected, $d=3$, and $\sigma = (1\; 2\; 3)$ is a cycle, then
\[
\sigma^*(\varphi_1,\varphi_2,\varphi_3) = (\varphi_2,\varphi_3,\varphi_1).
\]

Likewise, we define the group $f_*(\PGL_r|_T)\rtimes \cAut_S(T)$. With $g$ still as above, we have $g^*(\cEnd_{\cO|_{T\times_S X}}(\cO|_{T\times_S X}^r)) = \cEnd_{\cO|_{T\times_S X}}(\cO|_{T\times_S X}^r)$ and so we also have a semidirect product structure on $f_*(\PGL_r|_T)\rtimes \cAut_S(T)$ using the same formula, just when $\varphi$ is an algebra automorphism. Once again, when $T=\Sd$, this group becomes $(\PGL_r)^d\rtimes \SS_d$.

\begin{lem}\label{lem_mod_torsor_gerbe}
Consider objects of the form $(f\colon T\to S,\cO|_{T}^r) \in \fMod_r^{d\etale}(S)$.
\begin{enumerate}[label={\rm(\roman*)}]
\item \label{lem_mod_torsor_gerbe_i} We have that
\[
\cAut(T\to S,\cO|_T^r) \cong f_*(\GL_r|_T)\rtimes \cAut_S(T).
\]
\item \label{lem_mod_torsor_gerbe_ii} In particular,
\[
\cAut(\Sd \to S,\cO_{\Sd}^r) \cong (\GL_r)^d \rtimes \SS_d.
\]
\end{enumerate}
\end{lem}
\begin{proof}
\noindent\ref{lem_mod_torsor_gerbe_i}: Let $X\in \Sch_S$. A section in $\cAut(T\to S,\cO|_T^r)(X)$ is a triple of the form $(\Id_X,g,\varphi)$ where $g\in \cAut_S(T)(X)$ and $\varphi \colon \cO|_{T\times_S X}^r \iso g^*(\cO|_{T\times_S X}^r) = \cO|_{T\times_S X}^r$, and so $\varphi \in f_*(\GL_r|_T)(X)$. Unsurprisingly, the map of sheaves
\begin{align*}
\cAut(T\to S,\cO|_T^r) &\to f_*(\GL_r|_T)\rtimes \cAut_S(T) \\
(\Id_S,g,\varphi) &\mapsto g\cdot \varphi
\end{align*}
will give our desired isomorphism. It is clearly bijective and it is a group isomorphism since we have
\begin{align*}
&(\Id_S,h,\psi)(\Id_S,g,\varphi) = (\Id_S,hg,g^*(\psi)\varphi) \\
\mapsto &h \cdot g \cdot g^*(\psi) \cdot \varphi = h \cdot \psi \cdot g \cdot \varphi =(h\cdot \psi)(g\cdot \varphi) 
\end{align*}

\noindent\ref{lem_mod_torsor_gerbe_ii}: This follows immediately from \ref{lem_mod_torsor_gerbe_i}.
\end{proof}

Every degree $d$ \'etale cover is locally isomorphic to $\Sd\to S$ and likewise every rank $r$ finite locally free module over such a cover is locally isomorphic to $\cO|_{\Sd}^r$. Therefore, $\fMod_r^{d\etale}$ is a gerbe. By choosing $(\Sd\to S,\cO_{\Sd}^r)$ as the split object, we view the groupoid $\fMod_r^{d\etale}(S)$ as the groupoid of twisted forms of $(\Sd,\cO_{\Sd}^r)$. By Lemma \ref{lem_mod_torsor_gerbe}\ref{lem_mod_torsor_gerbe_ii} this groupoid is equivalent to the category of $(\GL_r^d)\rtimes \SS_d$--torsors. The isomorphism classes in $\fMod_r^{d\etale}(S)$ are classified by $H^1(S,(\GL_r)^d\rtimes \SS_d)$. In the notation of \eqref{AppenB_gerbe_tors_equiv}, $\fMod_r^{d\etale}$ is equivalent to the stack $\fF(\GL_r)^{d\etale}$.

\begin{lem}\label{lem_Azu_torsor_gerbe}
Consider objects of the form $(f\colon T\to S,\Mat_r(\cO|_{T})) \in \fAzu_r^{d\etale}(S)$.
\begin{enumerate}[label={\rm(\roman*)}]
\item \label{lem_Azu_torsor_gerbe_i} We have that
\[
\cAut(T\to S,\Mat_r(\cO|_{T})) \cong f_*(\PGL_r|_T)\rtimes \cAut_S(T).
\]
\item \label{lem_Azu_torsor_gerbe_ii} In particular,
\[
\cAut(\Sd \to S,\Mat_r(\cO|_{\Sd})) \cong (\PGL_r)^d \rtimes \SS_d.
\]
\end{enumerate}
\end{lem}
\begin{proof}
\noindent\ref{lem_Azu_torsor_gerbe_i}: Because $g^*(\cEnd_{\cO|_{T\times_S Y}}(\cO|_{T\times_S Y}^r)) = \cEnd_{\cO|_{T\times_S Y}}(\cO|_{T\times_S Y}^r)$, this proof is the same as the proof of Lemma \ref{lem_mod_torsor_gerbe}\ref{lem_mod_torsor_gerbe_i} except $\psi, \varphi \in f_*(\PGL_r|_T)$ will instead be algebra automorphisms.

\noindent\ref{lem_Azu_torsor_gerbe_ii}: This follows from \ref{lem_Azu_torsor_gerbe_i}.
\end{proof}

Via a similar argument as above, any object of $\fAzu_r^{d\etale}$ is locally isomorphic to $(\Sd\to S,\Mat_r(\cO|_{\Sd}))$. Since $(\Sd\to S,\Mat_r(\cO|_{\Sd}))$ is a global object, every fiber is non-empty. Hence, $\fAzu_r^{d\etale}$ is a gerbe. We choose $(\Sd\to S,\Mat_r(\cO|_{\Sd}))$ as the split object. We view $\fAzu_r^{d\etale}(S)$ as the groupoid of its twisted forms, see Proposition \ref{lem_gerbes_and_torsors}\ref{lem_gerbes_and_torsors_iv}. By Lemma \ref{lem_Azu_torsor_gerbe}\ref{lem_Azu_torsor_gerbe_ii}, this category is equivalent to the category of $(\PGL_r)^d\rtimes\SS_d$--torsors. The isomorphism classes in $\fAzu_r^{d\etale}(S)$ are classified by $H^1(S,(\PGL_r)^d\rtimes\SS_d)$. This stack is equivalent to the stack $\fF(\PGL_r)^{d\etale}$ of \eqref{AppenB_gerbe_tors_equiv}.

\subsection{Cohomology Maps}
Since we know from Example \ref{ex_split_norm} that the morphism $N_{\fMod}$ maps $(\Sd\to S,\cO_{\Sd}^r) \in \fMod_{r}^{d\etale}$ to $(\cO^r)^{\otimes d} \cong \cO^{(r^d)} \in \fMod_{r^d}$, functoriality yields an associated group homomorphism between the automorphism groups
\begin{equation}\label{eq_Nmod_group_hom}
N_{{\fMod},(\Sd\to S,\cO|_{\Sd}^r)}\colon (\GL_r)^d\rtimes \SS_d \to \GL_{r^d}
\end{equation}
and we seek to describe the resulting map $\widetilde{N_{\fMod}} \colon H^1(S,(\GL_r)^d\rtimes \SS_d) \to H^1(S,\GL_{r^d})$ on isomorphism classes.

First, we consider the Segre homomorphism 
\begin{align}
\Seg \colon \GL_r \times_S \dots \times_S \GL_r &\to \GL_{r^d} \label{eq_Seg}\\
(A_1, \dots, A_d) &\mapsto A_1 \otimes \dots \otimes A_d \nonumber
\end{align}
which we extend slightly. We view the isomorphism $\cO^{(r^d)}\cong(\cO^r)^{\otimes d}$ as an identification. Let $X\in \Sch_S$ by any scheme. Since $X$ may be disconnected, let $X=\bigsqcup_{i\in I}X_i$ be its decomposition into connected components. The $\cO(X)$--module $\cO^{(r^d)}(X)$ is spanned by elements of the form $x_1\otimes\ldots\otimes x_d$ with each $x_j=(x_{j,i})_{i\in I} \in \cO^r(X)=\prod_{i\in I}\cO^r(X_i)$.

For each $\sigma = (\sigma_i)_{i\in I} \in \SS_d(X)=\prod_{i\in I}\SS_d(\ZZ)$, we obtain a linear transformation of $\cO^{(r^d)}(X)$ by sending
\[
(x_{1,i})_{i\in I}\otimes\ldots\otimes (x_{d,i})_{i\in I} \mapsto (x_{\sigma_i^{-1}(1),i})_{i\in I}\otimes\ldots\otimes (x_{\sigma_i^{-1}(d),i})_{i\in I}.
\]
Here, $x_{k,i}$ is ending up in the $\sigma_i(k),i$--position. For example, if $X$ is connected, $d=3$, and $\sigma = (1\; 2\; 3)$ is a cycle, then
\[
x_1\otimes x_2 \otimes x_3 \mapsto x_3 \otimes x_1 \otimes x_2. 
\]
This yields an injective group homomorphism $j(X)\colon \SS_d(X) \inj \GL_{r^d}(X)$ and together for all $X\in \Sch_S$ these yield an injective morphism of group sheaves $j\colon \SS_d \inj \GL_{r^d}$. For clarity in the following computation we assume $X$ is connected, however the computation in the general case is the same but with added indices as above. For $A_i \in \GL_r(X)$ and  $\sigma \in \SS_d(X)$, we have 
\begin{align*}
&(\Seg(A_1,\ldots,A_d)\circ j(\sigma))(x_1\otimes\ldots\otimes x_d) \\
=& (A_1\otimes\ldots\otimes A_d)(x_{\sigma^{-1}(1)}\otimes\ldots\otimes x_{\sigma^{-1}(d)}) \\
=& A_1(x_{\sigma^{-1}(1)})\otimes\ldots\otimes A_d(x_{\sigma^{-1}(d)}) \\
=& j(\sigma) \big(A_{\sigma(1)}(x_1)\otimes\ldots\otimes A_{\sigma(d)}(x_d)\big) \\ 
=& (j(\sigma)\circ (A_{\sigma(1)}\otimes\ldots\otimes A_{\sigma(d)}))(x_1\otimes\ldots\otimes x_d) \\
=& (j(\sigma)\circ \Seg(A_{\sigma(1)},\ldots,A_{\sigma(d)}))(x_1\otimes\ldots\otimes x_d)
\end{align*}
which shows that $\Seg(A_1,\ldots,A_d)\circ j(\sigma) = j(\sigma)\circ \Seg(A_{\sigma(1)},\ldots,A_{\sigma(d)})$. Therefore, combining $\Seg$ with $j$ we get a well defined group homomorphism
\begin{equation}\label{eq_segre}
\Seg'\colon (\GL_r)^d \rtimes \SS_d \to  \GL_{r^d}.
\end{equation}

\begin{thm}\label{thm_norm_group_hom}
The group homomorphism
\[
N_{{\fMod},(\Sd\to S,\cO|_{\Sd}^r)} \colon (\GL_r)^d\rtimes \SS_d \to \GL_{r^d}
\]
of \eqref{eq_Nmod_group_hom} is the homomorphism $\Seg'$.
\end{thm}
\begin{proof}
Let $f \colon S^{\sqcup d} \to S$ be the canonical projection. By Example \ref{ex_split_norm}, $N_{\fMod}(S^{\sqcup d}\to S,(\cO|_{S^{\sqcup d}})^r) = \cO^{(r^d)}$ and the universal normic polynomial law is 
\begin{align*}
\bnu \colon f_*((\cO|_{S^{\sqcup d}})^r) = (\cO^r)^d &\to (\cO^r)^{\otimes d} = \cO^{(r^d)} \\
(x_1,\ldots,x_d) &\mapsto x_1\otimes \ldots \otimes x_d.
\end{align*}
Furthermore, $\bnu$ is stable under base change by Corollary \ref{cor_unviersal_base_change}. Therefore, $\bnu|_X$ also has the same universal property as $\bnu$.

Now, let $\varphi = (A_1,\ldots,A_d)\sigma \in ((\GL_r)^d\rtimes \SS_d)(X)$ be a section over some $X\in \Sch_S$. Denote by $f'\colon X^{\sqcup d} \to X$ the standard cover which is the pullback of $f$. Here as well we write as if $X$ is connected, but indices may be added for the general case. The automorphism $\varphi$ acts on $(\cO|_X^r)^d$ by
\[
(x_1,\ldots,x_d) \mapsto (A_1 x_{\sigma^{-1}(1)},\ldots,A_d x_{\sigma^{-1}(d)}).
\]
The composition $\bnu|_X \circ \varphi \colon f_*((\cO|_{X^{\sqcup d}})^r) \to \cO|_X^{(r^d)}$ is also a normic law and it is described by
\[
(x_1,\ldots,x_d) \mapsto A_1 x_{\sigma^{-1}(1)}\otimes\ldots\otimes A_d x_{\sigma^{-1}(d)}.
\]
It is therefore clear that the map $\Seg'(\varphi)$ makes the diagram below commute
\[
\begin{tikzcd}
f'_*((\cO|_{X^{\sqcup d}})^r) \arrow{d}{\varphi} \arrow{r}{\bnu} & \cO|_X^{(r^d)} \arrow{d}{\Seg'(\varphi)} \\
f'_*((\cO|_{X^{\sqcup d}})^r) \arrow{r}{\bnu} & \cO|_X^{(r^d)}
\end{tikzcd}
\]
and therefore, by Corollary \ref{cor_morphism_from_universal}, it is the unique such $\cO|_X$--module isomorphism which does so. This means $N_{X^{\sqcup d}/X}(\varphi)=\Seg'(\varphi)$. Since we have that $N_{\fMod,(\Sd\to S,\cO|_{\Sd}^r)}(X) = N_{X^{\sqcup d}/X}$ on morphisms by definition, we conclude that $N_{\fMod,(\Sd\to S,\cO|_{\Sd}^r)}=\Seg'$ as natural transformations, as desired.
\end{proof}

\begin{cor}\label{cor_norm_twist}
The map on cohomology induced by the Segre homomorphism $\Seg'$ of \eqref{eq_segre} is
\begin{align*}
\widetilde{\Seg'} \colon H^1(S,(\GL_r)^d\rtimes \SS_d) &\to H^1(S,\GL_{r^d})\\
[(T\to S,\cM)] &\mapsto [N_{T/S}(\cM)].
\end{align*}
\end{cor}
\begin{proof}
Since $\fMod_r^{d\etale}$ and $\fMod_{r^d}$ are gerbes and we know by Theorem \ref{thm_norm_group_hom} that $N_{\fMod,(\Sd\to S,\cO|_{\Sd}^r)} = \Seg'$, this follows by applying Lemma \ref{lem_loos_cohom}.
\end{proof}

Under the Segre homomorphism, the center of each $\GL_r$ maps into the center of $\GL_{r^d}$. The center of $\GL_r$ is also the kernel of the canonical projection $\GL_r \to \PGL_r$ and so there exists a group homomorphism $\PSeg'$ which makes the diagram
\begin{equation}\label{eq_seg_prime_und}
\begin{tikzcd}
(\GL_r)^d \rtimes \SS_d \ar[r,"\Seg'"] \ar[d] & \GL_{r^d} \ar[d]\\
(\PGL_r)^d \rtimes \SS_d \ar[r,"\PSeg'"] & \PGL_{r^d}.
\end{tikzcd}
\end{equation}
commute. By viewing an algebra isomorphism in $\PGL_r$ as simply a module isomorphism of a locally free $\cO$--module of rank $r^2$, we get a canonical inclusion $\PGL_r \inj \GL_{r^2}$ which fits into the commutative diagram
\begin{equation}\label{eq_PSeg_GLr2}
\begin{tikzcd}
(\PGL_r)^d \rtimes \SS_d \arrow[hookrightarrow]{r} \arrow{d}{\PSeg'} & (\GL_{r^2})^d \rtimes \SS_d \arrow{d}{\Seg'} \\
\PGL_{r^d} \arrow[hookrightarrow]{r} & \GL_{r^{2d}}.
\end{tikzcd}
\end{equation}

\begin{cor}\label{cor_norm_azu_group_hom}
Let $N_{\fAzu}$ be the morphism of \eqref{eq_N_Mod_Azu}. We know the automorphism group of $(\Sd\to S,\Mat_r(\cO|_{\Sd}))\in \fAzu_{r}^{d\etale}$ is $(\PGL_r)^d\rtimes \SS_d$ by Lemma \ref{lem_Azu_torsor_gerbe}\ref{lem_Azu_torsor_gerbe_ii} and we have that $N_{\fAzu}(\Sd\to S,\Mat_r(\cO|_{\Sd})) \cong \Mat_{(r^d)}(\cO)$. Therefore, we get a group homomorphism
\[
N_{\fAzu,(\Sd\to S,\Mat_r(\cO|_{\Sd}))} \colon (\PGL_r)^d\rtimes \SS_d \to \PGL_{r^d}.
\]
This homomorphism is $\PSeg'$.
\end{cor}
\begin{proof}
By Lemma \ref{lem_norm_etale_neutral_azu} and Example \ref{ex_split_norm}, we have that
\[
N_{\fAzu}\big((S^{\sqcup d}\to S,\Mat_r(\cO|_{S^{\sqcup d}})\big)\cong \Mat_{(r^d)}(\cO).
\]
Let $\varphi \in (\PGL_r)^d\rtimes \SS_d$. Let $\varphi'$ denote this isomorphism viewed as a morphism in $\fMod_{r^2}^{d\etale}$. We then have
\[
N_{\fAzu,(\Sd\to S,\Mat_r(\cO|_{\Sd}))}(\varphi) = N_{\fMod,(\Sd\to S,\Mat_r(\cO|_{\Sd}))}(\varphi') = \Seg'(\varphi').
\]
where the second equality is given by Theorem \ref{thm_norm_group_hom}. However, due to diagram \eqref{eq_PSeg_GLr2}, this is simply $\PSeg'(\varphi)$ viewed as a module morphism. Thus, $N_{\fAzu,(\Sd\to S,\Mat_r(\cO|_{\Sd}))}= \PSeg'$ as claimed.
\end{proof}

\begin{cor}\label{cor_norm_azu_cohomology}
The map on cohomology induced by the morphism $\PSeg'$ of \eqref{eq_seg_prime_und} is
\begin{align*}
\widetilde{\PSeg'} \colon H^1(S,(\PGL_r)^d \rtimes \SS_d) &\to H^1(S,\PGL_{r^d}) \\
[(T\to S,\cA)] &\mapsto [N_{T/S}(\cA)].
\end{align*}
\end{cor}
\begin{proof}
This follows from Lemma \ref{lem_loos_cohom} because of the result of Corollary \ref{cor_norm_azu_group_hom}.
\end{proof}

The gerbes $\fAzu_r^{d\etale}$ and $\fAzu_{r^d}$ fit into the commutative diagram of stack morphisms
\[
\begin{tikzcd}
\fAzu_r^{d\etale} \arrow{r} \arrow{d}{N_{\fAzu}} & \fMod_{r^2}^{d\etale} \arrow{d}{N_{\fMod}} \\
\fAzu_{r^d} \arrow{r} & \fMod_{r^{2d}}
\end{tikzcd}
\]
where the horizontal maps are the canonical inclusions (equivalently, forgetful functors). By Corollary \ref{cor_norm_twist} and Corollary \ref{cor_norm_azu_cohomology}, both the above diagram and \eqref{eq_PSeg_GLr2} induce the same diagram on cohomology, namely
\[
\begin{tikzcd}
H^1(S,(\PGL_r)^d \rtimes \SS_d) \arrow{r} \arrow{d}{\widetilde{\PSeg'}} & H^1(S,(\GL_{r^2})^d \rtimes \SS_d) \arrow{d}{\widetilde{\Seg'}} \\
H^1(S,\PGL_{r^d}) \arrow{r} & H^1(S,\GL_{r^{2d}})
\end{tikzcd}
\]
where the horizontal maps send isomorphism classes of algebras to their isomorphism class simply as modules.

\subsection{The Norm and the Brauer Group}
In this section we fix a degree $d$ \'etale cover $f\colon T \to S$ of our base scheme and we describe how the functor $N_{T/S}$ acts on the Brauer classes of Azumaya algebras. We work with Brauer-Grothendieck groups as in \cite{CTS}, which are second cohomology groups. For example, these are denoted $\Br(S)= H^2\fppf(S,\GG_m)$ and $\Br(T) = H^2\fppf(T,\GG_m|_T)$. This is in contrast to \cite[3.6.1.1]{CF} where the notation ``$\Br(S)$" is used for the Brauer-Azumaya group consisting of classes of Azumaya algebras up to Brauer equivalence. These two notions are not isomorphic in general, but they are isomorphic over fields or more broadly in the case covered by Gabber's Theorem, see \cite[4.2.1]{CTS}.

We will show in Proposition \ref{prop_brauer_group}\ref{prop_brauer_group_i} that the norm functor is compatible with the \emph{trace map} $H^2\fppf(T,\GG_m|_T) \to H^2\fppf(S,\GG_m)$ of \cite[IX.5.1.3]{SGA4}. The work in \cite{SGA4} uses \'etale cohomology, but by \cite[2.2.5.15]{CF} or \cite[III.3.9]{M}, this agrees with flat cohomology since $\GG_m$ is smooth. The trace map is defined as follows. First, as noted in \cite[IX.5.1]{SGA4}, since $f\colon T\to S$ is finite \'etale, there is an isomorphism $H^2\fppf(T,\GG_m|_T) \cong H^2\fppf(S,f_*(\GG_m|_T))$. Then, the product map $\mu \colon \GG_m^d \to \GG_m$ can have its domain twisted by the $\SS_d$--torsor $\cIsom(S^{\sqcup d},T)$ as in Lemma \ref{lem_twist_sheaf}, which yields the trace map $\tr \colon f_*(\GG_m|_T) \to \GG_m$ of \cite[IX.5.1.2]{SGA4}. This trace in turn induces the desired trace map between cohomology.

Further, for any group sheaf $\bG$ over $S$, we have a restriction map, $\res \colon \bG \to f_*(\bG|_T)$, which is the diagonal embedding $\bG \inj \bG^d$ twisted by $\cIsom(S^{\sqcup d},T)$. Alternatively, for $X\in \Sch_S$ there are the restriction maps 
\begin{align*}
\bG(X) &\to \bG(T\times_S X)=f_*(\bG|_T)(X) \\
\varphi &\mapsto \varphi|_{T\times_S X}
\end{align*}
which are part of the definition of the sheaf. These homomorphisms assemble into the restriction map $\res \colon \bG \to f_*(\bG|_T)$. Since $\GG_m$ is abelian, we have by \cite[IX.5.1.4]{SGA4} that the composition
\[
\GG_m \xrightarrow{\res} f_*(\GG_m|_T) \xrightarrow{\tr} \GG_m
\]
is the ``multiplication" by $d$ map, i.e., $x \mapsto x^d$ since $\GG_m$ is written multiplicatively. In turn, the composition on cohomology
\begin{equation}\label{eq_comp_d}
H^2(S,\GG_m) \xrightarrow{\res} H^2(S,f_*(\GG_m|_T)) \xrightarrow{\tr} H^2(S,\GG_m)
\end{equation}
is also multiplication by $d$.

We now define two new stacks and compute some of their automorphism sheaves. First, let $\fTMod_r$ be the stack with
\begin{enumerate}[label={\rm(\roman*)}]
\item objects $(X,\cM)$ where $X\in \Sch_S$ and $\cM$ is a locally free $\cO|_{T\times_S X}$--module of constant rank $r$,
\item morphisms $(g,\varphi)\colon (Y,\cM_1)\to (X,\cM_2)$ where $g\colon Y\to X$ is an $S$--scheme morphism and $\varphi \colon \cM_1 \iso \cM_2|_{T\times_S Y}$ is a $\cO|_{T\times_S Y}$--module isomorphism, where $\cM_2$ is restricted along the map $T\times_S Y \to T\times_S X$ which is the pullback of $g$, and
\item structure functor $(X,\cM) \mapsto X$ and $(g,\varphi)\mapsto g$. 
\end{enumerate}
It is clear that $\fTMod_r$ is fibered in groupoids and since two locally free modules of the same rank are locally isomorphic, it is also a gerbe. The fiber $\fTMod_r(S)$ is the groupoid of locally free $\cO|_T$--modules of constant rank $r$. We designate $(S,\cO|_T^r) \in \fTMod_r(S)$ as the split object.
\begin{lem}\label{lem_TMod_aut}
Let $f\colon T \to S$ be a degree $d$ \'etale cover and let $\fTMod_r$ be defined as above. Consider an object $(S,\cM) \in \fTMod_r(S)$.
\begin{enumerate}[label={\rm(\roman*)}]
\item \label{lem_TMod_aut_i} We have that
\[
\cAut(S,\cM) \cong f_*(\GL(\cM)).
\]
\item \label{lem_TMod_aut_ii} In particular,
\[
\cAut(S,\cO|_T^r) \cong f_*(\GL_r|_T).
\]
\end{enumerate}
\end{lem}
\begin{proof}
\noindent\ref{lem_TMod_aut_i}: For a scheme $X\in \Sch_S$, a section $(g,\varphi)\in \cAut(S,\cM)(X)$ is an automorphism of $(X,\cM|_{T\times_S X})$ in $\fTMod_r(X)$. Since it is a morphism in the fiber $\fTMod_r(X)$, it must have $g=\Id_X$ and therefore $\varphi \colon \cM|_{T\times_S X} \iso \cM_{T\times_S X}$ is an automorphism in $\GL(\cM)(T\times_S X)=f_*(\GL(\cM))(X)$. It is clear this yields an isomorphism of groups $\cAut(S,\cM)(X) \cong f_*(\GL(\cM))(X)$ and that these assemble into an automorphism of sheaves as claimed.

\noindent\ref{lem_TMod_aut_ii}: This is immediate from \ref{lem_TMod_aut_i} since $\GL(\cO|_T^r)=\GL_r|_T$.
\end{proof}

Second, we define an analogous stack for Azumaya algebras. Let $\fTAzu_r$ be the substack of $\fTMod_{r^2}$ with
\begin{enumerate}[label={\rm(\roman*)}]
\item objects $(X,\cA)$ where $X\in \Sch_S$ and $\cA$ is an Azumaya $\cO|_{T\times_S X}$--algebra of constant degree $r$,
\item morphisms $(g,\varphi)\colon (Y,\cA_1)\to (X,\cA_2)$ where $g\colon Y\to X$ is an $S$--scheme morphism and $\varphi \colon \cA_1 \iso \cA_2|_{T\times_S Y}$ is an $\cO|_{T\times_S Y}$--algebra isomorphism.
\end{enumerate}
Since all such Azumaya algebras are locally isomorphic, $\fTAzu_r$ is a gerbe as well. The fiber $\fTAzu_r(S)$ is the groupoid of degree $r$ Azumaya $\cO|_T$--algebras.
\begin{lem}\label{lem_TAzu_aut}
Let $f\colon T \to S$ be a degree $d$ \'etale cover and let $\fTAzu_r$ be defined as above. Consider an object $(S,\cA) \in \fTAzu_r(S)$.
\begin{enumerate}[label={\rm(\roman*)}]
\item \label{lem_TAzu_aut_i} We have that
\[
\cAut(S,\cA) \cong f_*(\PGL(\cA)).
\]
\item \label{lem_TAzu_aut_ii} In particular,
\[
\cAut(S,\Mat_r(\cO|_T)) \cong f_*(\PGL_r|_T).
\]
\end{enumerate}
\end{lem}
\begin{proof}
The proof of Lemma \ref{lem_TMod_aut} may be replicated here, replacing module automorphisms with algebra automorphisms and replacing $\GL$ with $\PGL$.
\end{proof}

Now, we define various stack morphism from which we will later extract a commutative diagram of group sheaves.
\begin{align}
\res \colon \fMod_r &\to \fTMod_r & \res \colon \fAzu_r &\to \fTAzu_r \nonumber \\
(X,\cM) &\mapsto (X,\cM|_{T\times_S X}) & (X,\cA) &\mapsto (X,\cA|_{T\times_S X}) \nonumber \\
(g,\varphi) &\mapsto (g,\varphi|_{T\times_S Y}) & (g,\varphi) &\mapsto (g,\varphi|_{T\times_S Y}) \nonumber \\[2pt]
\inc \colon \fTMod_r &\to \fMod_r^{d\etale} & \inc \colon \fTAzu_r &\to \fAzu_r^{d\etale} \nonumber \\
(X,\cM) &\mapsto (T\times_S X \to X,\cM) & (X,\cA) &\to (T\times_S X \to X,\cA) \nonumber \\
(g,\varphi) &\mapsto (g,g',\varphi) & (g,\varphi) &\mapsto (g,g',\varphi) \nonumber \\[2pt]
\cEnd \colon \fMod_r &\to \fAzu_r & \cTEnd \colon \fTMod_r &\to \fTAzu_r \nonumber \\
(X,\cM) &\mapsto (X,\cEnd_{\cO|_X}(\cM)) & (X,\cM) &\mapsto (X,\cEnd_{\cO|_{T\times_S X}}(\cM)) \nonumber \\
(g,\varphi) &\mapsto (g,\cEnd(\varphi)) & (g,\varphi) &\mapsto (g,\cEnd(\varphi)) \nonumber
\end{align}
\begin{align}
\cEnd^{d\etale} \colon \fMod_r^{d\etale} &\to \fAzu_r^{d\etale} \nonumber \\
(X'\to X,\cM) &\mapsto (X'\to X,\cEnd_{\cO|_Y}(\cM)) \nonumber \\
(f',g,\varphi) &\mapsto (f',g,\cEnd(\varphi)). \nonumber
\end{align}
where $g' \colon T\times_S Y \to T\times_S X$ denotes the pullback of $g\colon Y \to X$. We abuse notation by reusing $\res$ and $\inc$ for two different restriction and inclusion maps, however the second instance is the same map but on a substack. For the maps $\res \colon \fMod_r \to \fTMod_r$, if $(g,\varphi) \colon (Y,\cM_1) \to (X,\cM_2)$ is a morphism, then $\varphi \colon \cM_1 \iso \cM_2|_{Y}$ is an isomorphism. Since $(\cM_2|_{T\times_S X})|_{T\times_S Y} = (\cM_2|_Y)|_{T\times_S Y}$, the restricted morphism is
\[
\varphi|_{T\times_S Y} \colon \cM_1|_{T\times_S Y} \iso (\cM_2|_{T\times_S X})|_{T\times_S Y}
\]
as required and similarly for $\res\colon \fAzu_r \to \fTAzu_r$.

If $\varphi \colon \cM_1 \iso \cM_2|_Y$ is an $\cO|_Y$--module isomorphism (or similarly, over $\cO|_{T\times_S Y}$ or $\cO|_{Y'}$ as would be the case for $\cTEnd$ or $\cEnd^{d\etale}$), then $\cEnd(\varphi)$ is the algebra automorphism
\begin{align*}
\cEnd(\varphi) \colon \cEnd_{\cO|_{Y}}(\cM_1) &\iso \cEnd_{\cO|_{Y}}(\cM_2|_Y)=\cEnd_{\cO|_X}(\cM_2)|_Y \\
\alpha &\mapsto \varphi \circ \alpha \circ \varphi^{-1}.
\end{align*}

These morphisms fit into a commutative diagram
\begin{equation}\label{eq_functors_comm_i}
\begin{tikzcd}
\fMod_r \arrow{d}{\res} \arrow{r}{\cEnd} & \fAzu_r \arrow{d}{\res} \\
\fTMod_r \arrow{d}{\inc} \arrow{r}{\cTEnd} & \fTAzu_r \arrow{d}{\inc} \\
\fMod_r^{d\etale} \arrow{r}{\cEnd^{d\etale}} & \fAzu_r^{d\etale}.
\end{tikzcd}
\end{equation}

\begin{lem}\label{lem_functor_group_homs}
Tracing the image of $(S,\cO^r) \in \fMod_r$ through the diagram \eqref{eq_functors_comm_i} we obtain
\[
\begin{tikzcd}
(S,\cO^r) \arrow[mapsto]{r} \arrow[mapsto]{d} & (S,\cEnd_{\cO}(\cO^r)) \arrow[mapsto]{d} \\
(S,\cO|_T^r) \arrow[mapsto]{r} \arrow[mapsto]{d} & (S,\cEnd_{\cO}(\cO^r)|_T) = (S,\cEnd_{\cO|_T}(\cO|_T^r)) \arrow[mapsto]{d} \\
(T\to S,\cO|_T^r) \arrow[mapsto]{r} & (T\to S,\cEnd_{\cO|_T}(\cO|_T^r)).
\end{tikzcd}
\]
The corresponding induced homomorphisms between automorphism sheaves are given by the following diagram.
\[
\begin{tikzcd}
\GL_r \arrow{r}{\pi} \arrow{d}{\res} & \PGL_r \arrow{d}{\res} \\
f_*(\GL_r|_T) \arrow{r}{\pi'} \arrow[hookrightarrow]{d} & f_*(\PGL_r|_T) \arrow[hookrightarrow]{d} \\
f_*(\GL_r|_T)\rtimes \cAut_S(T) \arrow{r}{\pi'\times\Id} & f_*(\PGL_r|_T)\rtimes \cAut_S(T) 
\end{tikzcd}
\]
where $\pi$ and $\pi'$ are the canonical projections and the hooked arrows indicate the inclusions.
\end{lem}
\begin{proof}
The objects appearing in the first diagram have the corresponding automorphism sheaves in the second diagram either by definition or by Lemmas \ref{lem_mod_torsor_gerbe}\ref{lem_mod_torsor_gerbe_i}, \ref{lem_Azu_torsor_gerbe}\ref{lem_Azu_torsor_gerbe_i}, \ref{lem_TMod_aut}\ref{lem_TMod_aut_ii}, or  \ref{lem_TAzu_aut}\ref{lem_TAzu_aut_ii}.

Since the restriction maps send a morphism $\varphi$ over $X\in \Sch_S$ to $\varphi|_{T\times_S X}$, it is clear they induce the restriction map between groups.

The claim that the horizontal maps are the canonical projections is clear since, by definition, $\cEnd$, $\cTEnd$, and $\cEnd^{d\etale}$ send a module automorphism to its corresponding inner automorphism of the endomorphism algebra and additionally $\cEnd^{d\etale}$ acts as the identity on the scheme part of morphisms, i.e., it preserves $f'$ and $g$ in $(f',g,\varphi)$.

The fact that the hooked arrows are the inclusion is immediate since the inclusion maps sends morphisms of the form $(\Id_X,\varphi)$ to $(\Id_X,\Id_{T\times_S X},\varphi)$.
\end{proof}

Next, we extend the commutative diagram \eqref{eq_functors_comm_i} by appending the functors $N_{\fMod}$ and $N_{\fAzu}$ of \eqref{eq_N_Mod_Azu} on the bottom. This produces
\begin{equation}\label{eq_functors_comm_ii}
\begin{tikzcd}
\fMod_r \arrow{d}{\res} \arrow{r}{\cEnd} & \fAzu_r \arrow{d}{\res} \\
\fTMod_r \arrow{d}{\inc} \arrow{r}{\cTEnd} & \fTAzu_r \arrow{d}{\inc} \\
\fMod_r^{d\etale} \arrow{r}{\cEnd^{d\etale}} \arrow{d}{N_{\fMod}} & \fAzu_r^{d\etale} \arrow{d}{N_{\fAzu}}\\
\fMod_{r^d} \arrow{r}{\cEnd} & \fAzu_{r^d}.
\end{tikzcd}
\end{equation}
where the bottom square only commutes up to canonical isomorphism. In particular, for each $(X'\to X,\cM) \in \fMod_r^{d\etale}$ we have
\begin{align*}
(N_{\fAzu}\circ \cEnd^{d\etale})(X'\to X,\cM) &= \big(X,N_{X'/X}(\cEnd_{\cO|_{X'}}(\cM))\big)\\
(\cEnd\circ N_{\fMod})(X'\to X,\cM) &= \big(X,\cEnd_{\cO|_X}(N_{X'/X}(\cM))\big)
\end{align*}
and by Lemma \ref{lem_norm_etale_neutral_azu} there is a canonical isomorphism of $\cO|_X$--algebras
\[
\Psi_{(X'\to X,\cM)}\colon N_{X'/X}(\cEnd_{\cO|_{X'}}(\cM)) \iso \cEnd_{\cO|_X}(N_{X'/X}(\cM)).
\]

Tracing the object $(T\to S,\cO|_T^r)$ through the bottom square (and through its canonical isomorphism) produces
\[
\begin{tikzcd}
(T\to S,\cO|_T^r) \arrow[mapsto]{r} \arrow[mapsto]{d} & (T\to S,\cEnd_{\cO|_T}(\cO|_T^r)) \arrow[mapsto]{d} \\
(S,N_{T/S}(\cO|_T^r)) \arrow[mapsto]{r} & (S,\cEnd_{\cO}(N_{T/S}(\cO|_T^r)))
\end{tikzcd}
\]
with corresponding group sheaf homomorphisms
\begin{equation}\label{eq_Seg_twist}
\begin{tikzcd}
f_*(\GL_r|_T)\rtimes \cAut_S(T) \arrow{r}{\pi'\times \Id} \arrow{d}{\phi} & f_*(\PGL_r|_T)\rtimes \cAut_S(T) \arrow{d}{\phi'} \\
\GL(N_{T/S}(\cO|_T^r)) \arrow{r}{\pi''} & \PGL(N_{T/S}(\cO|_T^r))
\end{tikzcd}
\end{equation}
where $\pi''$ is the canonical projection. In fact, $\phi$ and $\phi'$ are twists of the modified Segre embeddings $\Seg'$ and $\PSeg'$ respectively. This can be seen by taking a sufficiently fine cover which splits $T$ and then applying Theorem \ref{thm_norm_group_hom} or Corollary \ref{cor_norm_azu_group_hom} respectively.

Combining the diagrams of Lemma \ref{lem_functor_group_homs} and \eqref{eq_Seg_twist} and extending the rows into their canonical short exact sequences, we obtain
\[
\begin{tikzcd}[column sep = 0.15in]
1 \arrow{r} & \GG_m \arrow{r} \arrow{d}{\res} & \GL_r \arrow{r}{\pi} \arrow{d}{\res} & \PGL_r \arrow{d}{\res} \arrow{r} & 1 \\
1 \arrow{r} & f_*(\GG_m|_T) \arrow{r} \arrow[equals]{d} & f_*(\GL_r|_T) \arrow{r}{\pi'} \arrow[hookrightarrow]{d} & f_*(\PGL_r|_T) \arrow[hookrightarrow]{d} \arrow{r} & 1 \\
1 \arrow{r} & f_*(\GG_m|_T) \arrow{r} \arrow[dashed]{d} & f_*(\GL_r|_T)\rtimes \cAut_S(T) \arrow{r}{\pi'\times\Id} \arrow{d}{\phi} & f_*(\PGL_r|_T)\rtimes \cAut_S(T) \arrow{d}{\phi'} \arrow{r} & 1 \\
1 \arrow{r} & \GG_m \arrow{r} & \GL(N_{T/S}(\cO|_T^r)) \arrow{r}{\pi''} & \PGL(N_{T/S}(\cO|_T^r)) \arrow{r} & 1.
\end{tikzcd}
\]
\begin{lem}
The dashed morphism in the diagram above is the trace homomorphism, $\tr \colon f_*(\GG_m|_T) \to \GG_m$.
\end{lem}
\begin{proof}
If we instead consider the following diagram involving the Segre homomorphism,
\[
\begin{tikzcd}
\GG_m^d \arrow{r} \arrow{d}{\tr} & (\GL_r)^d \rtimes \SS_d \arrow{d}{\Seg'} \\
\GG_m \arrow{r} & \GL_{(r^d)}
\end{tikzcd}
\]
where the trace morphism coincides with the multiplication map, it is clear this commutes since
\[
\Seg(c_1I,\ldots,c_dI) = c_1I\otimes \ldots \otimes c_dI = (c_1\dots c_d)I.
\]
This describes the dashed morphism we are interested in locally and by points (i) and (ii) after \cite[IX.5.1.3]{SGA4}, this characterizes the trace map. Therefore, the dashed morphism is $\tr \colon f_*(\GG_m|_T) \to \GG_m$ as claimed.
\end{proof}

At this point, the third row of the large diagram above is no longer needed and we consider the compressed diagram
\[
\begin{tikzcd}[column sep = 0.15in]
1 \arrow{r} & \GG_m \arrow{r} \arrow{d}{\res} & \GL_r \arrow{r}{\pi} \arrow{d}{\res} & \PGL_r \arrow{d}{\res} \arrow{r} & 1 \\
1 \arrow{r} & f_*(\GG_m|_T) \arrow{r} \arrow{d}{\tr} & f_*(\GL_r|_T) \arrow{r}{\pi'} \arrow{d}{\rho} & f_*(\PGL_r|_T) \arrow{d}{\rho'} \arrow{r} & 1 \\
1 \arrow{r} & \GG_m \arrow{r} & \GL(N_{T/S}(\cO|_T^r)) \arrow{r}{\pi''} & \PGL(N_{T/S}(\cO|_T^r)) \arrow{r} & 1
\end{tikzcd}
\]
where $\rho=\phi \circ \inc$ and $\rho'=\phi'\circ \inc$. Finally, we may take the associated diagram of long exact cohomology sequences to obtain the following result.
\begin{prop}\label{prop_brauer_group}
Let $T \to S$ be a degree $d$ \'etale cover. Let $\cB$ be an Azumaya $\cO|_T$--algebra of constant degree and $\cA$ be an Azumaya $\cO$--algebra of constant degree.
\begin{enumerate}[label={\rm(\roman*)}]
\item \label{prop_brauer_group_i} $[N_{T/S}(\cB)] = \tr([\cB]) \in \Br(S)$.
\item \label{prop_brauer_group_ii} $[N_{T/S}(\cA|_T)] = d[\cA] \in \Br(S)$.
\end{enumerate}
\end{prop}
\begin{proof}
The diagram of long exact cohomology contains the following,
\[
\begin{tikzcd}
H^1(S,\PGL_r) \arrow{r} \arrow{d}{\res} & H^2\fppf(S,\GG_m) = \Br(S) \arrow{d}{\res} \\
H^1(S,f_*(\PGL_r|_T)) \arrow{d} \arrow{r} & H^2\fppf(S,f_*(\GG_m|_T)) = \Br(T) \arrow{d}{\tr} \\
H^1(S,\PGL(N_{T/S}(\cO|_T^r))) \arrow{r} & H^2\fppf(S,\GG_m) = \Br(S)
\end{tikzcd}
\]
where the horizontal maps are the natural boundary morphisms taking an isomorphism class of an algebra to its Brauer class.

By Lemma \ref{lem_loos_cohom} and Lemma \ref{lem_functor_group_homs}, the maps on first cohomology induced by the group homomorphisms
\[
\PGL_r \xrightarrow{\res} f_*(\PGL_r|_T) \hookrightarrow f_*(\PGL_r|_T)\rtimes \cAut_S(T) \xrightarrow{\phi'} \PGL(N_{T/S}(\cO|_T^r))
\]
are the same as the maps induced by the functors
\[
\fAzu_r \xrightarrow{\res} \fTAzu_r \xrightarrow{\inc} \fAzu_r^{d\etale} \xrightarrow{N_{\fAzu}} \fAzu_{r^d}.
\]
Therefore, tracing the image of the isomorphism class $[\cB] \in H^1(T,\PGL_r|_T)$ $= H^1(S,f_*(\PGL_r|_T))$, we obtain
\[
\begin{tikzcd}
{[\cB]} \arrow[mapsto]{r} \arrow[mapsto]{d} & {[\cB]} \in \Br(T) \arrow[mapsto]{d} \\
{[N_{T/S}(\cB)]} \arrow[mapsto]{r} & {[N_{T/S}(\cB)]} = \tr([\cB]) \in \Br(S)
\end{tikzcd}
\]
justifying claim \ref{prop_brauer_group_i}. Similarily, for $[\cA] \in H^1(S,\PGL_r)$, we can chase it through the diagram as follows
\[
\begin{tikzcd}
{[\cA]} \arrow[mapsto]{r} \arrow[mapsto]{d} & {[\cA]}\in \Br(S) \arrow[mapsto]{d} \\
{[\cA|_T]} \arrow[mapsto]{d} \arrow[mapsto]{r} & {[\cA|_T]} \in \Br(T) \arrow[mapsto]{d} \\
{[N_{T/S}(\cA|_T)]} \arrow[mapsto]{r} & {[N_{T/S}(\cA|_T)]=d[\cA]} \in \Br(S)
\end{tikzcd}
\]
where the factor of $d$ appears by \eqref{eq_comp_d}. This justifies claim \ref{prop_brauer_group_ii}.
\end{proof}

\section{An Equivalence $A_1^2 \equiv D_2$}\label{Equivalence}
In this section we show that the norm functor restricts to a functor from objects of type $C$ to objects of type $D$, i.e., from certain Azumaya algebras with symplectic involutions to quadratic triples. In particular, we consider the following two gerbes. First, let $\fC_m^{d\etale}$ be the gerbe with
\begin{enumerate}[label={\rm(\roman*)}]
\item objects which are pairs $(T' \to T, (\cA,\sigma))$ where $T'\to T \in \Sch_S$ is a degree $d$ \'etale cover of schemes and $(\cA,\sigma)$ is an Azumaya $\cO|_{T'}$--algebra of degree $2m$ with a symplectic involution, and with
\item morphisms which are triples $(f,g,\varphi)\colon (X'\to X,(\cA_1,\sigma_1)) \to (T'\to T,(\cA_2,\sigma_2))$ where $f,g$ are scheme morphisms making
\[
\begin{tikzcd}
X' \ar[r,"g"] \ar[d] & T' \ar[d] \\
X \ar[r,"f"] & T
\end{tikzcd}
\]
commute and where $\varphi \colon (\cA_1,\sigma_1) \iso g^*(\cA_2,\sigma_2)$ is an isomorphism of Azumaya $\cO|_{X'}$--algebras with symplectic involution,
\item and whose structure functor $p\colon \fC_m^{d\etale} \to \Sch_S$ is $(T'\to T,(\cA,\sigma)) \mapsto T$ and $(g,f,\varphi) \mapsto f$.
\end{enumerate}
By Remark \ref{rem_torsors_equiv_auto}, this is equivalent to the gerbe $\fF(\PSP_{2m})^{d\etale}$ of \eqref{AppenB_gerbe_tors_equiv}, which is equivalent to the stack of $(\PSP_{2m}^d)\rtimes \SS_d$--torsors. We use the notation $\fC_m^{d\etale}$, with a subscript $m$, since for an Azumaya algebra of degree $2m$ with a symplectic involution $(\cA,\sigma)$, the group $\PSP_{\cA,\sigma}$ is a group of type $C_m$.

For the second gerbe, we define $\fD_m$ to have
\begin{enumerate}[label={\rm(\roman*)}]
\item objects which are pairs $(X,(\cA,\si,f))$ where $X\in \Sch_S$ and $(\cA,\si,f)$ is a quadratic triple with $\cA$ a degree $2m$ Azumaya $\cO|_X$--algebra,
\item morphisms which are pairs $(g,\varphi)\colon (Y,(\cA_1,\sigma_1,f_1)) \to (X,(\cA_2,\sigma_2,f_2))$ where $g\colon Y\to X$ is an $S$--scheme morphism and the map $\varphi \colon (\cA_1,\si_1,f_1)\to g^*(\cA_2,\si_2,f_2)$ is an isomorphism of quadratic triples over $Y$, and
\item structure functor which sends $(X,(\cA,\sigma,f))\mapsto X$ and $(g,\varphi)\mapsto g$.
\end{enumerate}
The stack $\fD_m$ is also a gerbe since we know by \cite[4.6]{GNR} that all quadratic triples are isomorphic \'etale locally and thus also fppf locally. Alternatively, since it contains the object $(S,(\Mat_{2m}(\cO),\sigma_{2m},f_{2m}))$ whose automorphism group is $\bPGO_{2m}$, it is equivalent to the stack of $\bPGO_{2m}$--torsors by Proposition \ref{lem_gerbes_and_torsors}\ref{lem_gerbes_and_torsors_iv} and Remark \ref{rem_gerbe_forms}. Here as well the subscript $m$ is used since the groups $\PGO_{\cA,\sigma,f}$ are of type $D_m$ when $\cA$ is degree $2m$. In \cite[2.7.0.30]{CF}, the stack $\fD_m$ is denoted as $\cP\!\mathit{aires}\cQ\mathit{uad}_{2m}$.

We show that when $d$ is even, the norm functor $N\alg$ of \ref{eq_normalg_stack_morphism} restricts to a functor
\[
\Psi' \colon \fC_m^{d\etale} \to \fD_{2^{d-1}m^d},
\]
defined precisely in \eqref{eq_stack_Psi_prime} below.

In the case when $m=1$, the objects of $\fC_1^{d\etale}$ involve quaternion algebras with symplectic involution. However, every quaternion algebra $\cA$ has a unique (up to isomorphism) symplectic involution defined by $\sigma_\cA(a) = \Trd_\cA(a)\cdot 1 - a$ and this provides the classical equivalence of categories $A_1 \equiv C_1$ between quaterion algebras on their own and quaterion algebras with symplectic involutions. Define $\fA_m^{d\etale}$ to be the gerbe with
\begin{enumerate}[label={\rm(\roman*)}]
\item objects $(T' \to T,\cA)$ where $T'\to T \in \Sch_S$ is a degree $d$ \'etale cover and $\cA$ is an Azumaya $\cO|_{T'}$--algebra of degree $m+1$,
\end{enumerate}
and with morphisms and structure functor mirroring the definitions for $\fC_m^{d\etale}$. This is equivalent to the gerbe $\fF(\PGL_{m+1})^{d\etale}$ of \eqref{AppenB_gerbe_tors_equiv}. We then have an equivalence of gerbes $\fA_1^{d\etale} \equiv \fC_1^{d\etale}$.

As our main result, we show that when $m=1$ and $d=2$, the norm functor induces an equivalence of gerbes $\fC_1^{2\etale} \iso \fD_2$, which of course then also gives us an equivalence
\[
N\colon \fA_1^{2\etale} \iso \fD_2.
\]
This equivalence of gerbes is the analogue of the equivalence of groupoids given in \cite[15.B]{KMRT}. We also show how the Clifford algebra construction provides a quasi-inverse to the Norm functor on the level of fibers. We define a morphism of stacks
\[
\Cl \colon \fD_2 \to \fA_1^{2\etale}
\]
based on sending a quadratic triple to its Clifford algebra such that for all $T\in \Sch_S$, the functors $N_T \colon \fA_1^{2\etale}(T) \to \fD_2(T)$ and $\Cl_T \colon \fD_2(T) \to \fA_1^{2\etale}$ are quasi-inverse equivalences. This generalizes results from \cite[\S 15.B]{KMRT}.

\subsection{A Quadratic Triple over $\ZZ$}\label{triple_over_Z}
As preparation, we begin by constructing a quadratic triple over the integers from a tensor product of symplectic involutions. Let $n_1,\ldots,n_d$ be an even number of positive integers (so $d$ is even), and let $n$ be the integer such that $2n=(2 n_1) \dots (2 n_d)$.
We then have an isomorphism of $\ZZ$--algebras
\[
\Mat_{2 n_1}(\ZZ) \otimes_\ZZ \dots \otimes_\ZZ
\Mat_{2 n_d}(\ZZ) \, \simlgr \, 
\Mat_{2 n}(\ZZ)
\]
given by the tensor product of matrices. On each $\Mat_{2n_i}(\ZZ)$, we consider 
the standard symplectic involution 
$\sigma_{n_i}$ 
defined by 
\[
\sigma_{n_i}(B)= J_{n_i}^{-1} B^T J_{n_i}
= - J_{n_i} B^T J_{n_i}
\text{ with } J_{n_i}=\begin{bmatrix} 0 & -I_{n_i} \\ I_{n_i} & 0\end{bmatrix}.
\]
The involution $\sigma_{n_i}$ is adjoint to the skew-symmetric bilinear form $\psi_{n_i}$ on $\ZZ^{2 n_i}$
defined by $\psi_{n_i}(v,v') = v^T J_{n_i} v'$, considering $v$ and $v'$ as columns. The tensor product $\sigma= \sigma_{n_1} \otimes \dots \otimes
\sigma_{n_d}$ of these involutions is then an orthogonal involution on $\Mat_{2n}(\ZZ)$. Precisely, it is adjoint to the regular symmetric
bilinear form $b$ on $\ZZ^{2n}={ \ZZ^{2n_1} \otimes_\ZZ \dots \otimes_\ZZ \ZZ^{2n_d}}$
defined by
\[
b( v_1 \otimes \dots \otimes v_d, v_1' \otimes \dots \otimes v'_d)= \psi_{n_1}(v_1,v'_1)\dots \psi_{n_d}(v_d,v'_d).
\]
Since each $\psi_{n_i}$ is symplectic, if $v=v_1\otimes\dots \otimes v_d$ is a pure tensor, then $b(v,v)=0$. Writing a general vector $w=w_1+ \dots + w_k$ as a sum of pure tensors, this means that
\[
b(w,w)=\sum_{i=1}^k b(w_i,w_i) + \sum_{i,j=1 \atop i<j}^k (b(w_i,w_j) + b(w_j,w_i)) = \sum_{i,j=1 \atop i<j}^k 2b(w_i,w_j),
\]
i.e., $b(w,w) \in 2\ZZ$. Therefore, we may define a quadratic $\ZZ$--form
\[
q_\tens(x)=\frac{1}{2}b(x,x)
\]
whose polar will be $b$. If we wish to write this without the use of $\frac{1}{2}$, for example after a base change as we do in Section \ref{subsec_twisting_triples}, it appears as
\[
q_\tens\big(\sum_{i=1}^k v_{1,i}\otimes\ldots\otimes v_{d,i}\big) = \sum_{i,j=1 \atop i<j}^k b(v_{1,i}\otimes\ldots\otimes v_{d,i},v_{1,j}\otimes\ldots\otimes v_{d,j}).
\]
The form $q_\tens$ is regular since $b$ is and therefore it has an adjoint involution $\sigma_\tens$. It follows from \cite[4.4(i)]{GNR} or \cite[2.7.0.31]{CF} that $\sigma_\tens$ is part of a quadratic pair $(\sigma_\tens,f_\tens)$. Since $\frac{1}{2}\in \QQ^\times$, after extension to $\QQ$ we must have that
\begin{equation}\label{eq_triple_over_Z}
f_\tens(s)= \frac{1}{2} \Tr(s)
\end{equation}
and hence this also holds for each $s \in \Symm( \Mat_{2n}(\ZZ), \sigma_\tens)$. Thus, the quadratic pair $(\sigma_\tens, f_\tens)$ is unique and therefore $q_\tens$ is unique also.

\begin{remark} \label{rem_canonical} {\rm  
The involution $\sigma_\tens$ is isomorphic over $\ZZ$ to the split involution $\eta_0$ of \cite[4.5 (b)]{GNR} and so for uniqueness reasons as in \cite[4.3(b)]{GNR}, the isomorphism is also one of quadratic pairs, i.e., $(\sigma_\tens,f_\tens) \cong (\eta_0,f_0)$. Further, when $d=2$ and $\sigma_\tens=\sigma_{n_1}\otimes \sigma_{n_2}$, there is a quadratic pair $(\sigma_\tens,f_{\otimes})$ on $\Mat_{4n_1n_2}(\ZZ)$ arising from the construction in \cite[5.6]{GNR}. Uniqueness also implies that $f_\tens = f_\otimes$.
}
\end{remark}

\subsection{Restricting the Segre Homomorphism}\label{section_restricting_segre}
We consider the orthogonal groups reviewed in Section \ref{sec_algebraic_groups} with respect to the quadratic form and quadratic triple defined in Section \ref{triple_over_Z} above. In particular, for this subsection we will consider them over the base scheme $\Spec(\ZZ)$. The geometric fibers of $\mathbf{O}_{q_\tens}^+$ are split semisimple groups of type $D_n$ by \cite[25.12]{KMRT} because $n\geq 2$, so $\mathbf{O}_{q_\tens}^+$ is a semisimple $\ZZ$-group scheme of type $D_n$. Since $\mathbf{O}_{q_\tens}^+|_{\QQ}$ is a split semisimple group, $\mathbf{O}_{q_\tens}^+$ is  a Chevalley $\ZZ$--group scheme in view of the uniqueness of integral models, as in \cite[1.4]{Con2}.

We also consider the symplectic groups $\SP_{2n_i} = \SP_{\Mat_{2n_i}(\ZZ),\sigma_{n_i}}$ and $\PSP_{2n_i}$ $= \cAut(\Mat_{2n_i}(\ZZ),\sigma_{n_i})$ as defined in Section \ref{sec_algebraic_groups} associated to the symplectic involutions defined in Section \ref{triple_over_Z}. The group $\SP_{2n_i}$ is isomorphic to the symplectic group of the alternating form $\psi_{n_i}$ and so by \cite[25.11]{KMRT}, the geometric fibers of $\SP_{2n_i}$ are split semisimple and simply connected groups of type $C_{n_i}$. Again, $\SP_{2n_i}|_{\QQ}$ is a split semisimple group, so $\SP_{2n_i}$ is a Chevalley $\ZZ$--group scheme.

\begin{lem}\label{lem_segre}
Consider the Segre homomorphism \eqref{eq_Seg} and its extension $\Seg'$ of \eqref{eq_segre} as well as the orthogonal and symplectic groups as reviewed above.
\begin{enumerate}[label={\rm(\roman*)}]
 \item \label{lem_segre1} 
The mapping  $\Seg \colon \mathbf{GL}_{2 n_1}\times_\ZZ
\dots \times_\ZZ
 \mathbf{GL}_{2n_d} \to \mathbf{GL}_{2n}$ induces 
a closed immersion of $\ZZ$--group schemes  
\[
h\colon \Bigl( \mathbf{Sp}_{2n_1} \times_\ZZ 
\dots \times_\ZZ \mathbf{Sp}_{2n_d} \Bigr)\!/ \! (\bmu_2)^{d,0} \to \mathbf{O}_{q_\tens}^+
\]
where $(\bmu_2)^{d,0}= \ker\bigl( (\bmu_2)^{d} \xrightarrow{\Pi} \bmu_2 \bigr)$ and $\Pi$ is the product map.
\smallskip

 \item \label{lem_segre2} If $n_1= \dots=n_d=m$ (so that $2n=(2m)^d$), then $h$  extends to a closed  
 immersion of $\ZZ$--group schemes
 \[
 \widetilde h\colon  \Bigl( \bigl(\mathbf{Sp}_{2m}\bigr)^d \! / \!(\bmu_2)^{d,0}\Bigr)  \rtimes_\ZZ   \SS_d
 \to \mathbf{O}_{q_\tens}
 \]
 where the permutation groups acts as in \eqref{eq_segre}.
 Furthermore, recalling the Dickson homomorphism from \eqref{eq_Dickson}, the composition map
 \[
 \SS_d \to \mathbf{O}_{q_\tens} 
 \xrightarrow{\textnormal{Dickson}} \ZZ/2\ZZ
 \]
  is the signature homomorphism if $m$ is odd and is trivial
  if $m$ is even.
 
  \item \label{lem_segre3} If 
  $m=1$ and $d=2$, then $\widetilde h$ is an isomorphism.
 
\end{enumerate}

\end{lem}

\begin{proof}  \ref{lem_segre1}: The Segre mapping, $\Seg\colon \mathbf{GL}_{2 n_1}\times_\ZZ
\dots \times_\ZZ
 \mathbf{GL}_{2n_d} \to \mathbf{GL}_{2n}$, induces 
a homomorphism  of $\ZZ$--group schemes  
\[
h'\colon \mathbf{Sp}_{2n_1} \times_\ZZ 
\dots \times_\ZZ \mathbf{Sp}_{2n_d}  \to \mathbf{O}_{q_\tens}.
\]
Since the symplectic groups have connected geometric fibers and $\mathbf{O}_{q_\tens}^+$ is the identity component of $\mathbf{O}_{q_\tens}$,
the map $h'$ factors through $\mathbf{O}_{q_\tens}^+$.
The kernel of $h'$ is the intersection of 
${\mathbf{Sp}_{2n_1} \times_\ZZ 
\dots \times_\ZZ \mathbf{Sp}_{2n_d}}$
with  the kernel of the Segre mapping, which is 
$\ker( \GG_m^d \xrightarrow{\Pi}  \GG_m)$. It follows
that  $\ker(h')=(\bmu_2)^{d,0}$.
According to \cite[VIII.5]{SGA3}, we can quotient out by the 
diagonalizable $\ZZ$--group $(\bmu_2)^{d,0}$
and get a monomorphism
\[
h\colon \Bigl( \mathbf{Sp}_{2n_1} \times_\ZZ 
\dots \times_\ZZ \mathbf{Sp}_{2n_d} \Bigr)\!/ \! (\bmu_2)^{d,0} \to \mathbf{O}_{q_\tens}^+.
\]
This is a closed immersion according to \cite[5.3.5]{Con1}.
\smallskip

\noindent \ref{lem_segre2}: This follows from the fact that 
the construction of $(\sigma_\tens, f_\tens)$ is equivariant
with respect to the action of the symmetric group $\SS_d$.
It  remains to deal with  the composition map
$\SS_d \to \mathbf{O}_{q_\tens} \to \ZZ/2\ZZ$. It is enough to check it over  the 
$\QQ$-points and, in this case, 
the Dickson map $\mathbf{O}_{q_\tens}(\QQ) \to (\ZZ/2\ZZ)(\QQ) \cong \bmu_2(\QQ)$
is nothing but the determinant by \cite[IV.5.1.2]{K}.
To prove our claim, it is then enough to compute the
 image of the transposition $(1\; 2)$   
 by the morphism $j\colon \SS_d \to \SS_{(2m)^d}$. Without loss of generality, we can assume that $d=2$
so that  $2n=(2m)^2$ and 
\[
j\bigl( (1\; 2) \bigr)= \prod_{1 \leq i < j \leq 2m} ((i,j)\; (j,i)),
\]
which is a product of $m(2m-1)$ transpositions.
It follows that we have $\det\Bigl(\Seg'\bigl( (1\; 2) \bigr) \Bigr)
=1$ if and only $m$ is even, which justifies the claim.
\smallskip

\noindent \ref{lem_segre3}: Since $(\mathbf{Sp}_{2})^2 \! / \!\bmu_2$ is smooth according to \cite[VI$_B$.9.2]{SGA3}, it is both flat and locally of finite presentation. Since $\bO_{q_\tens}^+$ is also smooth, the fiberwise isomorphism criterion of \cite[IV$_4$.17.9.5]{EGA}
allows us to reduce to the case of an algebraically closed field $k$.
The map  $(\mathbf{Sp}_{2,k})^2 /\bmu_2 \to 
(\mathbf{O}_{q_\tens}^+)_k$ 
 is a closed embedding between two smooth connected 
 algebraic groups of the same dimension (i.e., $6$), which
 is an isomorphism according to \cite[cor. 5.8]{GW}.
 Using  that the composition  $\ZZ/2\ZZ \to (\mathbf{O}_{q_\tens})_k \to \ZZ/2\ZZ$ is an isomorphism, we conclude that $\widetilde h_k$ is an isomorphism.
\end{proof}

\begin{remark} Another way to see that the map $\widetilde h$ in Lemma \ref{lem_segre}\ref{lem_segre3} is an isomorphism is 
to use the map $f\colon \mathbf{SL}_2 \times_\ZZ \mathbf{SL}_2 \to \mathbf{O}_{\det}^+$
defined by $f(B_1,B_2)(A)= B_1 A B_2^{-1}$. Here $\det \colon \Mat_2(\ZZ) \to \ZZ$ is a quadratic form where we view $\Mat_2(\ZZ)$ simply as a rank $4$ $\ZZ$--module. See \cite[C.6.3]{Con1}.
\end{remark}

To get maps of adjoint groups, we can quotient out the maps of Lemma \ref{lem_segre} by the center $\bmu_2$
of each $\mathbf{Sp}_{2n_i}$ on the left and the center $\bmu_2$ of $\mathbf{O}_{q_\tens}^+$ on the right. This yields a closed immersion
\begin{equation}\label{eq_h_underline}
\underline{h}\colon  \mathbf{PSp}_{2n_1} \times_\ZZ 
\dots \times_\ZZ \mathbf{PSp}_{2n_d}  \to \mathbf{PGO}_{q_\tens}^+.
\end{equation}
Of course, we may instead construct $\underline{h}$ from the restricted Segre homomorphism $h'\colon \mathbf{Sp}_{2n_1} \times_\ZZ \dots \times_\ZZ \mathbf{Sp}_{2n_d}  \to \mathbf{O}_{q_\tens}^+$ directly by quotienting out the centers on both sides at once, and this makes it clear that $\underline{h}$ will simply be a restriction of the projective Segre homomorphism
\begin{align*}
\PSeg \colon \PGL_{2n_1}\times_\ZZ \ldots\times_\ZZ \PGL_{2n_d} &\to \PGL_{2n} \\
(\varphi_1,\ldots,\varphi_d) &\mapsto \varphi_1\otimes\ldots\otimes\varphi_d.
\end{align*}
In the second case, we get a closed immersion
\begin{equation}\label{eq_h_tilde_underline}
\underline{h}'\colon (\mathbf{PSp}_{2m})^d
  \rtimes_\ZZ  \SS_d
 \to \mathbf{PGO}_{q_\tens},
\end{equation}
which is the restriction of the map $\PSeg'$ of \eqref{eq_seg_prime_und}. In particular, if $m=1$ and $d=2$, the group homomorphism $\underline{h}'$ is an isomorphism since the map $\tilde{h}$ of Lemma \ref{lem_segre}\ref{lem_segre3} is an isomorphism in this case.

\subsection{Twisting Quadratic Triples}\label{subsec_twisting_triples}
We now shift our attention from working over $\Spec(\ZZ)$ to working over our base scheme $S$. By abuse of notation, when we refer to objects and morphisms constructed in Sections \ref{triple_over_Z} and \ref{section_restricting_segre} we mean to consider their base change to $S$.

We will define a morphism of stacks in which the morphism $\underline{h}$ of \eqref{eq_h_underline} appears as an induced homomorphism between automorphism sheaves, see Lemma \ref{lem_twist_un}\ref{lem_twist_un_2}. Let the numbers $d$ and $n_1,\ldots,n_d$ be as in the beginning of Section \ref{triple_over_Z}. The first stack, which we denote by $\fC_{(n_1,\ldots,n_d)}$, will be the product of the $d$ stacks respectively consisting of Azumaya algebras of degree $n_i$ with symplectic involution. Concretely, this product stack has
\begin{enumerate}[label={\rm(\roman*)}]
\item objects $(X,(\cA_1,\sigma_1),\ldots,(\cA_d,\sigma_d))$ where $X\in \Sch_S$ and $(\cA_i,\sigma_i)$ is an Azumaya $\cO|_X$--algebra of degree $2n_i$ with symplectic involution,
\item morphisms
\[
(g,\varphi_1,\ldots,\varphi_d) \colon (Y,(\cA_{1,i},\sigma_{1,i})_{i=1}^d) \to (X,(\cA_{2,i},\sigma_{2,i})_{i=1}^d)
\]
where $g\colon Y\to X$ is an $S$--scheme morphism and $\varphi_i \colon (\cA_{1,i},\sigma_{1,i})\iso (\cA_{2,i},\sigma_{2,i})|_Y$ is an isomorphism of Azumaya $\cO|_Y$--algebras with involution, and
\item structure functor which sends $(X,(\cA_1,\sigma_1),\ldots,(\cA_d,\sigma_d))\mapsto X$ and $(g,\varphi_1,\ldots,\varphi_d)\mapsto g$.
\end{enumerate}
Since each of the factors of the product stack is a gerbe, $\fC_{(n_1,\ldots,n_d)}$ is also a gerbe. We take the split object to be
\[
\cM=(S,(\Mat_{2n_1}(\cO),\sigma_{n_1}),\ldots,(\Mat_{2n_d}(\cO),\sigma_{n_d}))
\]
and so $\cAut(\cM) = \PSP_{2n_1}\times \ldots \times \PSP_{2n_d}$. Hence, $\fC_{(n_1,\ldots,n_d)}$ is equivalent to the stack of $(\PSP_{2n_1}\times \ldots \times \PSP_{2n_d})$--torsors by Lemma \ref{lem_gerbes_and_torsors}\ref{lem_gerbes_and_torsors_iv} and Remark \ref{rem_gerbe_forms}.

The second stack is $\fD_n$, for $n\geq 2$ as defined at the start of Section \ref{Equivalence}. We take $(S,(\Mat_{2n}(\cO),\sigma_\tens,f_\tens))$ to be the split object where $(\Mat_{2n}(\cO),\sigma_\tens,f_\tens)$ is the quadratic triple defined in \eqref{eq_triple_over_Z} base changed to $S$. Then,
\[
\cAut(S,(\Mat_{2n}(\cO),\sigma_\tens,f_\tens)) = \bPGO_{q_\tens}.
\]
We are aiming to define a morphism of stacks $\fC_{(n_1,\ldots,n_d)} \to \fD_n$ which sends the $d$--tuple of Azumaya algebras with symplectic involutions to their tensor product, see \eqref{eq_stack_Psi}. In order to do so, we must define a choice of $f$ so that we have a quadratic triple on the resulting Azumaya algebra. We do so with the following straightforward generalization of \cite[5.6]{GNR}, which first requires a generalization of \cite[5.1]{GNR}. 

We remind ourselves of the definitions of the symmetric, skew-symmetric, and symmetrized elements from Section \ref{sec_quad_pairs}. Given Azumaya $\cO$--algebras with involution $(\cA_1,\si_1), \ldots, (\cA_d,\si_d)$ and setting $(\cA, \si) = ( \cA_1 \ot_\cO \cdots \ot_\cO A_d, \si_1 \ot \cdots \ot \si_d)$, we want to talk about
\begin{align*}
&\cSym_{\cA_1,\si_1}\otimes_{\cO} \cdots \ot_{\cO} \cSym_{\cA_d,\si_d}, \text{ or}\\
&\cSkew_{\cA_1,\si_1}\otimes_{\cO} \cdots \ot_{\cO} \cSkew_{\cA_d,\si_d}
\end{align*}
as submodules of $(\cA,\sigma)$. This is justified in the following lemma.
\begin{lem}\label{tensor_submodules}
Let $(\cA_1,\si_1), \ldots, (\cA_d,\si_d)$ be Azumaya $\cO$--algebras with involution. Set $(\cA, \si) = ( \cA_1 \ot_\cO \cdots \ot_\cO A_d, \si_1 \ot \cdots \ot \si_d)$.
\begin{enumerate}[label={\rm (\roman*)}]
\item\label{tensor_submodules_i} If all $\sigma_i$ are orthogonal, then the canonical map
\[
\cSym_{\cA_1,\si_1}\otimes_{\cO} \cdots \ot_{\cO} \cSym_{\cA_d,\si_d} \to \cA
\]
is injective.
\item\label{tensor_submodules_ii} If all $\sigma_i$ are weakly symplectic, then the canonical map
\[
\cSkew_{\cA_1,\si_1}\otimes_{\cO} \cdots \ot_{\cO} \cSkew_{\cA_d,\si_d} \to \cA
\]
is injective.
\end{enumerate}
\end{lem}
\begin{proof}
\noindent\ref{tensor_submodules_i}: Each $\cA_i$ is an Azumaya algebra, hence finite locally free, hence flat. Similarly, by \cite[3.3(ii)]{GNR} since $\sigma_i$ is orthogonal, each $\cSym_{\cA_i,\sigma_i}$ is finite locally free and hence flat. It is clear that our map is injective for one tensor factor, so arguing inductively we have that
\begin{align*}
&\cSym_{\cA_1,\sigma_1}\otimes_\cO \ldots\otimes_\cO \cSym_{\cA_{d-1},\sigma_{d-1}} \otimes_\cO \cSym_{\cA_d,\sigma_d} \\
\inj &\cA_1\otimes_\cO \ldots\otimes_\cO \cA_{d-1} \otimes_\cO \cSym_{\cA_d,\sigma_d} \\
\inj &\cA_1\otimes_\cO \ldots\otimes_\cO \cA_{d-1} \otimes_\cO \cA_d = \cA
\end{align*}
is injective, using flatness of $\cSym_{\cA_d,\sigma_d}$ in the first step and flatness of $\cA_1\otimes_\cO \ldots\otimes_\cO \cA_{d-1}$ in the second. This proves the claim.
\noindent\ref{tensor_submodules_ii}: The argument in \ref{tensor_submodules_i} can be repeated replacing $\cSym$ with $\cSkew$ since we know that when each $\sigma_i$ is weakly symplectic, then each $\cSkew_{\cA_i,\sigma_i}$ is finite locally free and hence flat by \cite[3.4(ii)]{GNR}.
\end{proof}
We will implicitly use the result of Lemma \ref{tensor_submodules} in the statements of results below.

\begin{lem}\label{lem_tensor_Symdecomp}
Let $(\cA_1,\si_1), \ldots, (\cA_d,\si_d)$ be Azumaya $\cO$--algebras with involution. Set $(\cA, \si) = ( \cA_1 \ot_\cO \cdots \ot_\cO A_d, \si_1 \ot \cdots \ot \si_d)$. 

\begin{enumerate}[label={\rm (\roman*)}]
\item \label{lem_tensor_Symdecomp_i} 
If $\si_i$ are all orthogonal, then
\[ \cSym_{\cA,\si} = \big(\cSymd_{\cA,\si} + (\cSym_{\cA_1,\si_1}\otimes_{\cO} 
\cdots \ot_{\cO} \cSym_{\cA_d,\si_d})\big)^\sharp, \]
\item \label{lem_tensor_Symdecomp_ii} If instead all $\si_i$ are weakly symplectic and $d$ is even, then 
    \[ 
    \cSym_{\cA,\si} = \big(\cSymd_{\cA,\si} + (\cSkew_{\cA_1,\si_1}\otimes_{\cO}
    \cdots \ot_{\cO} \cSkew_{\cA_d,\si_d})\big)^\sharp \]
\end{enumerate}
where $\sharp$ denotes sheafification.
\end{lem}
\begin{proof}
\noindent\ref{lem_tensor_Symdecomp_i}: The statement is clear for $d=1$ and is \cite[5.1(i)]{GNR} for $d=2$. We can therefore proceed by induction, assuming that the result holds for $d-1$ factors. Set $(\cC, \ga) = (\cA_2,\si_2)\ot_\cO \cdots \ot_\cO (\cA_d, \si_d)$. Then $(\cA,\sigma)=(\cA_1\otimes_\cO \cC,\sigma_1\otimes\gamma)$ and we can apply the result for the two factor case to obtain
\[ \cSym_{\cA, \si} = \big( \cSymd_{\cA, \si} + (\cSym_{\cA_1, \si_1} \ot_\cO \cSym_{\cC, \ga})\big)^\sharp.
\]
Now, applying the $d-1$ factor case to $\cSym_{\cC, \ga}$ in the right hand side of the equation above produces the expression
\begin{align*}
&\big(\cSymd_{\cA, \si} + (\cSym_{\cA_1, \si_1} \ot_\cO (\cSymd_{\cC,\gamma}+(\cSym_{\cA_2,\sigma_2}\otimes_\cO \ldots \otimes_\cO \cSym_{\cA_d,\sigma_d}))\big)^\sharp \\
= &\big(\cSymd_{\cA, \si} + (\cSym_{\cA_1, \si_1} \ot_\cO \cSymd_{\cC,\gamma})+(\cSym_{\cA_1,\sigma_1}\otimes_\cO \ldots \otimes_\cO \cSym_{\cA_d,\sigma_d}))\big)^\sharp
\end{align*}
However, the module $\cSym_{\cA_1, \si_1} \ot_\cO \cSymd_{\cC,\gamma}$ is already a submodule of $\cSymd_{\cA, \si}$ since $\sigma = \sigma_1\otimes \gamma$ and any $x \in \cSym_{\cA_1, \si_1} \ot_\cO \cSymd_{\cC,\gamma}$ is locally of the form $s\otimes(a+\gamma(a))$, which can be written as
\[
s\otimes a + (\sigma_1\otimes \gamma)(s\otimes a)
\]
because $\sigma_1(s)=s$. Therefore we obtain
\[ \cSym_{\cA,\si} = \big(\cSymd_{\cA,\si} + (\cSym_{\cA_1,\si_1}\otimes_{\cO} 
\cdots \ot_{\cO} \cSym_{\cA_d,\si_d})\big)^\sharp, \]
as claimed.

\noindent\ref{lem_tensor_Symdecomp_ii}: Since we have an even number of algebras with symplectic involutions, set $(\cB_i,\sigma_i') = (\cA_{2i-1}\otimes_\cO \cA_{2i},\sigma_{2i-1}\otimes\sigma_{2i})$, each of which are algebras with orthogonal involutions. Then $(\cA,\sigma) = (\cB_1,\sigma_1')\otimes\ldots\otimes(\cB_{\frac{d}{2}},\sigma_{\frac{d}{2}}')$. We can now apply \ref{lem_tensor_Symdecomp_i} to obtain that
\[
\cSym_{\cA,\sigma} = \big(\cSymd_{\cA,\sigma} + (\cSym_{\cB_1,\sigma_1'}\otimes_\cO \ldots \otimes_\cO \cSym_{\cB_{\frac{d}{2}},\sigma_{\frac{d}{2}}'})\big)^\sharp.
\]
Now, we apply \cite[5.1(ii)]{GNR} to $(\cB_1,\sigma_1')=(\cA_1\otimes_\cO \cA_2,\sigma_1\otimes \sigma_2)$ to obtain
\[
\cSym_{\cB_1,\sigma_1'} = \big(\cSymd_{\cB_1,\sigma_1'} + (\cSkew_{\cA_1,\sigma_1}\otimes_\cO \cSkew_{\cA_2,\sigma_2})\big)^\sharp.
\]
Similarly, the term with a $\cSymd$ factor will already be a submodule of $\cSymd_{\cA,\sigma}$, and so we obtain
\[
\cSym_{\cA,\sigma} = \big(\cSymd_{\cA,\sigma} + (\cSkew_{\cA_1,\sigma_1}\otimes_\cO \cSkew_{\cA_2,\sigma_2}\otimes_\cO \cSym_{\cB_2,\sigma_2'} \otimes_\cO \ldots \otimes_\cO \cSym_{\cB_{\frac{d}{2}},\sigma_{\frac{d}{2}}'})\big)^\sharp.
\]
Continuing to apply \cite[5.1(ii)]{GNR} down the line, we ultimately arrive at
\[
\cSym_{\cA,\sigma} = \big(\cSymd_{\cA,\sigma} + (\cSkew_{\cA_1,\sigma_1}\otimes_\cO \ldots \otimes_\cO \cSkew_{\cA_d,\sigma_d})\big)^\sharp,
\]
which finishes the proof.
\end{proof}

\begin{lem}\label{even_symp_tensor}
Let $(\cA_1,\sigma_1),\ldots,(\cA_d,\sigma_d)$ be an even number of Azumaya algebras with symplectic involutions. Set $(\cA,\sigma)=(\cA_1\otimes_\cO \ldots \otimes_\cO \cA_d,\sigma_1\otimes\ldots\otimes\sigma_d)$, which is an Azumaya algebra with orthogonal involution. Then, there exists a unique $f_\otimes \colon \cSym_{\cA,\sigma} \to \cO$ which makes $(\cA,\sigma,f_\otimes)$ a quadratic triple and such that
\[
f_\otimes(s_1\otimes\ldots\otimes s_d) = 0
\]
for all $s_i \in \cSkew_{\cA_i,\sigma_i}$.
\end{lem}
\begin{proof}
Since any semi-trace on $(\cA,\sigma)$ has predetermined values on $\cSymd_{\cA,\sigma}$, by Lemma \ref{lem_tensor_Symdecomp}\ref{lem_tensor_Symdecomp_ii} it will be uniquely determined by its values on $\cSkew_{\cA_1,\sigma_1}\otimes_\cO \ldots \otimes_\cO \cSkew_{\cA_d,\sigma_d}$. We demonstrate the existence of a semitrace $f_\otimes$ which is zero on the tensor product of skew elements, as claimed in the statement. To do so, we partition our tensor product as
\[
(\cA,\sigma) = (\cA_1,\sigma_1)\otimes_\cO (\cA_2\otimes_\cO \ldots \otimes_\cO \cA_d,\sigma_2\otimes\ldots\otimes \sigma_d)
\]
which is now the tensor product of two algebras with symplectic involutions, since the second factor has an odd number of terms. Thus, we apply \cite[5.6]{GNR} to this to obtain a semi-trace $f_\otimes$ which is zero on all elements of $\cSkew_{\cA_1,\sigma_1}\otimes_\cO \cSkew_{(\cA_2\otimes_\cO \ldots \otimes_\cO \cA_d,\sigma_2\otimes\ldots\otimes \sigma_d)}$. Due to the odd number of factors, it is clear that
\[
\cSkew_{\cA_2,\sigma_2}\otimes_\cO \ldots \otimes_\cO \cSkew_{\cA_d,\sigma_d} \subseteq \cSkew_{(\cA_2\otimes_\cO \ldots \otimes_\cO \cA_d,\sigma_2\otimes\ldots\otimes \sigma_d)}
\]
and so this $f_\otimes$ satisfies our requirements and we are done.
\end{proof}

Consider a morphism
\[
(g,\varphi_1,\ldots,\varphi_d)\colon (Y,(\cA_{1,i},\sigma_{1,i})_{i=1}^d) \to (X,(\cA_{2,i},\sigma_{2,i})_{i=1}^d)
\]
in $\fC_{(n_1,\ldots,n_d)}$. Set $(\cB_1,\sigma_1) = (\cA_{1,1}\otimes_{\cO|_Y}\ldots\otimes_{\cO|_Y}\cA_{1,d},\sigma_{1,1}\otimes\ldots\otimes\sigma_{1,d})$ and likewise for $(\cB_2,\sigma_2)$. We equip these with the semi-traces from Lemma \ref{even_symp_tensor} to produce quadratic triples $(\cB_1,\sigma_1,f_1)$ and $(\cB_2,\sigma_2,f_2)$. Since
\[
g^*(\cA_{2,1})\otimes_{\cO|_Y} \ldots \otimes_{\cO|_Y} g^*(\cA_{2,d}) \cong g^*(\cA_{2,1}\otimes_{\cO|_X}\ldots \otimes_{\cO|_X} \cA_{2,d})
\]
and similarly for the involutions, it is clear that after taking the tensor product of all the component algebras, we get an isomorphism
\[
\varphi_1\otimes\ldots\otimes\varphi_d \colon (\cB_1,\sigma_1) \to g^*(\cB_2,\sigma_2).
\]
We would like to see that this isomorphism also respects the semi-traces. This is straightforward though, since $g^*(f_2)\circ (\varphi_1\otimes\ldots\otimes\varphi_d)$ is a semi-trace on $(\cB_1,\sigma_1)$ which is zero on the submodule $\cSkew_{\cA_{1,1},\sigma_{1,1}}\otimes_{\cO|_Y} \ldots \otimes_{\cO|_Y} \cSkew_{\cA_{1,d},\sigma_{1,d}}$ because for a pure tensor $s_1\otimes\ldots\otimes s_d$ in this submodule,
\[
\varphi_1(s_1)\otimes\ldots\otimes\varphi_d(s_d)
\]
is also a tensor product of skew elements and is thus sent to zero by $g^*(f_2)$. Hence
\[
g^*(f_2)\circ (\varphi_1\otimes\ldots\otimes\varphi_d) = f_1
\]
by the uniqueness in Lemma \ref{even_symp_tensor}. This means that we are able to define the following functor,
\begin{align}
\Psi \colon \fC_{(n_1,\ldots,n_d)} &\to \fD_n \label{eq_stack_Psi} \\
(X,(\cA_1,\sigma_1),\ldots,(\cA_d,\sigma_d)) &\mapsto (X,(\cA_1\otimes_{\cO|_X}\ldots\otimes_{\cO|_X}\cA_d,\sigma_1\otimes\ldots\otimes\sigma_d,f_\otimes)) \nonumber \\
(g,\varphi_1,\ldots,\varphi_d) &\mapsto (g,\varphi_1\otimes\ldots\otimes \varphi_d) \nonumber
\end{align}
which is a morphism of stacks since $\fD_n$ is a gerbe.

\begin{lem} \label{lem_twist_un} 
We use numbers $d$ and $n_1,\ldots,n_d$ as in the beginning of Section {\rm \ref{triple_over_Z}}, thus $n_1,\ldots,n_d$ are an even number of positive integers, and likewise set $n=n_1(2n_2)\ldots(2n_d)$. Consider the morphism $\Psi$ of \eqref{eq_stack_Psi} defined above.
\begin{enumerate}[label={\rm(\roman*)}]
\item \label{lem_twist_un_1} The image of the split object in $\fC_{(n_1,\ldots,n_d)}$ under $\Psi$ is isomorphic to the split object in $\fD_n$, i.e.,
\[
\Psi\big((\Mat_{2n_1}(\cO),\sigma_{n_1}),\ldots,(\Mat_{2n_d}(\cO),\sigma_{n_d})\big) \cong (\Mat_{2n}(\cO),\sigma_\tens,f_\tens).
\]
\item \label{lem_twist_un_2} The group sheaf homomorphism induced by $\Psi$ between the automorphism sheaves of the split objects,
\[
\Psi_{\cM} \colon \PSP_{2n_1}\times\ldots\times \PSP_{2n_d} \to \bPGO_{q_\tens},
\]
is the composition
\[
\PSP_{2n_1}\times\ldots\times \PSP_{2n_d} \xrightarrow{\underline{h}} \bPGO_{q_\tens}^+ \hookrightarrow \bPGO_{q_\tens}
\]
where $\underline{h}$ is the map from \eqref{eq_h_underline}.
\end{enumerate}
\end{lem}
\begin{proof}
\noindent\ref{lem_twist_un_1}: Recall from Section \ref{triple_over_Z} that, by construction, the involution $\sigma_{n_1}\otimes\ldots\otimes \sigma_{n_d}$ is adjoint to the polar bilinear form of $q_\tens$, i.e., $\sigma_\tens = \sigma_{n_1}\otimes\ldots\otimes \sigma_{n_d}$. Thus, we only need to argue that $f_\tens$ vanishes on all elements of $\cSkew_{(\Mat_{2n_1}(\cO),\sigma_{n_1})}\otimes_\cO \ldots \otimes_\cO \cSkew_{(\Mat_{2n_d}(\cO),\sigma_{n_d})}$ and then Lemma \ref{even_symp_tensor} will guarantee that $f_\tens$ is the semi-trace produced by $\Psi$. However, if we instead reframe our situation as being of the form $d'=2$, $n_1' = n_1$ and $n_2' = n_2(2n_3)\ldots(2n_d)$ with symplectic involutions $(\Mat_{2n_1}(\cO),\sigma_{n_1})$ and
\[
(\Mat_{2n_2'}(\cO),\sigma_{n_2'})=(\Mat_{2n_2}(\cO)\otimes_\cO \ldots \otimes_\cO \Mat_{2n_d}(\cO), \sigma_{n_2}\otimes\ldots\otimes \sigma_{n_d}),
\]
then $\sigma_{n_1}\otimes \sigma_{n_2'} = \sigma_{n_1}\otimes\ldots\otimes \sigma_{n_d}$. Hence, by uniqueness from Remark \ref{rem_canonical}, the semi-trace $f_\tens$ agrees with the semi-trace $f_\otimes$ from \cite[5.6]{GNR} over $\ZZ$, and thus they also agree after being base changed to $S$. The defining property of $f_\otimes$ is that it vanishes on all elements of
\[
\cSkew_{(\Mat_{2n_1}(\cO),\sigma_{n_1})}\otimes_\cO \cSkew_{(\Mat_{2n_2'}(\cO),\sigma_{n_2'})}
\]
which contains the module $$\cSkew_{(\Mat_{2n_1}(\cO),\sigma_{n_1})}\otimes_\cO \ldots \otimes_\cO \cSkew_{(\Mat_{2n_d}(\cO),\sigma_{n_d})}$$ we are interested in. Hence, $f_\tens$ also vanishes on these elements and so we are done.

\noindent\ref{lem_twist_un_2}: Since $\underline{h}$ is a restriction of $\PSeg$ by the comments under \eqref{eq_h_underline}, it is clear that
 $\underline{h}$ agrees with the definition of how $\Psi_\cM$ acts on morphisms. Composing $\underline{h}$ with the inclusion into $\PGO_{q_\tens}$ to get the correct codomain produces the claimed description of $\Psi_\cM$, finishing the proof.
\end{proof}

As in Lemma \ref{lem_segre}\ref{lem_segre2}, we now consider the case $2n=(2m)^d$. We consider the object $\cM'=(\Sd \to S,(\Mat_{2m}(\cO),\sigma_m)|_{\Sd}) \in \fC_m^{d\etale}(S)$ to be the split object of $\fC_m^{d\etale}$. It has automorphism sheaf $(\PSP_{2m}^d)\rtimes \SS_d$. For an object $(X'\to X,(\cA,\sigma)) \in \fC_m^{d\etale}$, we may simply consider $\cA$ as an Azumaya algebra of degree $2m$ over a degree $d$ \'etale extension of $X$, and consider its norm $N_{X'/X}(\cA)$. By Lemma \ref{lem_ferrand} this is a degree $(2m)^d$ Azumaya algebra over $X$. Assuming that $d$ is even, we now equip this algebra with a quadratic pair using the information from $\sigma$.
\begin{lem}\label{lem_quad_pair_on_norm}
Let $X' \to X$ be a degree $d$ \'etale cover of schemes with $d$ even and let $(\cA,\sigma)$ be a degree $2m$ Azumaya algebra over $X'$ with a symplectic involution. Since $\sigma \colon \cA \to \cA$ is an $\cO|_{X'}$--module morphism, we obtain a $\cO|_X$--module morphism $N_{X'/X}(\sigma) \colon N_{X'/X}(\cA) \to N_{X'/X}(\cA)$. Then, we have the following.
\begin{enumerate}[label={\rm(\roman*)}]
\item \label{lem_quad_pair_on_norm_i} The map $N_{X'/X}(\sigma)$ is an orthogonal involution on $N_{X'/X}(\cA)$.
\item \label{lem_quad_pair_on_norm_ii} The $\cO|_X$--module $N_{X'/X}(\cSkew_{\cA,\sigma})$ is a submodule of $N_{X'/X}(\cA)$, i.e., the injection $\cSkew_{\cA,\sigma} \inj \cA$ is mapped by the norm to an injection
\[
N_{X'/X}(\cSkew_{\cA,\sigma}) \inj N_{X'/X}(\cA).
\]
\item \label{lem_quad_pair_on_norm_iii} $\cSym_{(N_{X'/X}(\cA),N_{X'/X}(\sigma))} = \big(\cSymd_{(N_{X'/X}(\cA),N_{X'/X}(\sigma))} + N_{X'/X}(\cSkew_{\cA,\sigma})\big)^\sharp$, where $\sharp$ denotes sheafification.
\item \label{lem_quad_pair_on_norm_iv} There exists a unique semi-trace $f_N \colon \cSym_{(N_{X'/X}(\cA),N_{X'/X}(\sigma))} \to \cO|_X$ which makes $(N_{X'/X}(\cA),N_{X'/X}(\sigma),f_N)$ a quadratic triple and such that $f_N(s)=0$ for all $s\in N_{X'/X}(\cSkew_{\cA,\sigma})$.
\end{enumerate}
\end{lem}
\begin{proof}
The justification of all four claims will follow from local considerations. Since $X' \to X$ is \'etale of degree $d$, we can fix a cover $\{T_i \to X\}_{i\in I}$ over which $X'\times_X T_i \cong T_i^{\sqcup d}$. Pulling the algebra with symplectic involution $(\cA,\sigma)$ back to one of these $X'\times_X T_i$ then yields
\[
(\cA,\sigma)|_{X'\times_X T_i} = \big((\cA_{i,1},\sigma_{i,1}),\ldots,(\cA_{i,d},\sigma_{i,d})\big)
\]
where we have $d$ algebras with symplectic involution over $T_i$, one over each component of the disjoint union $T_i^{\sqcup d}$. Since the norm $N_{\fMod}$ is a stack morphism, it is compatible with base change, and thus
\[
N_{X'/X}(\cA)|_{T_i} \cong N_{(X'\times_X T_i)/T_i}(\cA|_{X'\times_X T_i}) \cong \cA_{i,1}\otimes_{\cO|_{T_i}}\ldots\otimes_{\cO|_{T_i}}\cA_{i,d}
\]
where we know by Example \ref{ex_split_algebra_norm} that over the split \'etale cover the norm produces the tensor product as above with the usual tensor component-wise multiplication for a tensor product of algebras. We now proceed with proving our claims.

\noindent\ref{lem_quad_pair_on_norm_i}: The morphism $N_{X'/X}(\sigma)$ locally restricts to
\[
N_{T_i^{\sqcup d}/T_i}(\sigma_{i,1},\ldots,\sigma_{i,d}) = \sigma_{i,1}\otimes\ldots\otimes \sigma_{i,d}
\]
which is an orthogonal involution since it is the tensor product of an even number of symplectic involutions. Hence $N_{X'/X}(\sigma)$ is an orthogonal involution.

\noindent\ref{lem_quad_pair_on_norm_ii}: The induced morphism $N_{X'/X}(\cSkew_{\cA,\sigma}) \to N_{X'/X}(\cA)$ locally takes the form
\[
N_{T_i^{\sqcup d}/T_i}(\cSkew_{(\cB_{i,1},\ldots,\cB_{i,d}),(\sigma_{i,1},\ldots,\sigma_{i,d})}) \to \cB_{i,1}\otimes_\cO \ldots\otimes_\cO \cB_{i,d}.
\]
The $\cSkew$ submodule takes the form
\[
\cSkew_{(\cA_{i,1},\ldots,\cA_{i,d}),(\sigma_{i,1},\ldots,\sigma_{i,d})} = (\cSkew_{\cA_{i,1},\sigma_{i,1}},\ldots,\cSkew_{\cA_{i,d},\sigma_{i,d}})
\]
and thus our morphism is the canonical map
\[
\cSkew_{\cA_{i,1},\sigma_{i,1}}\otimes_\cO \ldots\otimes_\cO \cSkew_{\cA_{i,d},\sigma_{i,d}} \to \cA_{i,1}\otimes_\cO \ldots\otimes_\cO \cA_{i,d}.
\]
This is injective by Lemma \ref{tensor_submodules}\ref{tensor_submodules_ii} and so $N_{X'/X}(\cSkew_{\cA,\sigma}) \inj N_{X'/X}(\cA)$ is injective since it is locally injective.

\noindent\ref{lem_quad_pair_on_norm_iii}: From the argument in \ref{lem_quad_pair_on_norm_ii}, we know that locally
\[
N_{X'/X}(\cSkew_{\cA,\sigma})|_{X'\times_X T_i} = \cSkew_{\cA_{i,1},\sigma_{i,1}}\otimes_\cO \ldots\otimes_\cO \cSkew_{\cA_{i,d},\sigma_{i,d}}.
\]
By Lemma \ref{lem_tensor_Symdecomp}\ref{lem_tensor_Symdecomp_ii} we know that
\begin{align*}
&\cSym_{(\cA_{i,1}\otimes\ldots\otimes\cA_{i,d},\sigma_{i,1}\otimes\ldots\otimes\sigma_{i,d})} \\
= &\big(\cSymd_{(\cA_{i,1}\otimes\ldots\otimes\cA_{i,d},\sigma_{i,1}\otimes\ldots\otimes\sigma_{i,d})} + \cSkew_{\cA_{i,1},\sigma_{i,1}}\otimes_\cO \ldots\otimes_\cO \cSkew_{\cA_{i,d},\sigma_{i,d}}\big)^\sharp
\end{align*}
which we may write as
\begin{align*}
&\cSym_{(N_{X'/X}(\cSkew_{\cA,\sigma}),N_{X'/X}(\sigma))}|_{X'\times_X T_i} \\
=& \big(\cSymd_{(N_{X'/X}(\cSkew_{\cA,\sigma}),N_{X'/X}(\sigma))}|_{X'\times_X T_i} + N_{X'/X}(\cSkew_{\cA,\sigma})|_{X'\times_X T_i}\big)^\sharp.
\end{align*}
Thus, our claimed equation holds locally so we may conclude that globally we have
\[
\cSym_{(N_{X'/X}(\cA),N_{X'/X}(\sigma))} = \big(\cSymd_{(N_{X'/X}(\cA),N_{X'/X}(\sigma))} + N_{X'/X}(\cSkew_{\cA,\sigma})\big)^\sharp
\]
as desired.

\noindent\ref{lem_quad_pair_on_norm_iv}: It is clear from \ref{lem_quad_pair_on_norm_iii} that if such a semi-trace $f$ exists, then it is unique. Thus, we only need to show that a semi-trace exists which vanishes on $N_{X'/X}(\cSkew_{\cA,\sigma})$. We construct $f_N$ by defining it to be locally given by the semi-traces $f_{\otimes,i}$ guaranteed by Lemma \ref{even_symp_tensor}. It is clear that if this is well defined, then it will have the property we desire. So, we check that the local $f_{\otimes,i}$ agree on overlaps. However this is clear since both $f_{\otimes,i}|_{T_{ij}}$ and $f_{\otimes,j}|_{T_{ij}}$ will be semi-traces which vanish on $N_{X'/X}(\cSkew_{\cA,\sigma})|_{T_{ij}}$, which by \ref{lem_quad_pair_on_norm_iii} means that they are equal. Hence we obtain the claimed semi-trace $f_N$ and we are done.
\end{proof}

If we are given a morphism
\[
(g',g,\varphi)\colon (X_1'\to X_1,(\cA_1,\sigma_1)) \to (X_2'\to X_2,(\cA_2,\sigma_2)
\]
in $\fC_m^{d\etale}$, we know that the norm yields a morphism in the gerbe of Azumaya algebras $\fAzu_{(2m)^d}$,
\[
N_\fAzu(g',g,\varphi) \colon (X_1,N_{X_1'/X_1}(\cA_1)) \to (X_2,N_{X_2'/X_2}(\cA_2)).
\]
Viewing these simply as modules, we may consider $\sigma_i$ as a morphism of the form $(\Id_{X_i'},\Id_{X_i},\sigma_i)$ in $\fMod_{(2m)^2}^{d \etale}$. Then, since we have that
\[
(g',g,\varphi)\circ(\Id_{X_1'},\Id_{X_1},\sigma_1)=(\Id_{X_2'},\Id_{X_2},\sigma_2) \circ (g',g,\varphi),
\]
we get that
\[
N_\fAzu(g',g,\varphi) \circ (\Id_{X_1},N_{X_1'/X_1}(\sigma_1)) = (\Id_{X_2},N_{X_2'/X_2}(\sigma_2))\circ N_\fAzu(g',g,\varphi).
\]
In other words, this means that the morphism $N_\fAzu(g',g,\varphi)$ respects the orthogonal involutions. In a similar way, since the original morphism $(g',g,\varphi)$ restricts to a map $(X_1'\to X_1,\cSkew_{\cA_1,\sigma_1}) \to (X_2'\to X_2,\cSkew_{\cA_2,\sigma_2})$ in the stack of modules, the morphism $N_\fAzu(g',g,\varphi)$ will map $N_{X_1'/X_1}(\cSkew_{\cA_1,\sigma_1})$ into $N_{X_2'/X_2}(\cSkew_{\cA_2,\sigma_2})$ and therefore also respect the semi-traces $f_{N,1}$ and $f_{N,2}$ provided by Lemma \ref{lem_quad_pair_on_norm}\ref{lem_quad_pair_on_norm_iv}.

Therefore, we may use the norm functor to define the morphism of stacks
\begin{align}
\Psi' \colon \fC_m^{d\etale} &\to \fD_n \label{eq_stack_Psi_prime}\\
(X'\to X,(\cA,\sigma)) &\mapsto (X,(N_{X'/X}(\cA),N_{X'/X}(\sigma),f_N)) \nonumber \\
(g',g,\varphi) &\mapsto N_\fAzu(g',g,\varphi). \nonumber
\end{align}

\begin{lem}\label{prop_norm_triples}
Let $m\in \ZZ$ and let $n\in \ZZ$ such that $2n=(2m)^d$. Consider the morphism $\Psi'$ of \eqref{eq_stack_Psi_prime} defined above. Then $\Psi'$ sends the split object $\cM'=(\Sd \to S,(\Mat_{2m}(\cO),\sigma_m)|_{\Sd})$ to $(S,(\Mat_{2n}(\cO),\sigma_\tens,f_\tens))$ and the induced group sheaf homomorphism
\[
\Psi'_{\cM'} \colon (\PSP_{2m})^d\rtimes \SS_d \to \bPGO_{q_\tens}
\]
is $\underline{h}'$ of \eqref{eq_h_tilde_underline}.
\end{lem}
\begin{proof}
We know that for the split object $\cM'$, the norm simply tensors together the components, and so the image $\Psi'(\cM')$ will agree with the image of
\[
\cM = (S,(\Mat_{2m}(\cO),\sigma_m),\ldots,(\Mat_{2m}(\cO),\sigma_m))
\]
under the morphism $\Psi \colon \fC_{(m,m,\ldots,m)} \to \fD_n$ of \eqref{eq_stack_Psi}. Then, we know from Lemma \ref{lem_twist_un}\ref{lem_twist_un_1} that $\Psi(\cM) = (S,(\Mat_{2n}(\cO),\sigma_\tens,f_\tens))$.

To identify the induced group sheaf homomorphism, we point out that $\Psi'$ is defined to behave the same on morphisms as the norm functor $N_\fAzu$. By Corollary \ref{cor_norm_azu_group_hom} the group sheaf homomorphism induced by $N_\fAzu$ is
\[
\PSeg' \colon (\PGL_{2m})^d \rtimes \SS_d \to \PGL_{(2m)^d}.
\]
Therefore, $\Psi'_{\cM'}$ will simply be obtained from $\PSeg'$ by appropriately restricting the domain and codomain. However, this is exactly how $\underline{h}'$ was defined, and so we are done. 
\end{proof}

In the special case when $\underline{h}'$ is an isomorphism, we obtain an equivalence of stacks by Theorem \ref{lem_gerbe_equivalence}. We compose this with the equivalence of stacks
\begin{align*}
\rho \colon \fA_1^{2\etale} &\to \fC_1^{2\etale} \\
(X'\to X,\cA)&\mapsto (X'\to X,(\cA,\sigma_{\cA})),
\end{align*}
where $\fA_1^{2\etale}$ is the stack of quaternion algebras over a degree $2$ \'etale extension as defined at the beginning of Section \ref{Equivalence}, that equips a quaternion algebra $\cA$ with its canonical symplectic involution $\sigma_{\cA}$.
\begin{thm}\label{thm_norm_eq} Assume that $m=1$ and $d=2$. The morphism
\[
\begin{tikzcd}
\fA_1^{2\etale} \arrow{r}{\rho} & \fC_1^{2\etale} \arrow{r}{\Psi'} & \fD_2 \\[-4ex]
(X'\to X,\cA) \arrow[mapsto]{rr} & & (X,(N_{X'/X}(\cA),N_{X'/X}(\sigma),f_N))
\end{tikzcd}
\]
is an equivalence of gerbes. Furthermore, those stacks are equivalent to the following stacks over $\Sch_S$.
 \sm
\begin{enumerate}[label={\rm (\roman*)}]

 \item\label{thm_norm_eq_i} The stack of  $(\mathbf{PGL}_2 \times \mathbf{PGL}_2) \rtimes \ZZ/2\ZZ$-torsors.

\item\label{thm_norm_eq_ii} The stack of  
$\mathbf{PGO}_{q_\tens} \cong \mathbf{PGO}_{\Mat_4(\cO), \sigma_\tens,f_\tens}$-torsors.

\sm
 
\item\label{thm_norm_eq_iii} The stack of adjoint semisimple  group schemes
of type $A_1 \times A_1$.

\sm
 
\item\label{thm_norm_eq_iv} The stack of adjoint semisimple  group schemes
of type $D_2$.

\end{enumerate}
\end{thm}
\begin{proof} As noted above, by Theorem \ref{lem_gerbe_equivalence}, $\Psi'$ is an equivalence of stacks because $\underline{h}'$ of \eqref{eq_h_tilde_underline} is an isomorphism. The morphism $\rho$ is the canonical equivalence and so their composition is an equivalence as well.

All adjoint semisimple group schemes of type $A_1\times A_1$ are twisted forms of the split adjoint Chevalley group scheme of the same type, namely $\PGL_2\times \PGL_2$, which has automorphism group $(\mathbf{PGL}_2 \times \mathbf{PGL}_2) \rtimes \ZZ/2\ZZ$. This provides the equivalence between \ref{thm_norm_eq_i} and \ref{thm_norm_eq_iii} since both are gerbes. In turn, the category $\fA_1^{2\etale}$ is equivalent to \ref{thm_norm_eq_i} since $\fA_1^{2\etale} \cong \fF(\PGL_2)^{2\etale}$ by Remark \ref{rem_torsors_equiv_auto} and then $\fF(\PGL_2)^{2\etale} \cong \fTors((\PGL_2\times\PGL_2)\rtimes\SS_2)$ by construction in Section \ref{permutation_semi_direct}.

Similarly, the adjoint semisimple groups of type $D_2$ are twisted forms of the split adjoint Chevalley group of type $D_2$, which is $\bPGO_{q_\tens}^+$, and its automorphism group is $\bPGO_q$. As above, using again Theorem \ref{lem_gerbe_equivalence}, this provides the equivalence between \ref{thm_norm_eq_ii}, \ref{thm_norm_eq_iv}, and $\fD_2$.
\end{proof}

\begin{remarks}\label{rem_KMRT_generalized}
\begin{enumerate}[label={\rm (\roman*)}]
\item[]
\item\label{rem_KMRT_generalized_iii} We give a fiber-wise quasi-inverse to this equivalence in Section \ref{Clifford_Morphism} below. In particular, see Theorem \ref{thm_quasi_inverse}.
\item\label{rem_KMRT_generalized_i} The equivalence \ref{thm_norm_eq_iii} $\Leftrightarrow$ \ref{thm_norm_eq_iv} in Theorem \ref{thm_norm_eq} implies that there is an isomorphism of $S$--group schemes $(\mathbf{SL}_2\times\mathbf{SL}_2)/M \cong \bO_4^+$ where $M$ is the diagonal copy of $\bmu_2$. Such an isomorphism is constructed in \cite[C.6.3]{Con1}.
\item\label{rem_KMRT_generalized_ii} Let $S=\Spec(\ZZ)$ and let $\FF$ be a field. Then, the fiber over $\Spec(\FF)$ of the gerbes $\fA_1^{2\etale}$ and $\fD_2$ are the groupoids $A_1^2$ and $D_2$ of \cite[\S15]{KMRT} and the morphism between these fibers is the functor of \cite[\S15.B]{KMRT}. Since any equivalence of gerbes gives rise to an equivalence of the fibers, Theorem \ref{thm_norm_eq} gives a proof of \cite[15.7]{KMRT} which is different from the one in loc. cit. The analogous remark applies to Auel's result \cite[3.1]{A}, where it is assumed that 2 is invertible over $S$.  
\end{enumerate}
\end{remarks}

\subsection{The Clifford Morphism}\label{Clifford_Morphism}
As mentioned directly above in Remark \ref{rem_KMRT_generalized}\ref{rem_KMRT_generalized_ii}, our equivalence of stacks $N \colon \fA_1^{2\etale} \to \fD_2$ given by the norm agrees over certain fibers with the equivalence of categories appearing in \cite[\S15.B]{KMRT}. However, in \cite{KMRT} they elaborate by showing that their norm functor has an inverse equivalence given by the Clifford algebra functor. This also generalizes to our setting, which we explain now.

Given a quadratic triple $(\cA,\sigma,f)$ of degree $2n$ over a scheme $T\in \Sch_S$, we denote its Clifford algebra by $\Cl(\cA,\sigma,f)$ as constructed in \cite[4.2.0.13]{CF} (and denoted $\bC_{0,\cA,\sigma,f}$ there), or see also \cite{Rue23}. The Clifford algebra is compatible with base change. If we have a scheme morphism $X\to T$, then
\[
\Cl(\cA,\sigma,f)|_X = \Cl\big((\cA,\sigma,f)|_X\big).
\]
By \cite[4.2.0.15]{CF}, the Clifford algebra is an Azumaya algebra over its center $Z(\Cl(\cA,\sigma,f))$, which is a quadratic \'etale extension of $\cO|_T$. In particular, locally over some cover $\{T_i \to T\}_{i\in I}$ we will have that
\[
Z(\Cl(\cA,\sigma,f))|_{T_i} \cong \cO|_{T_i} \times \cO|_{T_i}
\]
and thus the Clifford algebra is locally the product of two Azumaya $\cO|_{T_i}$--algebras
\[
\Cl(\cA,\sigma,f)|_{T_i} \cong \cB_{i,1}\times \cB_{i,2}.
\]
If necessary, we refine our cover to assume $(\cA,\sigma,f)|_{T_i} \cong (\cEnd_{\cO|_{T_i}}(\cM_i),\sigma_{q_i},f_{q_i})$ for some regular quadratic $\cO|_{T_i}$--module $(\cM_i,q_i)$ of rank $2n$. Then \cite[4.1.0.14]{CF} tells us that
\[
\Cl(\cA,\sigma,f)|_{T_i} \cong \Cl_0(\cM_i,q_i)
\]
where $\Cl_0(\cM_i,q_i)$ is the even part of the Clifford algebra of a quadratic form, which is constructed in \cite[4.2.0.6]{CF} essentially the same way as is done for quadratic forms over rings in \cite[IV.1.1.2]{K}. Hence we know that each $\Cl_0(\cM_i,q_i)$ is rank $2^{2n-1}$ and so $\Cl(\cA,\sigma,f)$ is also rank $2^{2n-1}$. This means that each of the Azumaya algebras $\cB_{i,1}$ and $\cB_{i,2}$ appearing above are of rank $2^{2n-2}$, i.e., they are degree $2^{n-1}$ Azumaya $\cO|_{T_i}$--algebras. In particular, by \cite[(5)]{Rue23}, when we consider the split quadratic triple $(\cA,\sigma,f)=(\Mat_{2n}(\cO),\sigma_{2n},f_{2n})$, we have
\begin{equation}\label{eq_Clifford_of_split}
\Cl(\Mat_{2n}(\cO),\sigma_{2n},f_{2n}) \cong \Mat_{2^{n-1}}(\cO)\times\Mat_{2^{n-1}}(\cO).
\end{equation}

Now we recall both an anti-equivalence of categories and an equivalence of categories which we will use. The first is \cite[Tag 01SA]{Stacks} which states there is an anti-equivalence of categories
\begin{equation}\label{eq_affine_qc_anti_equiv}
\left\{ \begin{array}{c} \text{Schemes with affine} \\ \text{morphism to } T \end{array}\right\} \leftrightarrow \left\{\begin{array}{c}\text{Quasi-coherent sheaves} \\ \text{of commutative } \cO|_T\text{--algebras} \end{array}\right\}
\end{equation}
which associates $g_*(\cO|_X)$ to $g\colon X\to T$. Under this correspondence, finite locally free morphisms of schemes correspond to finite locally free $\cO|_T$--algebras, \cite[Tag 02KA]{Stacks}. Further, finite \'etale morphisms of schemes correspond to finite \'etale $\cO|_T$--algebras, which are locally isomorphic as algebras to $\cO|_T^n$ for some $n$. This correspondence follows from the discussion in \cite[2.5]{CF}.

Second, \cite[01SB]{Stacks} says that given an affine morphism $g\colon X \to T$, there is an equivalence of categories
\begin{equation}\label{eq_qc_equiv}
\left\{ \begin{array}{c} \text{Quasi-coherent} \\ \cO|_X\text{--modules over} X \end{array}\right\} \leftrightarrow \left\{ \begin{array}{c} \text{Quasi-coherent} \\ g_*(\cO|_X)\text{--modules over } T \end{array}\right\}
\end{equation}
which sends $\cM$ to $g_*(\cM)$. This correspondence also restricts to the respective categories of quasi-coherent algebras. Technically, \cite[01SB]{Stacks} is only stated for quasi-coherent modules on the ringed spaces $X$ and $T$, however by \cite[Tag 03DX]{Stacks}, there is an equivalence between the categories of quasi-coherent modules on the ringed space $T$ and the site wide notion of quasi-coherent modules on $\Sch_T$. Further, Lemma \ref{lem_star} ensures that this equivalence commutes with pushforwards when we are dealing with affine morphisms. Hence we may use \cite[01SB]{Stacks} as stated above.

Using the anti-equivalence \eqref{eq_affine_qc_anti_equiv}, the center of the Clifford algebra corresponds to a degree $2$ \'etale covering $g_{(\cA,\sigma,f)}\colon E_{(\cA,\sigma,f)} \to T$ such that the pushforward
\[
g_{(\cA,\sigma,f)*}(\cO|_{E_{(\cA,\sigma,f)}}) \cong Z(\Cl(\cA,\sigma,f))
\]
is the center of the Clifford algebra. Then, using the second equivalence above, because the Clifford algebra is a quasi-coherent algebra over its center, there exists a quasi-coherent $\cO|_{E_{(\cA,\sigma,f)}}$--algebra $\cB_{(\cA,\sigma,f)}$ over $E_{(\cA,\sigma,f)}$ such that
\[
g_{(\cA,\sigma,f)*}(\cB_{(\cA,\sigma,f)}) \cong \Cl(\cA,\sigma,f).
\]
Since $\Cl(\cA,\sigma,f)$ is a degree $2^{n-1}$ Azumaya algebra over its center, $\cB_{(\cA,\sigma,f)}$ is a degree $2^{n-1}$ Azumaya $\cO|_{E_{(\cA,\sigma,f)}}$--algebra.

To ultimately upgrade this construction into a morphism of stacks, we will need to use the following lemma.
\begin{lem}[{\cite[Tag 02KG]{Stacks}}]\label{lem_fiber_switch_pull_push}
Given a fiber product diagram of schemes
\[
\begin{tikzcd}
E_1 \ar[r,"h'"] \ar[d,"g'"] & E_2 \ar[d,"g"] \\
T_1 \ar[r,"h"]& T_2
\end{tikzcd}
\]
where $g$, and hence also $g'$, is an affine morphism, then for any quasi-coherent $\cO|_{E_2}$--module $\cM$ we have
\[
h^*(g_*(\cM)) = g'_*(h'^*(\cM)).
\]
\end{lem}
Once again, Lemma \ref{lem_fiber_switch_pull_push} is stated for quasi-coherent modules over ringed spaces, but holds in our context by \cite[Tag 03DX]{Stacks}. Of course, the lemma also holds equally well for quasi-coherent algebras. 

We recall the gerbes $\fD_n$ and $\fA_{2^{n-1}}^{2\etale}$ from the beginning of Section \ref{Equivalence}. Assume we have a morphism
\[
(h,\varphi) \colon (T_1,(\cA_1,\sigma_1,f_1))\to (T_2,(\cA_2,\sigma_2,f_2))
\]
in $\fD_n$. So $h\colon T_1 \to T_2$ is any morphism of schemes and $\varphi \colon (\cA_1,\sigma_1,f_1) \iso h^*(\cA_2,\sigma_2,f_2)$ is an isomorphism of quadratic triples over $T_1$. We will define a morphism between the objects $(E_{(\cA_1,\sigma_1,f_1)}\to T_1,\cB_{(\cA_1,\sigma_1,f_1)})$ and $(E_{(\cA_2,\sigma_2,f_2)}\to T_2,\cB_{(\cA_2,\sigma_2,f_2)})$ in the stack $\fA_{2^{n-1}}^{2\etale}$. The Clifford algebra construction is functorial and respects base change so we obtain an isomorphism
\[
\Cl(\varphi)\colon \Cl(\cA_1,\sigma_1,f_1) \iso h^*(\Cl(\cA_2,\sigma_2,f_2)).
\]
To shorten notation, let $E_{(\cA_i,\sigma_i,f_i)}=E_i$, $g_i\colon E_i \to T$ be the \'etale cover, and $\cB_{(\cA_i,\sigma_i,f_i)}=\cB_i$. If we consider the fiber product diagram
\[
\begin{tikzcd}
T_1\times_{T_2} E_2 \ar[r,"h'"] \ar[d,"g_2'"] & E_2 \ar[d,"g_2"] \\
T_1 \ar[r,"h"] & T_2,
\end{tikzcd}
\]
Lemma \ref{lem_fiber_switch_pull_push} says that $h^*(g_{2*}(\cO|_{E_2})) = g_{2*}'(h^*(\cO|_{E_2}))$. However, we also have that
\[
h^*(g_{2*}(\cO|_{E_2})) = h^*(Z(\Cl(\cA_2,\sigma_2,f_2)))
\]
is the center of $h^*(\Cl(\cA_2,\sigma_2,f_2))$. This shows that
\[
E_{h^*(\cA_2,\sigma_2,f_2)} = T_1\times_{T_2} E_2.
\]
Now, the restricted isomorphism
\[
\Cl(\varphi)\colon Z(\Cl(\cA_1,\sigma_1,f_1))\iso Z(h^*(\Cl(\cA_2,\sigma_2,f_2)))
\]
between the centers gives us an isomorphism $E_{h^*(\cA_2,\sigma_2,f_2)} \iso E_1$ in the opposite direction between \'etale covers of $T_1$. Since this is an isomorphism, if we define $\widetilde{\varphi}$ to be the composition
\[
\widetilde{\varphi} \colon E_1 \iso E_{h^*(\cA_2,\sigma_2,f_2)} = T_1\times_{T_2} E_2 \xrightarrow{h'} E_2
\] 
of $h'$ with the inverse isomorphism, then we obtain a fiber product diagram
\[
\begin{tikzcd}
E_1 \ar[r,"\widetilde{\varphi}"] \ar[d,"g_1"] & E_2 \ar[d,"g_2"] \\
T_1 \ar[r,"h"] & T_2.
\end{tikzcd}
\]
Because this is another fiber product diagram, Lemma \ref{lem_fiber_switch_pull_push} gives us that
\[
h^*(g_{2*}(\cB_2)) = g_{1*}(\widetilde{\varphi}^*(\cB_2))
\]
where $h^*(g_{2*}(\cB_2)) = h^*(\Cl(\cA_2,\sigma_2,f_2))$. Therefore the isomorphism
\[
\Cl(\varphi) \colon \Cl(\cA_1,\sigma_1,f_1) \to h^*(\Cl(\cA_2,\sigma_2,f_2)),
\]
which we can write as an isomorphism
\[
\Cl(\varphi) \colon g_{1*}(\cB_1) \to g_{1*}(\widetilde{\varphi}^*(\cB_2))
\]
lifts to an isomorphism $\widetilde{\Cl(\varphi)} \colon \cB_1 \iso \widetilde{\varphi}^*(\cB_2)$ of Azumaya $\cO|_{E_1}$--algebras.

Thus, we have obtained a morphism of stacks which we call the \emph{Clifford morphism},
\begin{align}
\Cl \colon \fD_n &\to \fA_{2^{n-1}-1}^{2\etale} \label{eq_clifford_stack_morphism} \\
(T,(\cA,\sigma,f)) &\mapsto (g_{(\cA,\sigma,f)}\colon E_{(\cA,\sigma,f)}\to T, \cB_{(\cA,\sigma,f)}) \nonumber \\
(h,\varphi) &\mapsto (h,\widetilde{\varphi},\widetilde{\Cl(\varphi)}). \nonumber
\end{align}
Because both stacks are fibered in groupoids, a morphism of stacks, in fact of gerbes, is simply a functor between the two categories preserving the canonical projections. The latter is clear by construction. We leave the straightforward verification of functoriality of the construction to the reader.

\begin{remark}
The Clifford algebra also comes with a canonical involution analogous to what occurs in the story over rings. By \cite[1.9]{Rue23}, if $n\equiv 0 \pmod{4}$ then the canonical involution is orthogonal and it can be equipped with a canonical semitrace as in \cite[2.7]{Rue23}. In this case there is a natural factorization of the Clifford functor $\fD_n \to \fD_{2^{n-2}}^{2\etale} \to \fA_{2^{n-1}-1}^{2\etale}$. If instead $n\equiv 2 \pmod{4}$, then the canonical involution is symplectic and thus the Clifford functor naturally factors as $\fD_n \to \fC_{2^{n-2}}^{2\etale} \to \fA_{2^{n-1}-1}^{2\etale}$.
\end{remark}

\begin{thm}\label{thm_quasi_inverse}
When $n=2$ and thus $2^{n-1}-1=1$, the Clifford morphism of \eqref{eq_clifford_stack_morphism} becomes a morphism $\Cl \colon \fD_2 \to \fA_1^{2\etale}$. This morphism is an equivalence of gerbes which is quasi-inverse on fibers to $N\colon \fA_1^{2\etale} \to \fD_2$ provided by the norm in Theorem \ref{thm_norm_eq}.
\end{thm}
\begin{proof}
Our strategy is to invoke Corollary \ref{cor_identity_equiv} with $\frF = \frA_1^{2\etale}$, $\frG = \frD_2$, $\vphi = N \co \frF \to \frG$ the norm functor, $\psi =\Cl $ the Clifford functor, and $x=\big(S\sqcup S \to S, \Mat_2(\cO|_{S\sqcup S})\big)$ the split object of $\frA_1^{2\etale}$. Thus, we  will define a concrete isomorphism $g\co (\Cl \circ N)\, (x) \to x$ in $\frA_1^{2\etale}$ satisfying the assumptions of the corollary.  
 
To begin, we know from Lemma \ref{prop_norm_triples} that 
\[ N \big( (S\sqcup S \to S, \Mat_2(\cO|_{S\sqcup S})\big) = 
     (S, \Mat_2(\cO) \ot_{\cO} \Mat_2(\cO), \si_{\rm tens}, f_{\rm tens})
\] \rm 
where the involution $\si_{\rm tens}$ and the semitrace $f_{\rm tens}$ are constructed with respect to the quadratic form $q_{\rm tens}\co \cO^2 \ot_\cO \cO^2 \to \cO$ of Section \ref{triple_over_Z}, base changed to $S$. Moreover, by \cite[4.2.0.14]{CF} there is a canonical isomorphism 
\[ 
  \Cl(\Mat_2(\cO)\ot_\cO \Mat_2(\cO), \si_{\rm tens}, f_{\rm tens})
   \cong \Cl_0(\cO^2\ot_\cO \cO^2, q_{\rm tens}),
\]
which we take as an identification. The center of $\Cl_0 (\cO^2\ot_\cO \cO^2, q_{\rm tens})$ is $\cO \times \cO$, which by \eqref{eq_affine_qc_anti_equiv} corresponds to the trivial degree $2$ \'etale cover $f\co S\sqcup S \to S$, and we know that $f_*(\Mat_2(\cO|_{S\sqcup S})) = \Mat_2(\cO) \times \Mat_2(\cO)$ by \eqref{eq_Clifford_of_split}. Thus, we aim to find a natural isomorphism of $\cO$--algebras
\[ \Phi \co \Cl_0(\cO^2 \ot_\cO \cO^2, q_{\rm tens} ) \simlgr \Mat_2(\cO) \times  \Mat_2(\cO).
\]
To this end, let us recall the construction of $q_{\rm tens}$ from Section \ref{triple_over_Z}. It uses the regular alternating form $\psi \co \cO^2 \times \cO^2 \to \cO$, given on sections $x$ and $y$ of $\cO^2$ by 
$\psi(x,y) = x_1 y_2 - x_2 y_1$. The associated regular symmetric form 
\[
   b_{\rm tens} \co (\cO^2\ot_\cO \cO^2) \times (\cO^2 \ot_\cO \cO^2) \to \cO 
\] 
is defined on appropriate sections as $b_{\rm tens}(u \ot v, \, x \ot y) = \psi(u,x) \, \psi(v,y)$. Finally, one puts
\[
q_{\rm tens}(\sum_{i=1}^k u_i\otimes v_i) = \sum_{i,j=1 \atop i< j}^k b_\tens(u_i\otimes v_i, u_j\otimes v_j)
\]
for a general section $\sum_{i=1}^k u_i\otimes v_i$ of $\cO^2\ot_\cO \cO^2$. A helpful observation for proving $i(a)^2 = q_{\rm tens}(a) 1$ below is that  $q_{\rm tens}(u\ot v) = 0$ for pure tensors.  

We use the canonical map
\[ \vphi \co \cO^2 \ot_\cO \cO^2 \to \cEnd_{\cO}(\cO^2) = \Mat_2(\cO) \]
associated with the regular form $\psi$, i.e., $\vphi(u \ot v)(x) = \psi(v,x)\cdot u$ over appropriate sections, to define an $\cO$--module map 
\begin{align*}
i \co \cO^2\ot_\cO \cO^2 &\to \Mat_2\big(\cEnd_\cO(\cO^2)\big) \\ 
u\ot v &\mapsto \begin{bmatrix} 0 & \vphi(u \ot v) \\ -\vphi(v\ot u) & 0 \end{bmatrix}.
\end{align*} 
Using the relations 
\begin{align*} 
  \vphi(u \ot v) \, \vphi(x \ot y) &= \psi(v,x)\, \vphi(u,y),\text{ and} \\  
 \vphi(u\ot v) - \vphi(v\otimes u) &= \psi(v,u) 1_{\Mat_2(\cO)}, 
\end{align*}   
one can compute that for $a=\sum_{i=1}^k u_i\otimes v_i$, we have that $i(a)^2$ is
\begin{align*}
&\; \sum_{i,j=1}^k \begin{bmatrix} -\psi(v_i,v_j)\varphi(u_i\otimes u_j) & 0 \\ 0 & -\psi(u_i,u_j)\varphi(v_i\otimes v_j) \end{bmatrix} \\
=&\; \sum_{i,j=1 \atop i<j}^k \left[ \begin{tikzcd}[ampersand replacement=\&,row sep = 0ex,column sep=-13ex] -\psi(v_i,v_j)(\varphi(u_i\otimes u_j)-\varphi(u_j\otimes u_i)) \& 0 \\ 0 \& -\psi(u_i,u_j)(\varphi(v_i\otimes v_j)-\varphi(v_j\otimes v_i)) \end{tikzcd} \right] \\
=&\; \sum_{i,j=1 \atop i<j}^k \begin{bmatrix} -\psi(v_i,v_j)\psi(u_j,u_i)1_{\Mat_2(\cO)} & 0 \\ 0 & -\psi(u_i,u_j)\psi(v_j,v_i)1_{\Mat_2(\cO)} \end{bmatrix} \\
=&\; \big(\sum_{i,j=1 \atop i<j}^k \psi(u_i,u_j)\psi(v_i,v_j)\big)\Id \\
=&\; \big(\sum_{i,j=1 \atop i<j}^k b(u_i\otimes u_j,v_i\otimes v_j)\big)\Id \\
=&\; q_\tens(a)\Id,
\end{align*}
so that $i$ extends to a graded $\cO$--algebra homomorphism 
 \[ \Cl(i) \co \Cl(\cO^2 \ot_\cO \cO^2, q_{\rm tens}) \to \Mat_2\big(\cEnd_\cO(\cO^2)\big).
 \]
We claim that $\Cl(i)$ is an isomorphism. Indeed, this can be checked by passing to fields, where it is easily seen in \cite[p.~99]{KMRT}. Consequently, by restriction we have an algebra isomorphism
\[ 
\Phi  \co \Cl_0(\cO^2 \ot_\cO \cO^2, q_{\rm tens} ) \simlgr \Mat_2(\cO) \times  \Mat_2(\cO).
\]    
uniquely determined  by 
\[ \Phi\big( (u\ot v) \cdot (x\ot y) \big) =  \big( -\psi(v,y)\vphi(u\ot x), \, - \psi(u,x)\vphi(v\ot y) \big)
\] 
where $(u\ot v) \cdot (x\ot y)$ on the left-hand side is the product in the Clifford algebra. As noted before, we can use $\Phi$ to define our desired isomorphism 
\[ g\co (\Cl \circ N)\,\big(S\sqcup S \to S, \Mat_2(\cO|_{S\sqcup S})\big)  \simlgr 
    \big(S\sqcup S \to S, \Mat_2(\cO|_{S\sqcup S})\big).
\]
The concrete construction of $\Phi$ makes it clear that $g$ is in fact an isomorphism in $\frA_1^{2\etale}$. It will satisfy the assumptions of Corollary \ref{cor_identity_equiv} if the composition 
\begin{align*} 
(\PGL_2 \times \PGL_2) \rtimes \SS_2 &\xrightarrow{\; N\;} {\mathbf {PGO}}_{q_{\rm tens}} \\ 
&\xrightarrow{\; \Cl\;} \cAut\big(\Cl_0(\cO^2 \ot\cO^2, q_{\rm tens} )\big) \\
&\xrightarrow{\Inn(\Phi)} (\PGL_2 \times \PGL_2) \rtimes \SS_2
\end{align*} 
is the identity. Note that the norm functor $N\co \frA_1^{2\etale} \to \frD_2$ induces the isomorphism $\underline{h}'$ of \eqref{eq_h_tilde_underline}, which in this case is
\[ \underline{h}' \co (\PGL_2 \times \PGL_2) \rtimes \SS_2 \simlgr {\mathbf{PGO}}_{q_{\rm tens}} 
\]
since $\PGL_2 \cong \mathbf{PSp}_2$. Thus, we are concerned with proving 
$\Inn(\Phi) \circ \Cl \circ \underline{h}' = \Id$. 

A final important fact about Clifford algebras that we will use below is \cite[1.14]{Rue23} or \cite[13.1]{KMRT} over fields: given an inner automorphism $\ze = \Inn( Z)$ of the quadratic triple $(\Mat_2\big(\cEnd_\cO(\cO^2)\big), \si_{\rm tens}, f_{\rm tens})$ where $Z$ is an orthogonal transformation of $q_{\rm tens}$, the induced isomorphism of the Clifford algebra acts by applying $Z$ to the tensor factors. In particular, 
\[ \Cl_0(\ze) (a \cdot b) = (Za) \cdot (Zb)\]
for $a$, $b\in \cO^2\ot_\cO \cO^2$.

Now, towards showing that the assumptions of Corollary \ref{cor_identity_equiv} are satisfied, consider a section $\big((\al_1, \al_2), \tau\big)$ of $(\PGL_2 \times \PGL_2) \rtimes \SS_2$. We may work sufficiently locally such that $\al_1$ and $\al_2$ become inner automorphisms associated with matrices in $\SL_2$ and $\tau$ becomes a diagonal section, i.e., either the identity or the switch. Therefore, it is sufficient to show that 
\begin{equation}\label{eq1} 
\Phi \circ \Cl_0\big(\Inn(A\ot B)\big)\circ \Phi\me  = \big( \Inn(A), \, \Inn(B)\big) 
\end{equation} 
for $A$, $B\in \uSL_2$, and that for the switch morphism $\sw_\ot \co A \ot B \mapsto B \ot A$ of tensors we have
\begin{equation}\label{eq2} 
  \Phi \circ \Cl_0(\sw_\ot) \circ \Phi\me = \sw_\times
\end{equation}  
where $\sw_\times \co (X, Y) \mapsto (Y,X)$ is the obvious switch of $\Mat_2(\cO)\times \Mat_2(\cO)$. 

{\em Proof of\/} \eqref{eq1}:  Since $A$ is orthogonal for $\psi$, it is immediate that we have the formula $A \vphi(u \ot v) A\me = \vphi(Au, Av)$ and therefore
\begin{align*} 
&\;\big( (\Inn(A), \Inn(B)) \circ \Phi\big)\, \big( (u\ot v) \cdot (x \ot y)\big)
\\ =&\; \big( -\psi(v,y)A \vphi(u\ot x)A\me, \, - \psi(u,x)B\vphi(v\ot y)B\me \big)
\\ =&\; \big( -\psi(Bv,By)\vphi(Au\ot Ax), \, - \psi(Au,Ax)\vphi(Bv\ot By) \big)
\\ =&\; \Phi\big((Au\ot Bv) \cdot (Ax \ot By)\big)
\\ =&\; \big(\Phi \circ \Cl_0(\Inn (A \ot B))\big)\, \big( (u\ot v) \cdot (x \ot y)\big),
\end{align*}
which implies \eqref{eq1} because $(u\ot v) \cdot (x \ot y)$ for the various sections of $\cO^2\otimes_\cO \cO^2$ is a generating set of the algebra $\Cl_0(\cO^2\ot_\cO \cO^2, q_{\rm tens})$.  

{\em Proof of\/} \eqref{eq2}: The morphism $\sw_\ot$ is the inner automorphism associated with the tensor module switch 
\begin{align*}
\sw'_\ot \co \cO^2\ot_\cO \cO^2 &\to \cO^2 \ot_\cO \cO^2 \\
u \ot v &\mapsto v \ot u 
\end{align*}
Hence,
\begin{align*}
  &\; (\sw_\times \circ \Phi)\, \big( (u\ot v) \cdot (x\ot y) \big) 
\\=&\; \big(  - \psi(u,x)\vphi(v\ot y) , \, -\psi(v,y)\vphi(u\ot x)\big)
\\=&\; \Phi\big( (v\ot u)\cdot (y \ot x)\big) 
\\=&\; \Phi\big( \sw'_\ot(u\ot v) \cdot \sw'_\ot (x\ot y)\big) 
\\=&\; \big(\Phi \circ \Cl_0(\sw_\ot)\big)\,  \big( (u\ot v) \cdot (x\ot y) \big)
\end{align*}
which implies \eqref{eq2}. Thus, Corollary \ref{cor_identity_equiv} applies and we are done.
\end{proof}

\appendix\section{Twisted Sheaves and Weil Restriction}\label{app_Weil}
Here we present a variation of \cite[III,2.3.2]{Gir} about how we can use twisting, i.e., contracted products as in Section \ref{sec_contracted_products}, to transition between different pushforwards of pullbacks of sheaves on $\Sch_S$. In particular, let $T,S' \in \Sch_S$ with structure morphisms $f\colon T \to S$ and $g\colon S' \to S$, and let $\cF$ be a sheaf of sets on $\Sch_S$. In some cases, the sheaves $f_*(f^*(\cF))$ and $g_*(g^*(\cF))$ are twisted forms of one another, and thus one can be obtained by twisting the other. As an application, we will use our general result when $T\to S$ is a finite \'etale cover of degree $d$. As a corollary, this describes the Weil restriction, when it exists, of $Y\times_S T$ for an $S$--scheme $Y$. We refer to \cite[\S 7.6]{BLR} for details on the Weil restriction.

\begin{lem}\label{u_group_hom}
Let $S' \in \Sch_S$ with structure morphism $g\colon S' \to S$ and let $\cF$ be a sheaf on $\Sch_S$. Then, there exists a homomorphism of group sheaves
\[
u \colon \cAut_S(S') \to \cAut(g_*(g^*(\cF))
\]
defined by the property that for all $X\in \Sch_S$, $a \in \cAut_S(S')(X)$, and $e\in g_*(g^*(\cF))(X)$ we have
\[
u(a)(e)= \cF(a)^{-1}(e).
\]
Here, $\cAut_S(S')$ is the sheaf of $S$--automorphisms of $S'$ as in Example \ref{hom_sheaf_of_schemes} and $\cAut(g_*(g^*(\cF))$ is the sheaf of internal automorphisms of the sheaf $g_*(g^*(\cF))$ on $\Sch_S$.
\end{lem}
\begin{proof}
For any $X\in \Sch_S$ and section $a\in \cAut_S(S')(X)$, because $a\colon S'\times_S X \iso S'\times_S X$ is an automorphism, we obtain a set bijection
\[
\cF(a) \colon \cF(S'\times_S X) \iso \cF(S'\times_S X)
\]
when we consider the restriction map for $\cF$ along $a$. Because we may write $\cF(S'\times_S X) = g_*(g^*(\cF))(X)$, this yields a function $\cAut_S(S')(X) \to \cAut(g_*(g^*(\cF)))(X)$. However, because $\cF$ is contravariant, i.e., we have that $\cF(a\circ b) = \cF(b)\circ \cF(a)$, this is not a group homomorphism. If instead we consider the function defined by $a \mapsto \cF(a)^{-1}$ we do obtain a group homomorphism. This construction is compatible with restrictions and base change, and so we obtain a sheaf automorphism $u(a) \colon g_*(g^*(\cF))|_X \iso g_*(g^*(\cF))|_X$ defined by
\begin{align*}
u(a)(X') \colon g_*(g^*(\cF))(X') &\iso g_*(g^*(\cF))(X') \\
e &\mapsto \cF(a|_{X'})^{-1}(e)
\end{align*}
over a scheme $X' \in \Sch_X$. In this way we obtain the homomorphism of group sheaves
\[
u \colon \cAut_S(S') \to \cAut(g_*(g^*(\cF)))
\]
with the desired property.
\end{proof}

\begin{lem}\label{lem_twist_sheaf-gen} Let $g\co S' \to S$ and $f\co T \to S$ be schemes over $S$ and abbreviate $\cA = \cAut_S(S')$. Assume that $\cIsom(S', T)$ is an $\cA$--torsor. Then, for any sheaf of sets $\cF$ on $\Sch_S$ there exists a canonical isomorphism of sheaves of sets
\[
\cIsom(S',T)\wedge^{\cA} g_* (g^*(\cF)) \iso f_*(f^*(\cF))
\]
where the contracted product is defined using the map $u \colon \cA \to \cAut(g_*(g^*(\cF)))$ of Lemma \ref{u_group_hom}.
\end{lem}
\begin{proof} Let $\cE = g_* (g^*(\cF))$. Then $\cE(X) = \cF(S'\times_S X)$ and the  contracted product $\cIsom(S',T)\wedge^\cA \cE$ is the sheaf associated with the presheaf on $\Sch_S$ 
\[
X \mapsto \big(\Isom(S' \times_S X, T\times_S X)\times \cE(X)\big)/\sim
\]
where the equivalence relation $\sim$ is given by $(\varphi\cdot a, \,  e) \sim (\varphi, \, a \cdot e)$ for all $\varphi \in \Isom(S'\times_S X,T\times_S X)$, $e\in \cE(X)= \cF(S' \times_S X)$, and $a \in \cA(X)$. We will show that there is an injection from this presheaf into $f_*(f^*(\cF))$ which is locally surjective, and therefore will induce the desired isomorphism of sheaves.

Consider the canonical map of presheaves defined over $X$ by
\begin{align*}
\big( \Isom(S' \times_S X, T\times_S X)\times \cF(S' \times X)\big) /\sim \;&\to \; \cF(T\times_S X)  \\
(\varphi, e) \quad &\mapsto \quad \cF(\varphi^{-1})(e)
\end{align*}
Observe that this is well-defined since $\cF(\varphi^{-1}) \co \cF(S' \times_S X) \to \cF(T \times_S X)$ because $\cF$ is contravariant. If $\cF(\varphi_1^{-1})(e_1) = \cF(\varphi_2^{-1})(e_2)$, then 
\[
e_2 = (\cF(\varphi_2)\circ \cF(\varphi_1^{-1}))(e_1) = \cF(\varphi_1^{-1}\varphi_2)(e_1),
\]
but since $\varphi_1^{-1}\varphi_2 \in \cA(X)$, we get
\[
e_2 = u(\varphi_1^{-1}\varphi_2)^{-1} \cdot e_1 = u(\varphi_2^{-1}\varphi_1)\cdot e_1
\]
by Lemma \ref{u_group_hom}, and therefore, under the relation $\sim$,
\[
(\varphi_2,e_2) = (\varphi_2,u(\varphi_2^{-1}\varphi_1)\cdot e_1) = (\varphi_2\varphi_2^{-1}\varphi_1, e_1) = (\varphi_1,e_1),
\]
and so the map of presheaves is injective. 

If there is an element $\varphi \in \Isom(S'\times_S X,T\times_S X)$, i.e., if $\cIsom(S',T)(X)\neq \O$, then for all $f\in \cF(T\times_S X)$ we have
\[
(\varphi, \cF(\varphi)(f)) \mapsto f,
\]
hence the map is surjective wherever $\cIsom(S',T)$ has a point. Finally, since $\cIsom(S',T)$ is an $\cA$--torsor, there is a cover of $S$ over which it has points and so the map of presheaves is locally surjective. Hence, it induces an isomorphism of sheaves as claimed.
\end{proof}

\begin{remarks}
\begin{enumerate}[label={(\roman*)}]
\item[]
\item The sheaf $\cIsom(S',T)$ is an $\cA=\cAut_S(S')$--torsor precisely when $S'$ and $T$ are fppf locally isomorphic to one another. I.e., when there exists an fppf cover $\{X_i \to S\}_{i\in I}$ such that $S'\times_S X' \cong S'\times_S T$ for all $i\in I$.
\item Lemma~\ref{lem_twist_sheaf-gen}  produces an isomorphism of sheaves of groups, abelian groups, rings, etc., whenever $\cF$ has such a structure and the structure on $g_*(g^*(\cF))$ is given canonically. 
\item The $\cA$--torsor $\cIsom(S',T)$ corresponds to a cohomology class $[T] \in H^1(S,\cA)$. The homomorphism $\cAut(S') \to \cAut(\cE)$ gives rise to a map 
   $ H^1(S, \cA) \to  H^1(S, \cAut(\cF))$ in cohomology. Under this map, the image of $[T]$ corresponds to the class $[\cIsom(S',T)\we^\cA g_*(g^*(\cF))]$.
\end{enumerate}
\end{remarks}

For easier reference, we explicitly state Lemma~\ref{lem_twist_sheaf-gen} for the case of finite \'etale covers.

\begin{cor}\label{lem_twist_sheaf} 
Let $f\colon T \to S$ be a degree $d$ \'etale cover. Let $\cF \colon \Sch_S \to \Sets$ be any sheaf and equip $\cF^d$ with the left action of $\SS_d$ by permutations. Then, there is a canonical isomorphism of sheaves of sets
\[
 \cIsom(S^{\sqcup d},T)\wedge^{\SS_d} \cF^d \iso f_*(\cF|_T).
\]
\end{cor}

\begin{proof} We apply Lemma~\ref{lem_twist_sheaf-gen} with $g \colon S^{\sqcup d} \to S$ the standard degree $d$ \'etale cover. In this case $\cA = \SS_d$, $g_*(g^*(\cF))=\cF^d$, and $u\colon \cA \to \cAut(g_*(g^*(\cF)))$ causes $\SS_d$ to act on $\cF^d$ on the left by permutations. 
\end{proof}
\sm

We now fix a finite \'etale cover $f\colon T\to S$ of degree $d$ and recall \cite[III, Thm.~4.3(a)]{M}: Since $\SS_d$ is an affine group scheme, the $\SS_d$--torsor $\cIsom(S^{\sqcup d},T)$ is representable by an $S$--scheme $\widetilde{T}\to S$ which is therefore also an $\SS_d$--torsor. For any scheme $Z$ which has a left action of $\SS_d$, the associated sheaf $h_Z = \Hom_S(\und,Z)$ also has a left action of $\SS_d$. If the sheaf  $\cIsom(S^{\sqcup d},T)\wedge^{\SS_d} \Hom_S(\und,Z)$ is representable, we denote the representing scheme by $\widetilde{T}\wedge_S^{\SS_d} Z$. For example, for $X\in \Sch_S$ we have $\widetilde{T} \wedge_S^{\SS_d} X^{\sqcup d} \cong X\times_S T$.

\begin{cor}\label{lem_weil}
Let $Y\to S$ be an $S$--scheme such that the Weil restriction $R_{T/S}(Y\times_S T)$ exists as a scheme. Then, using the notation above, there is an isomorphism
\[
\widetilde{T}\wedge_S^{\SS_d} Y^d \iso R_{T/S}(Y\times_S T),
\]
where $\SS_d$ permutes the factors of $Y^d=Y\times_S \ldots \times_S Y$.
\end{cor}
\begin{proof}
We apply Lemma \ref{lem_twist_sheaf} to $h_Y=\Hom_{S}(\und,Y)$. The sheaf $h_Y^d$ is represented by the scheme $Y^d$. Applying the lemma we get an isomorphism
\[
\cIsom(S^{\sqcup d},T)\wedge^{\SS_d} h_Y^d \iso f_*(h_Y|_T).
\]
We have $f_*(h_Y|_T) = \Hom_T(\und\times_S T, Y\times_S T)$ and by definition the Weil restriction $R_{T/S}(Y\times_S T)$ exists if and only if this sheaf is representable, in which case it is represented by the Weil restriction. Hence by assumption these sheaves are representable, and so the left hand sheaf is represented by $\widetilde{T}\wedge_S^{\SS_d} Y^d$. The claimed isomorphism then follows from the Yoneda Lemma.
\end{proof}
\begin{remark}
Since $T\to S$ is a finite \'etale cover, equivalently since $\widetilde{T} \to S$ is a Galois cover, the discussion at the end of \cite[6.2 B]{BLR} says that a sufficient condition for $R_{T/S}(Y\times_S T)$ to exist is that the morphism $Y\times_S T \to T$ be quasi-projective. This will occur if $Y\to S$ is quasi-projective by \cite[Tag 0B3G]{Stacks}.
\end{remark}

\section{Cohomology of Semi-direct Products}\label{app_semi_direct}
In this appendix we present a general construction for the gerbe of $\bH \rtimes \bK$--torsors for the semi-direct product of two sheaves of groups $\bH$ and $\bK$ on $\Sch_S$. Our construction presents $\bH\rtimes\bK$--torsors as a gerbe over the intermediate stack of $\bK$--torsors. This point of view then recovers the fiberwise analysis of $\bH\rtimes\bK$--torsors appearing in \cite{MS} and also recovers cohomological statements about semi-direct products, for example the statement \cite[2.6.3(ii)]{Gi}, which we discuss in Remark \ref{remark_gille_cohom}.

In the main portion of this paper, we are concerned with semi-direct products of groups as they appear in the Segre homomorphism and related maps surrounding the norm functor. For this reason, we end this appendix by focusing on semi-direct products with the group $\SS_d$, the constant group sheaf associated to the abstract permutation group on $d$ letters, and we provide a more concrete description of the stack of $\bG^d \rtimes \SS_d$--torsors using \'etale covers which lends itself to our work on the norm functor.

\subsection{The General Construction}\label{general_semi_direct_products}
We consider a sheaf of groups of the form $\bH\rtimes\bK$ on $\Sch_S$. We write $\alpha \colon \bK \to \bAut(\bH)$ for the homomorphism such that
\[
k\cdot h \cdot k^{-1} = \alpha_k(h)
\]
inside $\bH\rtimes\bK$ for appropriate sections $k\in \bK$ and $h\in \bH$.

We have the gerbe of $\bK$--torsors, $\fTors(\bK)\to \Sch_S$, which we recall has
\begin{enumerate}[label={\rm (\roman*)}]
\item objects $(T,\cP)$ where $T\in \Sch_S$ and $\cP$ is a $\bK|_T$--torsor,
\item morphisms $(g,\varphi)\colon (T_1,\cP_1) \to (T_2,\cP_2)$ where $g\colon T_1 \to T_2$ is a morphism in $\Sch_S$ and $\varphi \colon \cP_1 \iso \cP_2|_{T_1}$ is a morphism of $\bK|_{T_1}$--torsors, and
\item structure functor $\fTors(\bK) \to \Sch_S$ given by $(T,\cP)\mapsto T$ and $(g,\varphi)\mapsto g$.
\end{enumerate}
We denote the trivial torsor by $\overline{\bK}$ and likewise for other groups below. The stack $\fTors(\bK)\to \Sch_S$ inherits the structure of a site as in \cite[Tag \href{https://stacks.math.columbia.edu/tag/06NU}{06NU}]{Stacks}. In particular, the coverings are families of the form
\[
\{(g_i,\varphi_i)\colon (T_i,\cP_i) \to (T,\cP)\}_{i\in I}
\]
where $\{g_i\colon T_i \to T\}_{i\in I}$ is a covering in $\Sch_S$.

We define a functor of groups $\bH'\colon \fTors(\bK)\to \Grp$ on this new site $\fTors(\bK)$. On objects, we define
\[
\bH'(T,\cP) = (\cP \wedge^{\bK|_T} \bH|_T)(T)
\]
where $\bK|_T$ acts on $\bH|_T$ through $\alpha\colon \bK \to \bAut(\bH)$. Since this means that $\bK|_T$ acts on $\bH|_T$ by group automorphisms, the twisted object $\cP\wedge^{\bK|_T} \bH|_T$ is also a sheaf of groups over $T$ and $\bH'(T,\cP)$ are simply the global sections of this twisted group. Next, let $(g,\varphi)\colon (T_1,\cP_1)\to (T_2,\cP_2)$ be a morphism in $\fTors(\bK)$. We note that
\[
(\cP_2\wedge^{\bK|_{T_2}} \bH|_{T_2})|_{T_1} = \cP_2|_{T_1} \wedge^{\bK|_{T_1}} \bH|_{T_1}
\]
and thus
\[
(\cP_2\wedge^{\bK|_{T_2}} \bH|_{T_2})(T_1) = (\cP_2|_{T_1} \wedge^{\bK_{T_1}} \bH|_{T_1})(T_1).
\]
We use this to define $\bH'(g,\varphi)$ to be the composition
\begin{align*}
\bH'(T_2,\cP_2)&=(\cP_2\wedge^{\bK|_{T_2}} \bH|_{T_2})(T_2) \\
&\xrightarrow{\textrm{res}_g} (\cP_2\wedge^{\bK|_{T_2}} \bH|_{T_2})(T_1) \\
&= (\cP_2|_{T_1} \wedge^{\bK|_{T_1}} \bH|_{T_1})(T_1) \\
&\xrightarrow{\varphi^{-1} \wedge \Id} (\cP_1 \wedge^{\bK|_{T_1}} \bH|_{T_1})(T_1) \\
&= \bH'(T_1,\cP_1).
\end{align*}

\begin{lem}
$\bH' \colon \fTors(\bK) \to \Grp$ is a sheaf.
\end{lem}
\begin{proof}
For any cover $\{(T_i,\cP_i) \to (T,\cP)\}_{i\in I}$, the sheaf property for $\bH'|_{(T,\cP)}$ follows from the sheaf property for $\cP\wedge^{\bK|_T}\bH|_T$ over the cover $\{T_i \to T\}_{i\in I}$ of $\Sch_S$.
\end{proof}

We now briefly use the general notion of a torsor on a site as in \cite[Tag 03AH]{Stacks} to consider the gerbe of $\bH'$--torsors over $\fTors(\bK)$, denoted $\fTors(\bH') \to \fTors(\bK)$. It has
\begin{enumerate}[label={\rm (\roman*)}]
\item objects $((T,\cP),\cQ)$, where $(T,\cP)\in \fTors(\bK)$ and $\cQ$ is an $\bH'|_{(T,\cP)}$--torsor,
\item morphisms $((g,\varphi),\psi) \colon ((T_1,\cP_1),\cQ_1) \to ((T_2,\cP_2),\cQ_2)$ where the pair $(g,\varphi) \colon (T_1,\cP_1) \to (\cT_2,\cP_2)$ is a morphism in $\fTors(\bK)$ and $\psi \colon \cQ_1 \iso \cQ_2|_{(T_1,\cP_1)}$ is a morphism of $\bH'|_{(T_1,\cP_1)}$--torsors, and
\item structure functor $((T,\cP),\cQ) \mapsto (T,\cP)$ and $((g,\varphi),\psi) \mapsto (g,\varphi)$.
\end{enumerate}

Now that we have a gerbe over $\fTors(\bK)$, which is itself a gerbe over $\Sch_S$, it is natural to consider their composite. It turns out that this composition is the gerbe we are after.
\begin{lem}\label{composite_gerbe}
The composite functor $\fTors(\bH') \to \fTors(\bK) \to \Sch_S$ is a gerbe.
\end{lem}
\begin{proof}
This is a special case of \cite[Tag 06R3]{Stacks}. We include a proof for the reader's convenience. Both of the stacks $\fTors(\bH') \to \fTors(\bK)$ and $\fTors(\bK) \to \Sch_S$ are gerbes, so the composite $\fTors(\bH') \to \Sch_S$ is a stack fibered in groupoids by \cite[Tag 09WX]{Stacks}. To see it is a gerbe, we check the remaining two conditions. First, the fiber over any $T\in \Sch_S$ is non-empty since it contains at least the trivial object $((T,\overline{\bK|_T}),\overline{\bH'}|_{(T,\overline{\bK|_T})})$. Second, any two objects in the same fiber are also locally isomorphic, since if we are given $((T,\cP_1),\cQ_1)$ and $((T,\cP_2),\cQ_2)$, we may first find a cover $\{T_i\to T\}_{i\in I}$ over which $\cP_1|_{T_i} \cong \cP_2|_{T_i}$, and then because $\fTors(\bH')\to\fTors(\bK)$ is a gerbe, we may find refined covers $\{X_{ij} \to T_i\}_{j\in J_i}$ over which
\[
((T_i,\cP_1|_{T_i}),\cQ_1|_{(T_i,\cP_1|_{T_i})})|_{(X_{ij},\cP_1|_{X_{ij}})} \cong ((T_i,\cP_2|_{T_i}),\cQ_2|_{(T_i,\cP_2|_{T_i})})|_{(X_{ij},\cP_2|_{X_{ij}})}.
\]
Cleaning up notation, this means that $((T,\cP_1),\cQ_1)|_{X_{ij}} \cong ((T,\cP_2),\cQ_2)|_{X_{ij}}$ over the cover $\{X_{ij}\to T\}_{i\in I,j\in J_i}$. Thus, $\fTors(\bH')\to \Sch_S$ is a gerbe.
\end{proof}

\begin{lem}\label{Hprime_gerbe_equiv}
The gerbe $\fTors(\bH') \to \Sch_S$ is equivalent to the gerbe $\fTors(\bH\rtimes \bK) \to \Sch_S$ of $\bH\rtimes \bK$--torsors over $\Sch_S$.
\end{lem}
\begin{proof}
Our desired equivalence will come from Proposition \ref{lem_gerbes_and_torsors}\ref{lem_gerbes_and_torsors_iv} after we identify that the automorphism sheaf of the split object
\[
\cP_0 = ((S,\overline{\bK}),\overline{\bH'|_{(S,\overline{\bK})}}) \in \fTors(\bH')(S)
\]
is isomorphic to $\bH\rtimes \bK$. This is because $\fTors(\bH')$ is a gerbe and so $\fForms(\cP_0) = \fTors(\bH')$ by Remark \ref{rem_gerbe_forms}. So, now we identify $\bAut(\cP_0)$.

Let $T\in \Sch_S$ and consider a section $((\Id_T,\varphi),\psi) \in \bAut(\cP_0)(T)$. This means that $\varphi \colon \overline{\bK}|_T \to \overline{\bK}|_T$ is a $\bK|_T$--torsor automorphism of the trivial torsor and is therefore left multiplication by some element $k\in \bK(T)$. The morphism $\psi$ is an isomorphism
\[
\psi \colon \overline{\bH'|_{(T,\overline{\bK}|_T)}} \iso (\Id_T,\varphi)^*(\overline{\bH'|_{(T,\overline{\bK}|_T)}}),
\]
but any pullback of the trivial torsor is isomorphic to the trivial torsor, so without loss of generality we may assume that $\psi$ is an automorphism of $\overline{\bH'|_{(T,\overline{\bK}|_T)}}$ and is therefore left multiplication by an element of $\bH'(T,\overline{\bK}|_T)$. However, we have that
\[
\bH'(T,\overline{\bK}|_T) = (\overline{\bK}|_T \wedge^{\bK|_T} \bH|_T)(T) = \bH(T).
\]
Any such pair of choices $k\in \bK(T)$ and $h\in \bH(T)$ can be used to define a distinct section of $\bAut(\cP_0)$ and so it is clear that
\[
\bAut(\cP_0) \cong \bH \times \bK \colon \Sch_S \to \Sets
\]
as sheaves of sets. We are now left to identify the group structure. Let $k_1,k_2 \in \bK(T)$, $h_1,h_2\in \bH(T)$, write $\varphi_1,\varphi_2$ for the respective automorphisms of $\overline{\bK}|_T$, and write $\psi_1,\psi_2$ for the respective automorphisms of $\overline{\bH'|_{(T,\overline{\bK}|_T)}}$. The group product is given by the morphism composition
\[
((\Id_T,\varphi_1),\psi_1)\circ ((\Id_T,\varphi_2),\psi_2) = ((\Id_T,\varphi_1\circ \varphi_2),(\Id_T,\varphi_2)^*(\psi_1)\circ \psi_2)
\]
so we need to identify the element of $\bH(T)$ corresponding to the automorphism $(\Id_T,\varphi_2)^*(\psi_1)$. Since $\psi_1$ is left multiplication by $h_1$, the pullback of $\psi_1$ will be left multiplication by the image of $h_1$ along the restriction map
\[
\bH'(\Id_T,\varphi_2) \colon \bH'(T,\overline{\bK}|_T) \to \bH'(T,\overline{\bK}|_T).
\]
Since $\overline{\bK}|_T$ is the trivial torsor, the presheaf underlying $(\overline{\bK}|_T \wedge^{\bK|_T} \bH|_T)$ is already a sheaf, so we may use the commutative diagram
\[
\begin{tikzcd}
(\overline{\bK(T)}\times \bH(T))/\sim \ar[r] \ar[d,"\varphi_2^{-1}\times\Id"] & \bH(T) \ar[d,"{\bH'(\Id_T,\varphi_2)}"] \\
(\overline{\bK(T)}\times \bH(T))/\sim \ar[r] & \bH(T)
\end{tikzcd}
\]
where the symbols $\sim$ denote the equivalence relation defining the contracted product and the horizontal maps are bijections. The inverse of $\phi_2$ appears in this diagram since the restriction maps in $\bH'$ were defined to use $\varphi_2^{-1}\wedge \Id$. Starting with $h_1$ in the top right, we obtain
\[
\begin{tikzcd}
{[(1,h_1)]} \ar[d,mapsto] & h_1 \ar[l,mapsto] \ar[d,mapsto] \\
{[(k_2^{-1},h_1)]} \ar[r,mapsto] & \alpha_{k_2^{-1}}(h_1)
\end{tikzcd}
\]
because the equivalence relation causes $\bK$ to act on $\bH$ via $\alpha \colon \bK \to \bAut(\bH)$. Therefore, writing the product in terms of the sections of $\bK$ and $\bH$, we have
\[
(k_1,h_1)\cdot (k_2,h_2) = (k_1k_2,\alpha_{k_2^{-1}}(h_1)h_2)
\]
However, $\alpha$ defines the structure of the semidirect product $\bH\rtimes \bK$, so it is clear that we have a group isomorphism $\bAut(\cP_0) \cong \bH\rtimes \bK$ as desired, finishing the proof.
\end{proof}

\subsection{A Concrete Equivalence $\fTors(\bH')\equiv \fTors(\bH\rtimes\bK)$}\label{remark_gille_cohom}
We may also view an object $((T,\cP),\cQ) \in \fTors(\bH')$ as a pair of torsors over $T$. The $\bH'|_{(T,\cP)}$--torsor $\cQ$ defines a $(\cP\wedge^{\bK|_T} \bH|_T)$--torsor $\cW$ over $T$ by setting
\[
\cW(X) = \cQ(X,\cP|_X)
\]
for schemes $X\in \Sch_T$. 

Conversely, given a $(\cP\wedge^{\bK|_T} \bH|_T)$--torsor $\cW$ over $T$, we may define a functor
\begin{align*}
\cQ \colon \fTors(\bK)_{(T,\cP)} &\to \Sets \\
((X,\cP') \to (T,\cP)) &\mapsto \cW(X)
\end{align*}
whose restriction map along a morphism $(g,\varphi)\colon (X_1,\cP_1')\to (X_2,\cP_2')$ is simply the restriction $\cW(g)\colon \cW(X_2)\to \cW(X_1)$ in $\cW$. Since the covers in the site $\fTors(\bK)$ are of the form $\{(X_i,\cP_i') \to (X,\cP)\}_{i\in I}$ where $\{X_i \to X\}_{i\in I}$ is a cover in $\Sch_S$ and otherwise there are no additional conditions on the $\cP_i'$, this naive definition of $\cQ$ will be a sheaf on $\fTors(\bK)_{(T,\cP)}$. To give it an action of $\bH'|_{(T,\cP)}$, we do the following. Given $(g,\varphi) \colon (X,\cP') \to (T,\cP)$ in $\fTors(\bK)$, the isomorphism $\varphi \colon \cP' \to \cP|_X$ induces an isomorphism
\[
(\cP' \wedge^{\bK|_X} \bH|_X)(X) \iso (\cP|_X \wedge^{\bK|_X} \bH|_X)(X) = (\cP\wedge^{\bK|_T}\bH|_T)(X).
\]
The final group acts on $\cW(X)$ through its torsor structure, and so we define the $\bH'|_{(T,\cP)}(X,\cP') = (\cP' \wedge^{\bK|_X} \bH|_X)(X)$ action on $\cW(X)$ through the above isomorphism. This produces an $\bH'|_{(T,\cP)}$--torsor.

Up to canonical isomorphism, these two constructions are inverses to each other. Therefore, $\fTors(\bH')$ is equivalent to the gerbe consisting of
\begin{enumerate}[label={\rm (\roman*)}]
\item objects of the form $(T,\cP,\cW)$ where $T\in \Sch_S$, $\cP$ is a $\bK|_T$--torsor, and $\cW$ is a $(\cP\wedge^{\bK|_T}\bH|_T)$--torsor, also over $T$.
\end{enumerate}
To discuss morphisms, when given a $(\cP\wedge^{\bK|_T}\bH|_T)$--torsor $\cW$ and a morphism $\varphi \colon \cP' \to \cP$ of $\bK|_T$--torsors, we write $\varphi^*(\cW)$ for the $(\cP'\wedge^{\bK|_T}\bH|_T)$--torsor obtained from $\cW$ by acting on it through the induced isomorphism
\[
\varphi\wedge \Id \colon \cP'\wedge^{\bK|_T}\bH|_T \iso \cP\wedge^{\bK|_T}\bH|_T.
\]
Thus, this gerbe has
\begin{enumerate}[label={\rm (\roman*)}]
\setcounter{enumi}{1}
\item morphisms of the form $(g,\varphi,\psi)\colon (T',\cP',\cW')\to (T,\cP,\cW)$ where $g \colon T' \to T$ is in $\Sch_S$ and $\varphi \colon \cP' \iso g^*(\cP)$ is a $\bK|_{T'}$--torsor isomorphism as usual, and $\psi$ is a $(\cP'\wedge^{\bK|_{T'}} \bH|_{T'})$--torsor isomorphism
\[
\psi \colon \cW' \to \varphi^*(g^*(\cW)).
\]
\end{enumerate}

We may now use this description to give a more concrete equivalence of gerbes $\fTors(\bH\rtimes \bK) \to \fTors(\bH')$. We fix the notation
\begin{align*}
i \colon \bH &\inj \bH\rtimes \bK \\
s \colon \bK &\inj \bH\rtimes \bK \\
p \colon \bH\rtimes\bK &\surj \bK
\end{align*}
for the two canonical inclusions and the canonical projection onto $\bK$, respectively. Consider an object $(T,\cE) \in \fTors(\bH\rtimes\bK)$, so $T\in \Sch_S$ and $\cE$ is a $(\bH\rtimes\bK)|_T$--torsor. We want to associate this with an object $(T,\cP,\cW) \in \fTors(\bH')(T)$. We set
\[
\cP = p_*(\cE) = \cE\wedge^{(\bH\rtimes\bK)|_T}\overline{\bK}|_T,
\]
which is a $\bK|_T$--torsor as required. Now, we have three different left actions of $\bK$, namely
\begin{enumerate}[label={\rm(\roman*)}]
\item $\bK$ on itself via conjugation,
\item $\bK$ on $\bH\rtimes\bK$ via conjugation (after the inclusion $s\colon \bK \to \bH\rtimes \bK)$, and
\item $\bK$ on $\bH$ via $\alpha \colon \bK \to \bAut(\bH)$,
\end{enumerate}
and the split short exact sequence
\[
\begin{tikzcd}
1 \ar[r] & \bH \ar[r,"i"] & \bH\rtimes\bK \ar[r,"p"] & \bK \ar[r] \ar[l,swap,bend right=50,"s"] & 1
\end{tikzcd}
\]
is equivariant with respect to these actions. Therefore, we have a twisted split exact sequence
\begin{equation}\label{eq_twisted_ses}
\begin{tikzcd}[column sep=2.5ex]
1 \ar[r] & \cP\wedge^{\bK|_T}\bH|_T \ar[r,"i'"] & \cP\wedge^{\bK|_T}(\bH\rtimes\bK)|_T \ar[r,"p'"] & \cP\wedge^{\bK|_T}\bK|_T \ar[r] \ar[l,swap,bend right=30,"s'"] & 1.
\end{tikzcd}
\end{equation}
Here we note that
\begin{align*}
\cP\wedge^{\bK|_T}\bK|_T &\cong \bAut_{\fTors(\bK|_T)}(\cP)\text{, and} \\
\cP\wedge^{\bK|_T}(\bH\rtimes\bK)|_T &\cong \bAut_{\fTors((\bH\rtimes\bK)|_T)}(s_*(\cP)).
\end{align*}
Now we consider the sheaf $\cIsom_{\fTors((\bH\rtimes\bK)|_T)}(s_*(\cP),\cE)$, which is a torsor for the group $\bAut_{\fTors((\bH\rtimes\bK)|_T)}(s_*(\cP))$. Since $p\circ s = \Id_{\bK}$ we know that $p_*(s_*(\cP))=\cP$, and we defined $\cP = p_*(\cE)$, therefore we have a map
\begin{align*}
f\colon \cIsom_{\fTors((\bH\rtimes\bK)|_T)}(s_*(\cP),\cE) &\to \bAut_{\fTors(\bK|_T)}(\cP) \\
\varphi &\mapsto p_*(\varphi).
\end{align*}
Let $\cW$ be the kernel of $f$, i.e.,
\[
\cW(T') = \{\varphi \in \cIsom_{\fTors((\bH\rtimes\bK)|_T)}(s_*(\cP),\cE)(T') \mid p_*(\varphi)=\Id_\cP\}
\]
for $T'\in \Sch_T$. We claim that $\cW$ is a $(\cP\wedge^{\bK|_T}\bH|_T)$--torsor. Because $\cP\wedge^{\bK|_T}\bH|_T$ is the kernel of the map $p'$, it is clear that it has a right action on $\cW$ via the right action of $\cP\wedge^{\bK|_T}(\bH\rtimes\bK)|_T$ on $\cIsom_{\fTors((\bH\rtimes\bK)|_T)}(s_*(\cP),\cE)$. Furthermore, the sequence
\[
\begin{tikzcd}
1 \ar[r] & \cW \ar[r] & \cIsom_{\fTors((\bH\rtimes\bK)|_T)}(s_*(\cP),\cE) \ar[r] & \bAut_{\fTors(\cK)}(\cP) \ar[r] & 1
\end{tikzcd}
\]
is locally isomorphic to \eqref{eq_twisted_ses} (as sheaves of sets) and therefore $\cW$ is locally isomorphic to $\overline{\cP\wedge^{\bK|_T}\bH|_T}$. Thus $\cW$ is a $(\cP\wedge^{\bK|_T}\bH|_T)$--torsor as required and we have an object
\[
(T,\cP,\cW) \in \fTors(\bH')(T).
\]
Conversely, say we start with an object $(T,\cP,\cW) \in \fTors(\bH')$. Using $\cP$ we still have the sequence \eqref{eq_twisted_ses}. Therefore, $i'_*(\cW)$ is a $\cP\wedge^{\bK|_T}(\bH\rtimes\bK)|_T$--torsor. Since $\cP\wedge^{\bK|_T}(\bH\rtimes\bK)|_T \cong \bAut_{\fTors((\bH\rtimes\bK)|_T)}(s_*(\cP))$ we can construct
\[
\cE = i'_*(\cW) \wedge^{\bAut(s_*(\cP))}s_*(\cP)
\]
which is an $(\bH\rtimes\bK)|_T$--torsor since $s_*(\cP)$ is. Thus we have an object $(T,\cE)\in \fTors(\bH\rtimes\bK)$.

We claim that these two constructions are quasi-inverse to each other. Indeed, take $(T,\cE)\in \fTors(\bH\rtimes\bK)$ as we started with and construct $(T,\cP,\cW)$. We can construct a map
\[
i'_*(\cW) = \cW \wedge^{i'} \bAut(s_*(\cP)) \to \cIsom_{\fTors((\bH\rtimes\bK)|_T)}(s_*(\cP),\cE)
\]
by defining it (on the presheaf determining the contracted product) to simply be $(w,\varphi) \mapsto w\cdot \varphi$ whenever $\cW$ has a point. We may do so since $\cW$ is a subsheaf of $\cIsom_{\fTors((\bH\rtimes\bK)|_T)}(s_*(\cP),\cE)$ which has a right action of $\bAut(s_*(\cP))$. It is then clear that it is an isomorphism. In turn, it is clear that
\begin{align*}
&\; i'_*(\cW) \wedge^{\bAut(s_*(\cP))}s_*(\cP) \\
\cong&\; \cIsom_{\fTors((\bH\rtimes\bK)|_T)}(s_*(\cP),\cE) \wedge^{\bAut(s_*(\cP))}s_*(\cP) \\
\cong&\; \cE.
\end{align*}

Conversely, say we have started with $(T,\cP,\cW)\in \fTors(\bH')$ and we have constructed $\cE = i'_*(\cW)\wedge^{\bAut(s_*(\cP))} s_*(\cP)$. We first check that $p_*(\cE)\cong \cP$. We use that contracted products are associative, see \cite[III, 1.3.5]{Gir}, to write
\begin{align*}
p_*(\cE) &= \big(i'_*(\cW)\wedge^{\bAut(s_*(\cP))} s_*(\cP)\big)\wedge^{(\bH\rtimes\bK)|_T} \overline{\bK}|_T \\
&\cong i'_*(\cW)\wedge^{\bAut(s_*(\cP))} \big(s_*(\cP)\wedge^{(\bH\rtimes\bK)|_T} \overline{\bK}|_T\big) \\
&= i'_*(\cW)\wedge^{\bAut(s_*(\cP))} p_*(s_*(\cP)) \\
&\cong i'_*(\cW)\wedge^{\bAut(s_*(\cP))} \cP
\end{align*}
where $\bAut_{\fTors((\bH\rtimes\bK)|_T)}(s_*(\cP))$ acts on $\cP$ on the left via the homomorphism
\[
p' \colon \bAut_{\fTors((\bH\rtimes\bK)|_T)}(s_*(\cP)) \to \bAut_{\fTors(\bK|_T)}(\cP)
\]
appearing in \eqref{eq_twisted_ses}. However, we then have that
\begin{align*}
i'_*(\cW)\wedge^{\bAut(s_*(\cP))} \cP &= \big(\cW \wedge^{(\cP\wedge^{\bK|_T}\bH|_T)} \overline{\bAut(s_*(\cP))}\big)\wedge^{\bAut(s_*(\cP))} \cP \\
&\cong \cW \wedge^{(\cP\wedge^{\bK|_T}\bH|_T)} \cP
\end{align*}
where the left action on $\cP$ is through the map
\[
p'\circ i' \colon \cP\wedge^{\bK|_T}\bH|_T \to \bAut_{\fTors(\bK|_T)}(\cP),
\]
which is the identity. Hence, this twisting does not modify $\cP$ and we obtain that
\[
p_*(\cE) \cong \cP.
\]
Next, we can consider the map of sheaves of sets
\[
\psi \colon \cW \to \cIsom(s_*(\cP),\cE)
\]
defined as follows. For $X\in \Sch_T$ such that $\cW(X)\neq \O$, we have that
\[
i_*'(\cW)(X) = \big(\cW(X)\times \overline{\bAut_{\fTors((\bH\rtimes\bK)|_T)}(s_*(\cP))}(X)\big)/\sim
\]
and this has a point, so in turn
\[
\cE(X) = \big(i_*'(\cW)(X)\times s_*(\cP)(X)\big)/\sim.
\]
Thus, for such $X$ we can consider elements of $\cE(X)$ as equivalence classes of the form $[[w,g],x]$ with $w\in \cW(X)$, $g\in \bAut_{\fTors((\bH\rtimes\bK)|_T)}(s_*(\cP))(X)$, and $x \in s_*(\cP)(X)$. Now, for $w\in \cW(X)$, we set
\[
\psi(w) \colon s_*(\cP)|_X \to \cE|_X
\]
to be the map which behaves over $X'\in \Sch_X$ by
\begin{align*}
\psi(w)(X') \colon s_*(\cP)(X') &\to \cE(X') \\
x &\mapsto [[w|_{X'},1],x].
\end{align*}
Thus, $\psi(w) \in \cIsom(s_*(\cP),\cE)(X)$ and we have defined a map $\psi\colon \cW \to \cIsom(s_*(\cP),\cE)$.

We claim that $\psi$ maps into the kernel of $p_*\colon \cIsom(s_*(\cP),\cE) \to \bAut_{\fTors(\bK|_T)}(\cP)$. Elements of $s_*(\cP) = \cP\wedge^{\bK|_T} \overline{\bH\rtimes\bK}|_T$ are locally equivalence classes $[x,g]$ where $x\in\cP$ and $g\in \bH\rtimes\bK$. The isomorphism
\[
p_*(s_*(\cP)) = s_*(\cP)\wedge^{(\bH\rtimes\bK)|_T} \bK|_T \iso \cP
\]
is locally of the form $[[x,g],k] \mapsto x\cdot p(g)k$. Likewise, the isomorphism
\[
p_*(\cE) = \cE\wedge^{(\bH\rtimes\bK)|_T} \bK|_T \iso \cP
\]
is locally of the form
\[
[[[w,\gamma],[x,g]],k] \mapsto p'(i'(\gamma))(x\cdot p(g)k),
\]
where $\gamma \in \bAut_{\fTors((\bH\rtimes\bK)|_T)}(s_*(\cP))$ and all other sections belong where their notation suggests. Note that this is well-defined and that the $w\in \cW$ does not play a role. This is because any two $w_1,w_2\in \cW$ differ by some $h\in \cP\wedge^{\bK|_T}\bH|_T$, and so we have
\begin{align*}
[[[w_1,\gamma],[x,g]],k] &= [[[w_2h,\gamma],[x,g]],k] \\
&= [[[w_2,h\gamma],[x,g]],k]
\end{align*}
but then $p'(i'(h\gamma))=p'(i'(\gamma))$ since $i'(h)$ is the kernel of $p'$. Taking these isomorphisms into account, the pushforward $p_*(\psi(w))$ of one of the maps we constructed above behaves locally as
\begin{align*}
x\cdot p(g)k \cong [[x,g],k] \mapsto& [[[w|_{X'},1],[x,g]],k] \\
&\cong p'(i'(1))\big(x\cdot p(g)k\big) \\
&= x\cdot p(g)k.
\end{align*}
This shows that $p_*(\psi(w)) = \Id_\cP$ globally as well, as claimed.

Finally, we claim that $\psi$ is equivariant with respect to the right $\cP\wedge^{\bK|_T}\bH|_T$--action on $\cW$ and the left $\cP\wedge^{\bK|_T}\bH|_T$--action on $\cIsom(s_*(\cP),\cE)$ via $i'$. Indeed, for $\varphi \in (\cP\wedge^{\bK|_T}\bH|_T)(X)$ and $p\in s_*(\cP)(X')$ we have that
\begin{align*}
\psi(w\cdot \varphi)(p) &= [[(w\cdot\varphi)|_{X'},1],p] \\
&= [[w|_{X'},\varphi|_{X'}],p] \\
&= [[w|_{X'},1],i'(\varphi)(p)] \\
&= (\psi(w)\circ i'(\varphi)\big)(p)
\end{align*}
and so $\psi(w\cdot \varphi) = \psi(w)\cdot \varphi$, as claimed. Hence $\psi \colon \cW \to \Ker(p_*)$ is a morphism of $\cP\wedge^{\bK|_T}\bH|_T$--torsors, and thus an isomorphism. In this was, we also recover $\cW$ from $\cE$, and so we have recovered the original triple $(T,\cP,\cW)$. We leave it to the reader to trace the journey of a morphism through this story.

From this point of view we also recover some cohomological descriptions. Viewing cohomology sets as isomorphism classes as in Proposition \ref{lem_gerbes_and_torsors}\ref{lem_gerbes_and_torsors_iii}, for a fixed $(S,\cP)\in \fTors(\bK)$ we have maps
\[
\begin{tikzcd}[row sep=2pt]
H^1(S,\cP\wedge^\bK \bH) \ar[r] & H^1(S,\bH\rtimes \bK) \ar[r] & H^1(S,\bK) \\
{[(S,\cW)]} \ar[r,mapsto] & {[(S,\cP,\cW)]} \ar[r,mapsto] & {[(S,\cP)]}.
\end{tikzcd}
\]
The first map is need not be injective, since we may have non-isomorphic $(\cP\wedge^\bK \bH)$--torsors $\cW_1$ and $\cW_2$ such that $\cW_1 \cong \varphi^*(\cW_2)$ for some automorphism of $\cP$, and so $(S,\cP,\cW_1)\cong (S,\cP,\cW_2)$ in $\fTors(\bH\rtimes \bK)$. However, this is the only way that two classes have isomorphic images and thus by letting the group $\bAut(\cP)(S)$ act on the set of isomorphism classes $H^1(S,\cP\wedge^\bK \bH)$ by $\varphi\cdot [(S,\cW)] = [(S,\varphi^*(\cW))]$, we obtain that
\[
H^1(S,\cP\wedge^\bK \bH)/\bAut(\cP)(S) \inj H^1(S,\bH\rtimes \bK).
\]
This describes the fiber of $H^1(S,\bH\rtimes \bK)$ over $[(S,\cP)] \in H^1(S,\bK)$, and thus we have a decomposition
\[
H^1(S,\bH\rtimes \bK) = \bigsqcup_{[(S,\cP)] \in H^1(S,\bK)} H^1(S,\cP\wedge^\bK \bH)/\bAut(\cP)(S).
\]
Using the notation
\begin{align*}
\bK^\cP &= \cP\wedge^\bK \bK = \bAut(\cP) \\
\bH^\cP &= \cP \wedge^\bK \bH\text{, and}\\
H^0(S,\bK^\cP)&=\bAut(\cP)(S), 
\end{align*}
this is the decomposition appearing in \cite[2.6.3(2)]{Gi}.

\subsection{$\bG^d \rtimes \SS_d$ and \'Etale Covers}\label{permutation_semi_direct}
Let $\bG \colon \Sch_S \to \Grp$ be a sheaf of groups. Consider the semi-direct product $\bG^d \rtimes \SS_d$ defined by
\[
((g_1,\ldots,g_d),1)(1,\sigma) = (1,\sigma)((g_{\sigma(1)},\ldots,g_{\sigma(d)}),1)
\]
for $\sigma \in \SS_d$ and $(g_1,\ldots,g_d) \in\bG^d$. By the results of Lemmas \ref{composite_gerbe} and \ref{Hprime_gerbe_equiv}, we know that the gerbe $\fTors(\bG^d \rtimes \SS_d)$ factors as
\[
\fTors(\bG^d \rtimes \SS_d) \to \fTors(\SS_d) \to \Sch_S.
\]
Our goal here is to provide a more concrete description $\fTors(\bG^d\rtimes \SS_d)$. The intermediate gerbe, $\fTors(\SS_d)$, is equivalent to the stack of degree $d$ \'etale extensions. In detail, consider the fibered category $F\acute{E}t_d$ where
\begin{enumerate}[label={\rm (\roman*)}]
\item objects are degree $d$ \'etale covers $T \to X$ for some $X\in \Sch_S$,
\item morphisms are pairs $(f,g)\colon (T'\to X')\to (T\to X)$ where $f$ and $g$ are scheme morphisms making
\[
\begin{tikzcd}
T' \ar[r,"g"] \ar[d] & T \ar[d] \\
X' \ar[r,"f"] & X
\end{tikzcd}
\]
a pullback diagram of schemes, and
\item the structure functor simply sends $(T\to X) \mapsto X$ and $(f,g)\mapsto f$.
\end{enumerate}
We argue that $F\acute{E}t_d$ is a gerbe. It is immediate that it is a stack since it is clearly a substack of the stack defined in \cite[Tag 0BLY]{Stacks}. Next, $F\acute{E}t_d$ is fibered in groupoids since if $f=\Id_X$, then $g$ must be an isomorphism in order to have a pullback diagram. Second, since finite \'etale covers of degree $d$ are all locally isomorphic to $S^{\sqcup d}\to S$, we see $F\acute{E}t_d$ is a gerbe. Then since $\bAut_S(S^{\sqcup d}) \cong \SS_d$, Proposition \ref{lem_gerbes_and_torsors}\ref{lem_gerbes_and_torsors_iv} tells us that $F\acute{E}t_d$ is equivalent to $\fTors(\SS_d)$. In particular, the equivalence comes from the functor
\[
(T \to X) \mapsto \cIsom_X(X^{\sqcup d},T).
\]

Replacing $\fTors(\SS_d)$ by $F\acute{E}t_d$ and considering the gerbe $\fTors(\bG^d\rtimes \SS_d)$ as consisting of triples as in Remark \ref{remark_gille_cohom}, the objects of $\fTors(\bG^d\rtimes\SS_d)$ are of the form $(X,T\to X,\cW)$ where $\cW$ is a $(\cIsom(X^{\sqcup d},T)\wedge^{\SS_d|_X}\bG^d|_X)$--torsor ($\cP = \cIsom(X^{\sqcup d},T)$ is the $\SS_d$--torsor). We will cease to write the redundant first $X$ and simply discuss objects $(T\to X,\cW) \in \fTors(\bG^d\rtimes\SS_d)$. Letting the degree $d$ \'etale cover be $f\colon T \to X$, we use Lemma \ref{lem_twist_sheaf} to see that
\[
(\cIsom(X^{\sqcup d},T)\wedge^{\SS_d|_X}\bG^d|_X) \cong f_*(\bG|_T).
\]
Additionally, since $T\to X$ is finite \'etale, we have by \cite[XXIV, 8.4]{SGA3} that
\[
H^1(X,f_*(\bG|_T)) = H^1(T,\bG|_T).
\]
Therefore, we may take $\cW$ to be a $\bG|_T$--torsor. In summary, we consider the gerbe $\fF(\bG)^{d\etale}$ which we define as follows.
\begin{enumerate}[label={\rm(\roman*)}]
\item The objects are  pairs $(T\to X,\cW)$ where $X\in \Sch_S$, $T \to X$ is a finite \'etale cover of degree $d$,
and $\cW$ is a $\bG|_T$-torsor over $T$.
\item The morphisms are triples $(f,g,\varphi)\colon (T'\to X',\cW') \to (T\to X,\cW)$ where $f$ and $g$ are scheme morphisms such that
\[
\begin{tikzcd}
T' \ar[r,"g"] \ar[d] & T \ar[d] \\
X' \ar[r,"f"] & X
\end{tikzcd}
\]
is a pullback diagram of schemes and $\varphi \colon \cW' \iso g^*(\cW)$ is a $\bG|_{T'}$--torsor isomorphism. Composition of morphisms is given by
\[
(f,g,\varphi)\circ (h,j,\psi) = (f\circ h, g\circ j, j^*(\varphi)\circ \psi).
\]
\item The structure functor sends $(T\to X,\cW) \mapsto X$ and $(f,g,\varphi) \mapsto f$.
\end{enumerate}

From the preceding discussion, we have that
\begin{equation}\label{AppenB_gerbe_tors_equiv}
\fF(\bG)^{d\etale} \cong \fTors(\bG^d\rtimes \SS_d).
\end{equation}

\begin{remark}
The condition that $T\to X$ be finite \'etale of degree $d$ can be replaced by any condition $(p)$ on affine morphisms which satisfies base change and descent, i.e., the conditions (BC) and (DESC) in \cite[App. C]{GW}. These morphisms will then also form an intermediate stack and the above argument will show that the corresponding fibered category $\fF(\bG)^{(p)}$ is a stack. This is done with quasi-coherent modules in place of $\bG$--torsors in Appendix \ref{app_stack_morphism}.
\end{remark}

The object $(S^{\sqcup d}\to S, \overline{\bG|_{S^{\sqcup d}}})$ is the split object of $\fF(\bG)^{d\etale}(S)$. The trivial torsor $\overline{\bG|_{S^{\sqcup d}}}$ is described as follows. A morphism $X \to S^{\sqcup d}$ induces a decomposition $X = \sqcup_{i=1}^d X_i$ for $S$--schemes $X_i$ by taking preimages of the factors of $S^{\sqcup d}$, and then $\overline{\bG|_{S^{\sqcup d}}}(X) = \overline{\bG}(X_1)\times \ldots \times \overline{\bG}(X_d)$. In particular, if $\pi \colon S^{\sqcup d} \to S$ is the canonical morphism, then $\overline{\bG^d} = \pi_*(\overline{\bG|_{S^{\sqcup d}}})$. We know from the equivalence of gerbes $\fF(\bG)^{d\etale} \cong \fTors(\bG^d\rtimes\SS_d)$ that
\[
\cAut(\Sd \to S, \overline{\bG|_{\Sd}}) \simlgr \bG^d \rtimes \SS_d.
\]

\begin{remark}\label{rem_torsors_equiv_auto}
The gerbe $\fF(\bG)^{d\etale}$ can equivalently be defined to have objects which are pairs $(T\to X,\cE)$ where $X\in \Sch_S$ and $T\to X$ is a degree $d$ \'etale cover as before, but where $\cE$ is any twisted form of a designated split sheaf $\cE_0$ (of rings, modules, algebras, etc.) on $\Sch_T$ whose automorphism group is $\bG|_T$.
\end{remark}

\section{Quasi-coherent sheaves on $\Sch_S$}\label{app_quasi_coh}
Following \cite[Tag 03DK]{Stacks}, an $\cO$--module $\cE$ is called \emph{quasi-coherent} if for all $X\in \Sch_S$ there is a covering $\{X_i \to X\}_{i\in I}$ such that for each $i\in I$ there is an exact sequence of $\cO|_{X_i}$--modules
\[
\bigoplus_{j\in J_i} \cO|_{X_i} \to \bigoplus_{k\in K_i} \cO|_{X_i} \to \cE|_{X_i} \to 0
\]
for some index sets $J_i$ and $K_i$. If all $K_i$ can be taken to be finite sets, we say $\cE$ is \emph{finitely generated}. If both $J_i$ and $K_i$ can be finite for all $i\in I$, then we say $\cE$ is \emph{finitely presented}. In particular, finite locally free $\cO$--modules are quasi-coherent. Since quasi-coherence is a local condition, an $\cO$--module $\cE$ is quasi-coherent if and only if $\cE|_T$ is a quasi-coherent $\cO|_T$--module for all $T\in \Sch_S$.

By \cite[Tag 03DX]{Stacks}, there is an equivalence between this (site wide) notion of quasi-coherent $\cO$--module and the classical notion of a quasi-coherent $\cO_S$--module on the locally ringed space $S$. Given a classical quasi-coherent sheaf $E$ on $S$, it can be extended to a quasi-coherent $\cO$--module by setting $\cE(T)= g^*(E)(T)$. Conversely, for any quasi-coherent $\cO$--module $\cE$, there exists a classical quasi-coherent sheaf $E$ on $S$ such that for $X \in \Sch_S$ with structure morphism $g\co X \to S$, we have $\cE(T)= g^*(E)(T)$. Then, $E$ is simply the restriction of $\cE$ to the small Zariski site consisting of open subschemes of $S$. For such a pair, we use the notation $E\fppf = \cE$ and $\cE_{\mathrm{small}} = E$ (in \cite{Stacks} they write $E^a$ for $E\fppf$). These of course satisfy $(E\fppf)_{\mathrm{small}} = E$ and $(\cE_{\mathrm{small}})\fppf = \cE$.
 
The following is a key characterization of quasi-coherent $\cO$--modules in terms of their restriction to $\Aff_S$ under the equivalence of Lemma \ref{lem_equiv_affine_sheaves}.
\begin{lem}[{\cite[Tag 0GZV $(1)\Leftrightarrow (7)$]{Stacks}}]\label{lem_quasi_coh_characterization}
Let $\cM \colon \Aff_S \to \Ab$ be a presheaf of $\cO$--modules. Then, $\cM$ is a quasi-coherent $\cO$--module (in particular it is a sheaf) if and only if for every morphism $f \colon V\to U$ in $\Aff_S$, the morphism
\begin{align*}
\rho_f \colon \cM(U)\otimes_{\cO(U)} \cO(V) &\to \cM(V) \\
m\otimes s &\mapsto s\cdot m|_V
\end{align*}
is an isomorphism of $\cO(V)$--modules.
\end{lem}

\begin{lem}[{\cite[0GNC(6)]{Stacks}}]\label{lem_hom_quasi_coh}
Let $\cM_1$ and $\cM_2$ be quasi-coherent $\cO$--modules. If $\cM_1$ is finite locally free, then $\cHom_{\cO}(\cM_1,\cM_2)$ is quasi-coherent as well.
\end{lem}

When working on the small Zariski site or small \'etale site, it is sufficient to assume $\cM_1$ is finitely presented in Lemma \ref{lem_hom_quasi_coh}, see \cite[Tag 01I8(3)]{Stacks} and \cite[Tag 0GNB(6)]{Stacks}. However, this is not sufficient when working on a big site. The key difference is that all structure morphisms in these small sites are flat, while structure morphisms in a big site are arbitrary scheme maps.

Quasi-coherence interacts well with pullbacks. 
\begin{lem}[{\cite[Tag 03LC (1)]{Stacks}}]\label{lem_star0}
Let $g \colon X \to S$ be a morphism of schemes and let $E$ be a quasi-coherent sheaf on $S$. Then, we have an equality $(g^*(E))\fppf = g^*(E\fppf)$. 
\end{lem}
In particular, if $\cE$ is a quasi-coherent $\cO$--module, then $g^*\cE = (g^*(\cE_{\mathrm{small}}))\fppf$ and so pullbacks of quasi-coherent $\cO$--modules are quasi-coherent. In order to have a similar result for pushforwards we assume that $g$ is affine.
\begin{lem}\label{lem_star}
Let $g\colon X \to S$ be an affine morphism of schemes and let $\cF$ be a quasi-coherent $\cO|_X$--module. Then, $g_*\cF = (g_*(\cF_{\mathrm{small}}))\fppf$, which in particular means that $g_*\cF$ is a quasi-coherent $\cO$--module.

Equivalently, if $F$ is a quasi-coherent sheaf on $X$, then we have $g_*(F\fppf)=(g_*(F))\fppf$.
\end{lem}
\begin{proof}
This follows immediately from \cite[Tag 02KG]{Stacks} which shows that, since $g$ is affine, for any other morphism $h\colon S'\to S$ there is a commutative diagram
\[
\begin{tikzcd}
X\times_S S' \arrow{r}{h'} \arrow{d}{g'} & X \arrow{d}{g} \\
S' \arrow{r}{h} & S
\end{tikzcd}
\]
and $h^*(g_*(\cF_{\mathrm{small}})) = g'_*(h'^*(\cF_{\mathrm{small}}))$. The global sections of these sheaves are 
\begin{align*}
h^*(g_*(\cF_{\mathrm{small}}))(S') &= (g_*(\cF_{\mathrm{small}}))\fppf (S'), \text{ and}\\
g'_*(h'^*(\cF_{\mathrm{small}}))(S') &= (h'^*(\cF_{\mathrm{small}}))(X\times_S S') = \cF(X\times_S S') = (g_*\cF)(S')
\end{align*}
and so $g_*\cF = (g_*(\cF_{\mathrm{small}}))\fppf$ as claimed.
\end{proof}
\begin{remark}
Since our $\cO|_X$--modules are sheaves with respect to the fppf topology, it is not sufficient in Lemma \ref{lem_star} to only assume that $g$ is quasi-compact and quasi-separated as is done in \cite[Tag 01LC]{Stacks} for classical (small) quasi-coherent modules on $X$. An example demonstrating this is given in \cite[Tag 03LC (2)]{Stacks}.
\end{remark}

We may construct a quasi-coherent $\cO$--module from a collection of $\cO(U)$--modules for every $U\in \Aff_S$, given that the $\cO(U)$--modules satisfy compatibility conditions with respect to tensor products similar to those in Lemma \ref{lem_quasi_coh_characterization}.
\begin{lem}\label{lem_quasi_coh_affine_defn}
Assume that for each $U\in \Aff_S$ we are given an $\cO(U)$--module $M_U$ and for each morphism $f\colon V \to U$ in $\Aff_S$ we are given an isomorphism of $\cO(V)$--modules $\rho_f \colon M_U \otimes_{\cO(U)} \cO(V) \iso M_V$ such that the following conditions hold.
\begin{enumerate}[label={\rm(\roman*)}]
\item\label{lem_quasi_coh_affine_defn_i} $\rho_{\Id_U}\colon M_U \otimes_{\cO(U)} \cO(U) \iso M_U$ is the canonical isomorphism, and
\item\label{lem_quasi_coh_affine_defn_ii}  for morphisms $g\colon V' \to V$ and $f\colon V \to U$ in $\Aff_S$, the following diagram commutes
\[
\begin{tikzcd}
\big(M_U \otimes_{\cO(U)}\cO(V)\big)\otimes_{\cO(V)}\cO(V') \arrow{r}{\mathrm{can}} \arrow{d}{\rho_f \otimes \Id} & M_U \otimes_{\cO(U)} \cO(V') \arrow{d}{\rho_{f\circ g}} \\
M_V \otimes_{\cO(V)} \cO(V') \arrow{r}{\rho_g} & M_{V'}
\end{tikzcd}
\]
where $\mathrm{can}$ is the canonical isomorphism.
\end{enumerate}
Then, there exists a unique quasi-coherent $\cO$--module $\cM\colon \Sch_S \to \Ab$ such that for $U\in \Aff_S$, we have $\cM(U) = M_U$ and for each morphism $f\colon V \to U$ in $\Aff_S$, the restriction morphism is $\cM(f) = \rho_f \circ (\Id\otimes 1_{\cO(V)})$.
\end{lem}
\begin{proof}
We begin by defining a presheaf on $\Aff_S$. Namely, we define
\begin{align*}
\cM' \colon \Aff_S &\to \Ab \\
U &\mapsto M_U \\
(f\colon V \to U) &\mapsto \rho_f \circ (\Id \otimes 1_{\cO(V)}).
\end{align*}
Conditions \ref{lem_quasi_coh_affine_defn_i} and \ref{lem_quasi_coh_affine_defn_ii} certify that this is a well-defined presheaf. Furthermore, using the $\cO(U)$--action on each $M_U$ gives $\cM'$ the structure of a presheaf of $\cO$--modules. By construction, for each morphism $f\colon V \to U$ the map
\begin{align*}
\cM'(U)\otimes_{\cO(U)} \cO(V) &\to \cM'(V) \\
m\otimes s &\mapsto s \cdot \cM'(f)(m)
\end{align*}
is simply $\rho_f$ since
\[
s \cdot \cM'(f)(m) = s\cdot \rho_f(m\otimes 1) = \rho_f(m\otimes s)
\]
where we use the $\cO(V)$--linearity of $\rho_f$. Hence, these maps are isomorphisms and so Lemma \ref{lem_quasi_coh_characterization} says that $\cM'$ is a quasi-coherent $\cO$--module. Applying Lemma \ref{lem_equiv_affine_sheaves}, we obtain a unique quasi-coherent $\cO$--module $\cM \colon \Sch_S \to \Ab$ whose restriction to $\Aff_S$ is $\cM'$. This finishes the proof.
\end{proof}

In the remainder of the paper, in order to mirror the notation appearing, for example, in Lemma \ref{lem_quasi_coh_affine_defn}, we write the following. Given a map of rings $f\colon R \to Q$, there is a base change functor $b_f \colon \fMod_R \to \fMod_Q$ defined by the tensor product, $b_f(M)=M\otimes_R Q$. If we have another map of rings $f'\colon R' \to Q'$ as well as functors $\cF_{R'/R}\colon \fMod_{R'} \to \fMod_R$ and $\cF_{Q'/Q} \colon \fMod_{Q'}\to \fMod_Q$, we write
\begin{align*}
\cF_{R'/R}(\und)\otimes_R Q &= b_f \circ \cF_{R'/R}\\
\cF_{Q'/Q}(\und \otimes_{R'} Q') &= \cF_{Q'/Q}\circ b_{f'}
\end{align*}
for the two functors $\fMod_{R'} \to \fMod_Q$. Additionally, we will write
\[
\theta \colon \cF_{R'/R}(\und)\otimes_R Q \iso \cF_{Q'/Q}(\und \otimes_{R'} Q')
\]
to denote that $\theta$ is a natural isomorphism between the functors $b_f \circ \cF_{R'/R}$ and $\cF_{Q'/Q}\circ b_{f'}$.

The results above allow us to assemble certain families of functors between module categories into a functor between categories of quasi-coherent sheaves in the following technical lemma. We will use the concept of an affine morphism for which we refer to \cite[Tag 01S6]{Stacks}. In particular, affine morphisms are stable under base change by \cite[Tag 01SD]{Stacks} and so for an affine morphism $f\colon T\to S$ and any affine scheme $U\in \Sch_S$, the fiber product $U\times_S T$ is an affine scheme. Recall as well that $\QCoh(S)$ denotes the category of quasi-coherent $\cO$--modules on $\Sch_S$. The category $\QCoh(T)$ is then the category of quasi-coherent $\cO|_T$--modules.

\begin{lem}\label{lem_quasi_coh_functor_construction}
Let $T \to S$ be an affine morphism (as in \cite[Tag 01S6]{Stacks}) of schemes such that the following holds. Assume that we have a family of covariant functors
\[
\cF_U \colon \fMod_{\cO(T\times_S U)} \to \fMod_{\cO(U)}
\]
for each $U\in \Aff_S$. Further, suppose that for each morphism $f\colon V\to U$ in $\Aff_S$, we have an isomorphism of functors
\[
\theta_f \colon \cF_U(\und) \otimes_{\cO(U)} \cO(V) \iso \cF_V\big(\und \otimes_{\cO(T\times_S U)} \cO(T\times_S V)\big)
\]
satisfying the following conditions. For each pair of morphisms $g\colon V' \to V$ and $f\colon V \to U$ in $\Aff_S$, the following diagrams commute.
\begin{equation}\label{eq_lem_quasi_construction_i}
\begin{tikzcd}[column sep = 1.1in]
\cF_U(\und)\otimes_{\cO(U)}\cO(U) \arrow{dr}{\mathrm{can}} \arrow{d}{\theta_{\Id_U}} &   \\
\cF_U\big(\und \otimes_{\cO(T\times_S U)}\cO(T\times_S U)\big) \arrow[near start]{r}{\cF_U(\mathrm{can})} & \cF_U(\und)
\end{tikzcd}
\end{equation}
and \refstepcounter{equation}\label{eq_lem_quasi_construction_ii}\\
\noindent \textnormal{(\theequation)} \hspace{-2ex} \begin{tabular}{@{}c}
$\begin{tikzcd}[ampersand replacement=\&]
\big(\cF_U(\und)\otimes_{\cO(U)}\cO(V)\big)\otimes_{\cO(V)}\cO(V') \arrow{r}{\mathrm{can}} \arrow{d}{\theta_f\otimes\Id} \& \cF_U(\und)\otimes_{\cO(U)}\cO(V') \arrow{dd}{\theta_{f\circ g}} \\
\cF_V\big(\und \otimes_{\cO(U_T)}\cO(V_T)\big)\otimes_{\cO(V)} \cO(V') \arrow{d}{\theta_g} \& \\
\cF_{V'}\big((\und\otimes_{\cO(U_T)}\cO(V_T))\otimes_{\cO(V_T)}\cO(V'_T)\big) \arrow{r}{\cF_{V'}(\mathrm{can})} \& \cF_{V'}(\und\otimes_{\cO(U_T)}\cO(V'_T)).
\end{tikzcd}$
\end{tabular}\\
For brevity in the diagrams, we abuse notation by using $\mathrm{can}$ to denote various canonical isomorphisms. We also denote $T\times_S U = U_T$ and likewise for $V_T$ and $V'_T$. Then, there is a functor $\cF_{T/S} \colon \QCoh(T) \to \QCoh(S)$ defined as follows. For each quasi-coherent $\cO|_T$--module $\cM$, the $\cO$--module $\cF_{T/S}(\cM)$ has the following properties:
\begin{enumerate}[label={\rm(\roman*)}]
\item \label{lem_quasi_construction_i} $\cF_{T/S}(\cM)(U) = \cF_U(\cM(T\times_S U))$ for all $U\in \Aff_S$, and
\item \label{lem_quasi_construction_ii} for each morphism $f\colon V \to U$ in $\Aff_S$, the associated restriction map $\cF_{T/S}(\cM)(f)$ is defined by
\[
\begin{tikzcd}
\cF_U(\cM(U_T)) \arrow{r}{\Id\otimes 1} \arrow[swap, bend right=20]{ddr}{\cF_{T/S}(\cM)(f)} & \cF_U(\cM(U_T))\otimes_{\cO(U)}\cO(V) \arrow{d}{\theta_f(\cM(U_T))} \\
 & \cF_V\big(\cM(U_T)\otimes_{\cO(U_T)} \cO(V_T)\big) \arrow{d}{\cF_V(\rho_{f'})} \\
 & \cF_V(\cM(V_T))
\end{tikzcd}
\]
where $f' \colon T\times_S V \to T\times_S U$ is the pullback of $f$ and
\[
\rho_{f'} \colon \cM(T\times_S U)\otimes_{\cO(T\times_S U)} \cO(T\times_S V) \to \cM(T\times_S V)
\]
is the canonical isomorphism as in Lemma \ref{lem_quasi_coh_characterization} arising from $\cM$ being quasi-coherent and both $T\times_S U$ and $T\times_S V$ being affine schemes.
\item \label{lem_quasi_construction_iii}
For each morphism $\varphi \colon \cM_1 \to \cM_2$ of quasi-coherent $\cO|_T$--modules, the morphism $\cF_{T/S}(\varphi) \colon \cF(\cM_1) \to \cF(\cM_2)$ is defined over $U \in \Aff_S$ by
\[
\cF_{T/S}(\varphi)(U) = \cF_U(\varphi(T\times_S U)) \colon \cF_U(\cM_1(T\times_S U)) \to \cF_U(\cM_2(T\times_S U)).
\]
\end{enumerate}
\end{lem}
\begin{proof}
There are various compatibility conditions that need to be checked. Let $\cM$ be a quasi-coherent $\cO|_T$--module. We begin by using Lemma \ref{lem_quasi_coh_affine_defn} to define the sheaf $\cF_{T/S}(\cM)$. Because $\cM$ is quasi-coherent, by Lemma \ref{lem_quasi_coh_characterization} it comes with standard isomorphisms $\rho_f \colon \cM(U)\otimes_{\cO(U)}\cO(V) \iso \cM(V)$ for each morphism $f\colon V \to U$ in $\Aff_S$.

Now, for each $U\in \Aff_S$, we set $M_U = \cF_U(\cM(T\times_S U))$. For a morphism $f\colon V \to U$ in $\Aff_S$, we define the isomorphism $\phi_f \colon M_U \otimes_{\cO(U)}\cO(V) \iso M_V$ to be
\[
\rho_f = \cF_V(\rho_{f'})\circ \theta_f(\cM(U_T))
\]
where $f' \colon T\times_S V \to T\times_S U$ is the pullback of $f$. Because the map $\rho_{\Id_{U_T}} \colon \cM(U_T)\otimes_{\cO(U_T)} \cO(U_T) \to \cM(U_T)$ is the canonical isomorphism, the commutativity of \eqref{eq_lem_quasi_construction_i} implies that $\phi_{\Id_U}$ is the canonical isomorphism, as required in Lemma \ref{lem_quasi_coh_affine_defn}.

For morphisms $g\colon V' \to V$ and $f\colon V \to U$ in $\Aff_S$, passing $\cM(T\times_S U)$ into diagram \eqref{eq_lem_quasi_construction_ii} yields
\[
\begin{tikzcd}[column sep=-3ex]
 & \cF_U(\cM(U_T))\otimes_{\cO(U)}\cO(V') \arrow{dddd}{\theta_{f\circ g}(\cM(U_T))} \\
\big(\cF_U(\cM(U_T))\otimes_{\cO(U)}\cO(V)\big)\otimes_{\cO(V)}\cO(V') \arrow{ur}{\mathrm{can}} \arrow{d}{\theta_f(\cM(U_T))\otimes\Id} & \\
\cF_V\big(\cM(U_T) \otimes_{\cO(U_T)}\cO(V_T)\big)\otimes_{\cO(V)} \cO(V') \arrow{d}{\theta_g(\cM(U_T)\otimes_{\cO(U_T)}\cO(V_T))} & \\
\cF_{V'}\big((\cM(U_T)\otimes_{\cO(U_T)}\cO(V_T))\otimes_{\cO(V_T)}\cO(V'_T)\big) \arrow[swap]{dr}{\cF_{V'}(\mathrm{can})} & \\
 & \cF_{V'}(\cM(U_T)\otimes_{\cO(U_T)}\cO(V'_T)).\\
\end{tikzcd}
\]
This diagram can be extended to
\[
\begin{tikzcd}
 &[-30pt] \bullet \arrow{r}{\mathrm{can}} \arrow{d}{\theta_f(\cM(U_T))\otimes\Id} \arrow[swap]{dl}{\phi_f\otimes\Id} & \bullet \arrow[swap]{dd}{\theta_{f\circ g}(\cM(U_T))} \arrow[bend left]{ddd}{\phi_{f\circ g}} \\
\cF_V(\cM(V_T))\otimes_{\cO(V)}\cO(V') \arrow{ddr}{\theta_g(\cM(V_T))} \arrow[bend right=50]{ddrr}{\phi_g} & \bullet \arrow{d}{\theta_g(\cM(U_T)\otimes_{\cO(U_T)}\cO(V_T))} \arrow[swap]{l}{\cF_V(\rho_f')\otimes\Id} & \\
 & \bullet \arrow{r}{\cF_{V'}(\mathrm{can})} \arrow{d}{\cF_{V'}(\rho_{f'}\otimes\Id)} & \bullet \arrow[swap]{d}{\cF_{V'}(\rho_{(f'\circ g')})} \\
 & \cF_{V'}\big(\cM(V_T)\otimes_{\cO(V_T)}\cO(V'_T)\big) \arrow{r}{\cF_{V'}(\rho_{g'})} & \cF_{V'}(\cM(V'_T))
\end{tikzcd}
\]
where the bullets represent the entries in the previous diagram. The triangle in the top left of the diagram commutes by definition of $\phi_f$. The triangle below it commutes because $\theta_g$ is a natural transformation. The bottom most square commutes because $\rho_{(f'\circ g')}\circ \mathrm{can} = \rho_{g'}\circ (\rho_{f'}\otimes\Id)$ since $\cM$ is quasi-coherent, and this has simply been passed through $\cF_{V'}$. Finally, the portions of the diagram involving the curved arrows commute by definition of $\phi_g$ and $\phi_{f\circ g}$. Ultimately, this shows that $\phi_{f\circ g} \circ \mathrm{can} = \phi_g \circ (\phi_f\otimes\Id)$ as required by Lemma \ref{lem_quasi_coh_affine_defn}. Hence, applying Lemma \ref{lem_quasi_coh_affine_defn} produces a quasi-coherent $\cO$--module $\cF_{T/S}(\cM)$ with properties \ref{lem_quasi_construction_i} and \ref{lem_quasi_construction_ii} of the statement.

To finish constructing the functor $\cF_{T/S}$, we need to define the image of morphisms. Let $\varphi \colon \cM_1 \to \cM_2$ be a morphism of quasi-coherent $\cO|_T$--modules. Due to Lemma \ref{lem_equiv_affine_sheaves}, it is sufficient to only define $\cF_{T/S}(\varphi)$ over the schemes in $\Aff_S$. For $U\in \Aff_S$, we define $\cF_{T/S}(\varphi)(U)=\cF_U(\varphi(T\times_S U))$ as in condition \ref{lem_quasi_construction_iii} of the statement. It is clear that if $\cF_{T/S}(\varphi)$ is a well-defined morphism, then $\cF_{T/S}$ will preserve identities and compositions and hence will be a functor. Therefore, we just need to verify that $\cF_{T/S}(\varphi)$ is a well-defined natural transformation of functors. Let $f\colon V \to U$ be a morphism in $\Aff_S$. Once again, we consider a large diagram
\[
\begin{tikzcd}[column sep=0.8in]
\cF_U(\cM_1(U_T)) \arrow{r}{\cF_U(\varphi(U_T))} \arrow{d}{\Id\otimes 1} & \cF_U(\cM_2(U_T)) \arrow{d}{\Id\otimes 1} \\
\cF_U(\cM_1(U_T))\otimes_{\cO(U)}\cO(V) \arrow{r}{\cF_U(\varphi(U_T))\otimes \Id} \arrow{d}{\theta_f(\cM_1(U_T))} & \cF_U(\cM_2(U_T))\otimes_{\cO(U)}\cO(V) \arrow{d}{\theta_f(\cM_2(U_T))} \\
\cF_V\big(\cM_1(U_T)\otimes_{\cO(U_T)}\cO(V_T)\big) \arrow{r}{\cF_V(\varphi(U_T)\otimes\Id)} \arrow{d}{\cF_V(\rho_{1,f'})} & \cF_V\big(\cM_2(U_T)\otimes_{\cO(U_T)}\cO(V_T)\big) \arrow{d}{\cF_V(\rho_{2,f'})} \\
\cF_V(\cM_1(V_T)) \arrow{r}{\cF_V(\varphi(V_T))} & \cF_V(\cM_2(V_T)).
\end{tikzcd}
\]
Here, $\rho_{i,f}$ is the standard isomorphism from $\cM_i$ being quasi-coherent. The top square is trivially commutative. The middle square commutes because $\theta_f$ is a natural transformation. The bottom square commutes because $\rho_{2,f'} \circ (\varphi(U_T)\otimes\Id) = \varphi(V_T)\circ \rho_{1,f'}$ due to $\varphi$ being a morphism of $\cO$--modules and this has been passed through $\cF_V$. The compositions down each column are the restriction morphisms $\cF_{T/S}(\cM_i)(f)$ respectively, so the commutativity of the outermost paths shows that $\cF_{T/S}(\varphi)$ is well-defined. This finishes the construction of $\cF_{T/S}$ and hence concludes the proof.
\end{proof}

\subsection{Stack Morphism}\label{app_stack_morphism}
If we have compatible functors between module categories for an even wider ranger of morphisms, we can construct a morphism of stacks of quasi-coherent sheaves. The codomain of the stack morphism will be the substack $\QCoh$ of $\fSh$ consisting of quasi-coherent modules.
\begin{enumerate}[label={\rm(\roman*)}]
\item The objects of $\QCoh$ are pairs $(X,\cF)$ with $X\in \Sch_S$ and $\cF$ a quasi-coherent $\cO|_X$--module on $\Sch_X$.
\item The morphisms of $\QCoh$ are pairs $(g,\varphi)\colon (X',\cF') \to (X,\cF)$ where $g\colon X' \to X$ is a morphism of $S$--schemes and $\varphi \colon \cF' \to g^*(\cF)$ is a morphism of $\cO|_{X'}$--modules. Composition is given by $(g,\varphi)\circ (h,\psi) = (g\circ h,h^*(\varphi)\circ \psi)$.
\end{enumerate}

Second, consider the stack of affine morphisms $p\colon \fAff \to \Sch_S$ whose
\begin{enumerate}[label={\rm(\roman*)}]
\item objects are affine morphisms $T'\to T$ where $T\in \Sch_S$,
\item morphisms are pairs $(f,g)\colon (X'\to X)\to (T'\to T)$ where $f$ and $g$ are scheme morphisms making
\[
\begin{tikzcd}
X' \arrow{d}{j} \arrow{r}{g} & T' \arrow{d}{h} \\\
X \arrow{r}{f} & T
\end{tikzcd}
\]
a fiber product diagram in $\Sch_S$, and
\item the structure functor sends $(T'\to T)\mapsto T$ and $(f,g)\mapsto f$.
\end{enumerate}
The pullback diagram condition on morphisms makes this fibered in groupoids. It is indeed a stack since $\cHom(T'_1\to T,T'_2\to T) = \cIsom_T(T'_1,T'_2)$, which is a sheaf, and affine morphisms satisfy gluing by \cite[4.4.7]{Olsson}.

Let $\fI \subseteq \fAff$ be any substack as defined in \cite[4.1.6]{Vis}. If $\fI$ is a full subcategory, this is equivalent to choosing a family of affine morphisms which are stable under base change and allow descent. The substack $\fI \to \Sch_S$ gives $\fI$ an inherited site structure as in \cite[Tag 06NU]{Stacks}. In detail, the covers will be families of morphisms of the form $\{(f_i,g_i)\colon (T'_i \to T_i) \to (T'\to T)\}_{i\in I}$ where $\{f_i \colon T_i \to T\}_{i\in I}$ is an fppf cover (and every $(f_i,g_i)$ is cartesian, but this holds by default since $\fAff$ is fibered in groupoids). Since morphisms in $\fAff$ must define pullback diagrams, this means that $T'_i \iso T'\times_T T_i$ for each $i\in I$ and so $\{g_i \colon T'_i \to T'\}_{i\in I}$ is also an fppf cover.

Next, we define the stack of quasi-coherent sheaves over $\fI$, denoted $\QCoh_{\fI}$, as follows.
\begin{enumerate}[label={\rm(\roman*)}]
\item Objects are pairs $(h\colon T' \to T,\cM)$ with $h\in \fI$ and $\cM$ a quasi-coherent $\cO|_{T'}$--module.
\item Morphisms are triples $(f,g,\varphi)\colon (j\colon X' \to X,\cN) \to (h\colon T' \to T,\cM)$ where $(f,g)\colon j \to h$ is a morphism in $\fI$ and $\varphi \colon \cN \to g^*(\cM)$ is an isomorphism of $\cO|_{X'}$--modules. Composition is given by
\[
(f_1,g_1,\varphi_1)\circ (f_2,g_2,\varphi_2) = (f_1\circ f_2,g_1\circ g_2,g_2^*(\varphi_1)\circ \varphi_2).
\]
\item The structure functor is $p\colon \QCoh_{\fI} \to \fI$ which behaves as $(h\colon T' \to T,\cM) \mapsto (h\colon T' \to T)$ and $(f,g,\varphi)\mapsto (f,g)$ on objects and morphisms respectively.
\end{enumerate}
It is clear this is a stack since, for two objects $(T'\to T,\cM_1)$ and $(T'\to T,\cM_2)$ in the same fiber, $\cIsom_{\cO|_{T'}}(\cM_1,\cM_2)$ is a sheaf and quasi-coherent modules permit gluing along fppf covers. Since $\fI$ is fibered in groupoids and we require that $\varphi$ in a morphism of $\QCoh_\fI$ be an isomorphism, the stack $\QCoh_\fI \to \fI$ is also fibered in groupoids.

However, we ultimately want to consider $\QCoh_\fI$ as a stack over $\Sch_S$. By \cite[09WX]{Stacks}, the composition $\QCoh_{\fI} \to \fI \to \Sch_S$ does produce a stack and it is also fibered in groupoids.

Throughout the remainder of this section, we let $\fI$ be a substack of $\fAff$ and assume we are given the following data:
\begin{enumerate}[label={\rm(\thesubsection.\alph*)}]
\item a functor $\cF_h \colon \fMod_{\cO(U')} \to \fMod_{\cO(U)}$ for every object $h\colon U' \to U$ in $\fI$ for which $U,U' \in \Aff_S$, and
\item for every fiber product diagram
\[
D=
\begin{tikzcd}
V' \arrow{d}{h'} \arrow{r}{f'} & U' \arrow{d}{h} \\
V \arrow{r}{f} & U
\end{tikzcd}
\]
where $U,U',V$ and hence $V'$ are affine, $h,h' \in \fI$, and $f\in \Aff_S$ is an arbitrary morphism, we are given an isomorphism of functors
\[
\theta_D \colon \cF_h(\und)\otimes_{\cO(U)}\cO(V) \iso \cF_{h'}\big(\und \otimes_{\cO(U')}\cO(V')\big).
\] 
\end{enumerate}
Further, we assume that these functors satisfy the following compatibility conditions.
\begin{enumerate}[resume,label={\rm(\thesubsection.\alph*)}]
\item \label{assumption_1} For every fiber product diagram $D$ of the form below, the diagram on the right commutes
\[
D=
\begin{tikzcd}
U' \arrow{d}{h} \arrow[equals]{r} & U' \arrow{d}{h} \\
U \arrow[equals]{r} & U,
\end{tikzcd}
\hspace{0.2in}
\begin{tikzcd}[column sep=0.8in]
\cF_h(\und)\otimes_{\cO(U)}\cO(U) \arrow{d}{\theta_D} \arrow{dr}{\mathrm{can}} & \\
\cF_h\big(\und\otimes_{\cO(U')}\cO(U')\big) \arrow[near start]{r}{\cF_h(\mathrm{can})} & \cF_h(\und)
\end{tikzcd}
\]
where $h\in \fI$ and $U,U'\in \Aff_S$.
\item \label{assumption_2} For all fiber product diagrams
\[
D_f = \begin{tikzcd}
V' \arrow{d}{h'} \arrow{r}{f'} & U' \arrow{d}{h} \\
V \arrow{r}{f} & U
\end{tikzcd},\;
D_g = \begin{tikzcd}
W' \arrow{d}{h''} \arrow{r}{g'} & V' \arrow{d}{h'} \\
W \arrow{r}{g} & V
\end{tikzcd},
\text{ and }
D_{f\circ g} = \begin{tikzcd}
W' \arrow{d}{h''} \arrow{r}{f'\circ g'} & U' \arrow{d}{h} \\
W \arrow{r}{f \circ g} & U
\end{tikzcd}
\]
with $h\in \fI$ and $f,g\in \Aff_S$ the diagram
\[
\begin{tikzcd}
\big(\cF_h(\und)\otimes_{\cO(U)}\cO(V)\big)\otimes_{\cO(V)}\cO(W) \arrow{r}{\mathrm{can}} \arrow{d}{\theta_{D_f}\otimes\Id} & \cF_h(\und)\otimes_{\cO(U)}\cO(W) \arrow{dd}{\theta_{D_{f\circ g}}} \\
\cF_{h'}\big(\und\otimes_{\cO(U')}\cO(V')\big)\otimes_{\cO(V)}\cO(W) \arrow{d}{\theta_{D_g}} & \\
\cF_{h''}\big((\und\otimes_{\cO(U')}\cO(V'))\otimes_{\cO(V')}\cO(W')\big) \arrow{r}{\cF_{h''}(\mathrm{can})} & \cF_{h''}(\und \otimes_{\cO(U')}\cO(W'))
\end{tikzcd}
\]
commutes.
\end{enumerate}

With these assumptions, we work up to a stack version of Lemma \ref{lem_quasi_coh_functor_construction}. First, let $h\colon T' \to T$ be any morphism in $\fI$. For each $U\in \Aff_T$ and each morphism $f\colon V \to U$ in $\Aff_T$, we have fiber product diagrams
\[
D_U = \begin{tikzcd}
T'\times_T U \arrow{r} \arrow{d}{h'} & T' \arrow{d}{h} \\
U \arrow{r} & T
\end{tikzcd}
\text{ and }
D_f = \begin{tikzcd}
T'\times_T V \arrow{r}{f'} \arrow{d}{h''} & T'\times_T U \arrow{d}{h'} \\
V \arrow{r}{f} & U.
\end{tikzcd}
\]
If we set $\cF_U = \cF_{h'}$ and $\theta_f = \theta_{D_f}$, then it is clear that the compatibility conditions assumed above specialize into the requirements of Lemma \ref{lem_quasi_coh_functor_construction}. Therefore, we may apply the lemma to obtain a functor denoted $\cF_{T'/T}\colon \QCoh(T') \to \QCoh(T)$. These functors are related to one another in the following way.

\begin{lem}\label{lem_phi_isomorphism}
Let $h'\colon X' \to X$ and $h\colon T' \to T$ be two morphisms in $\fI$ and let $\cF_{X'/X}$ and $\cF_{T'/T}$ be the associated functors defined above. Then, for every fiber product diagram
\[
D = \begin{tikzcd}
X' \arrow{d}{h'} \arrow{r}{g} & T' \arrow{d}{h} \\
X \arrow{r}{f} & T
\end{tikzcd}
\]
with $f,g$ morphisms of $\Sch_S$, we have an isomorphism of functors
\[
\phi_D \colon \cF_{X'/X}\circ g^* \iso f^* \circ \cF_{T'/T}.
\]
\end{lem}
\begin{proof}
Since $D$ is a fiber product diagram, for $U\in \Aff_X$ we get another fiber product diagram
\[
D_U = \begin{tikzcd}
X'\times_X U \arrow{r}{g_U} \arrow{d}{h'_U} & T'\times_T U \arrow{d}{h_U} \\
U \arrow[equals]{r} & U
\end{tikzcd}
\]
where $\tilde{g}$ is an isomorphism. Now, let $\cM \in \QCoh(T')$ with its canonical isomorphisms $\rho_f$ for morphisms $f\in \Aff_{T'}$. We have an isomorphism
\[
\phi_D(\cM)(U) \colon \cF_{X'/X}(g^*(\cM))(U) \iso \cF_{T'/T}(\cM)(U) = f^*(\cF_{T'/T}(\cM))(U)
\]
defined by
\[
\begin{tikzcd}
\cF_{h'_U}(\cM(X'\times_X U)) \arrow{r}{\cF_{h'_U}(\rho_{g_U}^{-1})} \arrow[swap,bend right=10]{ddr}{\phi_D(\cM)(U)} & \cF_{h'_U}\big(\cM(T'\times_T U)\otimes_{\cO(T'\times_T U)}\cO(X'\times_X U)\big) \arrow{d}{\theta_{D_U}^{-1}} \\
& \cF_{h_U}(\cM(T'\times_T U))\otimes_{\cO(U)}\cO(U) \arrow{d}{\mathrm{can}} \\
& \cF_{h_U}(\cM(T'\times_T U)).
\end{tikzcd}
\]
We check that these isomorphisms are compatible with the restriction along a morphism $V\to U$ of affine schemes in $\Aff_X$, and hence via Lemma \ref{lem_equiv_affine_sheaves}, give a well-defined isomorphism of sheaves
\[
\phi_D(\cM) \colon \cF_{X'/X}(g^*(\cM)) \iso f^*(\cF_{T'/T}(\cM)).
\] 
This also follows from constructing a large commutative diagram. First, we have a commutative diagram
\[
\begin{tikzcd}
 &[-5ex] T'\times_T V \arrow{rr} \arrow[near start]{dd}{h_V} &[-5ex] &[-5ex] T'\times_T U \arrow{rr} \arrow[near start]{dd}{h_U} &[-5ex] & T' \arrow{dd}{h} \\
X'\times_X V \arrow{ur}{g_V} \arrow[crossing over]{rr} \arrow{dd}{h'_V} & & X'\times_X U \arrow{ur}{g_U} \arrow[crossing over]{rr}  & & X' \arrow{ur}{g} & \\ 
 & V \arrow{rr} & & U \arrow{rr} & & T \\
V \arrow{rr} \arrow[equals]{ur} & & U \arrow{rr} \arrow[equals]{ur} \arrow[from=uu, near start, crossing over, "h'_U"] & & X \arrow{ur}{f} \arrow[crossing over, near start, from=uu, "h'"] &
\end{tikzcd}
\]
where the vertical faces are all pullback diagrams. The commutativity of the left cube means that we have an equality of pullback diagrams
\begin{align*}
D' &= \begin{tikzcd}[ampersand replacement=\&]
X'\times_X V \ar[r] \ar[d,"h'_V"] \ar[rd,phantom,"D_1"] \& X'\times_X U \ar[r,"g_U"] \ar[d,"h'_U"] \ar[rd,phantom,"D_U"] \& T'\times_T U \ar[d,"h_U"] \\
V \ar[r] \& U \ar[r,equals] \& U
\end{tikzcd} \\
&= \begin{tikzcd}[ampersand replacement=\&]
X'\times_X V \ar[r,"g_V"] \ar[d,"h'_V"] \ar[rd,phantom,"D_V"] \& T'\times_T V \ar[r] \ar[d,"h_V"] \ar[rd,phantom,"D_2"] \& T'\times_T U \ar[d,"h_U"] \\
V \ar[r,equals] \& V \ar[r] \& U.
\end{tikzcd}
\end{align*}
Therefore, our compatibility assumptions produce the following commutative diagram. To save space, we use the abbreviations $T'\times_T U = U_{T'}$,  $X'\times_X U = U_{X'}$, and similarly for $V_{T'}$ and $V_{X'}$. We also use the abuse of notation $\und \otimes_{\cO(U)}\cO(V) = \und \otimes_U V$.
\[
\begin{tikzcd}[column sep=-12ex]
 & \cF_{h_U}(\cM(U_{T'}))\otimes_U V \arrow[dddd,"\theta_{D'}" description] & \\
\big(\cF_{h_U}(\cM(U_{T'}))\otimes_U U \big)\otimes_U V \ar[ur,"\mathrm{can}"] \ar[d,swap,"\theta_{D_U}\otimes\Id"] & & \big(\cF_{h_U}(\cM(U_{T'}))\otimes_U V \big)\otimes_V V \ar[ul,swap,"\mathrm{can}"] \ar[d,"\theta_{D_2}\otimes \Id"] \\
\cF_{h'_U}\big(\cM(U_{T'})\otimes_{U_{T'}} U_{X'}\big)\otimes_U V \ar[d,swap,"\theta_{D_1}"] & & \cF_{h_V}\big(\cM(U_{T'})\otimes_{U_{T'}} V_{T'}\big)\otimes_V V \ar[d,"\theta_{D_V}"] \\
\cF_{h'_V}\big((\cM(U_{T'})\otimes_{U_{T'}} U_{X'})\otimes_{U_{X'}} V_{X'}\big) \ar[dr,swap,near start,"\cF_{h'_V}(\mathrm{can})"] & & \cF_{h'_V}\big((\cM(U_{T'})\otimes_{U_{T'}}V_{T'}) \otimes_{V_{T'}} V_{X'}\big) \ar[dl,near start,"\cF_{h'_V}(\mathrm{can})"]\\
 & \cF_{h'_V}\big(\cM(U_{T'})\otimes_{U_{T'}} V_{X'}\big) & 
\end{tikzcd}
\]
we extend the left side of this diagram (denoting its entries with $\bullet$ as before) to
\[
\begin{tikzcd}
\cF_{h'_U}(\cM(U_{X'})) \ar[rr,"\phi_D(\cM)(U)"] \ar[dd,swap,"\Id\otimes 1"] \ar[ddddrr,swap,"\cF_{X'/X}(V\to U)",rounded corners,to path={-- ([xshift=-5ex]\tikztostart.west) -- ([xshift=-5ex]\tikztostart.west|-\tikztotarget)\tikztonodes -- (\tikztotarget)}] & & \cF_{h_U}(\cM(U_{T'})) \ar[d,"\Id\otimes 1"] & & \\
 & \bullet \ar[r] \ar[d] & \bullet \ar[dd] & \bullet \ar[l] \ar[d] & \\
\cF_{h'_U}(\cM(U_{X'}))\otimes_U V \ar[dr,phantom,near start,"A"] \ar[urr,bend left=35,"\phi_D(\cM)(U)\otimes\Id"] \ar[d,"\theta_{D_1}"] & \bullet \ar[d,"\theta_{D_1}"] \ar[l,swap,"*"] & & \bullet \ar[d] & \\
\cF_{h'_V}\big(\cM(U_{X'})\otimes_{U_{X'}} V_{X'}\big) \ar[drr,"*"]  & \bullet \ar[dr,phantom,"B"] \ar[r] \ar[l,swap,"*"] & \bullet \ar[d,"*"] & \bullet \ar[l] & \\
 & & \cF_{h'_V}(\cM(V_{X'})) & &
\end{tikzcd}
\]
where arrows labelled ``$*$" are the appropriate isomorphisms induced by the quasi-coherence of $\cM$. Square $A$ commutes since $\theta_{D_1}$ is a natural transformation and square $B$ commutes because $\cM$ is quasi-coherent. The remaining new squares commute by definition. We also extend the right side of the original diagram to
\[
\begin{tikzcd}
 & \cF_{h_U}(\cM(U_{T'})) \ar[d,"\Id\otimes 1"] \ar[ddrr,bend left=20,"\cF_{T'/T}(V\to U)\otimes 1" description] \ar[rr,"\cF_{T'/T}(V\to U)"] & & \cF_{h_V}(\cM(V_{T'})) \ar[dd,"\Id\otimes 1"] \\
\bullet \ar[r] \ar[d] & \bullet \ar[dd] & \bullet \ar[l] \ar[d] &  \\
\bullet \ar[d] & & \bullet \ar[d,"\theta_{D_V}"] \ar[r,"*"] & \cF_{h_V}(\cM(V_{T'}))\otimes_V V \ar[d,"\theta_{D_V}"] \\
\bullet \ar[r] & \bullet \ar[d,"*"] & \bullet \ar[ur,phantom,near end,"B"] \ar[l] \ar[r,"*"] & \cF_{h'_V}\big(\cM(V_{T'})\otimes_{V_{T'}} V_{X'}\big) \ar[lld,"*"] \\
 & \cF_{h'_V}(\cM(V_{X'})) \ar[ur,phantom,"A"] \ar[uuuurr,swap,"\phi_D(\cM)(V)",rounded corners,to path={-- ([xshift=2.6in]\tikztostart.east) -- ([xshift=2.6in]\tikztostart.east|-\tikztotarget.east)\tikztonodes -- (\tikztotarget)}] & &
\end{tikzcd}
\]
where the square $A$ commutes due to the quasi-coherence of $\cM$, the square $B$ commutes because $\theta_{D_V}$ is a natural transformation, and again the other new squares commute by definition. Considering the expanded diagram in its totality, the outermost path is the commutative diagram
\[
\begin{tikzcd}[column sep=1in]
\cF_{X'/X}(g^*(\cM))(U) \ar[r,"\phi_D(\cM)(U)"] \ar[d,"\cF_{X'/X}(V\to U)"] & f^*(\cF_{T'/T}(\cM))(U) \arrow[d,"\cF_{T'/T}(V\to U)"] \\
\cF_{X'/X}(g^*(\cM))(V) \ar[r,"\phi_D(\cM)(V)"] & f^*(\cF_{T'/T}(\cM))(V).
\end{tikzcd}
\]
Since this is commutative for all morphisms $V \to U$, we have a well defined isomorphism of sheaves $\phi_D(\cM)$.

Now, we must check that these isomorphisms are functorial in $\cM$ and thus provide our desired isomorphism of functors $\phi_D$. Luckily, this follows directly from the fact that the various $\cF_h$ for $h\in \fI$ are functors. No more large diagrams are needed, and we are done.
\end{proof}

\begin{prop}\label{prop_stack_morphism_assembled}
The functors $\cF_{T'/T}$ assemble into a morphism of stacks $\fF \colon \QCoh_{\fI} \to \QCoh$ defined by
\begin{enumerate}[label={\rm(\roman*)}]
\item $\fF(h\colon T' \to T,\cM) = (T,\cF_{T'/T}(\cM))$ on objects, and
\item for a morphism $(f,g,\varphi)\colon (h'\colon X' \to X,\cN) \to (h\colon T' \to T,\cM)$, we have a pullback diagram
\[
D = \begin{tikzcd}
X' \ar[r,"g"] \ar[d,"h'"] & T' \ar[d,"h"] \\
X \ar[r,"f"] & T
\end{tikzcd}
\]
and we set $\fF(f,g,\varphi)= (f,\fF(\varphi))$ where
\[
\fF(\varphi) \colon \cF_{X'/X}(\cN) \xrightarrow{\cF_{X'/X}(\varphi)} \cF_{X'/X}(g^*(\cM)) \xrightarrow{\phi_D(\cM)} f^*(\cF_{T'/T}(\cM))
\]
is an $\cO|_X$--module morphism between the appropriate modules.
\end{enumerate}
\end{prop}
\begin{proof}
We must argue that $\fF$ is a well defined functor and that it preserves cartesian morphisms. Since it is already clear that it respects the structure functors, this will be sufficient to conclude that $\fF$ is a morphism of stacks.

Consider the identity morphism of an object $(\Id,\Id,\Id_{\cM})\colon (h\colon T' \to T,\cM) \to (h\colon T' \to T,\cM)$ and its associated fiber product diagram $D$. Our assumption \ref{assumption_1} implies that the isomorphism $\phi_D$ constructed in Lemma \ref{lem_phi_isomorphism} is the identity. Therefore, $\fF(\Id_{\cM}) = \Id_{\cF_{T'/T}(\cM)}$ as required.

Assume we have a composition of morphisms
\[
(h_3\colon T'_3 \to T_3,\cM_3) \overset{(f_2,g_2,\varphi_2)}{\longrightarrow} (h_2\colon T'_2 \to T_2,\cM_2) \overset{(f_1,g_1,\varphi_1)}{\longrightarrow} (h_1 \colon T'_1 \to T_1,\cM_1).
\]
For $U\in \Aff_{T_3}$, let $T'_i \times_{T_i} U = U'_i$. Then, we have fiber product diagrams
\[
D = \begin{tikzcd}[column sep=10ex]
U'_3 \ar[d,"h_{3,U}"] \ar[r,"g_{2,U}"] \ar[dr,phantom,"D_2"] & U'_2 \ar[d,"h_{2,U}"] \ar[r,"g_{1,U}"] \ar[dr,phantom,"D_1"] & U'_1 \ar[d,"h_{1,U}"] \\
U \ar[r,equals] & U \ar[r,equals] & U.
\end{tikzcd}
\]
The commutativity of the following diagram
\[\small
\begin{tikzcd}[column sep=5ex]
 & & \cF_{h_{1,U}}(\cM_1(U'_1)) \\
 & & \cF_{h_{1,U}}(\cM_1(U'_1))\otimes_U U \otimes_U U \ar[u,swap,"\mathrm{can}"] \ar[d,"\theta_{D_1}\otimes\Id"] \\
 & & \cF_{h_{2,U}}(\cM_1(U'_1)\otimes_{U'_1} U'_2)\otimes_U U \ar[d,"*"] \\
 & \cF_{h_{2,U}}(\cM_2(U'_2))\otimes_U U \ar[r,"\cF_{h_{2,U}}(\varphi_1(U'_2))\otimes\Id" yshift=1ex] \ar[d,"\theta_{D_2}"] & \cF_{h_{2,U}}(\cM_1(U'_2))\otimes_U U \ar[d,"\theta_{D_2}"] \\
 & \cF_{h_{3,U}}(\cM_2(U'_2)\otimes_{U'_2} U'_3) \ar[r,"\cF_{h_{3,U}}(\varphi_1(U'_2)\otimes\Id)" yshift=1ex] \ar[d,"*"] & \cF_{h_{3,U}}(\cM_1(U'_2)\otimes_{U'_2} U'_3) \ar[d,"*"] \\
\cF_{h_{3,U}}(\cM_3(U'_3)) \ar[r,"\cF_{h_{3,U}}(\varphi_2(U'_3))"] \ar[uuuuurr,bend left,"(\fF(\varphi_1)\circ \fF(\varphi_2))(U)"] \ar[rr,swap,bend right,"\cF_{h_{3,U}}((g_2^*(\varphi_1)\circ\varphi_2)(U'_3))"] & \cF_{h_{3,U}}(\cM_2(U'_3)) \ar[r,"\cF_{h_{3,U}}(\varphi_1(U'_3))"] & \cF_{h_{3,U}}(\cM_1(U'_3))
\end{tikzcd}
\]
implies the commutativity of face $A$ in the diagram below
\[\small
\begin{tikzcd}
 & \cF_{h_{1,U}}(\cM_1(U'_1)) &[-3ex] \\
 & \cF_{h_{1,U}}(\cM_1(U'_1))\otimes_U U \otimes_U U \ar[u,swap,"\mathrm{can}"] \ar[r,"\mathrm{can}"] \ar[d,"\theta_{D_1}\otimes\Id"] & \cF_{h_{1,U}}(\cM_1(U'_1))\otimes_U U \ar[dd,"\theta_D"] \ar[dd,phantom,bend right=70,"B"] \ar[ul,swap,"\mathrm{can}"] \\
 & \cF_{h_{2,U}}(\cM_1(U'_1)\otimes_{U'_1}U'_2)\otimes_U U \ar[dl,bend right,start anchor=west,end anchor=north east,"*"] \ar[d,"\theta_{D_2}"] & \\
\cF_{h_{2,U}}(\cM_1(U'_2))\otimes_U U \hspace{-5ex} \ar[uuur,phantom,bend left=70,near start,"A"] \ar[dr,bend right,start anchor=south east,end anchor=west,"\theta_{D_2}"] & \cF_{h_{3,U}}(\cM_1(U'_1)\otimes_{U'_1}U'_2 \otimes_{U'_2} U'_3) \ar[r,"\cF_{h_{3,U}}(\mathrm{can})"] \ar[d,"*"] & \cF_{h_{3,U}}(\cM(U'_1)\otimes_{U'_1} U'_3) \ar[ddl,swap,"*"] \\
 & \cF_{h_{3,U}}(\cM_1(U'_2)\otimes_{U'_2} U'_3) \ar[d,"*"] & \\
\cF_{h_{3,U}}(\cM_3(U'_3)) \ar[r,"\cF_{h_{3,U}}((g_2^*(\varphi_1)\circ \varphi_2)(U'_3))"] \ar[uuuuur,rounded corners,to path={--([xshift=-2ex]\tikztostart.west)--([xshift=-2ex]\tikztostart.west|-\tikztotarget.west)--(\tikztotarget)\tikztonodes} ,"(\fF(\varphi_1)\circ \fF(\varphi_2))(U)"] \ar[uuuuur,rounded corners,to path={--([yshift=-1ex]\tikztostart.south)-|([xshift=66ex]\tikztostart.east)[pos=0.4]\tikztonodes|-(\tikztotarget.east)},"\fF(g_2^*(\varphi_1)\circ\varphi_2)(U)"]& \cF_{h_{3,U}}(\cM_1(U'_3)) & & 
\end{tikzcd}
\]
and face $B$ commutes by assumption \ref{assumption_2}. This diagram implies that $(\fF(\varphi_1)\circ \fF(\varphi_2))(U)=\fF(g_2^*(\varphi_1)\circ\varphi_2)(U)$ for all $U\in \Aff_{T_3}$ and therefore $\fF(\varphi_1)\circ \fF(\varphi_2)=\fF(g_2^*(\varphi_1)\circ\varphi_2)$ in general. This shows that $\fF$ respects composition and hence is a well-defined functor.

Finally, we address cartesian morphisms. Recall that by construction, all morphisms in $\QCoh_{\fI}$ are cartesian. Since $\QCoh$ is a substack of the stack $\fSh$ of Example \ref{ex_stack_sheaves}, a morphism $(f,\psi) \in \QCoh$ is cartesian if and only if $\psi$ is an isomorphism. Now, consider a morphism $(f,g,\varphi)\colon (h'\colon X' \to X,\cN)\to(h\colon T' \to T,\cM)$ in $\QCoh_\fI$. The map $\varphi$ is an isomorphism by definition and the map $\phi_D$ constructed in Lemma \ref{lem_phi_isomorphism} is an isomorphism as well. Therefore, $\fF(\varphi)=\phi_D(\cM)\circ \cF_{X'/X}(\varphi)$ is an isomorphism and thus $\fF(f,g,\varphi)=(f,\fF(\varphi))$ is a cartesian morphism. This concludes the proof.
\end{proof}


\begin{thebibliography}{SGA34}
\bibitem[Au]{A} A.~Auel, \emph{Surjectivity of the Total Clifford Invariant and Brauer Dimension}, J. Algebra 443, 395-421 (2015).

\bibitem[Al]{Alper} J.~Alper, \emph{Stacks and Moduli}, \url{https://sites.math.washington.edu/~jarod/moduli.pdf}, August 19, 2024 version.

\bibitem[Ba]{Bachmann} T.~Bachmann, {\em Some remarks on units in Grothendieck-Witt rings}, J.~Algebra \textbf{499} (2018), 229-271.

\bibitem[B:A2]{B:A2} N.~Bourbaki, \emph{\'El\'ements de math\'ematique. Alg\`ebre. Chapitres 4 \`a 7.} Paris etc.: Masson. VII, 422 p. (1981).

\bibitem[BLR]{BLR} S.~Bosch, W.~L\"utkebohmert, M.~Raynaud, {\em N\'eron models}, Ergebnisse der Mathematik und ihrer Grenzgebiete, Band \textbf{21}, Springer-Verlag 1990.

\bibitem[CF]{CF} B. Calm\`es, J. Fasel, {\it Groupes classiques},
     Autour des sch\'emas en groupes, vol II, Panoramas et Synth\`eses {\bf 46} (2015),
      1-133.

\bibitem[Co1]{Con1} B.~Conrad, {\em Reductive group schemes\/}, in {\em Autour
    des     sch\'emas en groupes, vol. I}, Panoramas et Synth\`eses \textbf{42-43},
    Soc. Math. France  2014.

\bibitem[Co2]{Con2} B.~Conrad, {\it Non-split reductive groups over
    $\mathbb{Z}$}, Autour des sch\'emas en groupes, vol II,
    Panoramas et Synth\`eses {\bf 46} (2015), 193-253.

\bibitem[CTS]{CTS} J.-L.~Colliot-Th\'el\`ene, A.~N.~Skorobogatov, \emph{The Brauer-Grothendieck Group}, Cham: Springer, (2021).

\bibitem[De]{Deligne} P.~Deligne, {\em Cohomologie \`a supports propres}, expos\'e XVII of SGA 4, {\em Th\'eorie des topos et cohomologie \'etale des sch\'emas}. Tome 3, Springer-Verlag, Berlin, 1973,
   Lecture Notes in Math., Vol. \textbf{305}.

\bibitem[EGA]{EGA} A.~Grothendieck (avec la collaboration de J. Dieudonn\'e),
   {\it El\'ements de G\'eom\'etrie Alg\'ebrique}, Publications
    math\'ematiques de l'I.H.\'E.S. no.~4, 8, 11, 17, 20, 24, 28, 32, 1960--1967.

\bibitem[Fe]{F} D.~Ferrand, \textit{Un foncteur norme}, Bulletin de la S.M.F. 126, No. 1 (1998), 1--49.

\bibitem[Fo]{Ford} T.~J. Ford, {\it Separable Algebras}, Graduate
    Studies in Mathematics {\bf 183}, Amer.~Math.~Soc, Provident, RI, (2017).

\bibitem[Gil]{Gi} P.~Gille, {\it Sur la classification des sch\'emas en groupes semi-simples}, Autour des sch\'emas en groupes, vol III, Panoramas et Synth\`eses {\bf 47} (2015), 39-110.
 
\bibitem[Gir]{Gir} J.~Giraud, {\it Cohomologie non ab\'elienne}, Die Grundlehren der Mathematischen Wissenschaften, Band 179, Berlin-Heidelberg-New York: Springer-Verlag, 1971.

\bibitem[GNR]{GNR} P.~Gille, E.~Neher, C.~Ruether, \emph{Azumaya Algebras and Obstructions to Quadratic Pairs over a Scheme}, \url{https://arxiv.org/abs/2209.07107}, to appear in Trans. Amer. Math. Soc.
 
\bibitem[Gr]{Gro-Brau} A.~Grothendieck, \emph{Le Groupe de Brauer I}, S\'em. Bourbaki, 1964/65, no 290.

\bibitem[GW]{GW} U.~G\"ortz and T.~Wedhorn, {\em Algebraic Geometry I}, second edition, Springer Fachmedien Wiesbaden, 2020.
    
\bibitem[Kn]{K} M.-A.~Knus, {\it Quadratic and Hermitian Forms over
    Rings}, Grundlehren der mathematischen Wissenschaften {\bf 294}
    (1991), Springer.
    
\bibitem[KMRT]{KMRT} M.-A.~Knus, A.~Merkurjev, M.~Rost and J.-P.~Tignol,
    {\it   The   Book of Involutions}, American Mathematical Society Colloquium Publications \textbf{44}, American Mathematical Society, Providence, RI (1998).

 \bibitem[KO]{KO75} M.-A.~Knus, M.~Ojanguren, {\it 
A norm for modules and algebras},
Math. Z. {\bf 142} (1975), 33-45.

\bibitem[Kr]{Kr} Simon Krsnik, {\em Der Multiplikative Transfer auf dem Grothendieck--Witt Ring}, Diplomarbeit, Universit\"at Bielefeld, 2006. 

\bibitem[Lu]{Lu} C.~Lundkvist, \emph{Counterexamples regarding symmetric tensors and divided powers}, J. Pure and Appl. Algebra 212, 2236--2249 (2008).
 
\bibitem[Mi]{M}  J.-S.~Milne, {\'Etale Cohomology}, Princeton University Press, 1980.

\bibitem[MS]{MS} I.~Martino, F.~Scavia, \emph{Motivic classes of classifying stacks of some semi-direct products}, J. Algebra 544, 62--74 (2020).

\bibitem[Ol]{Olsson} M.~Olsson, \emph{Algebraic Spaces and Stacks}, Providence, RI: American Mathematical Society (AMS) (2016).

\bibitem[Ri]{Riehm} C.~Riehm, \emph{The Corestriction of Algebraic Structures}, Invent. Math. 11, 73-98 (1970).

\bibitem[Rob]{Roby} N.~Roby, \emph{Lois Polyn\^omes et Lois Formelles en Th\'eorie des Modules}, Ann. Sci. \'Ec. Norm. Sup\'er., III. S\'er. 80, 213-348 (1963).

\bibitem[Ros]{Rost} M.~Rost, \emph{The Multiplicative Transfer for the Grothendieck-Witt Ring}, preprint (2003).

\bibitem[Ru]{Rue23} C.~Ruether, \emph{The Canonical Quadratic Pair on Clifford Algebras over Schemes}, Preprint: \url{https://arxiv.org/abs/2309.03077} [v1], (October 2023). To appear in the Israel Journal of Mathematics.
 
\bibitem[Ry]{Rydh} D.~Rydh, \emph{Family of cycles and the Chow scheme,} Ph.D. thesis, May 2008, KTH, Stockholm. Retrieved from author's homepage \url{https://people.kth.se/~dary/papers.html}.

\bibitem[RyII]{Rydh2} D.~Rydh, \emph{Families of zero-cycles and divided powers: II. The universal family,} part II of \cite{Rydh}.
 
\bibitem[SGA3]{SGA3} {\it S\'eminaire de G\'eom\'etrie alg\'ebrique de
    l'I.H.E.S., 1963-1964, sch\'emas en groupes, dirig\'e par M. Demazure et A.
    Grothendieck},  Lecture Notes in Math. 151-153. Springer (1970).

\bibitem[SGA4]{SGA4}   M.~Artin, A.~Grothendieck, and J.-L.~ Verdier, {\it Th\'eorie de topos et
cohomologie \'etale des sch\'emas I, II, III},
Lecture Notes in Mathematics, vol. 269, 270,
305, Springer (1971).

\bibitem[St]{Stacks}  The Stacks Project Authors, {\em Stacks project},
\url{http://stacks.math.columbia.edu/}

\bibitem[Ti]{Ti} J.-P.~Tignol,
    {\it  On the corestriction of central simple algebras},
     Math. Z. {\bf 194} (1987), 267-274.
     
\bibitem[Vi]{Vis} A.~Vistoli, \emph{Notes on Grothendieck Topologies, Fibered Categories and Descent Theory}, version 4, \url{https://arxiv.org/abs/math/0412512} (2008).

\end{thebibliography}
\end{document}